\documentclass[a4paper, 12pt]{article}
\usepackage[utf8]{inputenc}
\usepackage{pdfpages}
\usepackage{authblk}

\usepackage[T1]{fontenc}
\usepackage[a4paper, margin=2.5cm]{geometry}
\usepackage{needspace}

\usepackage{libertine} 
\usepackage{inconsolata} 

\usepackage{amsthm}
\usepackage{amssymb}
\usepackage{amsmath}
\usepackage{mathtools}
\usepackage{hyperref}
\usepackage{xcolor}

\usepackage{stmaryrd}

\usepackage{caption}
\usepackage{ifthen}

\usepackage{comment}

\usepackage{subcaption}

\usepackage{floatrow}
\usepackage{enumitem}
\setlist{nosep}

\newcommand{\redirXOne}{$\C_{X1}$}
\newcommand{\redirXTwo}{$\C_{X2}$}
\newcommand{\redirXThree}{$\C_{X3}$}
\newcommand{\redirXFour}{$\C_{X4}$}

\newcommand{\CV}{$\C_V$}
\newcommand{\CN}{$\C_N$}
\newcommand{\CDA}{$\C_{Da}$}
\newcommand{\CDB}{$\C_{Db}$}
\newcommand{\CVp}{$\C_{V}'$}
\newcommand{\CU}{$\C_{U}$}
\newcommand{\CTTA}{$\C_{T2A}$}
\newcommand{\CTTNA}{$\C_{T2NA}$}
\newcommand{\CTOne}{$\C_{T1}$}
\newcommand{\CDOne}{$\C_{D1}$}
\newcommand{\CDTwo}{$\C_{D2}$}

\newcommand{\nameDOne}{$D_1$}
\newcommand{\nameDTwo}{$D_2$}
\newcommand{\nameDThree}{$D_3$}
\newcommand{\nameDFour}{$D_4$}

\newcommand{\nameJOne}{$J_1$}
\newcommand{\nameJTwo}{$J_2$}
\newcommand{\nameJThree}{$J_3$}
\newcommand{\nameJFour}{$J_4$}
\newcommand{\nameJFive}{$J_5$}
\newcommand{\nameJSix}{$J_6$}

\newcommand{\nameRuOne}{$R_1$}
\newcommand{\nameRuTwo}{$R_2$}
\newcommand{\nameRuThree}{$R_3$}
\newcommand{\nameRuFour}{$R_4$}
\newcommand{\nameRuFive}{$R_5$}
\newcommand{\nameRuSix}{$R_6$}
\newcommand{\nameRuSeven}{$R_7$}
\newcommand{\nameRuEight}{$R_8$}
\newcommand{\nameRuNine}{$R_9$}

\newcommand{\patcol}{\textbf{Color requirements:}\ }
\newcommand{\patred}{\textbf{Reduction:}\ }
\newcommand{\patrec}{\textbf{Recoloring:}\ }

\newcommand{\nameXXCOa}{$(a)$}
\newcommand{\nameXXa}{$(b)$}
\newcommand{\nameXXb}{$(33o)$}
\newcommand{\nameXXbb}{$(d)$} 
\newcommand{\nameXXc}{$(e)$} 
\newcommand{\nameXXd}{$(c)$}
\newcommand{\nameXXe}{$(f)$}
\newcommand{\nameXXf}{$(g)$}
\newcommand{\nameXXg}{$(h)$}
\newcommand{\nameXXgOne}{$(h_1)$}
\newcommand{\nameXXgTwo}{$(h_2)$}
\newcommand{\nameXXgThree}{$(h_3)$}
\newcommand{\nameXXgFour}{$(h_4)$}
\newcommand{\nameXXgFive}{$(h_5)$}
\newcommand{\nameXXgSix}{$(h_6)$}
\newcommand{\nameXXh}{$(i)$}
\newcommand{\nameXXi}{$(j)$}
\newcommand{\nameXXj}{$(k)$}
\newcommand{\nameXXk}{$(l)$}
\newcommand{\nameXXl}{$(m_1)$}
\newcommand{\nameXXlp}{$(m_2)$}
\newcommand{\nameXXm}{$(n)$}
\newcommand{\nameXXn}{$(o)$}
\newcommand{\nameXXo}{$(p)$}
\newcommand{\nameXXp}{$(q)$}
\newcommand{\nameXXq}{$(r)$}
\newcommand{\nameXXr}{$(s)$}
\newcommand{\nameXXs}{$(t)$}
\newcommand{\nameXXt}{$(u)$}

\usepackage{tikz}%
\usetikzlibrary{arrows,shapes,decorations,automata,backgrounds,petri,calc,snakes}%
\usetikzlibrary{decorations.pathmorphing}%
\usetikzlibrary{positioning}%
\usetikzlibrary{trees}%

\newcommand\myEdge[2]{\draw[thick] (#1) -- (#2);} 
\newcommand\myDEdge[2]{\draw[dashed,thick] (#1) -- (#2);} 

\newcommand\myCEdge[4]{%
	\draw[#3,thick] (#1) -- (#2);%
	\ifthenelse{\equal{#4}{}}{}{\draw ($(#1)!0.5!(#2)$) node[#3,nodelabel] {#4};}%
}%

\newcommand\myCurvEdge[5]{\draw[#3,thick] (#1) edge[#5] (#2);
}%

\newcommand\myHalfEdge[4]{%
	\draw[#3,thick,endpath,#3] ($(#1) + (#2:0.75)$) -- (#1);%
	\ifthenelse{\equal{#4}{}}{}{\draw ($(#1) + (#2:0.4)$) node[#3,fill=white,draw=white, font=\scriptsize,inner sep=0pt,minimum size=0pt] {#4};}%
	}%

\newcommand\myHalfLongEdge[4]{\draw[#3,thick,endpath] ($(#1) + (#2:1.5)$) -- (#1);\draw ($(#1) + (#2:0.75)$) node[#3,fill=white,draw=white, font=\scriptsize,inner sep=0pt,minimum size=0pt] {#4};}%

\newcommand\mySubEdge[4]{
	\draw[sedge,sedgemixed]  ($(#1) + (#2:0.85)$) 
	 to node[above=-2pt,\myPurple,nodelabel,midway] {$S$} (#1);%
\ifthenelse{\equal{#4}{}}{}{\draw ($(#1) + (#2:0.4)$) node[#3,nodelabel] {#4};}%
\draw[sedgemixed,thick] ($(#1) + (#2+30:0.85)$)  to node[above=-2pt,\myPurple,nodelabel,near start] {$S$} (#1);%
\draw[sedgemixed,thick] ($(#1) + (#2-30:0.85)$)  to node[above=-2pt,\myPurple,nodelabel,near start] {$S$} (#1);%
}%
\newcommand\mySubRedEdge[4]{
\draw[blue,thick] ($(#1) + (#2+30:0.85)$)  to node[above=-2pt,blue,nodelabel,near start] {$P_2$} (#1);%
\draw[blue,thick] ($(#1) + (#2-30:0.85)$)   to node[above=-2pt,blue,nodelabel,near start] {$P_2$} (#1);%
\draw[sedge,red]  ($(#1) + (#2:0.85)$) 
to node[above=-2pt,red,nodelabel,midway] {$P_1$} (#1);%

}%
\newcommand\mySubBlueEdge[4]{
	\draw[red,thick] ($(#1) + (#2+30:0.85)$)  to node[above=-2pt,red,nodelabel,near start] {$P_1$} (#1);%
	\draw[red,thick] ($(#1) + (#2-30:0.85)$)   to node[above=-2pt,red,nodelabel,near start] {$P_1$} (#1);%
\draw[sedge,blue]  ($(#1) + (#2:0.85)$) 
to node[above=-2pt,blue,nodelabel,midway] {$P_2$} (#1);%
	
}%

\newcommand\Mynodes{} 
\newcommand\Myspecialnodes{} 
\newcommand\MyEdgesBefore{} 
\newcommand\CfN{}%
\newcommand\CfL{}%
\newcommand\MyEdgesRemoved{} 
\newcommand\MyEdgesAfter{} 
\newcommand\MyEdgesRemovedBis{} 
\newcommand\MyEdgesAfterBis{} 
\newcommand\MyEdgesRemovedTer{} 
\newcommand\MyEdgesAfterTer{} 
\newcommand\MyEdgesRemovedQua{}
\newcommand\MyEdgesAfterQua{}
\newcommand\MyEdgesRemovedQui{}
\newcommand\MyEdgesAfterQui{}
\newcommand\MyEdgesRemovedSix{}
\newcommand\MyEdgesAfterSix{}

\newcommand\myPurple{red!50!blue}%

\usetikzlibrary{patterns}
\tikzset{%
bicolor/.style 2 args={draw=#1!50!#2},
}%
\makeatletter
\tikzset{hatch distance/.store in=\hatchdistance,hatch distance=5pt,hatch thickness/.store in=\hatchthickness,hatch thickness=5pt}

\pgfdeclarepatternformonly[\hatchdistance,\hatchthickness]{vertical hatch}
{\pgfqpoint{-1pt}{-1pt}}
{\pgfqpoint{\hatchdistance}{\hatchdistance}}
{\pgfpoint{\hatchdistance-1pt}{\hatchdistance-1pt}}%
{
	\pgfsetcolor{\tikz@pattern@color}
	\pgfsetlinewidth{\hatchthickness}
	\pgfpathmoveto{\pgfqpoint{0pt}{0pt}}
	\pgfpathlineto{\pgfqpoint{0pt}{\hatchdistance}}
	\pgfusepath{stroke}
}
\pgfdeclarepatternformonly[\hatchdistance,\hatchthickness]{horizontal hatch}
{\pgfqpoint{-1pt}{-1pt}}
{\pgfqpoint{\hatchdistance}{\hatchdistance}}
{\pgfpoint{\hatchdistance-1pt}{\hatchdistance-1pt}}%
{
	\pgfsetcolor{\tikz@pattern@color}
	\pgfsetlinewidth{\hatchthickness}
	\pgfpathmoveto{\pgfqpoint{0pt}{0pt}}
	\pgfpathlineto{\pgfqpoint{\hatchdistance}{0pt}}
	\pgfusepath{stroke}
}

\pgfdeclarepatternformonly[\hatchdistance,\hatchthickness]{grid hatch}
{\pgfqpoint{-1pt}{-1pt}}
{\pgfqpoint{\hatchdistance}{\hatchdistance}}
{\pgfpoint{\hatchdistance-1pt}{\hatchdistance-1pt}}%
{
	\pgfsetcolor{\tikz@pattern@color}
	\pgfsetlinewidth{\hatchthickness}
	\pgfpathmoveto{\pgfqpoint{0pt}{0pt}}
	\pgfpathlineto{\pgfqpoint{0pt}{\hatchdistance}}
	\pgfusepath{stroke}
	\pgfpathmoveto{\pgfqpoint{0pt}{0pt}}
	\pgfpathlineto{\pgfqpoint{\hatchdistance}{0pt}}
	\pgfusepath{stroke}
}

\makeatother

\tikzset{snake it/.style={decorate, decoration=snake}}%
\tikzstyle{mynode}=[draw=black,minimum size=8pt,inner sep=0pt,circle]
\tikzstyle{whitenode}=[mynode,circle,fill=white]
\tikzstyle{rednode}=[mynode, fill=red, pattern=horizontal hatch,hatch distance=3pt,hatch thickness=1pt,pattern color=red]
\tikzstyle{bluenode}=[mynode, pattern=vertical hatch,hatch distance=3pt,hatch thickness=1pt,pattern color=blue]
\tikzstyle{purplenode}=[mynode, fill=\myPurple, pattern=grid hatch,hatch distance=3pt,hatch thickness=1pt,pattern color=\myPurple]
\tikzstyle{blacknode}=[mynode,circle,fill=black]
\tikzstyle{nodelabel}=[fill=white,draw=white, font=\scriptsize,inner sep=0pt,minimum size=0pt]

\tikzstyle{contactnode}=[whitenode]

\tikzstyle{graynode}=[mynode,circle,gray,fill=gray]
\tikzstyle{blacksquarenode}=[mynode,rectangle,fill=black]
\tikzstyle{oddnode}=[mynode,diamond]
\tikzstyle{evennode}=[mynode,rectangle]
\tikzstyle{ghost}=[mynode,draw=white,circle]

\tikzstyle{path}=[decorate,decoration=snake, thick]

\tikzstyle{sedgemixed}=[thick,\myPurple,-]
\tikzstyle{spathmixed}=[thick,\myPurple,path]
\tikzstyle{spath}=[->>,path,red,thick]  
\tikzstyle{sedge}=[->>,thick,red] 
\tikzstyle{endpath}=[->] 
\tikzstyle{optedge}=[dashdotted]  
\newcommand\MyScale{0.8}%

\newcommand\tkSymbol[1]{%
    \begin{tikzpicture}[scale=\MyScale, every node/.style={scale=\MyScale}, auto]%
        \ifthenelse{\isundefined{\NoFigures}}{%
                #1%
        }{}%
    \end{tikzpicture}%
}

\newcommand\tkcf[1]{
    #1%
    \begin{tikzpicture}[scale=\MyScale, every node/.style={scale=\MyScale}, auto]%
        \ifthenelse{\isundefined{\NoFigures}}{%
                \Mynodes
                \Myspecialnodes
                \MyEdgesBefore
        }{}%
                \node (u000) at (0,-0.6) {\CfN};%
    \end{tikzpicture}%
}%

\newcommand\ncf[1]{#1{\CfN}}%

\newcommand\Ncf[1]{\ncf{#1}}%
\newcommand\NcfP[2]{\ncf{#1} \ensuremath{\oplus} \ncf{#2}}%

\newcommand\TikzRedefOdd{
    \tikzstyle{rednode}=[fill=white,draw=black,mynode]
    \tikzstyle{bluenode}=[fill=white,draw=black,mynode]
    \tikzstyle{purplenode}=[fill=white,draw=black,mynode]
    \tikzstyle{evennode}=[mynode,diamond]
}%
\newcommand\tikzPartRule[1]{%
\begin{tikzpicture}[scale=\MyScale, every node/.style={scale=\MyScale},auto]%
    #1%
	\node (u000) at (0,-0.78) {}; 
	\node (u001) at (1.2,0) {};
\end{tikzpicture}%
}%

\newcommand\OrdonneeFleche{1}

\newcommand\drawRule[1]{%
        \tkcf{} %
        \petiteFlecheACentrer{\OrdonneeFleche}%
        \tikzstyle{endpath}=[->] %
        \tikzPartRule{\TikzRedefOdd \Mynodes \MyEdgesRemoved} %
        \petiteFlecheACentrer{\OrdonneeFleche}%
        \tikzstyle{endpath}=[-] 
        \tikzPartRule{\Myspecialnodes \Mynodes \MyEdgesAfter}\\
}%

\newcommand\drawDoubleRule[2]{ 
        \tkcf{}%
        \begin{minipage}{0.60\linewidth}%
            \petiteFlecheACentrer{\OrdonneeFleche}%
            \tikzstyle{endpath}=[->] %
            \tikzPartRule{\TikzRedefOdd \Mynodes \MyEdgesRemoved}%
            \petiteFlecheACentrer{\OrdonneeFleche}%
            \tikzstyle{endpath}=[] %
            \tikzPartRule{\Myspecialnodes \Mynodes \MyEdgesAfter}%

            \petiteFlecheACentrer{\OrdonneeFleche}%
            \tikzstyle{endpath}=[->] %
            \tikzPartRule{\Mynodes \MyEdgesRemovedBis}%
            \petiteFlecheACentrer{\OrdonneeFleche}%
            \tikzstyle{endpath}=[] %
            \tikzPartRule{\Myspecialnodes \Mynodes \MyEdgesAfterBis}%
        \end{minipage}\\
}%

\newcommand\drawTripleRule[2]{ 
        \tkcf{}%
        \begin{minipage}{0.60\linewidth}%
            $\leadsto$ \tikzstyle{endpath}=[->]
            \tikzPartRule{\TikzRedefOdd \Mynodes \MyEdgesRemoved}%
            $\leadsto$ \tikzstyle{endpath}=[-]
            \tikzPartRule{\Myspecialnodes \Mynodes \MyEdgesAfter}%

            $\leadsto$ \tikzstyle{endpath}=[->]
            \tikzPartRule{\Mynodes \MyEdgesRemovedBis}%
            $\leadsto$ \tikzstyle{endpath}=[-]
            \tikzPartRule{\Myspecialnodes \Mynodes \MyEdgesAfterBis}%

            $\leadsto$ \tikzstyle{endpath}=[->]
            \tikzPartRule{\Mynodes \MyEdgesRemovedTer}%
            $\leadsto$ \tikzstyle{endpath}=[-]
            \tikzPartRule{\Myspecialnodes \Mynodes \MyEdgesAfterTer}%
        \end{minipage}\\
}%

\newcommand\drawSixtRule[2]{ 
        \tkcf{}%
        \begin{minipage}{0.60\linewidth}%
            \petiteFlecheACentrer{\OrdonneeFleche}%
            \tikzstyle{endpath}=[->]%
            \tikzPartRule{\TikzRedefOdd \Mynodes \MyEdgesRemoved}%
            \petiteFlecheACentrer{\OrdonneeFleche}%
            \tikzstyle{endpath}=[-]%
            \tikzPartRule{\Myspecialnodes \Mynodes \MyEdgesAfter}%

            \petiteFlecheACentrer{\OrdonneeFleche}%
            \tikzstyle{endpath}=[->]%
            \tikzPartRule{\Mynodes \MyEdgesRemovedBis}%
            \petiteFlecheACentrer{\OrdonneeFleche}%
            \tikzstyle{endpath}=[-]%
            \tikzPartRule{\Myspecialnodes \Mynodes \MyEdgesAfterBis}%

            \petiteFlecheACentrer{\OrdonneeFleche}%
            \tikzstyle{endpath}=[->]%
            \tikzPartRule{\Mynodes \MyEdgesRemovedTer}%
            \petiteFlecheACentrer{\OrdonneeFleche}%
            \tikzstyle{endpath}=[-]%
            \tikzPartRule{\Myspecialnodes \Mynodes \MyEdgesAfterTer}%

            \petiteFlecheACentrer{\OrdonneeFleche}%
            \tikzstyle{endpath}=[->]%
            \tikzPartRule{\Mynodes \MyEdgesRemovedQua}%
            \petiteFlecheACentrer{\OrdonneeFleche}%
            \tikzstyle{endpath}=[-]%
            \tikzPartRule{\Myspecialnodes \Mynodes \MyEdgesAfterQua}%

            \petiteFlecheACentrer{\OrdonneeFleche}%
            \tikzstyle{endpath}=[->]%
            \tikzPartRule{\Mynodes \MyEdgesRemovedQui}%
            \petiteFlecheACentrer{\OrdonneeFleche}%
            \tikzstyle{endpath}=[-]%
            \tikzPartRule{\Myspecialnodes \Mynodes \MyEdgesAfterQui}%

            \petiteFlecheACentrer{\OrdonneeFleche}%
            \tikzstyle{endpath}=[->]%
            \tikzPartRule{\Mynodes \MyEdgesRemovedSix}%
            \petiteFlecheACentrer{\OrdonneeFleche}%
            \tikzstyle{endpath}=[-]%
            \tikzPartRule{\Myspecialnodes \Mynodes \MyEdgesAfterSix}%
        \end{minipage}\\
}%

\newcommand\drawRuleSansConfig{%
		\begin{tikzpicture}[scale=\MyScale, every node/.style={scale=\MyScale}, auto]
		\node (name) at (0,1) {\CfN};
		\node (q) at (0,0) {};
		\end{tikzpicture}%
        \petiteFlecheACentrer{\OrdonneeFleche}%
        \tikzstyle{endpath}=[->] %
        \tikzPartRule{\TikzRedefOdd \Mynodes \MyEdgesRemoved}%
        \petiteFlecheACentrer{\OrdonneeFleche}%
        \tikzstyle{endpath}=[-] 
        \tikzPartRule{\Myspecialnodes \Mynodes \MyEdgesAfter}\\
}%

\newcommand{\labelUOne}{$u_1$}
\newcommand{\labelUTwo}{$u_2$}
\newcommand{\labelV}{$v$}
\newcommand{\labelVP}{$v'$}
\newcommand{\labelVPP}{$v''$}
\newcommand{\labelVOne}{$v_1$}
\newcommand{\labelVTwo}{$v_2$}
\newcommand{\labelVThree}{$v_3$}
\newcommand{\labelVFour}{$v_4$}

\newcommand{\resetLabels}{
	\renewcommand{\labelUOne}{$u_1$}
	\renewcommand{\labelUTwo}{$u_2$}
	\renewcommand{\labelV}{$v$}
	\renewcommand{\labelVP}{$v'$}
	\renewcommand{\labelVPP}{$v''$}
	\renewcommand{\labelVOne}{$v_1$}
	\renewcommand{\labelVTwo}{$v_2$}
	\renewcommand{\labelVThree}{$v_3$}
	\renewcommand{\labelVFour}{$v_4$}
}
\resetLabels

\newcommand{\cfCFmFm}{%
    \renewcommand\Mynodes{}%
    \renewcommand\Myspecialnodes{%
        \node[blacknode,label=below:$u_1$] (u1) at (0,1) {};
        \node[blacknode,label=below:$u_2$] (u2) at (1.2,1) {};
    }%
    \renewcommand\MyEdgesBefore{%
        \myHalfEdge{u2}{0}{optedge, -}{}%
        \myHalfEdge{u2}{45}{optedge, -}{}%
        \myHalfEdge{u2}{-45}{optedge, -}{}%
        \myCEdge{u1}{u2}{path}{}
        \myHalfEdge{u1}{180}{optedge, -}{}%
        \myHalfEdge{u1}{135}{optedge, -}{}%
        \myHalfEdge{u1}{225}{optedge, -}{}%
    }%
    \renewcommand{\CfN}{$4^- \oplus 4^-$}
    \renewcommand{\CfL}{fig:C4m4m}%
}%

\newcommand{\cfCZC}{%
    \renewcommand\Mynodes{%
        \node[whitenode,label=above:$v_1$] (v1) at (-1,2) {};%
        \node[whitenode,label=above:$v_2$] (v2) at (-1,1) {};%
        \node[whitenode,label=above:$v_3$] (v3) at (-1,0) {};%
        \node[whitenode,label=above:$v_4$] (v4) at (2,2) {};%
        \node[whitenode,label=above:$v_5$] (v5) at (2,1) {};%
        \node[whitenode,label=above:$v_6$] (v6) at (2,0) {};%
    }%
    \renewcommand\Myspecialnodes{%
        \node[blacknode,label=below:$u_1$] (u1) at (0,1) {};%
        \node[blacknode,label=below:$u_2$] (u2) at (1,1) {};%
    }%
    \renewcommand\MyEdgesBefore{%
        \myCEdge{u1}{v1}{optedge}{}%
        \myCEdge{u1}{v2}{optedge}{}%
        \myCEdge{u1}{v3}{optedge}{}%
        \myCEdge{u1}{u2}{path}{}%
        \myCEdge{u2}{v4}{optedge}{}%
        \myCEdge{u2}{v5}{optedge}{}%
        \myCEdge{u2}{v6}{optedge}{}%
    }%
    \renewcommand{\CfN}{$(0C)$}%
    \renewcommand{\CfL}{fig:CZC}%
}%

\newcommand{\cfCOC}{%
    \renewcommand\Mynodes{%
        \node[whitenode,label=above:$v_2$] (v2) at (-1,0) {};%
        \node[whitenode,label=above:$v_1$] (v1) at (-1,1) {};%
        \node[whitenode,label=above:$v$] (v) at (0.5,2) {};%
        \node[whitenode,label=above:$v_3$] (v3) at (2,0) {};%
        \node[whitenode,label=above:$v_4$] (v4) at (2,1) {};%
        \node[graynode] (uvm) at (0.5,1) {};%
    }%
    \renewcommand\Myspecialnodes{%
        \node[blacknode,label=below:$u_1$] (u1) at (0,1) {};%
        \node[blacknode,label=below:$u_2$] (u2) at (1,1) {};%
    }%
    \renewcommand\MyEdgesBefore{%
        \myCEdge{u1}{v}{}{}%
        \myCEdge{u2}{v}{}{}%
        \myCEdge{u1}{v1}{optedge}{}%
        \myCEdge{u1}{v2}{optedge}{}%
        \myCEdge{u1}{u2}{}{}%
        \myCEdge{u2}{v3}{optedge}{}%
        \myCEdge{u2}{v4}{optedge}{}%
    }%
    \renewcommand{\CfN}{$(1C)$}%
    \renewcommand{\CfL}{fig:COC}%
}%

\newcommand{\cfCOCa}{%
    \renewcommand\Mynodes{
        \node[whitenode,label=above:$v_2$] (v2) at (-1,0) {};
        \node[whitenode,label=above:$v_1$] (v1) at (-1,1) {};
        \node[whitenode,label=above:$v$] (v) at (0.5,2) {};
        \node[whitenode,label=above:$v_4$] (v4) at (2,0) {};
        \node[whitenode,label=above:$v_5$] (v5) at (2,1) {};
        \node[graynode] (uvm) at (0.5,1) {};
    }%
    \renewcommand\Myspecialnodes{%
        \node[blacknode,label=below:$u_1$] (u1) at (0,1) {};
        \node[blacknode,label=below:$u_2$] (u2) at (1,1) {};
    }%
    \renewcommand\MyEdgesBefore{%
        \myCEdge{u1}{v}{}{}%
        \myCEdge{u2}{v}{}{}%
        \myCEdge{v1}{v}{dashed}{}
        \myCEdge{u1}{v1}{}{}
        \myCEdge{u1}{v2}{optedge}{}
        \myCEdge{u1}{u2}{}{}
        \myCEdge{u2}{v4}{optedge}{}
        \myCEdge{u2}{v5}{optedge}{}
    }%
    \renewcommand{\CfN}{$(1Ca)$}%
    \renewcommand{\CfL}{fig:COCa}%
}%

\newcommand{\cfCTC}{%
    \renewcommand\Mynodes{%
        \node[whitenode,label=above:$v_1$] (v1) at (-1,0) {};%
        \node[whitenode,label=above:$v_2$] (v2) at (3,0) {};%
        \node[whitenode,label=above:$v$] (v) at (1,1) {};%
        \node[whitenode,label=above:$v'$] (vp) at (1,2) {};%
        \node[graynode] (uvm) at (1,0) {};%
    }%
    \renewcommand\Myspecialnodes{%
        \node[blacknode,label=above left:$u_1$] (u1) at (0,0) {};%
        \node[blacknode,label=above right:$u_2$] (u2) at (2,0) {};%
    }%
    \renewcommand\MyEdgesBefore{%
        \myCEdge{u1}{v}{}{}%
        \myCEdge{u1}{vp}{}{}%
        \myCEdge{u2}{vp}{}{}%
        \myCEdge{u2}{v}{}{}%
        \myCEdge{u1}{v1}{optedge}{}%
        \myCEdge{u1}{u2}{}{}%
        \myCEdge{u2}{vp}{}{}%
        \myCEdge{u2}{v2}{optedge}{}%
    }%
    \renewcommand{\CfN}{$(2C)$}%
    \renewcommand{\CfL}{fig:CTCj}%
}%

\newcommand{\cfCFC}{%
    \renewcommand\Mynodes{%
        \node[whitenode,label=below:$v$] (v) at (1,1) {};%
        \node[whitenode,label=below:$v'$] (vp) at (1,2) {};%
        \node[whitenode,label=above:$v''$] (vs) at (1,3) {};%
        \node[graynode] (uvm) at (1,0) {};%
    }%
    \renewcommand\Myspecialnodes{%
        \node[blacknode,label=left:$u_1$] (u1) at (0,0) {};%
        \node[blacknode,label=right:$u_2$] (u2) at (2,0) {};%
    }%
    \renewcommand\MyEdgesBefore{%
        \myCEdge{u1}{v}{}{}%
        \myCEdge{u1}{vp}{}{}%
        \myCEdge{u2}{vp}{}{}%
        \myCEdge{u2}{v}{}{}%
        \myCEdge{u1}{vs}{}{}%
        \myCEdge{u1}{uvm}{}{}%
        \myCEdge{u2}{uvm}{}{}%
        \myCEdge{u2}{vp}{}{}%
        \myCEdge{u2}{vs}{}{}%
    }%
    \renewcommand{\CfN}{$(3C)$}%
    \renewcommand{\CfL}{fig:CFC}%
}%

\newcommand{\cfXXn}{%
    \renewcommand\Mynodes{
        \node[whitenode,label=below:$v$] (v) at (1,1) {};
        \node[whitenode,label=left:$v'$] (vp) at (1,2) {};
        \node[whitenode,label=above:$v''$] (vs) at (1,3) {};
        \node[graynode] (uvm) at (1,0) {};
    }%
    \renewcommand\Myspecialnodes{%
        \node[blacknode,label=left:$u_1$] (u1) at (0,0) {};
        \node[blacknode,label=right:$u_2$] (u2) at (2,0) {};

    }%
    \renewcommand\MyEdgesBefore{%
    	\myCEdge{v}{vp}{dashed}{} %
        \myCEdge{vp}{vs}{dashed}{} %
        \myCEdge{u1}{v}{}{}%
        \myCEdge{u1}{vp}{}{}%
        \myCEdge{u2}{vp}{}{}%
        \myCEdge{u2}{v}{}{}%
		\draw (u1) edge [thick, bend left = 20] (vs);
        \myCEdge{u1}{uvm}{}{}
        \myCEdge{uvm}{u2}{}{}
        \myCEdge{u2}{vp}{}{}
		\draw (u2) edge [thick, bend right = 20] (vs);
    }%
    \renewcommand{\CfN}{\nameXXn}%
    \renewcommand{\CfL}{fig:Cn}%
}%
\newcommand{\RuleXXn}{%
    \cfXXn{}%
    \renewcommand\MyEdgesRemoved{
    \myCEdge{v}{vp}{blue}{$Q$}
    \myCEdge{vp}{vs}{green}{$R$}
    }%
    \renewcommand\MyEdgesAfter{%
    	\myCEdge{v}{vp}{dashed}{} %
        \myCEdge{vp}{vs}{dashed}{} %
		\draw (u1) edge [red, thick, bend left = 20] (vs);
		\draw ($(u1)!0.5!(vs)+(-0.38,0)$) node[red,nodelabel] {$\ P\ $};
        \myCEdge{u1}{uvm}{red}{$P$}
        \myCEdge{uvm}{u2}{red}{$P$}
        \myCEdge{v}{u2}{red}{$P$}
        \myCEdge{v}{u1}{blue}{$Q$}
        \myCEdge{u1}{vp}{blue}{$Q$}
        \myCEdge{vp}{u2}{green}{$R$}
		\draw (u2) edge [green, thick, bend right = 20] (vs);
		\draw ($(u2)!0.5!(vs)+(0.38,0)$) node[green,nodelabel] {$\ R\ $};
    }%
    \drawRule{fig:Rn}%
}%

\newcommand{\cfXXo}{%
    \renewcommand\Mynodes{
        \node[whitenode,label=below:$v$] (v) at (1,1) {};
        \node[whitenode,label=left:$v'$] (vp) at (1,2) {};
        \node[whitenode,label=above:$v''$] (vs) at (1,3) {};
        \node[graynode] (uvm) at (1,0) {};
    }%
    \renewcommand\Myspecialnodes{%
        \node[blacknode,label=left:$u_1$] (u1) at (0,0) {};
        \node[blacknode,label=right:$u_2$] (u2) at (2,0) {};

    }%
    \renewcommand\MyEdgesBefore{%
    	\myCurvEdge{vp}{vs}{dashed}{}{} %
        \myCEdge{u1}{v}{}{}%
        \myCEdge{u1}{vp}{}{}%
        \myCEdge{u2}{vp}{}{}%
        \myCEdge{u2}{v}{}{}%
        \myCEdge{vp}{v}{}{}%
		\draw (u1) edge [thick, bend left = 20] (vs);
        \myCEdge{u1}{uvm}{}{}
        \myCEdge{uvm}{u2}{}{}
        \myCEdge{u2}{vp}{}{}
		\draw (u2) edge [thick, bend right = 20] (vs);
    }%
    \renewcommand{\CfN}{\nameXXo}%
    \renewcommand{\CfL}{fig:Co}%
}%
\newcommand{\RuleXXo}{%
    \cfXXo{}%
    \renewcommand\MyEdgesRemoved{
    \myCEdge{v}{vp}{blue}{$Q$}
    \myCEdge{vp}{vs}{green}{$R$}
    }%
    \renewcommand\MyEdgesAfter{%
    	\myCurvEdge{vp}{vs}{dashed}{}{} %
		\draw (u1) edge [red, thick, bend left = 20] (vs);
		\draw ($(u1)!0.5!(vs)+(-0.38,0)$) node[red,nodelabel] {$\ P\ $};
        \myCEdge{u1}{uvm}{red}{$P$}
        \myCEdge{uvm}{u2}{red}{$P$}
        \myCEdge{v}{u2}{red}{$P$}
        \myCEdge{v}{vp}{red}{$P$}
        \myCEdge{v}{u1}{blue}{$Q$}
        \myCEdge{u2}{uvm}{red}{$P$}
        \myCEdge{uvm}{u1}{red}{$P$}
        \myCEdge{u1}{vp}{blue}{$Q$}
        \myCEdge{vp}{u2}{green}{$R$}
		\draw (u2) edge [green, thick, bend right = 20] (vs);
		\draw ($(u2)!0.5!(vs)+(0.38,0)$) node[green,nodelabel] {$\ R\ $};
    }%
    \drawRule{fig:Ro}%
}%

\newcommand{\cfCrr}{%
    \renewcommand\Mynodes{
        \node[whitenode,label=above:$v_1$] (v1) at (-1,0+0.2) {};
        \node[whitenode,label=above:$v_2$] (v2) at (1.5,1+0.2) {};
        \node[whitenode,label=left:$v$] (v) at (1,0.5+0.2) {};
        \node[whitenode,label=above:$v'$] (vp) at (1,2.5+0.2) {};
        \node[graynode] (uvm) at (1,0) {};
    }%
    \renewcommand\Myspecialnodes{%
        \node[blacknode,label=below:$u_1$] (u1) at (0,0+0.2) {};
        \node[blacknode,label=below:$u_2$] (u2) at (2.5,0+0.2) {};

    }%
    \renewcommand\MyEdgesBefore{%
    	\myCEdge{vp}{v}{dashed}{} %
        \myCEdge{u1}{v}{}{}%
        \myCEdge{u1}{vp}{}{}%
        \myCEdge{u2}{vp}{}{}%
        \myCEdge{u2}{v}{}{}%
        \myCEdge{u1}{v1}{optedge}{}
        \myCEdge{u1}{u2}{}{}
        \myCEdge{u2}{vp}{}{}
        \myCEdge{u2}{v2}{}{}
    }%
    \renewcommand{\CfN}{$(C_r)$}%
    \renewcommand{\CfL}{fig:Crr}%
}%

\newcommand{\cfXXh}{%
    \renewcommand\Mynodes{
        \node[whitenode,label=above:$v_1$] (v1) at (-1,0+0.2) {};
        \node[whitenode,label=below:$v_2$] (v2) at (1.5,1+0.2) {};
        \node[whitenode,label=left:$v$] (v) at (1,0.5+0.2) {};
        \node[whitenode,label=above:$v'$] (vp) at (1,2.5+0.2) {};
        \node[graynode] (uvm) at (1,0+0.2) {};
    }%
    \renewcommand\Myspecialnodes{%
        \node[blacknode,label=below:$u_1$] (u1) at (0,0+0.2) {};
        \node[blacknode,label=below:$u_2$] (u2) at (2.5,0+0.2) {};
    }%
    \renewcommand\MyEdgesBefore{%
    	\myCEdge{v2}{vp}{dashed}{}
        \myCEdge{vp}{v}{dashed}{}
        \myCEdge{u1}{v}{}{}%
        \myCEdge{u1}{vp}{}{}%
        \myCEdge{u2}{vp}{}{}%
        \myCEdge{u2}{v}{}{}%
        \myCEdge{u1}{v1}{optedge}{}
        \myCEdge{u1}{u2}{}{}
        \myCEdge{u2}{vp}{}{}
        \myCEdge{u2}{v2}{}{}
    }%
    \renewcommand{\CfN}{\nameXXh}%
    \renewcommand{\CfL}{fig:Ch}%
}%
\newcommand{\RuleXXh}{%
    \cfXXh{}%
    \renewcommand\MyEdgesRemoved{
        \myCEdge{v}{vp}{blue}{$Q$}
        \myCEdge{vp}{v2}{green}{$R$}
    }%
    \renewcommand\MyEdgesAfter{%
    	\myCEdge{v2}{vp}{dashed}{}
        \myCEdge{vp}{v}{dashed}{}
        \myCEdge{u1}{v1}{red,optedge}{$P$}
        \myCEdge{u1}{uvm}{red}{$P$}
        \myCEdge{uvm}{u2}{red}{$P$}
        \myCEdge{v}{u2}{red}{$P$}
        \myCEdge{v}{u1}{blue}{$Q$}
        \myCEdge{u1}{vp}{blue}{$Q$}
        \myCEdge{vp}{u2}{green}{$R$}
        \myCEdge{u2}{v2}{green}{$R$}
    }%
    \drawRule{fig:Rh}%
}%

\newcommand{\cfXXi}{%
    \renewcommand\Mynodes{
        \node[whitenode,label=above:$v_1$] (v1) at (-1,0+0.2) {};
        \node[whitenode,label=below:$v_2$] (v2) at (1.5,1+0.2) {};
        \node[whitenode,label=left:$v$] (v) at (1,0.5+0.2) {};
        \node[whitenode,label=above:$v'$] (vp) at (1,2.5+0.2) {};
        \node[graynode] (uvm) at (1,0+0.2) {};
    }%
    \renewcommand\Myspecialnodes{%
        \node[blacknode,label=below:$u_1$] (u1) at (0,0+0.2) {};
        \node[blacknode,label=below:$u_2$] (u2) at (2.5,0+0.2) {};

    }%
    \renewcommand\MyEdgesBefore{%
    	\myDEdge{vp}{v}{}%
        \myCEdge{v2}{vp}{}{}
        \myCEdge{v2}{v}{}{}
        \myCEdge{u1}{v}{}{}%
        \myCEdge{u1}{vp}{}{}%
        \myCEdge{u2}{vp}{}{}%
        \myCEdge{u2}{v}{}{}%
        \myCEdge{u1}{v1}{optedge}{}
        \myCEdge{u1}{u2}{}{}
        \myCEdge{u2}{vp}{}{}
        \myCEdge{u2}{v2}{}{}
    }%
    \renewcommand{\CfN}{\nameXXi}%
    \renewcommand{\CfL}{fig:Ci}%
}%
\newcommand{\RuleXXi}{%
    \cfXXi{}%
    \renewcommand\MyEdgesRemoved{
        \myCEdge{v}{vp}{blue}{$Q$}
        \myCEdge{vp}{v2}{green}{$R$}
        \myDEdge{v2}{v}{}%
}%
    \renewcommand\MyEdgesAfter{%

    	\myDEdge{vp}{v}{}%
        \myCEdge{u1}{v1}{red,optedge}{$P$}
        \myCEdge{u1}{uvm}{red}{$P$}
        \myCEdge{uvm}{u2}{red}{$P$}
        \myCEdge{v}{u2}{red}{$P$}
        \myCEdge{v}{v2}{red}{$P$}
        \myCEdge{vp}{v2}{red}{$P$}
        \myCEdge{v}{u1}{blue}{$Q$}
        \myCEdge{u1}{vp}{blue}{$Q$}
        \myCEdge{vp}{u2}{green}{$R$}
        \myCEdge{u2}{v2}{green}{$R$}
    }%
    \drawRule{fig:Ri}%
}%

\newcommand{\cfXXj}{%
    \renewcommand\Mynodes{
        \node[whitenode,label=below:$v$] (v) at (1,1+0.2) {};
        \node[evennode,label=left:$v'$] (vp) at (1,2+0.2) {};
        \node[graynode] (uvm) at (1,0+0.2) {};
    }%
    \renewcommand\Myspecialnodes{%
        \node[blacknode,label=below:$u_1$] (u1) at (0,0+0.2) {};
        \node[blacknode,label=below:$u_2$] (u2) at (2,0+0.2) {};

    }%
    \renewcommand\MyEdgesBefore{%
    	\myCEdge{vp}{v}{dashed}{}
        \myCEdge{u1}{v}{}{}%
        \myCEdge{u1}{vp}{}{}%
        \myCEdge{u2}{v}{}{}%
        \myCEdge{u2}{vp}{}{}
        \myCEdge{u2}{uvm}{}{}
        \myCEdge{u1}{uvm}{}{}
    }%
    \renewcommand{\CfN}{\nameXXj}%
    \renewcommand{\CfL}{fig:Cj}%
}%
\newcommand{\RuleXXj}{%
    \cfXXj{}%
    \renewcommand\MyEdgesRemoved{
        \myCEdge{v}{vp}{blue}{$Q$}
        \myHalfEdge{vp}{90}{green}{$R$}%
    }%
    \renewcommand\MyEdgesAfter{%
    	\myCEdge{vp}{v}{dashed}{}
        \myCEdge{u1}{uvm}{red}{$P$}
        \myCEdge{uvm}{u2}{red}{$P$}
        \myCEdge{v}{u2}{red}{$P$}
        \myCEdge{v}{u1}{blue}{$Q$}
        \myCEdge{u1}{vp}{blue}{$Q$}
        \myHalfEdge{vp}{90}{green}{$R$}%
        \myCEdge{vp}{u2}{green}{$R$}
    }%
    \drawRule{fig:Rj}%
}%

\newcommand{\cfXXk}{%
    \renewcommand\Mynodes{
        \node[oddnode,label=left:$v$] (v) at (1,1.2+0.2) {};
        \node[oddnode,label=left:$v'$] (vp) at (1,2.5+0.2) {};
        \node[graynode] (uvm) at (1,0+0.2) {};
    }%
    \renewcommand\Myspecialnodes{%
        \node[blacknode,label=below:$u_1$] (u1) at (0,0+0.2) {};
        \node[blacknode,label=below:$u_2$] (u2) at (2,0+0.2) {};

    }%
    \renewcommand\MyEdgesBefore{%
    	\myDEdge{vp}{v}{}%
        \myCEdge{u1}{v}{}{}%
		\path (u1) edge [thick, bend left = 10] (vp);
		\path (u2) edge [thick, bend right = 10] (vp);
        \myCEdge{u2}{v}{}{}%
        \myCEdge{u2}{uvm}{}{}
        \myCEdge{u1}{uvm}{}{}
    }%
    \renewcommand{\CfN}{\nameXXk}%
    \renewcommand{\CfL}{fig:Ck}%
}%
\newcommand{\RuleXXk}{%
    \cfXXk{}%
    \renewcommand\MyEdgesRemoved{
        \myHalfEdge{v}{-90}{blue}{$Q$}%
        \myHalfEdge{vp}{90}{green}{$R$}%
        \myDEdge{vp}{v}{}%
    }%
    \renewcommand\MyEdgesAfter{%
        \myDEdge{vp}{v}{}%
		\draw (u1) edge [red, thick, bend left = 10] (vp);
		\draw ($(u1)!0.5!(vp)+(-0.15,0)$) node[red,nodelabel] {$\ P\ $};
        \myCEdge{u1}{uvm}{red}{$P$}
        \myCEdge{uvm}{u2}{red}{$P$}
        \myCEdge{v}{u2}{red}{$P$}
        \myHalfEdge{v}{-90}{blue}{$Q$}%
        \myCEdge{u1}{v}{blue}{$Q$}
        \myHalfEdge{vp}{90}{green}{$R$}%
		\path (u2) edge [green, thick, bend right = 10] (vp);
		\draw ($(u2)!0.5!(vp)+(0.15,0)$) node[green,nodelabel] {$\ R\ $};
    }%
    \drawRule{fig:Rk}%
}%

\newcommand{\cfXXl}{%
    \renewcommand\Mynodes{%
		\node[whitenode,label={[label distance=-0.1cm]180:$v'$}] (v) at (1,1+0.2) {};%
        \node[whitenode,label=left:$v$] (vp) at (1,2+0.2) {};%
        \node[whitenode,label=below left:$v''$] (vpp) at (1,0+0.2) {};%
    }%
    \renewcommand\Myspecialnodes{%
        \node[blacknode,label=below:$u_1$] (u1) at (0,0+0.2) {};%
        \node[blacknode,label=below:$u_2$] (u2) at (2,0+0.2) {};%
    }%
    \renewcommand\MyEdgesBefore{%
        \myCEdge{u1}{v}{}{}%
        \myCEdge{u1}{vp}{}{}%
        \myCEdge{v}{vp}{}{}%
        \myCEdge{vpp}{v}{}{}%
        \myCEdge{u2}{vp}{}{}%
        \myCEdge{u2}{v}{}{}%
        \myCEdge{u2}{vp}{}{}%
        \myCEdge{u2}{vpp}{}{}%
        \myCEdge{u1}{vpp}{}{}%
    }%
    \renewcommand{\CfN}{\nameXXl}%
    \renewcommand{\CfL}{fig:Cl}%
}%
\newcommand{\RuleXXl}{%
    \cfXXl{}%
    \renewcommand\MyEdgesRemoved{%
        \myCEdge{v}{vp}{blue}{$Q$}%
        \myCEdge{v}{vpp}{green}{$R$}%
    }%
    \renewcommand\MyEdgesAfter{%
        \myCEdge{vp}{u1}{red}{}%
        \myCEdge{vp}{v}{red}{$P$}%
        \myCEdge{v}{vpp}{red}{$P$}%
        \myCEdge{vpp}{u2}{red}{$P$}%
        \myCEdge{v}{u1}{green}{$R$}%
        \myCEdge{u1}{vpp}{green}{$R$}%
        \myCEdge{v}{u2}{blue}{$Q$}%
        \myCEdge{u2}{vp}{blue}{$Q$}%
    }%
    \drawRule{fig:Rl}%
}%

\newcommand{\cfXXlp}{%
    \renewcommand\Mynodes{
        \node[whitenode,label=left:$v$] (v) at (1,2+0.2) {};
		\node[whitenode,label={[label distance=-0.1cm]180:$v'$}] (vp) at (1,1+0.2) {};%
        \node[whitenode,label=below:$v_1$] (v1) at (-1,0+0.2) {};
        \node[whitenode,label=below:$v_2$] (v2) at (3,0+0.2) {};
        \node[whitenode,label=below left:$v''$] (vpp) at (1,0+0.2) {};
    }%
    \renewcommand\Myspecialnodes{%
        \node[blacknode,label=below:$u_1$] (u1) at (0,0+0.2) {};
        \node[blacknode,label=below:$u_2$] (u2) at (2,0+0.2) {};
    }%
    \renewcommand\MyEdgesBefore{%
        \myCEdge{u1}{v}{}{}%
        \myCEdge{u1}{vp}{}{}%
        \myCEdge{v}{vp}{}{}%
        \myCEdge{vpp}{vp}{}{}%
        \myCEdge{u2}{v}{}{}%
        \myCEdge{u1}{v1}{}{}%
        \myCEdge{u2}{v2}{optedge}{}%
        \myCEdge{u2}{vpp}{}{}%
        \myCEdge{u1}{vpp}{}{}%
        \myCEdge{u2}{vp}{}{}%
        \myBentEdge{v1}{v}{optedge}{left}{0}%
    }%
    \renewcommand{\CfN}{\nameXXlp}%
    \renewcommand{\CfL}{fig:Clp}%
}%
\newcommand{\RuleXXlp}{%
    \cfXXlp{}%
    \renewcommand\MyEdgesRemoved{
        \myCEdge{v1}{v}{blue}{$Q$}
        \myCEdge{v}{vp}{green}{$R$}
        \myEdge{vp}{vpp}
    }%
    \renewcommand\MyEdgesAfter{%
        \myCEdge{u1}{v}{blue}{}%
        \myCEdge{u1}{vp}{red}{$P$}%
        \myCEdge{v}{vp}{red}{}%
        \myCEdge{vpp}{vp}{black}{}%
        \myCEdge{u1}{v1}{blue}{$Q$}%
        \myCEdge{u2}{vpp}{red}{$P$}%
        \myCEdge{u1}{vpp}{red}{$P$}%
        \myCEdge{u2}{vp}{green}{$R$}%
        \myCEdge{u2}{v}{green}{$R$}%
        \myCEdge{u2}{v2}{red,optedge}{}%
        \myBentEdge{v1}{v}{red,optedge}{left}{0}%
    }%
    \drawRule{fig:Rlp}%
}%

\newcommand{\cfXXm}{%
    \renewcommand\Mynodes{
		\node[whitenode,label={[label distance=-0.1cm]180:$v$}] (v) at (1,1+0.2) {};%
        \node[whitenode,label=left:$v'$] (vp) at (1,2+0.2) {};
        \node[whitenode,label=below:$v_1$] (v1) at (-1,0+0.2) {};
        \node[whitenode,label=below:$v_2$] (v2) at (3,0+0.2) {};
    }%
    \renewcommand\Myspecialnodes{%
        \node[blacknode,label=below:$u_1$] (u1) at (0,0+0.2) {};
        \node[blacknode,label=below:$u_2$] (u2) at (2,0+0.2) {};
    }%
    \renewcommand\MyEdgesBefore{%
        \myCEdge{u1}{v}{}{}%
        \myCEdge{u1}{vp}{}{}%
        \myCEdge{v}{vp}{}{}%
        \myCEdge{u2}{vp}{}{}%
        \myCEdge{u2}{v}{}{}%
        \myCEdge{u2}{vp}{}{}
        \myCEdge{u2}{u1}{}{}
        \myCEdge{u2}{v2}{optedge}{}%
        \myCEdge{u1}{v1}{optedge}{}%
    }%
    \renewcommand{\CfN}{\nameXXm}%
    \renewcommand{\CfL}{fig:Cm}%
}%
\newcommand{\RuleXXm}{%
    \cfXXm{}%
    \renewcommand\MyEdgesRemoved{
        \myCEdge{v}{vp}{blue}{$Q$}
    }%
    \renewcommand\MyEdgesAfter{%
        \myCEdge{v1}{u1}{red,optedge}{$P$}
        \myCEdge{vp}{u1}{red}{$P$}
        \myCEdge{vp}{v}{red}{$P$}
        \myCEdge{v}{u2}{red}{$P$}
        \myCEdge{v2}{u2}{red,optedge}{$P$}
        \myCEdge{v}{u1}{blue}{$Q$}
        \myCEdge{u2}{u1}{blue}{$Q$}
        \myCEdge{u2}{vp}{blue}{$Q$}
    }%
    \drawRule{fig:Rm}%
}%

\newcommand{\cfTTb}{%
    \renewcommand\Mynodes{%
        \node[whitenode,label=below:$v_1$] (v1) at (-1,0+1) {};%
        \node[whitenode,label=below:$v_2$] (v2) at (2,0+1) {};%
        \node[whitenode,label=above:$v$] (v) at ($(0,0+1) + (60:1)$) {};%
    }%
    \renewcommand\Myspecialnodes{%
        \node[blacknode,label=below:$u_1$] (u1) at (0,0+1) {};%
        \node[blacknode,label=below:$u_2$] (u2) at (1,0+1) {};%
    }%
    \renewcommand\MyEdgesBefore{%
        \myEdge{u1}{v1}%
        \myEdge{u1}{u2}%
        \myEdge{u2}{v2}%
        \myEdge{u1}{v}%
        \myEdge{u2}{v}%
        \myCEdge{v1}{v}{}{}%
        \myCEdge{v2}{v}{}{}%
    }%
    \renewcommand{\CfN}{$(33b)$}%
    \renewcommand{\CfL}{fig:CTTb}%
}%

\newcommand{\cfFFb}{%
    \renewcommand\Mynodes{%
        \node[whitenode,label=below:$v_1$] (v1) at (-0.5,1) {};%
        \node[whitenode,label=below:$v_2$] (v2) at (1.5,1) {};%
        \node[whitenode,label=below:$v_3$] (v3) at (1.8,0) {};%
        \node[whitenode,label=below:$v_4$] (v4) at (-0.8,0) {};%
        \node[whitenode,label=above:$v$] (v) at ($(0,1) + (60:1)$) {};%
    }%
    \renewcommand\Myspecialnodes{%
        \node[blacknode,label=below:$u_1$] (u1) at (0,1) {};%
        \node[blacknode,label=below:$u_2$] (u2) at (1,1) {};%
    }%
    \renewcommand\MyEdgesBefore{%
        \myEdge{u1}{v1}%
        \myEdge{u1}{u2}%
        \myEdge{u2}{v2}%
        \myEdge{u1}{v}%
        \myEdge{u2}{v}%
        \myEdge{v3}{u2}%
        \myCEdge{v4}{u1}{optedge}{}%
        \myCurvEdge{v3}{v}{}{}{in=0, out=45}%
        \myCurvEdge{v4}{v}{optedge}{}{in=180, out=135}%
        \myCEdge{v1}{v}{}{}%
        \myCEdge{v2}{v}{}{}%
    }%
    \renewcommand{\CfN}{$(44b)$}%
    \renewcommand{\CfL}{fig:CFFb}%
}%

\newcommand{\cfTwoAny}{%
    \renewcommand\Mynodes{%
        \node[whitenode,label=below:$v_1$] (v1) at (2,1) {};%
        \node[whitenode,label=below:$v_2$] (v2) at (2.3,0) {};%
        \node[whitenode,label=above:$v$] (v) at ($ (0,1) + (60:1) $) {};%
        \node[graynode] (uvm) at (0.5,1) {};%
    }%
    \renewcommand\Myspecialnodes{%
        \node[blacknode,label=below:$u_1$] (u1) at (0,1) {};%
        \node[blacknode,label=below:$u_2$] (u2) at (1,1) {};%
    }%
    \renewcommand\MyEdgesBefore{%
        \myEdge{u1}{u2}%
        \myCEdge{u2}{v1}{optedge}{}%
        \myEdge{u1}{v}%
        \myEdge{u2}{v}%
        \myCEdge{v2}{u2}{optedge}{}%
        \myCurvEdge{v2}{v}{optedge}{}{in=0, out=45}%
        \myCEdge{v1}{v}{optedge}{}%
    }%
    \renewcommand{\CfN}{$(2x)$}%
    \renewcommand{\CfL}{fig:CTwoAny}%
}%

\newcommand{\cfXXp}{%
    \renewcommand\Mynodes{
        \node[whitenode,label=right:\labelVOne] (v1) at (2,1) {};%
        \node[whitenode,label=below:\labelVTwo] (v2) at (2,0.5) {};%
        \node[whitenode,label=above:\labelV] (v) at ($ (0,1) + (60:1) $) {};%
    }%
    \renewcommand\Myspecialnodes{%
        \node[blacknode,label=below:\labelUOne] (u1) at (0,1) {};%
        \node[blacknode,label=below:\labelUTwo] (u2) at (1,1) {};%
    }%
    \renewcommand\MyEdgesBefore{%
        \myEdge{u1}{u2}%
        \myEdge{u2}{v1}%
        \myEdge{u1}{v}%
        \myEdge{u2}{v}%
        \myCEdge{v2}{u2}{optedge}{}%
        \myCEdge{v1}{v}{}{}%
    }%
    \renewcommand{\CfN}{\nameXXp}%
    \renewcommand{\CfL}{fig:Cp}%
}%

\newcommand{\relabeledXXp}[5]{
	\renewcommand{\labelUOne}{#1}
	\renewcommand{\labelUTwo}{#2}
	\renewcommand{\labelV}{#3}
	\renewcommand{\labelVOne}{#4}
	\renewcommand{\labelVTwo}{#5}
	\tkcf{\cfXXp}
	\resetLabels
}

\newcommand{\RuleXXp}{%
    \cfXXp{}%
    \renewcommand\MyEdgesRemoved{
        \myCEdge{v}{v1}{blue}{$Q$}
    }%
    \renewcommand\MyEdgesAfter{%
        \myCEdge{v}{u1}{blue}{$Q$}
        \myCEdge{u1}{u2}{blue}{$Q$}
        \myCEdge{v1}{u2}{blue}{$Q$}

        \myCEdge{v}{v1}{red}{$P$}
        \myCEdge{v}{u2}{red}{$P$}
        \myCEdge{v2}{u2}{optedge,red}{$P$}
    }%
    \drawRule{fig:Rp}%
}%

\newcommand{\cfCN}{%
    \renewcommand\Mynodes{
        \node[purplenode,label=below:$v_1$] (v1) at ($(0,1) + (210:1)$) {};%
        \node[purplenode,label=above:$v_2$] (v2) at ($(0,1) + (150:1)$) {};%
    }%
    \renewcommand\Myspecialnodes{%
        \node[blacknode,label=above:$u_1$] (u1) at (0,1) {};%
    }%
    \renewcommand\MyEdgesBefore{%
    	\mySubEdge{u1}{0}{}{}%
        \myEdge{u1}{v1}%
        \myEdge{u1}{v2}%
    }%
    \renewcommand{\CfN}{$(\C_N)$}%
    \renewcommand{\CfL}{fig:CCN}%
}%

\newcommand{\cfCNP}{%
    \renewcommand\Mynodes{%
        \node[whitenode,label=above:\labelVOne] (v1) at (1,2) {};%
        \node[whitenode,label=left:\labelVThree] (v3) at (1,0) {};%
        \node[whitenode,label=above:\labelVTwo] (v2) at (0,1) {};%
    }%
    \renewcommand\Myspecialnodes{%
		\node[blacknode,label={[label distance=-0.08cm]45:\labelUOne}] (u1) at (1,1) {};
    }%
    \renewcommand\MyEdgesBefore{%
		\draw[sedge]  ($(u1) + (0:0.75)$) 
		-- (u1);
        \myCEdge{u1}{v1}{optedge}{}%
        \myCEdge{u1}{v3}{optedge}{}%
        \myCEdge{u1}{v2}{optedge}{}%
    }%
    \renewcommand{\CfN}{$(\C_N^+)$}%
    \renewcommand{\CfL}{fig:CCNP}%
}%

\newcommand{\relabeledCNP}[5]{
	\renewcommand{\labelUOne}{#1}
	\renewcommand{\labelVOne}{#2}
	\renewcommand{\labelVTwo}{#3}
	\renewcommand{\labelVThree}{#4}
	\renewcommand{\labelVFour}{#5}
	\tkcf{\cfCNP}
	\resetLabels
}

\newcommand{\cfCZa}{%
    \renewcommand\Mynodes{
        \node[whitenode,label=below:$v_1$] (v1) at (0,1) {};
        \node[whitenode,label=below:$v_2$] (v2) at (2,1) {};
    }%
    \renewcommand\Myspecialnodes{%
        \node[blacknode,label=below:$u_1$] (u1) at (1,0) {};
    }%
    \renewcommand\MyEdgesBefore{%
        \myEdge{u1}{v1}
        \myEdge{u1}{v2}
        \myDEdge{v1}{v2}
    }%
    \renewcommand{\CfN}{\nameXXCOa}%
    \renewcommand{\CfL}{fig:C0a}%
}%
\newcommand{\RuleZeroa}{%
    \cfCZa
    \renewcommand\MyEdgesRemoved{
        \myCEdge{v1}{v2}{blue}{$Q$}
    }%
    \renewcommand\MyEdgesAfter{%
        \myCEdge{v1}{u1}{blue}{$Q$}
        \myCEdge{u1}{v2}{blue}{$Q$}
        \myDEdge{v1}{v2}

    }%
    \drawRule{fig:R0a}%
}%

\newcommand{\cfXXr}{%
    \renewcommand\Mynodes{
        \node[whitenode,label=above:\labelV] (v) at (60:1) {};
        \node[whitenode,label=above:\labelVOne] (v1) at ($ (v) + (-1,0) $) {};
        \node[whitenode,label=above:\labelVTwo] (v2) at ($ (v) + (1,0) $) {};
    }%
    \renewcommand\Myspecialnodes{%
        \node[blacknode,label=left:\labelUOne] (u1) at (0,0) {};
        \node[blacknode,label=right:\labelUTwo] (u2) at (1,0) {};
    }%
    \renewcommand\MyEdgesBefore{%
        \myEdge{u1}{u2}
        \myEdge{u1}{v}
        \myEdge{v1}{v}
        \myCEdge{v2}{v}{optedge}{}
        \myEdge{u2}{v}
    }%
    \renewcommand{\CfN}{\nameXXq}%
    \renewcommand{\CfL}{fig:Cr}%
}%

\newcommand{\relabeledXXr}[5]{
	\renewcommand{\labelUOne}{#1}
	\renewcommand{\labelUTwo}{#2}
	\renewcommand{\labelV}{#3}
	\renewcommand{\labelVOne}{#4}
	\renewcommand{\labelVTwo}{#5}
	\tkcf{\cfXXr}
	\resetLabels
}

\newcommand{\RuleXXr}{%
    \cfXXr
    \renewcommand\MyEdgesRemovedBis{
        \myCEdge{v}{v1}{blue}{$Q$}
        \myCEdge{v}{v2}{blue}{$Q$}
    }%
    \renewcommand\MyEdgesAfterBis{%
        \myCEdge{v}{v1}{blue}{$Q'$}
        \myCEdge{v}{u1}{blue}{$Q'$}
        \myCEdge{u1}{u2}{red}{$P$}
        \myCEdge{v}{u2}{red}{$P$}
        \myCEdge{v}{v2}{red}{$P$}
    }%
    \renewcommand\MyEdgesRemoved{
        \myCEdge{v2}{v}{optedge}{}
        \myCEdge{v1}{v}{blue,endpath}{$Q$}
    }%
    \renewcommand\MyEdgesAfter{%
        \myCEdge{v}{v1}{blue}{$Q$}
        \myCEdge{v}{u1}{blue}{$Q$}
        \myCEdge{u1}{u2}{red}{$P$}
        \myCEdge{v}{u2}{red}{$P$}
        \myCEdge{v2}{v}{optedge}{}
    }%
		\begin{minipage}{0.25\linewidth}
		\tkcf{}%
		\end{minipage}
		\hspace{-0.4cm}
        \begin{minipage}{0.65\linewidth}%
            \begin{tikzpicture}[scale=\MyScale, every node/.style={scale=\MyScale}, auto]
		\node[label=above:\rotatebox{30}{\Large{$\leadsto$}}] (arrow) at (0,\OrdonneeFleche) {};
		\node (q) at (0,0.4) {};
		\end{tikzpicture}%
            \tikzstyle{endpath}=[->] %
            \tikzPartRule{\TikzRedefOdd \Mynodes \MyEdgesRemoved}%
            \petiteFlecheACentrer{\OrdonneeFleche}%
            \tikzstyle{endpath}=[] %
            \tikzPartRule{\Myspecialnodes \Mynodes \MyEdgesAfter}%

            \begin{tikzpicture}[scale=\MyScale, every node/.style={scale=\MyScale}, auto]
		\node[label=above:\rotatebox{-20}{\Large{$\leadsto$}}] (arrow) at (0,\OrdonneeFleche) {};
		\node (q) at (0,-0.4) {};
		\end{tikzpicture}%
            \tikzstyle{endpath}=[->] %
            \tikzPartRule{\Mynodes \MyEdgesRemovedBis}%
            \petiteFlecheACentrer{\OrdonneeFleche}%
            \tikzstyle{endpath}=[] %
            \tikzPartRule{\Myspecialnodes \Mynodes \MyEdgesAfterBis}%
        \end{minipage}\\
}%

\newcommand{\cfRx}{%
    \renewcommand\Mynodes{
        \node[whitenode,label=below:$v_1$] (v1) at (0,0+0.2) {};
        \node[whitenode,label=above:$v_2$] (v2) at (0,3+0.2) {};
        \node[whitenode,label=below:$v_3$] (v3) at (-3,1.5+0.2) {};
        \node[whitenode,label=below:$v_4$] (v4) at (3,1.5+0.2) {};
    }%
    \renewcommand\Myspecialnodes{%
        \node[blacknode,label=below:$u_1$] (u1) at (-1.5,1.5+0.2) {};
        \node[blacknode,label=below:$u_2$] (u2) at (1.5,1.5+0.2) {};
    }%
    \renewcommand\MyEdgesBefore{%
    	\myHalfEdge{u1}{0}{optedge}{}%
        \myHalfEdge{u2}{180}{optedge}{}%
        \myEdge{u1}{v1}
        \myEdge{u1}{v2}
        \myEdge{u1}{v3}
        \myEdge{u2}{v1}
        \myEdge{u2}{v2}
        \myEdge{u2}{v4}
        \myEdge{v3}{v1}
        \myEdge{v3}{v2}
        \myEdge{v4}{v1}
        \myEdge{v4}{v2}
    }%
    \renewcommand{\CfN}{$(Rx)$}%
    \renewcommand{\CfL}{fig:CRx}%
}%

\newcommand{\cfXXe}{%
    \renewcommand\Mynodes{
        \node[whitenode,label=below:$v_1$] (v1) at (0,0+0.2) {};
        \node[whitenode,label=above:$v_2$] (v2) at (0,3+0.2) {};
        \node[whitenode,label=above:$v_5$] (v5) at (-0.5,1.5+0.2) {};
        \node[whitenode,label=above:$v_6$] (v6) at (0.5,1.5+0.2) {};
        \node[whitenode,label=below:$v_3$] (v3) at (-3,1.5+0.2) {};
        \node[whitenode,label=below:$v_4$] (v4) at (3,1.5+0.2) {};
    }%
    \renewcommand\Myspecialnodes{%
        \node[blacknode,label=below:$u_1$] (u1) at (-1.5,1.5+0.2) {};
        \node[blacknode,label=below:$u_2$] (u2) at (1.5,1.5+0.2) {};
    }%
    \renewcommand\MyEdgesBefore{%
    	\myCEdge{u1}{v5}{optedge}{}%
        \myCEdge{u2}{v6}{optedge}{}%
        \myEdge{v3}{v2}
        \myEdge{u1}{v1}
        \myEdge{u1}{v2}
        \myEdge{u1}{v3}
        \myEdge{u2}{v1}
        \myEdge{u2}{v2}
        \myEdge{u2}{v4}
        \myEdge{v3}{v1}
        \myEdge{v4}{v1}
        \myEdge{v4}{v2}
        \myDEdge{v2}{v1}%
    }%
    \renewcommand{\CfN}{\nameXXe}%
    \renewcommand{\CfL}{fig:Ce}%
}%
\newcommand{\RuleXXe}{%
    \cfXXe
    \renewcommand\MyEdgesRemoved{%
        \myCEdge{v3}{v2}{}{}
        \myCEdge{v1}{v2}{blue}{$Q$}
        \myCEdge{v1}{v4}{green}{$R$}
        \myDEdge{v3}{v1}%
        \myDEdge{v2}{v4}%

    }%
    \renewcommand\MyEdgesAfter{%
    	\myCEdge{v1}{v2}{dashed}{}%

        \myCEdge{v3}{v2}{}{}

        \myCEdge{u2}{v2}{red}{$P$}
        \myCEdge{v3}{u1}{red}{$P$}
        \myCEdge{v3}{u1}{red}{$P$}
        \myCEdge{v1}{v3}{red}{$P$}
        \myCEdge{v1}{v4}{red}{$P$}
        \myCEdge{v2}{v4}{red}{$P$}
        \myCEdge{v2}{u2}{red}{$P$}

        \myCEdge{u1}{v1}{blue}{$Q$}
        \myCEdge{u1}{v2}{blue}{$Q$}
        \myHalfEdge{u2}{180}{red,optedge}{$P$}%
        \myHalfEdge{u1}{0}{red,optedge}{$P$}%
        \myCEdge{v1}{u2}{green}{$R$}
        \myCEdge{u2}{v4}{green}{$R$}
    }%
    \drawRule{fig:Re}%
}%

\newcommand{\cfXXf}{%
    \renewcommand\Mynodes{
        \node[whitenode,label=below:$v_1$] (v1) at (0,0+0.2) {};
        \node[whitenode,label=left:$v_2$] (v2) at (0,3+0.2) {};
        \node[whitenode,label=above:$v_5$] (v5) at (-0.5,1.5+0.2) {};
        \node[whitenode,label=above:$v_6$] (v6) at (0.5,1.5+0.2) {};

        \node[whitenode,label=below:$v_3$] (v3) at (-3,1.5+0.2) {};
        \node[whitenode,label=below:$v_4$] (v4) at (3,1.5+0.2) {};

    }%
    \renewcommand\Myspecialnodes{%

        \node[blacknode,label=below:$u_1$] (u1) at (-1.5,1.5+0.2) {};
        \node[blacknode,label=below:$u_2$] (u2) at (1.5,1.5+0.2) {};

    }%
    \renewcommand\MyEdgesBefore{%
    	\myEdge{v3}{v1}
        \myEdge{v3}{v2}
        \myEdge{v4}{v2}%
        \myCEdge{u1}{v5}{optedge}{}%
        \myCEdge{u2}{v6}{optedge}{}%

        \myCurvEdge{v3}{v4}{}{$P$}{in=90, out=90,path}%

        \myEdge{v1}{v2}
        \myEdge{u1}{v1}
        \myEdge{u1}{v2}
        \myEdge{u1}{v3}
        \myEdge{u2}{v1}
        \myEdge{u2}{v2}
        \myEdge{v4}{v1}
        \myEdge{v4}{u2}
    }%
    \renewcommand{\CfN}{\nameXXf}%
    \renewcommand{\CfL}{fig:Cf}%
}%
\newcommand{\RuleXXf}{%
    \cfXXf
    \renewcommand\MyEdgesRemoved{%
        \myCurvEdge{v3}{v4}{}{}{in=90, out=90,path,dashed}%
        \myCEdge{v1}{v2}{blue}{$Q$}
        \myCEdge{v1}{v4}{green}{$R$}
        \myEdge{v1}{v3}%
        \myEdge{v2}{v4}%
        \myEdge{v2}{v3}%
    }%
    \renewcommand\MyEdgesAfter{%
    	\myEdge{v3}{v1}
        \myEdge{v3}{v2}
        \myEdge{v4}{v2}%
        \myCurvEdge{v3}{v4}{red}{$P$}{in=90, out=90,path}%
        \myCEdge{v1}{v2}{red}{$P$}
        \myCEdge{u2}{v2}{red}{$P$}
        \myCEdge{v3}{u1}{red}{$P$}
        \myCEdge{v1}{u1}{red}{$P$}
        \myCEdge{v1}{v4}{red}{$P$}
        \myCEdge{u1}{v1}{blue}{$Q$}
        \myCEdge{u1}{v2}{blue}{$Q$}
        \myHalfEdge{u2}{180}{red,optedge}{$P$}%
        \myHalfEdge{u1}{0}{red,optedge}{$P$}%
        \myCEdge{v1}{u2}{green}{$R$}
        \myCEdge{u2}{v4}{green}{$R$}
    }%
    \drawRule{fig:Rf}%
}%

\newcommand{\cfXXg}{%
    \renewcommand\Mynodes{
        \node[whitenode,label=below:$v_1$] (v1) at (0,0+0.2) {};
        \node[whitenode,label=above:$v_2$] (v2) at (0,2.5+0.2) {};
        \node[whitenode,label=below:$v_5$] (v5) at (-0.5,1.7+0.2) {};
        \node[whitenode,label=above:$v_6$] (v6) at (0.5,1+0.2) {};

        \node[whitenode,label=below:$v_3$] (v3) at (-2.5,1.5+0.2) {};
        \node[whitenode,label=below:$v_4$] (v4) at (2.5,1.5+0.2) {};
    }%
    \renewcommand\Myspecialnodes{%
        \node[blacknode,label=below:$u_1$] (u1) at (-1.5,1.5+0.2) {};
        \node[blacknode,label=below:$u_2$] (u2) at (1.5,1.5+0.2) {};
    }%
    \renewcommand\MyEdgesBefore{%
        \myCurvEdge{v3}{v4}{}{$P$}{in=90, out=90,path,dashed}%
    	\myEdge{v3}{v1}
        \myEdge{v3}{v2}
        \myEdge{v4}{v2}
        \myEdge{v4}{v1}

        \myEdge{u1}{v1}
        \myEdge{u1}{v2}
        \myEdge{u1}{v3}
        \myEdge{u2}{v1}
        \myEdge{u2}{v2}
        \myEdge{u2}{v4}
        \myCEdge{u1}{v5}{optedge}{}%
        \myCEdge{u2}{v6}{optedge}{}%

    }%
    \renewcommand{\CfN}{\nameXXg}%
    \renewcommand{\CfL}{fig:Cg}%
}%

\newcommand{\RuleXXgOne}{
	\cfXXg%
	\renewcommand\MyEdgesRemoved{
     \tikzstyle{endpath}=[-]
        \myCEdge{v3}{v4}{blue}{$Q$}
        \myEdge{v2}{v3}%
        \myEdge{v1}{v3}%
        \myEdge{v1}{v4}
        \myEdge{v2}{v4}
        \myHalfLongEdge{v4}{105}{blue,optedge,path}{}
        \draw ($(v3)+(0.45,0.6)$) node[blue,nodelabel] {$\ Q_1$};
        \myHalfLongEdge{v3}{75}{blue,optedge,path}{}
        \draw ($(v4)+(-0.5,0.6)$) node[blue,nodelabel] {$Q_2$};

        \path ($(v2) + (-0.5,0)$) edge [->, dashed, bend right = 0] ($(v3) + (0.5,1)$);
        \path ($(v2) + (0.5,0)$) edge [->, dashed, bend right = 0] ($(v4) + (-0.6,1)$);
    }%
    \renewcommand\MyEdgesAfter{%
        \myHalfLongEdge{v4}{100}{blue,optedge,path}{}
        \myHalfLongEdge{v3}{80}{blue,optedge,path}{}
        \draw ($(v3)+(0.45,0.6)$) node[blue,nodelabel] {$\ Q_1$};
        \draw ($(v4)+(-0.5,0.6)$) node[blue,nodelabel] {$Q_2$};
		\myEdge{v3}{v1}
        \myEdge{v4}{v2}
        \myEdge{v4}{v1}

        \myCEdge{v3}{u1}{blue}{$Q'$}
        \myCEdge{u1}{v2}{blue}{$Q'$}
        \myCEdge{v2}{u2}{blue}{$Q'$}
        \myCEdge{u2}{v4}{blue}{$Q'$}
        \myCEdge{u1}{v1}{red}{$P$}
        \myCEdge{v1}{u2}{red}{$P$}
        \myCEdge{u2}{v6}{red,optedge}{$P$}%
        \myCEdge{u1}{v5}{red,optedge}{$P$}%
        \myEdge{v2}{v3}
    }%
    \renewcommand{\CfN}{\nameXXgOne}%
	\drawRuleSansConfig%
}

\newcommand{\RuleXXgTwo}{
	\cfXXg%
	\renewcommand\MyEdgesRemoved{
     \tikzstyle{endpath}=[-]
        \myCEdge{v3}{v4}{blue}{$Q$}%
        \myEdge{v2}{v3}%
        \myEdge{v1}{v3}
        \myEdge{v1}{v4}
        \myEdge{v2}{v4}

        \myCurvEdge{v3}{v2}{blue}{}{in=180,out=90,path}
        \draw ($(v3)+(0.6,0.6)$) node[blue,nodelabel] {$Q_1$};
        \myCurvEdge{v4}{v1}{blue}{}{in=0,out=-90,path}
        \draw ($(v4)+(-0.6,-0.8)$) node[blue,nodelabel] {$Q_2$};
        \myHalfLongEdge{v2}{0}{blue,optedge,path}{}%
        \draw ($(v2)+(0.8,0.3)$) node[blue,nodelabel] {$Q_1$};
        \myHalfLongEdge{v1}{180}{blue,optedge,path}{}
        \draw ($(v1)+(-0.8,-0.3)$) node[blue,nodelabel] {$Q_2$};

		\path ($(v5) + (-0.2,-0.4)$) edge [->, dashed, bend right = 0] ($(v1) + (-1.2,0.2)$);
		\path ($(v6) + (0.2,0.4)$) edge [->, dashed, bend right = 0] ($(v2) + (1.2,-0.2)$);
    }%
    \renewcommand\MyEdgesAfter{%
        \myEdge{v4}{v2}
        \myEdge{v4}{v1}
        \myEdge{v3}{v1}

        \myHalfLongEdge{v1}{180}{red,optedge,path}{}%
        \myCurvEdge{v4}{v1}{red}{$P$}{in=0,out=-90,path}
        \myCEdge{u2}{v4}{red}{$P$}
        \myCEdge{u2}{v2}{red}{$P$}
        \myCEdge{u1}{v2}{red}{$P$}
        \myCEdge{u1}{v5}{red,optedge}{$P$}%

        \myHalfLongEdge{v2}{0}{blue,optedge,path}{}%
        \myCurvEdge{v3}{v2}{blue}{$Q'$}{in=180,out=90,path}
        \myCEdge{u1}{v3}{blue}{$Q'$}
        \myCEdge{u1}{v1}{blue}{$Q'$}
		\myCEdge{u2}{v6}{blue,optedge}{$Q'$}%
        \myCEdge{u2}{v1}{blue}{$Q'$}

        \draw ($(v3)+(0.6,0.6)$) node[blue,nodelabel] {$Q'$};
        \draw ($(v4)+(-0.5,-0.8)$) node[red,nodelabel] {$P$};
        \myEdge{v3}{v2}

        \draw ($(v2)+(0.8,0.3)$) node[blue,nodelabel] {$Q'$};
        \draw ($(v1)+(-0.8,-0.3)$) node[red,nodelabel] {$P$};
    }%
    \renewcommand{\CfN}{\nameXXgTwo}%
	\drawRuleSansConfig%
}

\newcommand{\RuleXXgThree}{
	\cfXXg%
	\renewcommand\MyEdgesRemoved{
     \tikzstyle{endpath}=[-]
        \myCEdge{v3}{v4}{blue}{$Q$}%
        \myEdge{v2}{v3}%
        \myEdge{v1}{v4}
        \myEdge{v2}{v4}

        \myCurvEdge{v3}{v2}{blue}{}{in=180,out=90,path}
        \myCurvEdge{v4}{v1}{blue}{}{in=0,out=-90,path}
        \myHalfLongEdge{v2}{0}{blue,optedge,path}{}
        \myHalfLongEdge{v1}{180}{blue,optedge,path}{}

        \myCEdge{v1}{v3}{green}{$\overline{Q}$}

        \path ($(v5) + (-0.2,0.1)$) edge [->, dashed, bend right = 0] ($0.5*(v2) + 0.5*(v3) + (-0.3,0.3)$);

        \draw ($(v3)+(0.6,0.6)$) node[blue,nodelabel] {$R_1$};
        \draw ($(v4)+(-0.6,-0.8)$) node[blue,nodelabel] {$Q_2$};
        \draw ($(v2)+(0.8,0.3)$) node[blue,nodelabel] {$R_1'$};
        \draw ($(v1)+(-0.8,-0.3)$) node[blue,nodelabel] {$Q_2$};
    }%
    \renewcommand\MyEdgesAfter{%
        \myCEdge{v1}{v3}{red}{$P$}

        \myCEdge{v3}{u1}{green}{$\overline{Q}$}
        \myCEdge{u1}{v2}{red}{$P$}
        \myCEdge{v2}{u2}{blue}{$Q'$}
        \myCEdge{u2}{v4}{blue}{$Q'$}
        \myCEdge{v1}{u1}{green}{$\overline{Q}$}
        \myCEdge{v1}{u2}{red}{$P$}

        \myCEdge{u2}{v6}{red,optedge}{$P$}%
        \myCEdge{u1}{v5}{red,optedge}{$P$}%

        \myHalfLongEdge{v2}{0}{blue,optedge,path}{}%
        \myHalfLongEdge{v1}{180}{blue,optedge,path}{}%
        \myCurvEdge{v4}{v1}{blue}{}{in=0,out=-90,path}
        \myCurvEdge{v3}{v2}{red}{}{in=180,out=90,path}

        \draw ($(v3)+(0.5,0.6)$) node[red,nodelabel] {$P$};
        \draw ($(v4)+(-0.6,-0.8)$) node[blue,nodelabel] {$Q'$};
        \draw ($(v2)+(0.8,0.3)$) node[blue,nodelabel] {$Q'$};
        \draw ($(v1)+(-0.8,-0.3)$) node[blue,nodelabel] {$Q'$};

        \myEdge{v3}{v2}
        \myEdge{v4}{v2}
        \myEdge{v4}{v1}
    }%
    \renewcommand{\CfN}{\nameXXgThree}%
	\drawRuleSansConfig%
}

\newcommand{\RuleXXgFour}{
	\cfXXg%
	\renewcommand\MyEdgesRemoved{
     \tikzstyle{endpath}=[-]
        \myCEdge{v1}{v2}{blue,path}{}
        \myCEdge{v3}{v4}{blue}{}
        \myEdge{v2}{v3}%
        \myCEdge{v1}{v3}{green}{$\overline{Q}$}
        \myEdge{v1}{v4}
        \myEdge{v2}{v4}

        \myCurvEdge{v3}{v2}{blue}{}{in=180,out=90,path}
        \myHalfLongEdge{v1}{180}{blue,optedge,path}{}%
        \myHalfLongEdge{v4}{90}{blue,optedge,path}{}

        \path ($(v5) + (-0.2,0.1)$) edge [->, dashed, bend right = 0] ($0.5*(v2) + 0.5*(v3) + (-0.3,0.3)$);

        \draw ($(v3)+(0.6,0.6)$) node[blue,nodelabel] {$R_1''$};
        \draw ($(v4)+(-0.4,0.8)$) node[blue,nodelabel] {$Q_2$};
        \draw ($(v2)+(0,-0.6)$) node[blue,nodelabel] {$R_1'$};
        \draw ($(v1)+(-0.8,-0.3)$) node[blue,nodelabel] {$R_1$};
    }%
    \renewcommand\MyEdgesAfter{%
        \myEdge{v4}{v2}
        \myEdge{v4}{v1}
        \myCEdge{v1}{u1}{green}{$\overline{Q}$}
        \myCEdge{v3}{u1}{green}{$\overline{Q}$}

        \myHalfLongEdge{v1}{180}{red,optedge,path}{}%
        \myCEdge{v1}{v3}{blue}{$Q'$}
        \myCEdge{v2}{u2}{red}{$P$}
        \myCEdge{u2}{v4}{red}{$P$}
        \myHalfLongEdge{v4}{90}{red,optedge,path}{}
        \myCurvEdge{v3}{v2}{blue}{}{in=180,out=90,path}
        \myCEdge{u1}{v2}{blue}{$Q'$}
        \myCEdge{v1}{v2}{red,path}{}

        \myCEdge{u2}{v6}{blue,optedge}{$Q'$}%
        \myCEdge{v1}{u2}{blue}{$Q'$}
        \myCEdge{u1}{v5}{blue,optedge}{$Q'$}%
        \myCEdge{u1}{v2}{blue}{$Q'$}

        \draw ($(v3)+(0.6,0.6)$) node[blue,nodelabel] {$Q'$};
        \draw ($(v4)+(-0.3,0.8)$) node[red,nodelabel] {$P$};
        \draw ($(v2)+(0,-0.6)$) node[red,nodelabel] {$P\ $};
        \draw ($(v1)+(-0.8,-0.3)$) node[red,nodelabel] {$P$};

        \myEdge{v2}{v3}%
    }%
    \renewcommand{\CfN}{\nameXXgFour}%
	\drawRuleSansConfig%
}

\newcommand{\RuleXXgFive}{
	\cfXXg%
	\renewcommand\MyEdgesRemoved{
     \tikzstyle{endpath}=[-]
        \myCEdge{v1}{v2}{blue,path}{}
        \myCEdge{v3}{v4}{blue}{}
        \myEdge{v2}{v3}
        \myCEdge{v1}{v3}{green}{$\overline{Q}$}
        \myEdge{v1}{v4}
        \myEdge{v2}{v4}

        \myCurvEdge{v3}{v2}{blue}{}{in=180,out=90,path}
        \myHalfLongEdge{v1}{180}{blue,optedge,path}{}
        \myHalfLongEdge{v4}{90}{blue,optedge,path}{}

		\path ($(v5) + (-0.2,-0.2)$) edge [->, dashed, bend right = 80] ($0.35*(v2) + 0.65*(v1) + (-0.2,0)$);
        \path ($(v6) + (-0.2,-0.2)$) edge [->, dashed, bend left = 0] ($(v1) + (-0.6,0.2)$);

        \draw ($(v3)+(0.6,0.6)$) node[blue,nodelabel] {$R_1''$};
        \draw ($(v4)+(-0.4,0.8)$) node[blue,nodelabel] {$Q_2$};
        \draw ($(v2)+(0,-0.6)$) node[blue,nodelabel] {$R_1'$};
        \draw ($(v1)+(-0.8,-0.3)$) node[blue,nodelabel] {$R_1$};
    }%
    \renewcommand\MyEdgesAfter{%
        \myCEdge{v1}{u1}{green}{$\overline{Q}$}
        \myCEdge{v3}{u1}{green}{$\overline{Q}$}

        \myHalfLongEdge{v1}{180}{red,optedge,path}{}%
        \myCEdge{v2}{u2}{red}{$P$}
        \myCEdge{u2}{v4}{blue}{$Q'$}
        \myHalfLongEdge{v4}{90}{blue,optedge,path}{}
        \myCurvEdge{v3}{v2}{red}{$P$}{in=180,out=90,path}
        \myCEdge{v1}{v2}{blue,path}{}

        \myCEdge{v1}{v3}{red}{$P$}

        \myCEdge{v1}{u2}{blue}{$Q'$}
        \myCEdge{u2}{v6}{red,optedge}{$P$}%
        \myCEdge{u1}{v5}{blue,optedge}{$Q'$}%
        \myCEdge{u1}{v2}{blue}{$Q'$}

        \draw ($(v3)+(0.5,0.6)$) node[red,nodelabel] {$P$};
        \draw ($(v4)+(-0.3,0.8)$) node[blue,nodelabel] {$Q'$};
        \draw ($(v2)+(0,-0.6)$) node[blue,nodelabel] {$Q'\ $};
        \draw ($(v1)+(-0.8,-0.3)$) node[red,nodelabel] {$P$};

        \myEdge{v3}{v2}
        \myEdge{v4}{v2}
        \myEdge{v4}{v1}
    }%
    \renewcommand{\CfN}{\nameXXgFive}%
	\drawRuleSansConfig%
}

\newcommand{\RuleXXgSix}{
	\cfXXg%
	\renewcommand\MyEdgesRemoved{
     \tikzstyle{endpath}=[-]
        \myCEdge{v1}{v2}{blue,path}{}
        \myCEdge{v3}{v4}{blue}{}
        \myEdge{v2}{v3}%
        \myEdge{v1}{v3}%

        \myEdge{v1}{v4}
        \myEdge{v2}{v4}

        \myCurvEdge{v3}{v2}{blue}{}{in=180,out=90,path}
        \myHalfLongEdge{v1}{180}{blue,optedge,path}{}%
        \myHalfLongEdge{v4}{90}{blue,optedge,path}{}

		\path ($(v5) + (-0.2,-0.1)$) edge [->, dashed, bend right = 80] ($0.25*(v2) + 0.75*(v1) + (-0.2,0)$);
        \path ($(v6) + (0.2,0.1)$) edge [->, dashed, bend right = 80] ($0.25*(v1) + 0.75*(v2) + (0.2,0)$);

        \draw ($(v3)+(0.6,0.6)$) node[blue,nodelabel] {$R_1''$};
        \draw ($(v4)+(-0.4,0.8)$) node[blue,nodelabel] {$Q_2$};
        \draw ($(v2)+(0,-0.6)$) node[blue,nodelabel] {$R_1'$};
        \draw ($(v1)+(-0.8,-0.3)$) node[blue,nodelabel] {$R_1$};
    }%
    \renewcommand\MyEdgesAfter{%
        \myEdge{v3}{v1}
        \myEdge{v3}{v2}
        \myEdge{v4}{v2}
        \myEdge{v4}{v1}
        \myEdge{v1}{v3}%

        \myCEdge{v1}{u1}{blue}{$Q'$}
        \myCEdge{v3}{u1}{blue}{$Q'$}

        \myHalfLongEdge{v1}{180}{blue,optedge,path}{}%
        \myCEdge{v2}{u2}{blue}{$Q'$}
        \myCEdge{u2}{v4}{blue}{$Q'$}
        \myHalfLongEdge{v4}{90}{blue,optedge,path}{}
        \myCurvEdge{v3}{v2}{blue}{}{in=180,out=90,path}
        \myCEdge{u1}{v2}{red}{$P$}
        \myCEdge{v1}{v2}{red,path}{}

        \myCEdge{v1}{u2}{red}{$P$}
        \myCEdge{u2}{v6}{red,optedge}{$P$}%
        \myCEdge{u1}{v5}{red,optedge}{$P$}%

        \draw ($(v3)+(0.6,0.6)$) node[blue,nodelabel] {$Q'$};
        \draw ($(v4)+(-0.3,0.8)$) node[blue,nodelabel] {$Q'$};
        \draw ($(v2)+(0,-0.6)$) node[red,nodelabel] {$P\ $};
        \draw ($(v1)+(-0.8,-0.3)$) node[blue,nodelabel] {$Q'$};
    }%
    \renewcommand{\CfN}{\nameXXgSix}%
	\drawRuleSansConfig%
}

\newcommand{\cfXXa}{%
    \renewcommand\Mynodes{%
        \node[whitenode,label=below:\labelVOne] (v1) at (0,0.2) {};%
        \node[whitenode,label=above:\labelV] (v) at ($ (0,1) + (60:1) $) {};%
        \node[evennode,label=below:\labelVTwo] (v2) at (1,0.2) {};%
        \node[whitenode,label=below:\labelVThree] (v3) at (2,1) {};%
    }%
    \renewcommand\Myspecialnodes{%
        \node[blacknode,label=left:\labelUOne] (u1) at (0,1) {};%
		\node[blacknode,label={[label distance=-0.15cm]20:\labelUTwo}] (u2) at (1,1) {};%
    }%
    \renewcommand\MyEdgesBefore{%
        \myEdge{u1}{v1}%
        \myEdge{u1}{u2}%
        \myEdge{u2}{v2}%
        \myEdge{u1}{v}%
        \myEdge{u2}{v}%
        \myCEdge{u2}{v3}{optedge}{}%

    }%
    \renewcommand{\CfN}{\nameXXa}%
    \renewcommand{\CfL}{fig:Ca}%
}%

\newcommand{\relabeledXXa}[6]{
	\renewcommand{\labelUOne}{#1}
	\renewcommand{\labelUTwo}{#2}
	\renewcommand{\labelV}{#3}
	\renewcommand{\labelVOne}{#4}
	\renewcommand{\labelVTwo}{#5}
	\renewcommand{\labelVThree}{#6}
	\tkcf{\cfXXa}
	\resetLabels
}

\newcommand{\RuleXXa}{%
    \cfXXa
    \renewcommand\MyEdgesRemoved{
        \myHalfEdge{v2}{0}{blue}{$Q$}%
    }%
    \renewcommand\MyEdgesAfter{%
        \myCEdge{v1}{u1}{red}{$P$}
        \myCEdge{u1}{v}{red}{$P$}
        \myCEdge{v}{u2}{red}{$P$}
        \myHalfEdge{u2}{0}{optedge,red}{$P$}%
        \myCEdge{u1}{u2}{blue}{$Q$}
        \myCEdge{u2}{v2}{blue}{$Q$}
        \myHalfEdge{v2}{0}{blue}{$Q$}%
    }%
    \drawRule{fig:Ra}%
}%

\newcommand{\cfXXd}{%
    \renewcommand\Mynodes{
        \node[whitenode,label=below:\labelVOne] (v1) at (-0.5,0) {};
        \node[whitenode,label=below:\labelVTwo] (v2) at (1.5,0) {};
        \node[whitenode,label=below:\labelVThree] (v3) at (-1,1) {};
        \node[whitenode,label=below:\labelVFour] (v4) at (2,1) {};
        \node[whitenode,label=above:\labelV] (v) at ($ (0,1) + (60:1) $) {};
    }%
    \renewcommand\Myspecialnodes{%
        \node[blacknode,label={[label distance=-0.15cm]135:\labelUOne}] (u1) at (0,1) {};%
        \node[blacknode,label={[label distance=-0.15cm]20:\labelUTwo}] (u2) at (1,1) {};%
    }%
    \renewcommand\MyEdgesBefore{%
    	\myDEdge{v1}{v2}%
        \myEdge{u1}{v1}
        \myEdge{u1}{u2}
        \myEdge{u2}{v2}
        \myEdge{u1}{v}
        \myEdge{u2}{v}%
        \myCEdge{u1}{v3}{optedge}{}
        \myCEdge{u2}{v4}{optedge}{}
    }%
    \renewcommand{\CfN}{\nameXXd}%
    \renewcommand{\CfL}{fig:Cd}%
}%

\newcommand{\relabeledXXd}[7]{
	\renewcommand{\labelUOne}{#1}
	\renewcommand{\labelUTwo}{#2}
	\renewcommand{\labelV}{#3}
	\renewcommand{\labelVOne}{#4}
	\renewcommand{\labelVTwo}{#5}
	\renewcommand{\labelVThree}{#6}
	\renewcommand{\labelVFour}{#7}
	\tkcf{\cfXXd}
	\resetLabels
}

\newcommand{\RuleXXd}{%
    \cfXXd
    \renewcommand\MyEdgesRemoved{%
        \myCEdge{v1}{v2}{blue}{$Q$}%
    }%
    \renewcommand\MyEdgesAfter{%
    	\myDEdge{v1}{v2}%
        \myHalfEdge{u1}{180}{optedge,red}{$P$}%
        \myCEdge{u1}{v}{red}{$P$}
        \myCEdge{v}{u2}{red}{$P$}
        \myHalfEdge{u2}{0}{optedge,red}{$P$}%
        \myCEdge{v1}{u1}{blue}{$Q$}
        \myCEdge{u1}{u2}{blue}{$Q$}
        \myCEdge{u2}{v2}{blue}{$Q$}%
    }%
    \drawRule{fig:Rd}%
}%

\newcommand{\cfXXb}{%
    \renewcommand\Mynodes{
        \node[oddnode,label=above:$v_1$] (v1) at (-1,1) {};
        \node[oddnode,label=above:$v_2$] (v2) at (2,1) {};
        \node[oddnode,label=left:$v$] (v) at ($ (0,1) + (60:1) $) {};
    }%
    \renewcommand\Myspecialnodes{%
        \node[blacknode,label=below:$u_1$] (u1) at (0,1) {};
        \node[blacknode,label=below:$u_2$] (u2) at (1,1) {};
    }%
    \renewcommand\MyEdgesBefore{%
        \myEdge{u1}{v1}
        \myEdge{u1}{u2}
        \myEdge{u2}{v2}
        \myEdge{u1}{v}
        \myEdge{u2}{v}%
        \myEdge{v1}{v}
        \myEdge{v2}{v}%
    }%
    \renewcommand{\CfN}{\nameXXb}%
    \renewcommand{\CfL}{fig:Cb}%
}%

\newcommand{\cfXXbb}{%
    \renewcommand\Mynodes{
        \node[whitenode,label=above:$v_1$] (v1) at (-1,1) {};
        \node[whitenode,label=above:$v_2$] (v2) at (2,1) {};
        \node[oddnode,label=left:$v$] (v) at ($ (0,1) + (60:1) $) {};
    }%
    \renewcommand\Myspecialnodes{%
        \node[blacknode,label=below:$u_1$] (u1) at (0,1) {};
        \node[blacknode,label=below:$u_2$] (u2) at (1,1) {};
    }%
    \renewcommand\MyEdgesBefore{%
        \myEdge{u1}{v1}
        \myEdge{u1}{u2}
        \myEdge{u2}{v2}
        \myEdge{u1}{v}
        \myEdge{u2}{v}%
		\draw (v1) edge [thick, bend right = 70] (v2);
    }%
    \renewcommand{\CfN}{\nameXXbb}%
    \renewcommand{\CfL}{fig:Cbb}%
}%

\newcommand{\RuleXXbb}{%
    \cfXXbb
    \renewcommand\MyEdgesRemoved{%
        \myHalfEdge{v}{90}{green}{$Q$}%
		\draw (v1) edge [dashed, thick, bend right = 70] (v2);
    }%
    \renewcommand\MyEdgesAfter{%
        \myHalfEdge{v}{90}{green}{$Q$}%

        \myCEdge{u1}{v}{green}{$Q$}

        \myCEdge{v}{u2}{red}{$P$}
		\myCEdge{u2}{u1}{green}{$Q$}
        \myCEdge{u1}{v1}{red}{$P$}

		\myCEdge{u2}{v2}{red}{$P$}

		\draw (v1) edge [red, thick, bend right = 70] (v2);
		\draw ($(v1)!0.5!(v2)+(0,-0.87)$) node[red,nodelabel] {$\ P\ $};

    }%
    \drawRule{fig:Rbb}%
}%

\newcommand{\cfXXc}{%
    \renewcommand\Mynodes{
        \node[oddnode,label=above:$v_1$] (v1) at (-1,0.5) {};
        \node[whitenode,label=above:$v_2$] (v2) at (2,0.5) {};
        \node[evennode,label=left:$v$] (v) at ($ (0,0.5) + (60:1) $) {};
    }%
    \renewcommand\Myspecialnodes{%
        \node[blacknode,label=below:$u_1$] (u1) at (0,0.5) {};
        \node[blacknode,label=below:$u_2$] (u2) at (1,0.5) {};
    }%
    \renewcommand\MyEdgesBefore{%
        \myEdge{u1}{v1}
        \myEdge{u1}{u2}
        \myEdge{u2}{v2}
        \myEdge{u1}{v}
        \myEdge{u2}{v}%
        \myCEdge{v1}{v}{}{}%
    }%
    \renewcommand{\CfN}{\nameXXc}%
    \renewcommand{\CfL}{fig:Cc}%
}%
\newcommand{\RuleXXc}{%
    \cfXXc
    \renewcommand\MyEdgesRemoved{%
        \myCEdge{v1}{v}{dashed}{}%
        \myHalfEdge{v}{90}{green}{$Q$}%
    }%
    \renewcommand\MyEdgesAfter{%
        \myCEdge{u1}{u2}{green}{$Q$}
        \myCEdge{v2}{u2}{red}{$P$}
        \myCEdge{u2}{v}{red}{$P$}
        \myCEdge{v}{v1}{red}{$P$}
		\myCEdge{u1}{v1}{red}{$P$} 
        \myHalfEdge{v}{90}{green}{$Q$}%
        \myCEdge{v}{u1}{green}{$Q$}
    }%
    \drawRule{fig:Rc}%
}%

\newcommand{\cfTwoTwoN}{%
    \renewcommand\Mynodes{
        \node[whitenode,label=below:$v_1$] (v1) at (0,0.2) {};
        \node[whitenode,label=above:$v_2$] (v2) at (0,2.8) {};
    }%
    \renewcommand\Myspecialnodes{%
        \node[blacknode,label=below:$u_1$] (u1) at (-1.5,1.5) {};
        \node[blacknode,label=below:$u_2$] (u2) at (1.5,1.5) {};

    }%
    \renewcommand\MyEdgesBefore{%
    	\myCEdge{v1}{v2}{dashed}{}%
        \myEdge{u1}{v1}
        \myEdge{u1}{v2}
        \myEdge{u2}{v1}
        \myEdge{u2}{v2}
    }%
    \renewcommand{\CfN}{$(22n)$}%
    \renewcommand{\CfL}{fig:CtwoTwoN}%
}%

\newcommand{\cfXXs}{%
    \renewcommand\Mynodes{
        \node[whitenode,label=left:\labelVOne] (v1) at (0+1,0.2) {};
        \node[oddnode,label=left:\labelVTwo] (v2) at (0+1,2.8) {};
    }%
    \renewcommand\Myspecialnodes{%
        \node[blacknode,label=below:\labelUOne] (u1) at (-1.5+1,1.5) {};
        \node[blacknode,label=below:\labelUTwo] (u2) at (1.5+1,1.5) {};
    }%
    \renewcommand\MyEdgesBefore{%
        \myEdge{v2}{v1}
        \myEdge{u1}{v1}
        \myEdge{u1}{v2}
        \myEdge{u2}{v1}
        \myEdge{u2}{v2}
    }%
    \renewcommand{\CfN}{\nameXXr}%
    \renewcommand{\CfL}{fig:Cs}%
}%

\newcommand{\relabeledXXs}[4]{
	\renewcommand{\labelUOne}{#1}
	\renewcommand{\labelUTwo}{#2}
	\renewcommand{\labelVOne}{#3}
	\renewcommand{\labelVTwo}{#4}
	\tkcf{\cfXXs}
	\resetLabels
}

\newcommand{\RuleXXs}{%
    \cfXXs
    \renewcommand\MyEdgesRemoved{%
        \myHalfEdge{v2}{90}{blue}{$Q$}%
        \myCEdge{v1}{v2}{green}{$R$}
    }%
    \renewcommand\MyEdgesAfter{%
        \myCEdge{u2}{v1}{red}{$P$}
        \myCEdge{v1}{v2}{red}{$P$}

        \myHalfEdge{v2}{90}{blue}{$Q$}%
        \myCEdge{u2}{v2}{blue}{$Q$}

        \myCEdge{v1}{u1}{green}{$R$}
        \myCEdge{v2}{u1}{green}{$R$}
    }%
    \drawRule{fig:Rs}%
}%

\newcommand{\cfXXt}{%
    \renewcommand\Mynodes{
        \node[evennode,label=left:\labelVOne] (v1) at (0.5,0.2) {};
        \node[evennode,label=left:\labelVTwo] (v2) at (0.5,2.8) {};
    }%
    \renewcommand\Myspecialnodes{%
        \node[blacknode,label=below:\labelUOne] (u1) at (-1.5+0.5,1.5) {};
        \node[blacknode,label=below:\labelUTwo] (u2) at (1.5+0.5,1.5) {};
    }%
    \renewcommand\MyEdgesBefore{%
        \myHalfEdge{v2}{90}{}{}%
        \myHalfEdge{v1}{-90}{}{}%
        \myEdge{v2}{v1}
        \myEdge{u1}{v1}
        \myEdge{u1}{v2}
        \myEdge{u2}{v1}
        \myEdge{u2}{v2}
    }%
    \renewcommand{\CfN}{\nameXXs}%
    \renewcommand{\CfL}{fig:Ct}%
}%

\newcommand{\relabeledXXt}[4]{
	\renewcommand{\labelUOne}{#1}
	\renewcommand{\labelUTwo}{#2}
	\renewcommand{\labelVOne}{#3}
	\renewcommand{\labelVTwo}{#4}
	\tkcf{\cfXXt}
	\resetLabels
}

\newcommand{\RuleXXt}{%
    \cfXXt
    \renewcommand\MyEdgesRemoved{%
        \myHalfEdge{v2}{90}{blue}{$Q$}%
        \myHalfEdge{v1}{-90}{green}{$R$}%
        \myCEdge{v1}{v2}{dashed}{}
    }%
    \renewcommand\MyEdgesAfter{%
        \myHalfEdge{v2}{90}{blue}{$Q$}%
        \myHalfEdge{v1}{-90}{green}{$R$}%
        \myCEdge{u2}{v1}{red}{$P$}
        \myCEdge{v1}{v2}{red}{$P$}
        \myCEdge{u1}{v2}{red}{$P$}

        \myHalfEdge{v2}{90}{blue}{$Q$}%
        \myCEdge{u2}{v2}{blue}{$Q$}

        \myCEdge{v1}{u1}{green}{$R$}
    }%
    \drawRule{fig:Rt}%
}%

\newcommand{\cfXXu}{%
    \renewcommand\Mynodes{
        \node[whitenode,label=below:$v_1$] (v1) at (0,0+0.2) {};
        \node[whitenode,label=above:$v_2$] (v2) at (0,3+0.2) {};
        \node[whitenode,label=below:$v_3$] (v3) at (3,1.5+0.2) {};
        \node[whitenode,label=below:$v_4$] (v4) at (0.5,1.5+0.2) {};
    }%
    \renewcommand\Myspecialnodes{%
        \node[blacknode,label=below:$u_1$] (u1) at (-1.5,1.5+0.2) {};
        \node[blacknode,label=below:$u_2$] (u2) at (1.5,1.5+0.2) {};
    }%
    \renewcommand\MyEdgesBefore{%
    	\myCEdge{u2}{v4}{optedge}{}%
        \myCEdge{v1}{v2}{dashed}{}%
        \myEdge{u1}{v1}
        \myEdge{u1}{v2}
        \myEdge{u2}{v1}
        \myEdge{u2}{v2}
        \myEdge{u2}{v3}
        \myEdge{v3}{v2}
    }%
    \renewcommand{\CfN}{$(23n)$}%
    \renewcommand{\CfL}{fig:Cu}%
}%

\newcommand{\cfXXup}{%
    \renewcommand\Mynodes{
        \node[whitenode,label=below:\labelVOne] (v1) at (0,0+0.2) {};
        \node[whitenode,label=above:\labelVTwo] (v2) at (0,3+0.2) {};
        \node[whitenode,label=below:\labelVThree] (v3) at (3,1.5+0.2) {};
        \node[whitenode,label=below:\labelVFour] (v4) at (0.5,1.5+0.2) {};
    }%
    \renewcommand\Myspecialnodes{%
        \node[blacknode,label=below:\labelUOne] (u1) at (-1.5,1.5+0.2) {};
        \node[blacknode,label=below:\labelUTwo] (u2) at (1.5,1.5+0.2) {};
    }%
    \renewcommand\MyEdgesBefore{%
    	\myCEdge{u2}{v4}{optedge}{}%
        \myEdge{v1}{v2}%
        \myEdge{u1}{v1}%
        \myEdge{u1}{v2}%
        \myEdge{u2}{v1}%
        \myEdge{u2}{v2}%
        \myEdge{u2}{v3}%
        \myEdge{v3}{v2}%
    }%
    \renewcommand{\CfN}{\nameXXt}%
    \renewcommand{\CfL}{fig:Cup}%
}%

\newcommand{\relabeledXXup}[6]{
	\renewcommand{\labelUOne}{#1}
	\renewcommand{\labelUTwo}{#2}
	\renewcommand{\labelVOne}{#3}
	\renewcommand{\labelVTwo}{#4}
	\renewcommand{\labelVThree}{#5}
	\renewcommand{\labelVFour}{#6}
	\tkcf{\cfXXup}
	\resetLabels
}

\newcommand{\RuleXXup}{%
    \cfXXup%
    \renewcommand\MyEdgesRemoved{%
        \myCEdge{v1}{v2}{blue}{$Q$}%
        \myCEdge{v3}{v2}{green}{$R$}
    }%
    \renewcommand\MyEdgesAfter{%
        \myCEdge{v1}{v2}{red}{$P$}%
        \myCEdge{u2}{v2}{green}{$R$}%
        \myCEdge{v2}{v3}{red}{$P$}%
        \myCEdge{u1}{v1}{blue}{$Q$}%
        \myCEdge{u1}{v2}{blue}{$Q$}%
        \myHalfEdge{u2}{180}{red,optedge}{$P$}%
        \myCEdge{v1}{u2}{red}{$P$}%
        \myCEdge{u2}{v3}{green}{$R$}%
    }%
    \drawRule{fig:Rup}%
}%

\newcommand{\cfCExt}{%
    \renewcommand\Mynodes{%
        \node[bluenode,label=above:$v_1$] (v1) at ($(0,1) + (180:1)$) {};%
    }%
    \renewcommand\Myspecialnodes{%
        \node[blacknode,label=above:$u_1$] (u1) at (0,1) {};%
    }%
    \renewcommand\MyEdgesBefore{%
    	\mySubEdge{u1}{0}{}{}%
        \myEdge{u1}{v1}%
    }%
    \renewcommand{\CfN}{$(\C_{EXT})$}%
    \renewcommand{\CfL}{fig:CCext}%
}%
\newcommand{\RuleCExt}{%
    \cfCExt%
    \renewcommand\MyEdgesRemoved{}%
    \renewcommand\MyEdgesAfter{%
    	\mySubEdge{u1}{0}{}{}%
        \myCEdge{v1}{u1}{red}{$P$}%
    }%
    \drawRule{fig:RCext}%
}%

\newcommand{\cfCV}{%
    \renewcommand\Mynodes{%
        \node[purplenode,label=below:$v_1$] (v1) at ($(0,1) + (210:1)$) {};%
        \node[purplenode,label=above:$v_2$] (v2) at ($(0,1) + (150:1)$) {};%
    }%
    \renewcommand\Myspecialnodes{%
        \node[blacknode,label=above:$u_1$] (u1) at (0,1) {};%
    }%
    \renewcommand\MyEdgesBefore{%
		\mySubEdge{u1}{0}{}{}%
        \myEdge{u1}{v1}%
        \myEdge{u1}{v2}%
        \myCEdge{v1}{v2}{dashed}{}%
    }%
    \renewcommand{\CfN}{$(\C_V)$}%
    \renewcommand{\CfL}{fig:CCv}%
}%
\newcommand{\RuleCV}{%
    \cfCV%
    \renewcommand\MyEdgesRemoved{%
        \myCEdge{v1}{v2}{green}{$Q$}%
    }%
    \renewcommand\MyEdgesAfter{%
		\mySubEdge{u1}{0}{}{}%
        \myCEdge{v1}{u1}{green}{$Q$}%
        \myCEdge{u1}{v2}{green}{$Q$}%
        \myCEdge{v1}{v2}{dashed}{}%
    }%
    \drawRule{fig:RCV}%
}%

\newcommand{\cfCvTwo}{%
    \renewcommand\Mynodes{
        \node[bluenode,label=above:$v_2$] (v2) at ($ (0,1) + (150:1) $) {};
        \node[purplenode,evennode,label=below:$v_1$] (v1) at ($ (0,1) + (210:1) $) {};
    }%
    \renewcommand\Myspecialnodes{%
        \node[blacknode,label=above:$u_1$] (u1) at (0,1) {};
    }%
    \renewcommand\MyEdgesBefore{%
    	\mySubRedEdge{u1}{0}{}{}%

        \myEdge{u1}{v1}
        \myEdge{u1}{v2}
        \myCEdge{v1}{v2}{}{}%
    }%
    \renewcommand{\CfN}{$(\C_{Ne})$}%
    \renewcommand{\CfL}{fig:CCv2}%
}%
\newcommand{\RuleCvTwo}{%
    \cfCvTwo
    \renewcommand\MyEdgesRemoved{%
        \myCEdge{v1}{v2}{black}{}%
        \myHalfEdge{v1}{180}{orange}{$Q$}%
    }%
    \renewcommand\MyEdgesAfter{%
		\mySubRedEdge{u1}{0}{}{}

        \myCEdge{v1}{v2}{black}{}%
        \myHalfEdge{v1}{180}{orange}{$Q$}%
        \myCEdge{v1}{u1}{orange}{$Q$}%
        \myCEdge{u1}{v2}{red}{$P_1$}%
    }%
    \drawRule{fig:RCv2}%
}%

\newcommand{\cfCvThree}{%
    \renewcommand\Mynodes{
        \node[bluenode,label=above:$v_2$] (v2) at ($ (0,1) + (150:1) $) {};
        \node[bluenode,oddnode,label=below:$v_1$] (v1) at ($ (0,1) + (210:1) $) {};
    }%
    \renewcommand\Myspecialnodes{%
        \node[blacknode,label=above:$u_1$] (u1) at (0,1) {};
    }%
    \renewcommand\MyEdgesBefore{%
    	\mySubRedEdge{u1}{0}{}{}%
        \myEdge{u1}{v1}
        \myEdge{u1}{v2}
        \myCEdge{v1}{v2}{}{}%
    }%
    \renewcommand{\CfN}{$(\C_{No})$}%
    \renewcommand{\CfL}{fig:CCv3}%
}%
\newcommand{\RuleCvThree}{%
    \cfCvThree
    \renewcommand\MyEdgesRemoved{%
        \myHalfEdge{v1}{180}{orange}{$Q$}%
        \myDEdge{v1}{v2}

    }%
    \renewcommand\MyEdgesAfter{%
    	\mySubRedEdge{u1}{0}{}{}%
        \myCEdge{v1}{v2}{red}{$P_1$}%
        \myHalfEdge{v1}{180}{orange}{$Q$}%
        \myCEdge{v1}{u1}{orange}{$Q$}%
        \myCEdge{u1}{v2}{red}{$P_1$}%
    }%
    \drawRule{fig:RCv3}%
}%

\newcommand{\cfCvp}{%
    \renewcommand\Mynodes{
        \node[purplenode,label=below:$v_1$] (v1) at ($ (0,1) + (210:1) $) {};
        \node[purplenode,label=above:$v_2$] (v2) at ($ (0,1) + (150:1) $) {};

    }%
    \renewcommand\Myspecialnodes{%
        \node[blacknode,label=above:$u_1$] (u1) at (0,1) {};
    }%
    \renewcommand\MyEdgesBefore{%
    	\mySubEdge{u1}{0}{}{}%

        \myEdge{u1}{v1}
        \myEdge{u1}{v2}
        \myCEdge{v1}{v2}{sedgemixed}{$P$}%
    }%
    \renewcommand{\CfN}{$(\C_{V}')$}%
    \renewcommand{\CfL}{fig:CCvp}%
}%
\newcommand{\RuleCvp}{%
    \cfCvp
    \renewcommand\MyEdgesRemoved{%
        \myCEdge{v1}{v2}{green}{$Q$}%
    }%
    \renewcommand\MyEdgesAfter{%
		\mySubEdge{u1}{0}{}{}

        \myCEdge{v1}{u1}{green}{$Q$}
        \myCEdge{u1}{v2}{green}{$Q$}
        \myCEdge{v1}{v2}{sedgemixed}{$P$}%
    }%
    \drawRule{fig:RCvp}%
}%

\newcommand{\cfCVPlus}{%
    \renewcommand\Mynodes{%
        \node[whitenode,label=above:\labelVOne] (v1) at (1,2) {};
        \node (v2) at (2,1) {};%
        \node[whitenode,label=left:\labelVThree] (v3) at (1,0) {};%
        \node[whitenode,label=above:\labelVTwo] (v4) at (0,1) {};%
    }%
    \renewcommand\Myspecialnodes{%
		\node[blacknode,label={[label distance=-0.08cm]45:\labelUOne}] (u1) at (1,1) {};
    }%
    \renewcommand\MyEdgesBefore{%
        \myCEdge{u1}{v1}{optedge}{}%
        \myCEdge{u1}{v3}{optedge}{}%
        \myCEdge{u1}{v4}{optedge}{}%
        \myCEdge{v2}{u1}{sedge}{}%
        \myDEdge{v1}{v4}%
    }%
    \renewcommand{\CfN}{$(\C_{V}^+)$}%
    \renewcommand{\CfL}{fig:CVP}%
}%

\newcommand{\relabeledCVPlus}[5]{
	\renewcommand{\labelUOne}{#1}
	\renewcommand{\labelVOne}{#2}
	\renewcommand{\labelVTwo}{#3}
	\renewcommand{\labelVThree}{#4}
	\renewcommand{\labelVFour}{#5}
	\tkcf{\cfCVPlus}
	\resetLabels
}

\newcommand{\cfCFoura}{%
    \renewcommand\Mynodes{
        \node[purplenode,label=above:$v_1$] (v1) at (1,2) {};
        \node (v2) at (2,1) {};
        \node[whitenode,label=below:$v_3$] (v3) at (1,0) {};
        \node[purplenode,label=above:$v_2$] (v4) at (0,1) {};
    }%
    \renewcommand\Myspecialnodes{%
        \node[blacknode,label={[label distance=-0.08cm]45:$u_1$}] (u1) at (1,1) {};
    }%
    \renewcommand\MyEdgesBefore{%
        \myEdge{u1}{v1}
        \myEdge{u1}{v3}
        \myEdge{u1}{v4}
        \myCEdge{v1}{v4}{dashed}{}%
        \myCEdge{v2}{u1}{sedge}{}%
    }%
    \renewcommand{\CfN}{$(\C_{4a})$}%
    \renewcommand{\CfL}{fig:C4a}%
}%
\newcommand{\RuleCFoura}{%
    \cfCFoura

    \renewcommand\MyEdgesRemoved{%
        \myCEdge{v1}{v4}{blue}{$Q$}%
    }%
    \renewcommand\MyEdgesAfter{%
        \myCEdge{v1}{u1}{blue}{$Q$}%
        \myCEdge{v4}{u1}{blue}{$Q$}%
        \myCEdge{v2}{u1}{red}{$P$}%
        \myCEdge{v3}{u1}{red}{$P$}%
        \myCEdge{v1}{v4}{dashed}{}%
    }%
    \drawRule{fig:R4a}%
}%

\newcommand{\cfCFourb}{%
    \renewcommand\Mynodes{
        \node[whitenode,label=above:$v_1$] (v1) at (1,2) {};
        \node (v2) at (2,1) {};
        \node[whitenode,label=below:$v_3$] (v3) at (1,0) {};
        \node[whitenode,label=above:$v_2$] (v4) at (0,1) {};
    }%
    \renewcommand\Myspecialnodes{%
		\node[blacknode,label={[label distance=-0.08cm]45:$u_1$}] (u1) at (1,1) {};
    }%
    \renewcommand\MyEdgesBefore{%
        \myEdge{u1}{v1}
        \myEdge{u1}{v3}
        \myEdge{u1}{v4}
        \myCEdge{v1}{v4}{}{}%
        \myCEdge{v3}{v4}{}{}%
        \myCEdge{v2}{u1}{sedge}{}%
    }%
    \renewcommand{\CfN}{$(\C_{4b})$}%
    \renewcommand{\CfL}{fig:C4b}%
}%
\newcommand{\RuleCFourb}{%
    \cfCFourb

    \renewcommand\MyEdgesRemoved{%
        \myCEdge{v1}{v4}{blue}{$Q$}%
        \myDEdge{v3}{v4}%
    }%
    \renewcommand\MyEdgesAfter{%
        \myCEdge{v1}{u1}{blue}{$Q$}%
        \myCEdge{v4}{u1}{blue}{$Q$}%
        \myCEdge{v2}{u1}{sedge}{}%
        \myCEdge{v3}{u1}{red}{$P$}%
        \myCEdge{v3}{v4}{red}{$P$}%
        \myCEdge{v4}{v1}{red}{$P$}%
    }%
    \drawRule{fig:R4b}%
}%

\newcommand{\cfCOnea}{%
    \renewcommand\Mynodes{
    }%
    \renewcommand\Myspecialnodes{%
        \node[blacknode,label=above:$u_1$] (u1) at (0,1) {};
    }%
    \renewcommand\MyEdgesBefore{
    \myHalfEdge{u1}{0}{sedge}{}}%
    \renewcommand{\CfN}{$(\C_{1a})$}%
    \renewcommand{\CfL}{fig:C1a}%
}%
\newcommand{\cfCU}{%
    \renewcommand\Mynodes{
        \node[purplenode,label=above:$v_2$] (v2) at ($ (1,1) + (60:1) $) {};
        \node[purplenode,label=above:$v_1$] (v1) at ($ (0,1) + (120:1) $) {};
    }%
    \renewcommand\Myspecialnodes{%
        \node[blacknode,label=left:$u_1$] (u1) at (0,1) {};
        \node[blacknode,label=right:$u_2$] (u2) at (1,1) {};
    }%
    \renewcommand\MyEdgesBefore{%
		\mySubEdge{u1}{-90}{}{}%
        \mySubEdge{u2}{-90}{}{}

        \myEdge{u1}{v1}
        \myEdge{u2}{v2}
        \myEdge{u1}{u2}
        \myCEdge{v1}{v2}{dashed}{}%
    }%
    \renewcommand{\CfN}{$(\C_U)$}%
    \renewcommand{\CfL}{fig:CU}%
}%
\newcommand{\RuleCU}{%
    \cfCU
    \renewcommand\MyEdgesRemoved{%
        \myCEdge{v1}{v2}{green}{$Q$}%
    }%
    \renewcommand\MyEdgesAfter{%
		\mySubEdge{u1}{-90}{}{}%
        \mySubEdge{u2}{-90}{}{}

        \myCEdge{v1}{u1}{green}{$Q$}%
        \myCEdge{u1}{u2}{green}{$Q$}%
        \myCEdge{u2}{v2}{green}{$Q$}%
        \myCEdge{v1}{v2}{dashed}{}%
    }%
    \drawRule{fig:RCU}%
}%

\newcommand{\cfCDOnenea}{%
    \renewcommand\Mynodes{
        \node[whitenode,label=right:$v_1$] (v1) at (1.5,1.5) {};
        \node[whitenode,label=right:$v_2$] (v2) at (1.5,0.5) {};
        \node[whitenode,label=above:$v$] (v) at (0.75, 1) {};
    }%
    \renewcommand\Myspecialnodes{%
        \node[blacknode,label=left:$u_1$] (u1) at (0,2) {};
        \node[blacknode,label=left:$u_2$] (u2) at (0,0) {};
    }%
    \renewcommand\MyEdgesBefore{%
		\myHalfEdge{u1}{135}{sedgemixed}{$S$}%
        \myHalfEdge{u1}{20}{sedgemixed}{$S$}%
        \myHalfEdge{u2}{-135}{sedgemixed}{$S$}%
        \myHalfEdge{u2}{-20}{sedgemixed}{$S$}%
        \myCEdge{u1}{u2}{sedgemixed}{$S$}

        \myEdge{u1}{v}
        \myEdge{v}{v2}
        \myEdge{v}{v1}
        \myEdge{v1}{v2}
        \myEdge{u1}{v1}
        \myEdge{u2}{v}
        \myEdge{u2}{v2}
    }%
    \renewcommand{\CfN}{$(\C_{Da})$}%
    \renewcommand{\CfL}{fig:CCD1a}%
}%
\newcommand{\RuleCDOnenea}{%
    \cfCDOnenea
    \renewcommand\MyEdgesRemoved{%
        \myCEdge{v}{v1}{green}{$Q$}%
        \myEdge{v}{v2}%
        \myDEdge{v1}{v2}%
    }%
    \renewcommand\MyEdgesAfter{%
		\myHalfEdge{u1}{135}{sedgemixed}{$S$}%
        \myHalfEdge{u1}{20}{sedgemixed}{$S$}%
        \myHalfEdge{u2}{-135}{sedgemixed}{$S$}%
        \myHalfEdge{u2}{-20}{sedgemixed}{$S$}%

        \myCEdge{v}{u2}{green}{$Q$}%
        \myCEdge{u2}{u1}{green}{$Q$}%
        \myCEdge{u1}{v1}{green}{$Q$}%
        \myEdge{v}{v2}%
        \myCEdge{u1}{v}{sedgemixed}{$S$}%
        \myCEdge{v}{v1}{sedgemixed}{$S$}%
        \myCEdge{v1}{v2}{sedgemixed}{$S$}%
        \myCEdge{v2}{u2}{sedgemixed}{$S$}%
    }%
    \drawRuleSansFigure{1.25}
}%

\newcommand{\cfCDOneneb}{%
    \renewcommand\Mynodes{
        \node[purplenode,label=right:$v_1$] (v1) at (1.5,1.5) {};
        \node[purplenode,label=right:$v_2$] (v2) at (1.5,0.5) {};
        \node[whitenode,label=above:$v$] (v) at (0.75, 1) {};
    }%
    \renewcommand\Myspecialnodes{%
        \node[blacknode,label=left:$u_1$] (u1) at (0,2) {};
        \node[blacknode,label=left:$u_2$] (u2) at (0,0) {};
    }%
    \renewcommand\MyEdgesBefore{%
		\myHalfEdge{u1}{135}{sedgemixed}{$S$}%
        \myHalfEdge{u1}{20}{sedgemixed}{$S$}%
        \myHalfEdge{u2}{-135}{sedgemixed}{$S$}%
        \myHalfEdge{u2}{-20}{sedgemixed}{$S$}%
        \myCEdge{u1}{u2}{sedgemixed}{$S$}

        \myEdge{u1}{v}
        \myEdge{v}{v2}
        \myEdge{v}{v1}
        \myDEdge{v1}{v2}
        \myEdge{u1}{v1}
        \myEdge{u2}{v}
        \myEdge{u2}{v2}
    }%
    \renewcommand{\CfN}{$(\C_{Db})$}%
    \renewcommand{\CfL}{fig:CCD1b}%
}%
\newcommand{\RuleCDOneneb}{%
    \cfCDOneneb
    \renewcommand\MyEdgesRemoved{%
        \myCEdge{v1}{v2}{green}{$Q$}%
        \myEdge{v}{v1}%
        \myEdge{v}{v2}%
    }%
    \renewcommand\MyEdgesAfter{%
		\myHalfEdge{u1}{135}{sedgemixed}{$S$}%
        \myHalfEdge{u1}{20}{sedgemixed}{$S$}%
        \myHalfEdge{u2}{-135}{sedgemixed}{$S$}%
        \myHalfEdge{u2}{-20}{sedgemixed}{$S$}%

        \myCEdge{v}{u2}{sedgemixed}{$S$}%
        \myCEdge{u2}{u1}{green}{$Q$}%
        \myCEdge{u1}{v1}{green}{$Q$}%
        \myEdge{v}{v2}%
        \myCEdge{u1}{v}{sedgemixed}{$S$}%
        \myEdge{v}{v1}%
        \myDEdge{v1}{v2}%
        \myCEdge{v2}{u2}{green}{$Q$}%
    }%
    \drawRuleSansFigure{1.25}
}%

\newcommand{\cfCDOnenec}{%
    \renewcommand\Mynodes{
		\node[whitenode,label={[label distance=-0.06cm]-90:$v_1$}] (v1) at (0.75,1.5) {};
        \node[whitenode,label={[label distance=-0.06cm]90:$v_2$}] (v2) at (0.75,0.5) {};
        \node[whitenode,label=above:$v$] (v) at (1.5,1) {};
        \node[whitenode,label=left:$w_1/w_2$] (v') at (0,1) {};
    }%
    \renewcommand\Myspecialnodes{%
        \node[blacknode,label=left:$u_1$] (u1) at (0,2) {};
        \node[blacknode,label=left:$u_2$] (u2) at (0,0) {};
    }%
    \renewcommand\MyEdgesBefore{%
		\myHalfEdge{u1}{120}{sedgemixed}{$S$}%
        \myHalfEdge{u1}{20}{sedgemixed}{$S$}%
        \myHalfEdge{u2}{-120}{sedgemixed}{$S$}%
        \myHalfEdge{u2}{-20}{sedgemixed}{$S$}%
        \myCEdge{u1}{v'}{sedgemixed}{$S$}
        \myCEdge{v'}{u2}{sedgemixed}{$S$}

        \myBentEdge{u1}{v}{black}{left}{30}
        \myEdge{v}{v2}
        \myEdge{v}{v1}
        \myEdge{u1}{v1}
        \myBentEdge{u2}{v}{black}{right}{30}
        \myEdge{u2}{v2}
        \myDEdge{v'}{v1}
    }%
    \renewcommand{\CfN}{$(\C_{D1c})$}%
    \renewcommand{\CfL}{fig:CCD1c}%
}%
\newcommand{\RuleCDOnenec}{%
    \cfCDOnenec
    \renewcommand\MyEdgesAfter{%
		\myHalfEdge{u1}{120}{sedgemixed}{$S$}%
        \myHalfEdge{u1}{20}{sedgemixed}{$S$}%
        \myHalfEdge{u2}{-120}{sedgemixed}{$S$}%
        \myHalfEdge{u2}{-20}{sedgemixed}{$S$}%

		\patternCV{u1}{v1}{v'}{green}%

        \myBentEdge{u1}{v}{sedgemixed}{left}{30}
        \draw ($(v1)+(0.4,0.1)$) node[\myPurple,nodelabel] {$S$};
        \myEdge{v}{v2}
        \myEdge{v}{v1}
        \myBentEdge{u2}{v}{sedgemixed}{right}{30}
        \draw ($(v2)+(0.4,-0.1)$) node[\myPurple,nodelabel] {$S$};
        \myEdge{u2}{v2}
        \patternCN{u2}{v'}{v2}{red}
    }%
	\drawSimpleRuleSansFigure{1.1}
}%

\newcommand{\cfCDOnened}{%
    \renewcommand\Mynodes{
        \node[whitenode,label={[label distance=-0.06cm]-90:$v_1$}] (v1) at (0.75,1.5) {};
        \node[whitenode,label={[label distance=-0.06cm]90:$v_2$}] (v2) at (0.75,0.5) {};
        \node[whitenode,label=above:$v$] (v) at (1.5,1) {};
        \node[whitenode,label=left:$w_1/w_2$] (v') at (0,1) {};
    }%
    \renewcommand\Myspecialnodes{%
        \node[blacknode,label=left:$u_1$] (u1) at (0,2) {};
        \node[blacknode,label=left:$u_2$] (u2) at (0,0) {};
    }%
    \renewcommand\MyEdgesBefore{%
		\myHalfEdge{u1}{120}{sedgemixed}{$S$}%
        \myHalfEdge{u1}{20}{sedgemixed}{$S$}%
        \myHalfEdge{u2}{-120}{sedgemixed}{$S$}%
        \myHalfEdge{u2}{-20}{sedgemixed}{$S$}%
        \myCEdge{u1}{v'}{sedgemixed}{$S$}
        \myCEdge{v'}{u2}{sedgemixed}{$S$}

        \myBentEdge{u1}{v}{black}{left}{30}
        \myEdge{v}{v2}
        \myEdge{v}{v1}
        \myEdge{u1}{v1}
        \myBentEdge{u2}{v}{black}{right}{30}
        \myEdge{u2}{v2}
        \myEdge{v'}{v1}
        \myEdge{v'}{v2}
        \myDEdge{v}{v'}
    }%
    \renewcommand{\CfN}{$(\C_{D1d})$}%
    \renewcommand{\CfL}{fig:CCD1d}%
}%
\newcommand{\RuleCDOnened}{%
    \cfCDOnened
    \renewcommand\MyEdgesAfter{%
		\myHalfEdge{u1}{120}{sedgemixed}{$S$}%
        \myHalfEdge{u1}{20}{sedgemixed}{$S$}%
        \myHalfEdge{u2}{-120}{sedgemixed}{$S$}%
        \myHalfEdge{u2}{-20}{sedgemixed}{$S$}%

		\myCEdge{u1}{v'}{green}{}%
		\myBentEdge{u1}{v}{green}{left}{30}%
		\node at ($(u1)+(1.2,0)$) {\textcolor{green}{\CV}};

        \myCEdge{u1}{v1}{sedgemixed}{$S$}
        \myEdge{v}{v2}
        \myCEdge{v}{v1}{sedgemixed}{$S$}
        \myBentEdge{u2}{v}{sedgemixed}{right}{30}
        \draw ($(v2)+(0.4,-0.1)$) node[\myPurple,nodelabel] {$S$};
        \myEdge{u2}{v2}
        \myEdge{v'}{v1}
        \patternCN{u2}{v'}{v2}{red}
        \myDEdge{v}{v'}
    }%
	\drawSimpleRuleSansFigure{1.1}
}%

\newcommand{\cfCDOnenee}{%
    \renewcommand\Mynodes{
        \node[whitenode,label=below:$v_1$] (v1) at (1.5,1.5) {};
        \node[whitenode,label=below:$v_2$] (v2) at (1.5,0.5) {};
        \node[whitenode,label=left:$v$] (v) at (0.75, 1) {};
        \node[whitenode,label=above:$w_1$] (v1'') at (0.75,2.2) {};
    }%
    \renewcommand\Myspecialnodes{%
        \node[blacknode,label=left:$u_1$] (u1) at (0,2) {};
        \node[blacknode,label=left:$u_2$] (u2) at (0,0) {};
    }%
    \renewcommand\MyEdgesBefore{%
		\myHalfEdge{u1}{120}{sedgemixed}{$S$}%
        \myCEdge{u1}{v1''}{sedgemixed}{$S$}%
        \myHalfEdge{u2}{-120}{sedgemixed}{$S$}%
        \myHalfEdge{u2}{-20}{sedgemixed}{$S$}%
        \myPath{u1}{u2}{\subColor}{left}{0}

        \myPath{v1''}{v2}{sedgemixed}{left}{70}
        \draw ($(v1)+(0,0.4)$) node[\myPurple,nodelabel] {$S$};
        \myHalfEdge{v2}{20}{sedgemixed}{$S$}%

        \myEdge{u1}{v}
        \myEdge{v}{v2}
        \myEdge{v}{v1}
        \myEdge{u1}{v1}
        \myEdge{u2}{v}
        \myEdge{u2}{v2}
        \myDEdge{v1''}{v}
    }%
    \renewcommand{\CfN}{$(\C_{D1+})$}%
    \renewcommand{\CfL}{fig:CCD1e}%
}%
\newcommand{\RuleCDOnenee}{%
    \cfCDOnenee
    \renewcommand\MyEdgesAfter{%
		\myHalfEdge{u1}{120}{sedgemixed}{$S$}%

        \myHalfEdge{u2}{-120}{sedgemixed}{$S$}%
        \myHalfEdge{u2}{-20}{sedgemixed}{$S$}%
		\patternCV{u1}{v}{v1''}{orange}%

    	\myPath{v1''}{v2}{black}{left}{70}
        \myHalfEdge{v2}{20}{sedgemixed}{$S$}%

        \myCEdge{v}{u2}{green}{}
		\myPath{u1}{u2}{\subColor}{left}{0}

        \myCEdge{u1}{v1}{sedgemixed}{}
        \draw ($(v)!0.5!(v1)+(0,0.4)$) node[\myPurple,nodelabel] {$S$};
        \myCEdge{v}{v2}{sedgemixed}{$S$}%
        \myCEdge{v}{v1}{sedgemixed}{$S$}%
        \myCEdge{v2}{u2}{green}{}

        \node at ($(v)+(0.05,-0.45)$) {\textcolor{green}{\CVp}};
    }%
	\drawSimpleRuleSansFigure{1.1}
}%

\newcommand{\cfCDOneneeBis}{%
    \renewcommand\Mynodes{
        \node[whitenode,label=below:$v_1$] (v1) at (1.2,1.5) {};
        \node[whitenode,label=below:$v_2$] (v2) at (1.5,0.5) {};
        \node[whitenode,label=left:$v$] (v) at (0.75, 1) {};
        \node[whitenode,label=above:$w_1$] (v1'') at (0.9,2.4) {};
    }%
    \renewcommand\Myspecialnodes{%
        \node[blacknode,label=left:$u_1$] (u1) at (0,2) {};
        \node[blacknode,label=left:$u_2$] (u2) at (0,0) {};
    }%
    \renewcommand\MyEdgesBefore{%
		\myHalfEdge{u1}{120}{sedgemixed}{$P$}%
        \myCEdge{u1}{v1''}{sedgemixed}{$P$}%
        \myHalfEdge{u2}{-120}{sedgemixed}{$P$}%
        \myHalfEdge{u2}{-20}{sedgemixed}{$P$}%
        \myPath{u1}{u2}{\subColor}{left}{0}

        \myPath{v1''}{v2}{\subColor}{left}{70}
        \draw ($(v1)+(0.6,0.4)$) node[\myPurple,nodelabel] {};
        \myHalfEdge{v2}{10}{sedgemixed}{$S$}%

        \myEdge{u1}{v}
        \myEdge{v}{v2}
        \myEdge{v}{v1}
        \myEdge{u1}{v1}
        \myEdge{u2}{v}
        \myEdge{u2}{v2}

        \myDEdge{v1}{v1''}
    }%
    \renewcommand{\CfN}{$(\C_{D1+})$}%
    \renewcommand{\CfL}{fig:CCD1e}%
}%
\newcommand{\RuleCDOneneeBis}{%
    \cfCDOneneeBis
    \renewcommand\MyEdgesAfter{%
		\myHalfEdge{u1}{120}{sedgemixed}{$S$}%

        \myHalfEdge{u2}{-120}{sedgemixed}{$S$}%
        \myHalfEdge{u2}{-20}{sedgemixed}{$S$}%
		\patternCV{u1}{v1''}{v1}{orange}

    	\myPath{v1''}{v2}{black}{left}{70}
        \myHalfEdge{v2}{10}{sedgemixed}{$S$}%

        \myCEdge{v}{u2}{green}{}
		\myPath{u1}{u2}{\subColor}{left}{0}

        \myCEdge{u1}{v}{sedgemixed}{$S$}%
        \myCEdge{v}{v2}{sedgemixed}{$S$}%
		\myEdge{v}{v1}
        \myCEdge{v2}{u2}{green}{}

        \node at ($(v)+(0.05,-0.45)$) {\textcolor{green}{\CVp}};
    }%
	\drawSimpleRuleSansFigure{1.1}
}%

\newcommand{\cfCDTwowod}{%
    \renewcommand\Mynodes{
        \node[whitenode,label=above right:$v$] (v) at (0.75,1) {};
        \node[whitenode,label=above right:$v'$] (v') at (1.5,1) {};
        \node[whitenode,label=left:$w_1$] (v1) at (-0.5,1.4) {};
        \node[whitenode,label=left:$w_2$] (v2) at (-0.5,0.6) {};
    }%
    \renewcommand\Myspecialnodes{%
        \node[blacknode,label=left:$u_1$] (u1) at (0,2) {};
        \node[blacknode,label=left:$u_2$] (u2) at (0,0) {};
    }%
    \renewcommand\MyEdgesBefore{%
		\myHalfEdge{u1}{120}{sedgemixed}{$S$}%
        \myHalfEdge{u1}{20}{sedgemixed}{$S$}%
        \myHalfEdge{u2}{-120}{sedgemixed}{$S$}%
        \myHalfEdge{u2}{-20}{sedgemixed}{$S$}%

        \myCEdge{u1}{v1}{sedgemixed}{$S$}
        \myCEdge{u2}{v2}{sedgemixed}{$S$}
        \myPath{v1}{v2}{\subColor}{left}{0}

        \myEdge{u1}{v}
        \myBentEdge{u1}{v'}{black}{left}{40}
        \myEdge{u2}{v}
        \myBentEdge{u2}{v'}{black}{right}{40}
        \myDEdge{v1}{v}

        \myEdge{v}{v'}
    }%
    \renewcommand{\CfN}{$(\C_{D2d})$}%
    \renewcommand{\CfL}{fig:CCD2d}%
}%
\newcommand{\RuleCDTwowod}{%
    \cfCDTwowod
    \renewcommand\MyEdgesRemoved{
        \myCEdge{v1}{v}{green}{$Q$}
        \myPath{v1}{v2}{black}{left}{0}
        \myEdge{v}{v'}
    }%
    \renewcommand\MyEdgesAfter{%
		\myHalfEdge{u1}{120}{sedgemixed}{$S$}%
        \myHalfEdge{u1}{20}{sedgemixed}{$S$}%
        \myHalfEdge{u2}{-120}{sedgemixed}{$S$}%
        \myHalfEdge{u2}{-20}{sedgemixed}{$S$}%

        \myEdge{v}{v'}

		\patternCV{u1}{v}{v1}{green}
        \patternCN{u2}{v2}{v}{red}
        \myPath{v1}{v2}{black}{left}{0}

        \myBentEdge{u1}{v'}{sedgemixed}{left}{40}
        \myBentEdge{u2}{v'}{sedgemixed}{right}{40}
        \draw ($(v)!0.5!(v')+(0,0.7)$) node[\myPurple,nodelabel] {$S$};
        \draw ($(v)!0.5!(v')+(0,-0.7)$) node[\myPurple,nodelabel] {$S$};
    }%
    \drawSimpleRuleSansFigure{1.1}
}%

\newcommand{\cfCH}{%
    \renewcommand\Myspecialnodes{
        \node[blacknode,label=left:$u_1$] (u1) at (-0.5,3) {};
        \node[blacknode,label=below:$u_2$] (u2) at (-0.5,0) {};
        \node[blacknode,label=right:$u_3$] (u3) at (4,3) {};
        \node[blacknode,label=right:$u_4$] (u4) at (4,0) {};
    }
    \renewcommand\Mynodes{

    	\node[whitenode,label=below:$v_1$] (v1') at (1,1) {};
    	\node[whitenode,label=left:$v_2$] (v2') at (-2,2) {};
    }
    \renewcommand\MyEdgesBefore{
        \myPath{u1}{v2'}{\subColor}{left}{0}%
        \myPath{v2'}{u3}{\subColor}{left}{80}%
        \myPath{u1}{u3}{\subColor}{left}{0}%
        \myPath{v1'}{u2}{\subColor}{left}{0}%
        \myPath{u4}{v1'}{\subColor}{left}{20}%
        \myPath{u2}{u4}{\subColor}{right}{50}%
        \myPath{u1}{u2}{\subColor}{left}{0}%
        \myPath{u3}{u4}{\subColor}{left}{0}%

        \myEdge{u1}{v1'}
        \myEdge{u2}{v2'}

		\draw ($(v1')+(1.25,-0.9)$) node[\myPurple,nodelabel] {$P_2$};
		\draw ($(v1')+(-0.6,-0.6)$) node[\myPurple,nodelabel] {$P_2$};
		\draw ($(v2')+(0.9,0.4)$) node[\myPurple,nodelabel] {$P_1$};
		\draw ($(v2')+(2.5,2.45)$) node[\myPurple,nodelabel] {$P_1$};
    }%
    \renewcommand{\CfN}{$(\C_H)$}%
    \renewcommand{\CfL}{fig:CCH}%
}%
\newcommand{\RuleCH}{%
    \cfCH
    \renewcommand\MyEdgesAfter{%
        \myPath{u1}{v2'}{black}{left}{0}%
        \myPath{v2'}{u3}{\subColor}{left}{80}%
        \myPath{u1}{u3}{\subColor}{left}{0}%
        \myPath{v1'}{u2}{black}{left}{0}%
        \myPath{u4}{v1'}{\subColor}{left}{20}%
        \myPath{u2}{u4}{\subColor}{right}{50}%
        \myPath{u1}{u2}{\subColor}{left}{0}%
        \myPath{u3}{u4}{\subColor}{left}{0}%

        \myBentEdge{u1}{v1'}{\subColor}{left}{0}
        \myBentEdge{u2}{v2'}{\subColor}{left}{0}

		\draw ($(v1')!0.5!(u1) + (0.1,0.1)$) node[\myPurple,nodelabel] {$P_1'$};
		\draw ($(v2')!0.5!(u2) + (-0.1,-0.1)$) node[\myPurple,nodelabel] {$P_2'$};
		\draw ($(v1')+(1.25,-0.9)$) node[\myPurple,nodelabel] {$P_1'$};
		\draw ($(v2')+(2.5,2.45)$) node[\myPurple,nodelabel] {$P_2'$};
    }%
    \drawSimpleRuleSansFigure{2.8}
}%

\newcommand{\cfCTOne}{%
    \renewcommand\Mynodes{
        \node[whitenode,bluenode,label=above:$v_1$] (v1) at (-1,1+0.5) {};
        \node[whitenode,label=above:$v_2$] (v2) at (1,1+0.5) {};
        \node[whitenode,label=above left:$v$] (v) at (0, 1+0.5) {};
    }%
    \renewcommand\Myspecialnodes{%
        \node[blacknode,label=left:$u_1$] (u1) at (0,0+0.5) {};
        \node[blacknode,label=left:$u_2$] (u2) at (0,2+0.5) {};
    }%
    \renewcommand\MyEdgesBefore{%
		\mySubRedEdge{u1}{-90}{}{}%
        \mySubBlueEdge{u2}{90}{}{}

        \myEdge{u1}{v}
        \myEdge{u2}{v}
        \myEdge{u2}{v2}
        \myEdge{u1}{v1}
        \myEdge{v}{v1}
        \myEdge{v}{v2}
    }%
    \renewcommand{\CfN}{$(\C_{T1})$}%
    \renewcommand{\CfL}{fig:CTOne}%
}%

\newcommand{\cfCTOneb}{%
    \renewcommand\Mynodes{
        \node[bluenode,oddnode,label=above:$v_1$] (v1) at (-1,1+0.5) {};
        \node[rednode,oddnode,label=above:$v_2$] (v2) at (1,1+0.5) {};
        \node[whitenode,label=above left:$v$] (v) at (0, 1+0.5) {};
    }%
    \renewcommand\Myspecialnodes{%
        \node[blacknode,label=left:$u_1$] (u1) at (0,0+0.5) {};
        \node[blacknode,label=left:$u_2$] (u2) at (0,2+0.5) {};
    }%
    \renewcommand\MyEdgesBefore{%
		\mySubRedEdge{u1}{-90}{}{}%
        \mySubBlueEdge{u2}{90}{}{}

        \myEdge{u1}{v}
        \myEdge{u2}{v}
        \myEdge{u2}{v2}
        \myEdge{u1}{v1}
        \myEdge{v}{v1}
        \myEdge{v}{v2}
    }%
    \renewcommand{\CfN}{$(\C_{T1b})$}%
    \renewcommand{\CfL}{fig:CTOneb}%
}%
\newcommand{\RuleCTOneb}{%
    \cfCTOneb
    \renewcommand\MyEdgesRemoved{%
        \myDEdge{v}{v1}
        \myDEdge{v}{v2}
        \myHalfEdge{v1}{180}{green}{$Q$}%
        \myHalfEdge{v2}{0}{orange}{$R$}%
    }%
    \renewcommand\MyEdgesAfter{%
		\mySubRedEdge{u1}{-90}{}{}%
        \mySubBlueEdge{u2}{90}{}{}

        \myHalfEdge{v2}{0}{orange}{$R$}%
        \myCEdge{u2}{v2}{orange}{$R$}%

        \myHalfEdge{v1}{180}{green}{$Q$}%
        \myCEdge{u1}{v1}{green}{$Q$}%

        \myCEdge{u1}{v}{red}{$P_1$}%
        \myCEdge{v}{v1}{red}{$P_1$}%
        \myCEdge{v}{u2}{blue}{$P_2$}%
        \myCEdge{v}{v2}{blue}{$P_2$}%
    }%
    \drawRule{fig:RCTOneb}%
}%

\newcommand{\cfCTOnea}{%
    \renewcommand\Mynodes{%
        \node[purplenode,evennode,label=above:$v_1$] (v1) at (-1,1+0.5) {};
        \node[bluenode,label=above:$v_2$] (v2) at (1,1+0.5) {};
        \node[bluenode,label=above left:$v$] (v) at (0, 1+0.5) {};
    }%
    \renewcommand\Myspecialnodes{%
        \node[blacknode,label=left:$u_1$] (u1) at (0,0+0.5) {};
        \node[blacknode,label=left:$u_2$] (u2) at (0,2+0.5) {};
    }%
    \renewcommand\MyEdgesBefore{%
		\mySubRedEdge{u1}{-90}{}{}%
        \mySubEdge{u2}{90}{}{}
        \myEdge{u1}{v}
        \myEdge{u2}{v}
        \myEdge{u2}{v2}
        \myEdge{u1}{v1}
        \myEdge{v}{v1}
        \myEdge{v}{v2}
    }%
    \renewcommand{\CfN}{$(\C_{T1a})$}%
    \renewcommand{\CfL}{fig:CTOnea}%
}%
\newcommand{\RuleCTOnea}{%
    \cfCTOnea%
    \renewcommand\MyEdgesRemoved{%
        \myHalfEdge{v1}{180}{orange}{$T$}%
        \myCEdge{v}{v2}{green}{$R$}%
        \myCEdge{v}{v1}{black}{}%
    }%
    \renewcommand\MyEdgesAfter{%
    	\mySubRedEdge{u1}{-90}{}{}%
        \mySubEdge{u2}{90}{}{}%
        \myHalfEdge{v1}{-180}{orange}{$Q$}%
        \myCEdge{u1}{v1}{orange}{$Q$}%
        \myCEdge{v}{v1}{black}{}%
        \myCEdge{v}{u2}{green}{$R$}%
        \myCEdge{u2}{v2}{green}{$R$}%
        \myCEdge{u1}{v}{red}{$P_1$}%
        \myCEdge{v}{v2}{red}{$P_1$}%
    }%
	\drawRule{}%
}%

\newcommand{\cfCTTNA}{%
    \renewcommand\Mynodes{
        \node[whitenode,label=above:$v$] (v) at (-1,1+0.5) {};
        \node[bluenode,label=above:$v'$] (vp) at (1,1+0.5) {};
    }%
    \renewcommand\Myspecialnodes{%
        \node[blacknode,label=left:$u_1$] (u1) at (0,0+0.5) {};
        \node[blacknode,label=left:$u_2$] (u2) at (0,2+0.5) {};
    }%
    \renewcommand\MyEdgesBefore{%
		\mySubRedEdge{u1}{-90}{}{}%
        \mySubBlueEdge{u2}{90}{}{}

        \myEdge{u1}{v}
        \myEdge{u2}{v}
        \myEdge{u2}{vp}
        \myEdge{u1}{vp}
        \myDEdge{v}{vp}
    }%
    \renewcommand{\CfN}{$(\C_{T2NA})$}%
    \renewcommand{\CfL}{fig:CTTNA}%
}%

\newcommand{\cfCTTNAa}{%
    \renewcommand\Mynodes{
        \node[purplenode,evennode,label=above:$v$] (v) at (-1,1+0.5) {};
        \node[bluenode,label=above:$v'$] (vp) at (1,1+0.5) {};
    }%
    \renewcommand\Myspecialnodes{%
        \node[blacknode,label=left:$u_1$] (u1) at (0,0+0.5) {};
        \node[blacknode,label=left:$u_2$] (u2) at (0,2+0.5) {};
    }%
    \renewcommand\MyEdgesBefore{%
		\mySubRedEdge{u1}{-90}{}{}%
        \mySubEdge{u2}{90}{}{}

        \myEdge{u1}{v}
        \myEdge{u2}{v}
        \myEdge{u2}{vp}
        \myEdge{u1}{vp}
        \myDEdge{v}{vp}
    }%
    \renewcommand{\CfN}{$(\C_{T2NAa})$}%
    \renewcommand{\CfL}{fig:CTTNAa}%
}%
\newcommand{\RuleCTTNAa}{%
    \cfCTTNAa
    \renewcommand\MyEdgesRemoved{%
        \myHalfEdge{v}{-180}{orange}{$Q$}%
        \myCEdge{v}{vp}{green}{$R$}%
    }%
    \renewcommand\MyEdgesAfter{%
    	\mySubRedEdge{u1}{-90}{}{}%
        \mySubEdge{u2}{90}{}{}%

        \myHalfEdge{v}{-180}{orange}{$Q$}%
        \myCEdge{u1}{v}{orange}{$Q$}%

        \myCEdge{v}{u2}{green}{$R$}%
        \myCEdge{u2}{vp}{green}{$R$}%

        \myCEdge{u1}{vp}{red}{$P_1$}%
        \myDEdge{v}{vp}
    }%
    \drawRule{fig:RCTTNAa}%
}%

\newcommand{\cfCTTNAb}{%
    \renewcommand\Mynodes{
        \node[rednode,oddnode,label=above:$v$] (v) at (-1,1+0.5) {};
        \node[bluenode,oddnode,label=above:$v'$] (vp) at (1,1+0.5) {};
    }%
    \renewcommand\Myspecialnodes{%
        \node[blacknode,label=left:$u_1$] (u1) at (0,0+0.5) {};
        \node[blacknode,label=left:$u_2$] (u2) at (0,2+0.5) {};
        \mySubRedEdge{u1}{-90}{}{}%
        \mySubBlueEdge{u2}{90}{}{}%
    }%
    \renewcommand\MyEdgesBefore{%
        \myEdge{u1}{v}
        \myEdge{u2}{v}
        \myEdge{u2}{vp}
        \myEdge{u1}{vp}
        \myDEdge{v}{vp}
    }%
    \renewcommand{\CfN}{$(\C_{T2NAb})$}%
    \renewcommand{\CfL}{fig:CTTNAb}%
}%
\newcommand{\RuleCTTNAb}{%
    \cfCTTNAb
    \renewcommand\MyEdgesRemoved{%
        \myDEdge{v}{vp}
        \myHalfEdge{v}{-180}{orange}{$Q$}%
        \myHalfEdge{vp}{0}{green}{$R$}%
    }%
    \renewcommand\MyEdgesAfter{%
        \myHalfEdge{v}{-180}{orange}{$Q$}%
        \myCEdge{u1}{v}{orange}{$Q$}%

        \myCEdge{v}{u2}{blue}{$P_2$}%
        \myCEdge{u2}{vp}{green}{$R$}%
        \myHalfEdge{vp}{0}{green}{$R$}%

        \myCEdge{u1}{vp}{red}{$P_1$}%
        \myDEdge{v}{vp}
    }%
    \drawRule{fig:RCTTNAb}%
}%

\newcommand{\cfCTTA}{%
    \renewcommand\Mynodes{
        \node[whitenode,label=above:$v$] (v) at (-1,1+0.5) {};
        \node[bluenode,label=above:$v'$] (vp) at (1,1+0.5) {};
    }%
    \renewcommand\Myspecialnodes{%
        \node[blacknode,label=left:$u_1$] (u1) at (0,0+0.5) {};
        \node[blacknode,label=left:$u_2$] (u2) at (0,2+0.5) {};
        \mySubRedEdge{u1}{-90}{}{}%
        \mySubBlueEdge{u2}{90}{}{}%
    }%
    \renewcommand\MyEdgesBefore{%
        \myEdge{u1}{v}
        \myEdge{u2}{v}
        \myEdge{u2}{vp}
        \myEdge{u1}{vp}
        \myEdge{v}{vp}
    }%
    \renewcommand{\CfN}{$(\C_{T2A})$}%
    \renewcommand{\CfL}{fig:CTTA}%
}%

\newcommand{\cfCTTAa}{%
    \renewcommand\Mynodes{
        \node[bluenode,oddnode,label=above:$v$] (v) at (-1,1+0.5) {};
        \node[bluenode,label=above:$v'$] (vp) at (1,1+0.5) {};
    }%
    \renewcommand\Myspecialnodes{%
        \node[blacknode,label=left:$u_1$] (u1) at (0,0+0.5) {};
        \node[blacknode,label=left:$u_2$] (u2) at (0,2+0.5) {};
        \mySubRedEdge{u1}{-90}{}{}%
        \mySubEdge{u2}{90}{}{}%
    }%
    \renewcommand\MyEdgesBefore{%
        \myEdge{u1}{v}
        \myEdge{u2}{v}
        \myEdge{u2}{vp}
        \myEdge{u1}{vp}
        \myEdge{v}{vp}
    }%
    \renewcommand{\CfN}{$(\C_{T2Aa})$}%
    \renewcommand{\CfL}{fig:CTTAa}%
}%
\newcommand{\RuleCTTAa}{%
    \cfCTTAa
    \renewcommand\MyEdgesRemoved{%
        \myHalfEdge{v}{-180}{orange}{$Q$}%
        \myCEdge{v}{vp}{green}{$R$}%
    }%
    \renewcommand\MyEdgesAfter{%
        \myHalfEdge{v}{-180}{orange}{$Q$}%
        \myCEdge{u1}{v}{orange}{$Q$}%

        \myCEdge{v}{u2}{green}{$R$}%
        \myCEdge{u2}{vp}{green}{$R$}%

        \myCEdge{u1}{vp}{red}{$P_1$}%
        \myCEdge{v}{vp}{red}{$P_1$}%
    }%
    \drawRule{fig:RCTTAa}%
}%

\newcommand{\cfCTTAb}{%
    \renewcommand\Mynodes{
        \node[evennode,label=above:$v$] (v) at (-1,1+0.5) {};
        \node[bluenode,evennode,label=above:$v'$] (vp) at (1,1+0.5) {};
    }%
    \renewcommand\Myspecialnodes{%
        \node[blacknode,label=left:$u_1$] (u1) at (0,0+0.5) {};
        \node[blacknode,label=left:$u_2$] (u2) at (0,2+0.5) {};
        \mySubRedEdge{u1}{-90}{}{}%
        \mySubBlueEdge{u2}{90}{}{}%
    }%
    \renewcommand\MyEdgesBefore{%
        \myEdge{u1}{v}
        \myEdge{u2}{v}
        \myEdge{u2}{vp}
        \myEdge{u1}{vp}
        \myEdge{v}{vp}
    }%
    \renewcommand{\CfN}{$(\C_{T2Ab})$}%
    \renewcommand{\CfL}{fig:CTTNAb}%
}%
\newcommand{\RuleCTTAb}{%
    \cfCTTAb
    \renewcommand\MyEdgesRemoved{%
        \myDEdge{v}{vp}
        \myHalfEdge{v}{-180}{orange}{$Q$}%
        \myHalfEdge{vp}{0}{green}{$R$}%
    }%
    \renewcommand\MyEdgesAfter{%
        \myHalfEdge{v}{-180}{orange}{$Q$}%
        \myCEdge{u1}{v}{orange}{$Q$}%

        \myCEdge{v}{u2}{blue}{$P_2$}%
        \myCEdge{u2}{vp}{green}{$R$}%
        \myHalfEdge{vp}{0}{green}{$R$}%

        \myCEdge{u1}{vp}{red}{$P_1$}%
        \myCEdge{v}{vp}{red}{$P_1$}%
    }%
    \drawRule{fig:RCTTAb}%
}%

\newcommand{\AllRulesDesc}{%

\renewcommand{\MyScale}{0.84}
\renewcommand\OrdonneeFleche{1.4}

\ \\

{\center
\RuleZeroa

}

\begin{itemize}
\item \ncf{\cfCZa}: The special vertex $u_1$ has degree $2$: it has two non-adjacent neighbors $v_1,v_2$.

\patred
In the reduced graph, we add the edge $v_1v_2$.

\patrec
In $G$, we deviate the color of $v_1v_2$ on $u_1$.

\renewcommand{\MyScale}{0.9}

{\center
\RuleXXa

}

\item \ncf{\cfXXa}: The two special vertices $u_1,u_2$ are adjacent and they have precisely one common neighbor $v$. The special vertex $u_1$ has degree $3$: it has another neighbor $v_1$, and $u_2$ has degree $3$ or $4$, with a neighbor $v_2$ and maybe another $v_3$. The vertex $v_2$ has an even degree in $G$.

\patred
In the reduced graph, $v_2$ has an odd degree: let $Q$ be a path of the coloring of $G'$ that ends on $v_2$.

\patrec
In $G$, we extend $Q$ to the edges $v_2u_2$, $u_2u_1$. We use the extra color on the path $P = (v_1,u_1,v,u_2)$ and maybe the edge $u_2v_3$ if it is in $G$.

\renewcommand\OrdonneeFleche{1.6}

{\center
\RuleXXd 

}

\item \ncf{\cfXXd}: Each of the two special vertices $u_1,u_2$ has degree $3$ or $4$. They are adjacent, they have precisely one common neighbor $v$ and each of $u_1,u_2$ has another neighbor, $v_1,v_2$ respectively. Each of $u_1,u_2$ may have a third neighbor $v_3,v_4$ respectively. The vertices $v_1,v_2$ are non-adjacent.

\patred
In the reduced graph, we add the edge $v_1v_2$.

\patrec
In $G$, we deviate the color of $v_1v_2$ on $u_1,u_2$. We use the extra color on the path $P = (u_1,v,u_2)$, and maybe on the edges $v_3u_1$ and $u_2v_4$ if they belong to $G$.

\ifthenelse{\equal{\isThesis}{true}}
{\vfill
\pagebreak}
{}

{\center
\RuleXXbb

}

\item \ncf{\cfXXbb}: The two special vertices $u_1,u_2$ have degree $3$. They are adjacent, they have precisely one common neighbor $v$ and each of $u_1,u_2$ has another neighbor, $v_1,v_2$ respectively, which is adjacent to $v$. The vertices $v_1,v_2$ are adjacent, and $v$ has an odd degree in $G$.

\patred
In the reduced graph, we remove the edge $v_1v_2$.
The vertex $v$ keeps an odd degree: let $Q$ be a path of the coloring of $G'$ that ends on it.

\patrec
In $G$, we extend $Q$ on the edges 
$vu_1$ and $u_1u_2$.
We use the extra color on the path $(v,u_2,v_2,v_1,u_1)$.

\renewcommand\OrdonneeFleche{1.4}

{\center
\RuleXXc

}

\item \ncf{\cfXXc}: The two special vertices $u_1,u_2$ have degree $3$. They are adjacent, they have precisely one common neighbor $v$ and each of $u_1,u_2$ has another neighbor, $v_1,v_2$ respectively. The vertex $v_1$ is adjacent to $v$. The vertex $v$ has an even degree in $G$. 

\patred
In the reduced graph, we remove the edge $vv_1$. 
The vertex $v$ has an odd degree in $G'$, so let $Q$ be a path of the coloring of $G'$ that ends on $v$.

\patrec
In $G$, we extend $Q$ on the edges 
$vu_1$ and $u_1u_2$.
We use the extra color on the path $P = (u_1,v_1,v,u_2,v_2)$.

\renewcommand{\MyScale}{0.9}
\renewcommand\OrdonneeFleche{2.0}

{\center
\RuleXXe

}

\item \ncf{\cfXXe}: Each of the two special vertices $u_1,u_2$ has degree $3$ or $4$. They are non-adjacent, they have precisely two common neighbors $v_1,v_2$ and each of $u_1,u_2$ has another neighbor $v_3,v_4$ respectively. Both $v_3,v_4$ are adjacent to both $v_1,v_2$. The vertices $v_1,v_2$ are non-adjacent. Each of $u_1,u_2$ may have another neighbor, $v_5,v_6$ respectively.

\patred
In the reduced graph, we add the edge $v_1v_2$ and remove the edges $v_1v_3$ and $v_2v_4$.

\patrec
In $G$, we deviate the color of $v_1v_2$ on $u_1$, and the color of $v_1v_4$ on $u_2$. We use the extra color on the path $P = (u_1,v_3,v_1,v_4,v_2,u_2)$, and maybe on the edges $v_5u_1$ and $u_2v_6$ if they belong to $G$.

{\center
\RuleXXf

}

\item \ncf{\cfXXf}: Each of the two special vertices $u_1,u_2$ has degree $3$ or $4$. They are non-adjacent, they have precisely two common neighbors $v_1,v_2$ and each of $u_1,u_2$ has another neighbor $v_3,v_4$ respectively. Both $v_3,v_4$ are adjacent to both $v_1,v_2$. The vertices $v_1,v_2$ are adjacent. Each of $u_1,u_2$ may have another neighbor, $v_5,v_6$ respectively. There is a path $P_{34}$ between $v_3$ and $v_4$ in $G$ that is vertex-disjoint from the other $6$ vertices.

\patred
In the reduced graph, we remove the edges of the path $P_{34}$.

\patrec
In $G$, we deviate the color of $v_1v_2$ on $u_1$, and the color of $v_1v_4$ on $u_2$. We use the extra color on the path $P = (u_1,v_3,P_{34},v_4,v_1,v_2,u_2)$, and maybe on the edges $v_5u_1$ and $u_2v_6$ if they belong to $G$.

\vfill
\pagebreak

\renewcommand{\MyScale}{1.1}
\renewcommand\OrdonneeFleche{2.0}

{\center
\tkcf{\cfXXg}

}

\item \ncf{\cfXXg}: Each of the two special vertices $u_1,u_2$ has degree $3$ or $4$. They are non-adjacent and have precisely two common neighbors $v_1,v_2$. Each of $u_1,u_2$ has another neighbor $v_3,v_4$ respectively, both adjacent to both $v_1,v_2$. The set $\{v_1,v_2\}$ is a $2$-cut that separates $u_1,u_2$. Each of $u_1,u_2$ may have another neighbor $v_5,v_6$ respectively.

\patred
In the reduced graph, we add the edge $v_3v_4$. Let $Q$ be the path of the coloring of $G'$ induced by the color of $v_3v_4$. We denote it $Q = (Q_1,v_3,v_4,Q_2)$.

\patrec
We distinguish between six cases, depending on the properties of the path $Q$.
Rule {\nameXXgOne} covers the case where the path $Q$ avoids the vertex $v_2$.
In rules {\nameXXgTwo} and {\nameXXgThree}, the vertices $v_2, v_3, v_4$ and $v_1$ appear in that order on the path $Q$.
Rule {\nameXXgTwo} covers the case where $v_5$ is on the subpath between $v_2$ and $v_3$, and $v_6$ is on the subpath between $v_4$ and $v_1$ (the rules are a bit more general and only require $v_5$ and $v_6$ to avoid some subpaths of $Q$).
Rule {\nameXXgThree} covers the case where the vertex $v_5$ (if it exists) avoids the subpath of $Q$ between $v_2$ and $v_3$.
By symmetry, this also covers the case where $v_6$ (if it exists) avoids the subpath of $Q$ between $v_4$ and $v_1$.
In rules {\nameXXgFour}, {\nameXXgFive} and {\nameXXgSix}, the vertices $v_1$, $v_2$, $v_3$ and $v_4$ appear in this order on the path $Q$.
Rule {\nameXXgFour} covers the case where the vertex $v_5$ avoids the subpath of $Q$ between $v_2$ and $v_3$, while in rules {\nameXXgFive} and {\nameXXgSix} the vertex $v_5$ avoids the subpath of $Q$ between $v_1$ and $v_2$.
In rule {\nameXXgFive}, the vertex $v_6$ avoids the subpath of $Q$ after $v_1$, and in rule {\nameXXgSix} it avoids the subpath between $v_1$ and $v_2$.
Since each of $v_5,v_6$ avoids at least one subpath of $Q$, this covers all possible cases.

\renewcommand{\MyScale}{0.95}

{\center
\hspace{8mm}
\RuleXXgOne

}

\begin{enumerate}
\item $Q$ does not touch $v_2$ in $G'$.

	In $G$, we replace the path $Q$ with the path $Q' = (Q_1,v_3,u_1,v_2,u_2,v_4,Q_2)$. We use the extra color on the path $P = (u_1,v_1,u_2)$ and maybe on the edges $v_5u_1$ and $u_2v_6$ if they belong to $G$.

{\center
\RuleXXgTwo

}

\item $Q_1$ touches $v_2$ and $Q_2$ touches $v_1$. The vertex $v_5$ does not touch $Q_2$, and $v_6$ does not touch $Q_1$.

	In $G$, we replace $Q$ with a path $Q' = (Q_1,v_3,u_1,v_1,u_2)$ and maybe extend $Q'$ on $u_2v_6$. We use the extra color on the path $P = (Q_2,v_4,u_2,v_2,u_1)$ and maybe on the edge $u_1v_5$.

{\center
\RuleXXgThree

}

\item $Q_1$ touches $v_2$ and $Q_2$ touches $v_1$. We denote $Q_1 = (v_3,R_1,v_2,R_1')$. The vertex $v_5$ does not touch $R_1$. Note that by planarity, $v_6$ cannot touch $R_1$.

	In $G$, we replace $Q$ with a path $Q' = (R_1',v_2,u_2,v_4,Q_2)$ and we deviate the color of $v_3v_1$ in $G'$ on $u_1$. We use the extra color on the path $P = (v_5,u_1,v_2,R_1,v_3,v_1,u_2)$ and maybe on the edge $u_2v_6$.

{\center
\RuleXXgFour

}

\item $Q_1$ touches both $v_1,v_2$: we denote it $Q_1 = (R_1,v_1,R_1',v_2,R_1'',v_3)$. The vertex $v_5$ does not touch $R_1''$. Again by planarity, $v_6$ cannot touch $R_1''$.

	In $G$, we replace $Q$ with a path $Q' = (u_1,v_2,R_1'',v_3,v_1,u_2)$ and maybe extend $Q'$ on the edges $v_5u_1$ and $u_2v_6$. We deviate the color of $v_3v_1$ on $u_1$ and we use the extra color on the path $P = (Q_2,v_4,u_2,v_2,R_1',v_1,R_1)$.

\ifthenelse{\equal{\isThesis}{true}}
{\vfill
\pagebreak}
{}

{\center
\RuleXXgFive

}

\item $Q_1$ touches both $v_1,v_2$: we denote it $Q_1 = (R_1,v_1,R_1',v_2,R_1'',v_3)$. The vertex $v_5$ does not touch $R_1'$, and $v_6$ does not touch $R_1$. Again, note that by planarity $v_5$ cannot touch $Q_2$ and $v_6$ cannot touch $R_1''$.

	In $G$, we replace $Q$ with a path $Q' = (u_1,v_2,R_1',v_1,u_2,v_4,Q_2)$ and maybe extend $Q'$ on the edge $v_5u_1$. We deviate the color of $v_3v_1$ on $u_1$ and we use the extra color on the path $P = (R_1,v_1,v_3,R_1'',v_2,u_2)$ and maybe on the edge $u_2v_6$.

{\center
\RuleXXgSix

}

\item $Q_1$ touches both $v_1,v_2$: we denote it $Q_1 = (R_1,v_1,R_1',v_2,R_1'',v_3)$. None of $v_5, v_6$ touch $R_1'$.

	In $G$, we replace $Q$ with a path $Q' = (R_1,v_1,u_1,v_3,R_1'',v_2,u_2,v_4,Q_2)$. We use the extra color on the path $P = (u_1,v_2,R_1',v_1,u_2)$ and maybe on the edges $v_5u_1$ and $u_2v_6$.
\end{enumerate}

\renewcommand{\MyScale}{1}
\renewcommand\OrdonneeFleche{1.5}

{\center
\RuleXXh

}

\item \ncf{\cfXXh}: The special vertex $u_1$ has degree $3$ or $4$, and $u_2$ has degree $4$.
The special vertices $u_1,u_2$ have (at least) two common neighbors $v,v'$ that are non-adjacent. The special vertex $u_2$ has another neighbor $v_2$, non-adjacent to $v'$ nor $u_1$, and $u_1$ may have another neighbor $v_1$, non-adjacent to $u_2$.
If $u_1,u_2$ are non-adjacent, they have another common neighbor $w$; let us denote $P_{12}$ the path $(u_1,u_2)$ if $u_1,u_2$ are adjacent, and $(u_1,w,u_2)$ otherwise.

\patred
In the reduced graph, we add the edges $vv'$ and $v_2v'$.

\patrec
In $G$, we deviate the color of $vv'$ on $u_1$ and the color of $v_2v'$ on $u_2$. We use the extra color on the path $P = (v,u_2,P_{12},u_1)$ and maybe on the edge $u_1v_1$ if it belongs to $G$.

{\center
\RuleXXi

}

\item \ncf{\cfXXi}: The special vertex $u_1$ has degree $3$ or $4$, and $u_2$ has degree $4$.
The special vertices $u_1,u_2$ have (at least) two common neighbors $v,v'$ that are non-adjacent. The special vertex $u_2$ has another neighbor $v_2$, adjacent to $v,v'$ but not $u_1$, and $u_1$ may have another neighbor $v_1$, non-adjacent to $u_2$.
If $u_1,u_2$ are non-adjacent, they have another common neighbor $w$; let us denote $P_{12}$ the path $(u_1,u_2)$ if $u_1,u_2$ are adjacent, and $(u_1,w,u_2)$ otherwise.

\patred
In the reduced graph, we add the edge $vv'$ and remove the edge $vv_2$.

\patrec
In $G$, we deviate the color of $vv'$ on $u_1$ and the color of $v_2v'$ on $u_2$. We use the extra color on the path $P = (v',v_2,v,u_2,P_{12},u_1)$ and maybe on the edge $u_1v_1$.

{\center
\RuleXXj

}

\item \ncf{\cfXXj}: The special vertices $u_1,u_2$ have degree $3$, with at least two common neighbors $v,v'$ that are non-adjacent.
If $u_1,u_2$ are non-adjacent, they have another common neighbor $w$; let us denote $P_{12}$ the path $(u_1,u_2)$ if $u_1,u_2$ are adjacent, and $(u_1,w,u_2)$ otherwise.
The vertex $v'$ has an even degree in $G$.

\patred
In the reduced graph, we add the edge $vv'$. The vertex $v'$ now has an odd degree, so let $R$ be a path of the coloring of $G'$ that ends on $v'$.

\patrec
In $G$, we deviate the color of $vv'$ on $u_1$, and extend $R$ on the edge $v'u_2$. We use the extra color on the path $P = (u_1,P_{12},u_2,v)$.

{\center
\RuleXXk

}

\item \ncf{\cfXXk}: The special vertices $u_1,u_2$ have degree $3$, with at least two common neighbors $v,v'$ that are non-adjacent and both have an odd degree in $G$.
If $u_1,u_2$ are non-adjacent, they have another common neighbor $w$; let us denote $P_{12}$ the path $(u_1,u_2)$ if $u_1,u_2$ are adjacent, and $(u_1,w,u_2)$ otherwise.

\patred
In the reduced graph, $v,v'$ now have an odd degree, so let $Q,R$ be paths of the coloring of $G'$ that end on $v,v'$ respectively.

\patrec
In $G$, we extend $Q$ on the edge $vu_1$ and $R$ on $v'u_2$. We use the extra color on the path $P = (v',u_1,P_{12},u_2,v)$.

\renewcommand\MyScale{1.1}
{\center
\RuleXXl

}

\item \ncf{\cfXXl}: The two special vertices $u_1,u_2$ have degree $3$, they are non-adjacent and have three common neighbors $v,v',v''$, such that $v'$ is adjacent to $v$ and $v''$.

\patrec
In $G$, we deviate the color of $vv'$ in $G'$ on $u_2$ and the color of $v'v''$ on $u_1$. We use the extra color on the path $P = (u_1,v,v',v'',u_2)$.

\renewcommand\MyScale{0.92}

{\center
\RuleXXlp

}

\item \ncf{\cfXXlp}: The special vertex $u_1$ has degree $4$, and $u_2$ has degree $3$ or $4$: they are non-adjacent and have precisely three common neighbors $v,v',v''$, such that $v'$ is adjacent to $v$ and $v''$. The special vertex $u_1$ has another neighbor $v_1$, and $u_2$ may have another neighbor $v_2$, such that $v_1\neq v_2$.

\patred
In the reduced graph, we add the edge $vv_1$ if it does not already belong to $G$.

\patrec
In $G$, we deviate the color of $vv_1$ in $G'$ on $u_1$, and the color of $vv'$ on $u_2$. We use the extra color on the path $P = (v,v',u_1,v'',u_2)$ and maybe on the edges $vv_1$ and $u_2v_2$ if they belong to $G$.

{\center
\RuleXXm

}

\item \ncf{\cfXXm}: Each of the two special vertices $u_1,u_2$ has degree $3$ or $4$. They are adjacent and have precisely two common neighbors $v,v'$ that are adjacent. Each of $u_1,u_2$ may have another neighbor, $v_1,v_2$ respectively.

\patrec
In $G$, we deviate the color of $vv'$ in $G'$ on $u_1,u_2$. We use the extra color on the path $P = (u_1,v',v,u_2)$ and maybe on the edges $v_1u_1$ and $u_2v_2$ if they belong to $G$.

\renewcommand{\MyScale}{0.95}

{\center
\RuleXXn

}

\item \ncf{\cfXXn}: The special vertices $u_1,u_2$ have degree $4$, with at least three common neighbors $v,v',v''$ such that $v'$ is non-adjacent to $v$ and $v''$.
If $u_1,u_2$ are non-adjacent, they have another common neighbor $w$; let us denote $P_{12}$ the path $(u_1,u_2)$ if $u_1,u_2$ are adjacent, and $(u_1,w,u_2)$ otherwise.

\patred
In the reduced graph, we add the edges $vv'$ and $v'v''$.

\patrec
In $G$, we deviate the color of $vv'$ on $u_1$ and the color of $v'v''$ on $u_2$. We use the extra color on the path $P = (v'',u_1,P_{12},u_2,v)$.

{\center
\RuleXXo

}

\item \ncf{\cfXXo}: The special vertices $u_1,u_2$ have degree $4$, with at least three common neighbors $v,v',v''$ such that $v'$ is adjacent to $v$ and non-adjacent to $v''$.
If $u_1,u_2$ are non-adjacent, they have another common neighbor $w$; let us denote $P_{12}$ the path $(u_1,u_2)$ if $u_1,u_2$ are adjacent, and $(u_1,w,u_2)$ otherwise.

\patred
In the reduced graph, we add the edge $v'v''$.

\patrec
In $G$, we deviate the color of $vv'$ on $u_1$ and the color of $v'v''$ on $u_2$. We use the extra color on the path $P = (v'',u_1,P_{12},u_2,v,v')$.

\renewcommand{\MyScale}{1.1}

{\center
\RuleXXp

}

\item \ncf{\cfXXp}: The special vertex $u_1$ has degree $2$, and $u_2$ has degree $3$ or $4$. The special vertices $u_1,u_2$ are adjacent and have precisely one common neighbor $v$. The special vertex $u_2$ has another neighbor $v_1$ adjacent to $v$, and maybe another neighbor $v_2$.

\patrec
In $G$, we deviate the color of the edge $vv_1$ in $G'$ on $u_1,u_2$. We use the extra color on the path $P = (u_2,v,v_1)$, and maybe on the edge $v_2u_2$ if it belongs to $G$.

\renewcommand\OrdonneeFleche{1.1}

{\center
\RuleXXr

}

\item \ncf{\cfXXr}: The two special vertices $u_1,u_2$ have degree $2$: they are adjacent and have one common neighbor $v$. 
The vertex $v$ has at least one other neighbor $v_1$.

In the reduced graph, we examine two cases.
\begin{itemize}
\item 
In the first case, a path $Q$ of the coloring of $G'$ ends on $v$ (through the edge $v_1v$).

\patrec
In $G$, we extend $Q$ on the edge $vu_1$. We use the extra color on the path $P = (u_1,u_2,v)$.

\item In the second case, no path of the coloring of $G'$ ends on $v$. Let $Q$ be a path of the coloring such that $Q = (Q_1,v_1,v,v_2,Q_2)$, with $v_2$ another neighbor of $v$.

\patrec
In $G$, we replace $Q$ with a path $Q' = (Q_1,v_1,v,u_1)$ and we use the extra color on the path $P = (u_1,u_2,v,v_2,Q_2)$.
\end{itemize}

\renewcommand\OrdonneeFleche{1.7}

{\center
\RuleXXs

}

\item \ncf{\cfXXs}: The two special vertices $u_1,u_2$ have degree $2$: they are non-adjacent and have two common neighbors $v_1,v_2$ that are adjacent. The vertex $v_2$ has an odd degree in $G$.

\patred
In the reduced graph, $v_2$ keeps an odd degree: let $Q$ be a path of the coloring that ends on $v_2$.

\patrec
In $G$, we deviate the color of $v_1v_2$ in $G'$ on $u_1$, and we extend the path $Q$ on the edge $v_2u_2$.

\ifthenelse{\equal{\isThesis}{true}}
{}
{\vfill
\pagebreak}

{\center
\RuleXXt

}

\item \ncf{\cfXXt}: The two special vertices $u_1,u_2$ have degree $2$: they are non-adjacent and have two common neighbors $v_1,v_2$ that are adjacent and have an even degree in $G$.

\patred
In the reduced graph, we remove the edge $v_1v_2$. The vertices $v_1,v_2$ have odd degrees in $G'$: let $Q,R$ be two paths of the coloring that end on $v_2,v_1$ respectively.

\patrec
In $G$, we extend the path $Q$ on the edge $v_2u_2$ and the path $R$ on $v_1u_1$. We use the extra color on the path $P = (u_1,v_2,v_1,u_2)$.

\renewcommand{\MyScale}{0.9}
\renewcommand\OrdonneeFleche{1.9}

{\center
\RuleXXup

}

\item \ncf{\cfXXup}: The special vertex $u_1$ has degree $2$, and $u_2$ has degree $3$ or $4$. The special vertices $u_1,u_2$ are non-adjacent and have two common neighbors $v_1,v_2$ that are adjacent. The vertex $u_2$ has another neighbor $v_3$ adjacent to $v_2$. The vertex $u_2$ may have another neighbor $v_4$.

\patrec
In $G$, we deviate the color of the edge $v_1v_2$ in $G'$ on $u_1$ and the color of $v_2v_3$ on $u_2$. We use the extra color on the path $P = (v_3,v_2,v_1,u_2)$ and maybe on the edge $u_2v_4$ if it belongs to $G$.

\end{itemize}

}%

\newcommand{\AllSemiRulesTwoDesc}{%

	{ 
	\tikzstyle{bluenode}=[mynode,fill=white]
	\tikzstyle{purplenode}=[mynode,fill=white]
	\renewcommand\mySubEdge[4]{\draw[sedge]  ($(##1) + (##2:0.75)$) 
	-- (##1);}
	\renewcommand\mySubRedEdge[4]{\draw[sedge]  ($(##1) + (##2:0.75)$) 
	-- (##1);}

	\renewcommand{\MyScale}{1.0}
	\renewcommand\OrdonneeFleche{1.4}

{\center
\hspace{8mm}
\RuleCExt

}

\begin{itemize}
\item \ncf{\cfCExt}: The special vertex $u_1$ has exactly one remaining neighbor $v_1$.

\patrec
In $G$, we extend the extra color on the edge $u_1v_1$.

{\center
\RuleCV

}
\item \ncf{\cfCV}: The special vertex $u_1$ has exactly two remaining neighbors $v_1,v_2$ that are non-adjacent.

\patred
In the reduced graph, we add the edge $v_1v_2$.

\patrec
In $G$, we deviate the color of $v_1v_2$ on $u_1$.

{\center
\RuleCvTwo

}
\item \ncf{\cfCvTwo}: The special vertex $u_1$ has exactly two remaining neighbors $v_1,v_2$ that are adjacent. The vertex $v_1$ has an even degree in $G$.

\patred
In the reduced graph, $v_1$ has an odd degree: let $R$ be a path of the coloring of $G'$ that ends on $v_1$.

\patrec
In $G$, we extend the path $R$ on the edge $v_1u_1$.

{\center
\RuleCvThree

}
\item \ncf{\cfCvThree}: The special vertex $u_1$ has exactly two remaining neighbors $v_1,v_2$ that are adjacent. The vertex $v_1$ has an odd degree in $G$.

\patred
In the reduced graph, we remove the edge $v_1v_2$. The vertex $v_1$ keeps an odd degree in $G'$: let $R$ be a path of the coloring of $G'$ that ends on $v_1$.

\patrec
In $G$, we extend the path $R$ on the edge $v_1u_1$, and we extend the extra color on the edges $u_1v_2$ and $v_2v_1$.

{\center
\RuleCFoura

}
\item \ncf{\cfCFoura}: The special vertex $u_1$ has exactly three remaining neighbors $v_1,v_2,v_3$, such that $v_1,v_2$ are non-adjacent (remark that $v_1,v_2$ are not necessarily consecutive in the cyclic order of the neighbors of $u_1$).

\patred
In the reduced graph, we add the edge $v_1v_2$.

\patrec
In $G$, we deviate the color of $v_1v_2$ on $u_1$ and extend the extra color on the edge $u_1v_3$.

\ifthenelse{\equal{\isThesis}{true}}
{\vfill
\pagebreak}
{}

{\center
\RuleCFourb

}
\item \ncf{\cfCFourb}: The special vertex $u_1$ has exactly three remaining neighbors $v_1,v_2,v_3$, such that the edges $v_1v_2$ and $v_2v_3$ belong to $G$.

\patred
In the reduced graph, we remove the edge $v_2v_3$.

\patrec
In $G$, we deviate the color of $v_1v_2$ on $u_1$ and extend the extra color on the edges $u_1v_3$, $v_3v_2$, $v_2v_1$.

\end{itemize}

	} 
	
}%

\newcommand{\AllAliasDesc}{%
	{ 
	\tikzstyle{bluenode}=[mynode,fill=white]
	\tikzstyle{purplenode}=[mynode,fill=white]
	\renewcommand\mySubEdge[4]{\draw[sedge]  ($(##1) + (##2:0.75)$) 
	-- (##1);}
	\renewcommand\mySubRedEdge[4]{\draw[sedge]  ($(##1) + (##2:0.75)$) 
	-- (##1);}

\renewcommand{\MyScale}{1.0}
\renewcommand\OrdonneeFleche{1.2}

{\center
\hspace{10mm}

\tkcf{\cfCVPlus} %
	\TexteCentre{$:=\{$}{\OrdonneeFleche} %
	\tkcf{\cfCOnea} %
	\TexteCentre{$\vert$}{\OrdonneeFleche} %
	\tkcf{\cfCV} %
	\TexteCentre{$\vert$}{\OrdonneeFleche} %
	\tkcf{\cfCFoura} %
	\TexteCentre{$\}$}{\OrdonneeFleche}%

}

\begin{itemize}
\item \ncf{\cfCVPlus}: The special vertex $u_1$ has either $0$, $2$ or $3$ remaining neighbors $v_1,v_2,v_3$, such that at least two of them are non-adjacent (not necessarily consecutive in the cyclic order of the neighbors). If it has $0$, this is configuration \ncf{\cfCOnea}; if it has exactly $2$ and they are non-adjacent this is configuration \ncf{\cfCV}; and if it has three remaining neighbors, and at least two of them are non-adjacent, this is configuration \ncf{\cfCFoura}.

{\center

\tkcf{\cfCN} %
			  \TexteCentre{$:=\{$}{\OrdonneeFleche}%
			  \tkcf{\cfCV} %
			  \TexteCentre{$\vert$}{\OrdonneeFleche} %
			  \tkcf{\cfCvTwo} %
			  \TexteCentre{$\vert$}{\OrdonneeFleche} %
			  \tkcf{\cfCvThree}%
			  \TexteCentre{$\}$}{\OrdonneeFleche}%

}

\item \ncf{\cfCN}: The special vertex $u_1$ has $2$ remaining neighbors $v_1,v_2$. If $v_1,v_2$ are non-adjacent, this is configuration \ncf{\cfCV}. Otherwise, if one of $v_1,v_2$ has an even degree in $G$, this is configuration \ncf{\cfCvTwo}, and if both have an odd degree in $G$, this is configuration \ncf{\cfCvThree}.

\ifthenelse{\equal{\isThesis}{true}}
{\vfill
\pagebreak}
{}

{\center

\tkcf{\cfCNP} %
			  \TexteCentre{$:=\{$}{\OrdonneeFleche} %
			  \tkcf{\cfCOnea} %
			  \TexteCentre{$\vert$}{\OrdonneeFleche} %
			  \tkcf{\cfCExt} %
			  \TexteCentre{$\vert$}{\OrdonneeFleche} %
			  \tkcf{\cfCN} %
			  \TexteCentre{$\vert$}{\OrdonneeFleche} %
			  \tkcf{\cfCFoura} %
			  \TexteCentre{$\vert$}{\OrdonneeFleche} %
			  \tkcf{\cfCFourb} %
			  \TexteCentre{$\}$}{\OrdonneeFleche}%

}

\item \ncf{\cfCNP}: The special vertex $u_1$ has between $0$ and $3$ remaining neighbors $v_1,v_2,v_3$. If it has $0$, $1$ or $2$, this is configuration \ncf{\cfCOnea}, \ncf{\cfCExt} and \ncf{\cfCN} respectively. If it has $3$ remaining neighbors, then if two of them (not necessarily consecutive in the cyclic order) are non-adjacent this is configuration \ncf{\cfCFoura} and otherwise \ncf{\cfCFourb}.
\end{itemize}
	
	} 

}%

\newcommand\myprivateBentEdge[6]{
	\draw[#3, thick, segment aspect=0, bend #4 = #5] (#1) to node[#3,nodelabel,midway] {#6} (#2);
}
\newcommand\myBentEdge[5]{
	\ifthenelse{\equal{#3}{\subColor}}
{\myprivateBentEdge{#1}{#2}{#3}{#4}{#5}{$S$}
}
{\ifthenelse{\equal{#3}{\subColorOne}}
	{\myprivateBentEdge{#1}{#2}{#3}{#4}{#5}{$P_1$}}{
		\ifthenelse{\equal{#3}{\subColorTwo}}
		{\myprivateBentEdge{#1}{#2}{#3}{#4}{#5}{$P_2$}}
		{
			\myprivateBentEdge{#1}{#2}{#3}{#4}{#5}{}
		}
	}
}
}%

\newcommand\myPathPrivate[6]{
\draw[#3, thick, decorate, decoration=snake, segment aspect=0, segment amplitude=0.5mm, segment length = 4mm, bend #4 = #5] (#1) to node[#3,nodelabel,midway] {#6} (#2);}
\newcommand\myPath[5]{
	\ifthenelse{\equal{#3}{\subColor}}
		{\myPathPrivate{#1}{#2}{#3}{#4}{#5}{$S$}
		}
		{\ifthenelse{\equal{#3}{\subColorOne}}
		{\myPathPrivate{#1}{#2}{#3}{#4}{#5}{$P_1$}}{
			\ifthenelse{\equal{#3}{\subColorTwo}}
			{\myPathPrivate{#1}{#2}{#3}{#4}{#5}{$P_2$}}
			{
			\myPathPrivate{#1}{#2}{#3}{#4}{#5}{}
			}
		}
	}
}%

\newcommand\patternCV[4]{%
	\myCEdge{#1}{#2}{#4}{}%
	\myCEdge{#1}{#3}{#4}{}%
	\myDEdge{#2}{#3}%
	\node at ($0.33*(#1) + 0.33*(#2) + 0.33*(#3)$)  {\textcolor{#4}{$\mathcal{C}_V$}};
}%

\newcommand\patternCN[4]{%
    \myCEdge{#1}{#2}{#4}{}%
	\myCEdge{#1}{#3}{#4}{}%
    \myCEdge{#2}{#3}{#4}{}%
    \node at ($0.33*(#1) + 0.33*(#2) + 0.33*(#3)$)  {\textcolor{#4}{$\mathcal{C}_N$}};
}%

\newcommand\patternCDOnene[6]{%
    \myCEdge{#1}{#2}{#6}{}%
	\myCEdge{#1}{#3}{#6}{}%
    \myCEdge{#2}{#3}{#6}{}%
    \node at ($0.33*(#1) + 0.33*(#2) + 0.33*(#3)$)  {\textcolor{#6}{$\mathcal{C}_{D1}$}};
    \myCEdge{#4}{#3}{#6}{}%
	\myCEdge{#4}{#5}{#6}{}%
    \myCEdge{#3}{#5}{#6}{}%
}%

\newcommand\tkcfSansNomFigure{
    \begin{tikzpicture}[scale=\MyScale, every node/.style={scale=\MyScale}, auto]
                \Myspecialnodes
                \Mynodes
                \MyEdgesBefore
    \end{tikzpicture}%
}%

\newcommand\flecheACentrer[1]{%
		\begin{tikzpicture}[scale=\MyScale, every node/.style={scale=\MyScale}, auto]
		\node[label=above:\Huge{$\leadsto$}] (arrow) at (0,#1) {};
		\node (q) at (0,0) {};
		\end{tikzpicture}%
}%

\newcommand\drawSimpleRule[3]{ 
    \begin{figure}%
        \ifthenelse{\isundefined{\NoFigures}}{%
        \centering
        \tkcfSansNomFigure{}%
        \flecheACentrer{#2}%
        \tikzPartRule{\Myspecialnodes \Mynodes \MyEdgesAfter}%
        }{}%
        \caption{#3}%
        \label{#1}%
    \end{figure}%
}%

\newcommand\drawSimpleRuleSansFigure[1]{
		\centering
        \tkcfSansNomFigure{}%
        \petiteFlecheACentrer{#1}%
        \tikzPartRule{\Myspecialnodes \Mynodes \MyEdgesAfter}%
}

\newcommand\petiteFlecheACentrer[1]{%
		\begin{tikzpicture}[scale=\MyScale, every node/.style={scale=\MyScale}, auto]
		\node[label=above:\Large{$\leadsto$}] (arrow) at (0,#1) {};
		\node (q) at (0,0) {};
		\end{tikzpicture}%
}%

\newcommand\TexteCentre[2]{
		\begin{tikzpicture}[scale=\MyScale, every node/.style={scale=\MyScale}, auto]%
		\node[label=above:\huge{#1}] (text) at (0,#2) {};%
		\node (q) at (0,0) {};%
		\end{tikzpicture}%
}%

\newcommand\drawRuleSansFigure[1]{ 
	\tkcf{}%
    \petiteFlecheACentrer{#1}%
    \tikzstyle{endpath}=[->]%
    \tikzPartRule{\TikzRedefOdd \Mynodes \MyEdgesRemoved}%
    \petiteFlecheACentrer{#1}%
    \tikzstyle{endpath}=[-]%
    \tikzPartRule{\Myspecialnodes \Mynodes \MyEdgesAfter}%
}

\newcommand{\RuleRouting}{%
	\renewcommand\Mynodes{
        \node[whitenode,label=left:$v_1$] (v1) at (-0.5,1-0.5) {};
        \node[whitenode,label=right:$v_2$] (v2) at (0.5,1-0.5) {};
        \node[whitenode,label=above right:$w$] (w) at (1,0-0.5) {};
    }%
    \renewcommand\Myspecialnodes{%
        \node[blacknode,label=left:$u$] (u1) at (0,0-0.5) {};
        \node[blacknode,label=left:$u'$] (u2) at (0,3-0.5) {};
        \node[blacknode,label=right:$u''$] (u'') at (1.5,2.5-0.5) {};
        \node (u''') at (2.5,0-0.5) {};
    }%
    \renewcommand\MyEdgesBefore{%
        \myEdge{u1}{v1}
        \myEdge{u1}{v2}
        \myEdge{v1}{v2}
        \myPath{v1}{u2}{\subColor}{left}{0}
 		\draw ($(v1)+(0.12,1.095)$) node[\myPurple,nodelabel] {$P$};
		\myPath{v1}{u''}{\subColor}{left}{0}
		\draw ($(v2)+(-0.1,0.85)$) node[\myPurple,nodelabel] {$P$};

        \myDEdge{v1}{w}
        \myEdge{v2}{w}
        \myCEdge{u1}{w}{sedgemixed}{$S$}
        \myPath{w}{u'''}{\subColor}{left}{0}
        \draw ($(w)+(0.77,0.1)$) node[\myPurple,nodelabel] {$S$};

		\myHalfEdge{u1}{-90+30}{sedgemixed}{$S$}%
        \myHalfEdge{u1}{-90-30}{sedgemixed}{$S$}%
    }%
    \renewcommand\MyEdgesRemoved{
        \myCEdge{u1}{v1}{sedgemixed}{$P'$}
        \myEdge{u1}{v2}
        \myEdge{v1}{v2}
        \myPath{v1}{u2}{\subColor}{left}{0}
        \draw ($(v1)+(0.1,1.1)$) node[\myPurple,nodelabel] {$P'$};
		\myPath{v1}{u''}{black}{left}{0}

        \myDEdge{v1}{w}
        \myEdge{v2}{w}
        \myCEdge{u1}{w}{black}{}
        \myPath{w}{u'''}{black}{left}{0}

		\myHalfEdge{u1}{-90+30}{sedgemixed}{$S$}%
        \myHalfEdge{u1}{-90-30}{sedgemixed}{$S$}%
    }%
    \renewcommand\MyEdgesAfter{%
    	\myDEdge{v1}{w}

		\myCEdge{u1}{v1}{black}{}
        \myCEdge{u1}{v2}{sedgemixed}{$P'$}
        \myCEdge{v1}{v2}{sedgemixed}{$P'$}
        \myPath{v1}{u2}{\subColor}{left}{0}
        \draw ($(v1)+(0.1,1.1)$) node[\myPurple,nodelabel] {$P'$};
		\myPath{v1}{u''}{black}{left}{0}

        \myEdge{v2}{w}
        \myCEdge{u1}{w}{black}{}
        \myPath{w}{u'''}{black}{left}{0}

		\myHalfEdge{u1}{-90+30}{sedgemixed}{$S$}%
        \myHalfEdge{u1}{-90-30}{sedgemixed}{$S$}%
    }%
    \renewcommand{\CfN}{}%
    \begin{figure}[ht]%
    	\ifthenelse{\isundefined{\NoFigures}}{%
        \centering%
        \tkcf{} %
        \petiteFlecheACentrer{1} %
        \tikzPartRule{\Myspecialnodes \Mynodes \MyEdgesRemoved}%
        \petiteFlecheACentrer{1} %
        \tikzPartRule{\Myspecialnodes \Mynodes \MyEdgesAfter}%
        }{}%
		\caption{On the left: $u$ causes a distant problem on a path $P = u'\sim u''$ of the subdivision. In the middle: a new subdivision is considered for some reduction rule, with a new path $P' = u'\sim u$. On the right: the routing operation modifies the path $P' = u'\sim u$ in order to choose two non-adjacent vertices ($v_1,w$) as remaining neighbors for $u$.}%
        \label{fig:routing}%
    \end{figure}%
}%

\newcommand{\PbDist}{%
	\renewcommand\Mynodes{
        \node[whitenode,label=left:$v_1$] (v1) at (-0.5,1-0.5) {};
        \node[whitenode,label=right:$v_2$] (v2) at (0.5,1-0.5) {};
    }%
    \renewcommand\Myspecialnodes{%
        \node[blacknode,label=left:$u$] (u1) at (0,0-0.5) {};
        \node[blacknode,label=left:$u'$] (u2) at (-1.5,2.5) {};
        \node[blacknode,label=right:$u''$] (u'') at (1.5,2.5) {};
    }%
    \renewcommand\MyEdgesBefore{%
		\myEdge{u1}{v1}
        \myEdge{u1}{v2}
        \myEdge{v1}{v2}
        \myPath{u2}{v1}{\subColor}{left}{0}
		\myPath{v1}{u''}{\subColor}{left}{0}

		\myHalfEdge{u1}{-90}{sedgemixed}{$S$}%
		\myHalfEdge{u1}{-90+40}{sedgemixed}{$S$}%
        \myHalfEdge{u1}{-90-40}{sedgemixed}{$S$}%
    }%
    \renewcommand\MyEdgesAfter{%
		\myCEdge{u1}{v1}{black}{}
        \myPath{u2}{v1}{\subColor}{left}{0}
		\myPath{v1}{u''}{\subColor}{left}{0}

		\myHalfEdge{u1}{-90}{sedgemixed}{$S$}%
		\myHalfEdge{u1}{-90+40}{sedgemixed}{$S$}%
        \myHalfEdge{u1}{-90-40}{sedgemixed}{$S$}%
    }%
    \renewcommand{\CfN}{}%
    \begin{figure}[ht]%
    	\ifthenelse{\isundefined{\NoFigures}}{%
        \centering%
        \tkcf{} %
        }{}%
        \caption{Distant problem caused by a special vertex $u$}
        \label{fig:pbdist}%
    \end{figure}%
}%

\newcommand{\cfCDOnene}{%
    \renewcommand\Mynodes{
        \node[whitenode,label=right:$v_1$] (v1) at (1.5,1.5) {};
        \node[whitenode,label=right:$v_2$] (v2) at (1.5,0.5) {};
        \node[whitenode,label=above:$v$] (v) at (0.75, 1) {};
    }%
    \renewcommand\Myspecialnodes{%
        \node[blacknode,label=left:$u_1$] (u1) at (0,2) {};
        \node[blacknode,label=left:$u_2$] (u2) at (0,0) {};
    }%
    \renewcommand\MyEdgesBefore{%
		\myHalfEdge{u1}{135}{sedgemixed}{$S$}%
        \myHalfEdge{u1}{20}{sedgemixed}{$S$}%
        \myHalfEdge{u2}{-135}{sedgemixed}{$S$}%
        \myHalfEdge{u2}{-20}{sedgemixed}{$S$}%
        \myPath{u1}{u2}{\subColor}{left}{0}

        \myEdge{u1}{v}
        \myEdge{v}{v2}
        \myEdge{v}{v1}
        \myEdge{u1}{v1}
        \myEdge{u2}{v}
        \myEdge{u2}{v2}
    }%
    \renewcommand{\CfN}{$(\C_{D1})$}%
    \renewcommand{\CfL}{fig:CDOneNE}%
}%

\newcommand{\cfCDTwowo}{%
    \renewcommand\Mynodes{
        \node[whitenode,label=above right:$v'$] (v') at (1.5,1) {};
        \node[whitenode,label=above:$v$] (v) at (0.75, 1) {};
    }%
    \renewcommand\Myspecialnodes{%
        \node[blacknode,label=left:$u_1$] (u1) at (0,2) {};
        \node[blacknode,label=left:$u_2$] (u2) at (0,0) {};
    }%
    \renewcommand\MyEdgesBefore{%
		\myHalfEdge{u1}{135}{sedgemixed}{$S$}%
        \myHalfEdge{u1}{20}{sedgemixed}{$S$}%
        \myHalfEdge{u2}{-135}{sedgemixed}{$S$}%
        \myHalfEdge{u2}{-20}{sedgemixed}{$S$}%
        \myPath{u1}{u2}{\subColor}{left}{0}

        \myEdge{u1}{v}
        \myEdge{u2}{v}
        \myBentEdge{u1}{v'}{black}{left}{20}
        \myBentEdge{u2}{v'}{black}{right}{20}
    }%
    \renewcommand{\CfN}{$(\C_{D2})$}%
    \renewcommand{\CfL}{fig:CDTwoWO}%
}%

\newcommand{\cfSettledA}{%
    \renewcommand\Mynodes{
        \node[whitenode,label=left:$v_1$] (v1) at (-0.6,1) {};
        \node[whitenode,label=above:$v$] (v) at (0.4, 1.25) {};

        \node (w1) at (-0.6, 2.25) {};
        \node (w2) at (1.4, 0.25) {};
    }%
    \renewcommand\Myspecialnodes{%
        \node[blacknode,label=left:$u_1$] (u1) at (0,0) {};
    }%
    \renewcommand\MyEdgesBefore{%
		\myHalfEdge{u1}{-120}{sedgemixed}{$S$}%
        \myHalfEdge{u1}{0}{sedgemixed}{$S$}%
        \myHalfEdge{u1}{-60}{sedgemixed}{$S$}%

        \myPath{w1}{v}{\subColor}{left}{0}
        \myPath{v}{w2}{\subColor}{left}{0}

        \myEdge{u1}{v}
        \myEdge{v}{v1}
        \myEdge{u1}{v1}
    }%
    \renewcommand{\CfN}{}%
    \renewcommand{\CfL}{fig:SettledA}%
}%

\newcommand{\cfSettledB}{%
    \renewcommand\Mynodes{
        \node[whitenode,label=left:$v_1$] (v1) at (0.4,1.25) {};
        \node[whitenode,label=right:$v_2$] (v2) at (1.6,1.25) {};
        \node[whitenode,label=above:$v$] (v) at (1, 0.75) {};
    }%
    \renewcommand\Myspecialnodes{%
        \node[blacknode,label=right:$u_1$] (u1) at (0,0) {};
        \node[blacknode,label=left:$u_2$] (u2) at (2,0) {};
    }%
    \renewcommand\MyEdgesBefore{%
		\myHalfEdge{u1}{-120}{sedgemixed}{$S$}%
        \myHalfEdge{u1}{180}{sedgemixed}{$S$}%
        \myHalfEdge{u1}{300}{sedgemixed}{$S$}%
        \myHalfEdge{u2}{-120}{sedgemixed}{$S$}%
        \myHalfEdge{u2}{0}{sedgemixed}{$S$}%
        \myHalfEdge{u2}{-60}{sedgemixed}{$S$}%

        \myEdge{u1}{v}
        \myEdge{v}{v2}
        \myEdge{v}{v1}
        \myEdge{u1}{v1}
        \myEdge{u2}{v}
        \myEdge{u2}{v2}
    }%
    \renewcommand{\CfN}{}%
    \renewcommand{\CfL}{fig:SettledB}%
}%

\newcommand{\cfSettledC}{%
    \renewcommand\Mynodes{
        \node[whitenode,label=above:$v_1$] (v1) at (1.5,1.7) {};
        \node[whitenode,label=below:$v_2$] (v2) at (1.5,-0.2) {};
    }%
    \renewcommand\Myspecialnodes{%
        \node[blacknode,label=below right:$u_1$] (u1) at (0,1.5) {};
        \node[blacknode,label=above right:$u_2$] (u2) at (0,0) {};
        \node[blacknode,label=left:$u_3$] (u3) at (3,0.75) {};
    }%
    \renewcommand\MyEdgesBefore{%
		\myHalfEdge{u1}{90}{sedgemixed}{$S$}%
        \myHalfEdge{u1}{90+60}{sedgemixed}{$S$}%
        \myHalfEdge{u1}{90+2*60}{sedgemixed}{$S$}%
        \myHalfEdge{u2}{-90}{sedgemixed}{$S$}%
        \myHalfEdge{u2}{-90-60}{sedgemixed}{$S$}%
        \myHalfEdge{u2}{-90-2*60}{sedgemixed}{$S$}%
        \myHalfEdge{u3}{60}{sedgemixed}{$S$}%
        \myHalfEdge{u3}{0}{sedgemixed}{$S$}%
        \myHalfEdge{u3}{-60}{sedgemixed}{$S$}%

        \myEdge{u1}{v1}
        \myEdge{u1}{u2}
        \myEdge{u2}{v2}
        \myEdge{u3}{v1}
        \myEdge{u3}{v2}
        \myDEdge{v1}{v2}
    }%
    \renewcommand{\CfN}{}%
    \renewcommand{\CfL}{fig:SettledC}%
}%

\newcommand{\RuleCDDRed}{%
	\renewcommand\Myspecialnodes{
        \node[blacknode,label=below:$u_1$] (u1) at (-3,0) {};
        \node[blacknode,label=below:$u_2$] (u2) at (3,0) {};
        \node[blacknode,label=left:$u_3$] (u3) at (0,2) {};
        \node[blacknode,label=left:$u_4$] (u4) at (0,6) {};
    }%
    \renewcommand\Mynodes{
    	\node[whitenode,label=right:$v_1$] (v1) at (0,3) {};
    	\node[whitenode,label=left:$v_1'$] (v1') at (-0.8,4) {};
    	\node[whitenode,label=left:$v_2$] (v2) at (0,3.8) {};
    	\node[whitenode,label=right:$v_2'$] (v2') at (1.2,4) {};
    	\node[whitenode] (w4) at (0,4.6) {};
    	\node[whitenode] (v3) at (0,0.2) {};
    	\node[whitenode] (v3') at (0.8,0.6) {};
    }%
    \renewcommand\MyEdgesBefore{
    	\myBentEdge{u1}{u2}{black}{right}{20}
        \myPath{u1}{u3}{\subColor}{left}{0}
        \myPath{u2}{u3}{\subColor}{left}{0}
        \myPath{u1}{u4}{\subColor}{left}{30}
        \myPath{u2}{u4}{\subColor}{right}{30}
        \myPath{u3}{v1}{\subColor}{left}{0}
        \myPath{v1}{v2}{black}{left}{0}
        \myPath{v2}{w4}{\subColor}{left}{0}
        \myBentEdge{w4}{u4}{\subColor}{left}{0}
		\myBentEdge{u1}{v1}{\subColor}{left}{15}
        \myBentEdge{u1}{v1'}{black}{left}{15}
        \myBentEdge{u2}{v2}{\subColor}{right}{15}
        \myBentEdge{u2}{v2'}{black}{right}{15}
        \myEdge{v1}{v1'}
        \myEdge{v2}{v2'}
        \myEdge{u3}{v3}
        \myEdge{u3}{v3'}

		\myEdge{u4}{v2'}
		\myEdge{u4}{v1'}
    }%
    \renewcommand\MyEdgesAfter{
        \myPath{u1}{u3}{\subColor}{left}{0}
        \myPath{u2}{u3}{\subColor}{left}{0}
        \myPath{u1}{u4}{\subColor}{left}{30}
        \myPath{u2}{u4}{\subColor}{right}{30}
        \myPath{u3}{v1}{\subColor}{left}{0}
        \myPath{v1}{v2}{black}{left}{0}
        \myPath{v2}{w4}{black}{left}{0}
		\myBentEdge{u1}{v1}{\subColor}{left}{15}
        \myBentEdge{u2}{v2}{\subColor}{right}{15}
        \myEdge{v1}{v1'}
        \myBentEdge{v2}{v2'}{\subColor}{left}{0}
        \myBentEdge{u4}{v2'}{\subColor}{left}{0}

        \patternCN{u3}{v3}{v3'}{\PatternColorIII}

        \myBentEdge{u1}{u2}{\PatternColorI}{right}{20}
        \myBentEdge{u1}{v1'}{\PatternColorI}{left}{15}
        \myBentEdge{u2}{v2'}{\PatternColorI}{right}{15}
        \node at ($(u1)+(1.5,0)$)  {\textcolor{\PatternColorI}{$\mathcal{C}_{U}$}};
		\patternCN{u4}{w4}{v1'}{\PatternColorII}
    }%
    \begin{figure}%
    	\ifthenelse{\isundefined{\NoFigures}}{%
        \centering%
        \tkcfSansNomFigure{}%
        \petiteFlecheACentrer{3.5}%
        \tikzPartRule{\Myspecialnodes \Mynodes \MyEdgesAfter}%
        }{}%
        \caption{Reduction of configuration $C_{4+}^{U}$.}%
        \label{fig:C4DDRed}%
    \end{figure}%
}%

\newcommand{\cfExcGraph}{%
    \renewcommand\Mynodes{
        \node[whitenode] (v1) at (-126:2) {};
        \node[whitenode] (v2) at (162:2) {};
        \node[whitenode] (v3) at (90:2) {};
        \node[whitenode] (v4) at (18:2) {};
        \node[whitenode] (v5) at (-54:2) {};

        \node[whitenode] (w1) at (-4, -2) {};
        \node[whitenode] (w5) at (4, -2) {};
    }%
    \renewcommand\Myspecialnodes{%
    }%
    \renewcommand\MyEdgesBefore{%
		 \myPath{w1}{v1}{black}{left}{0}
		 \myEdge{v1}{v2}
		 \myEdge{v1}{v3}
		 \myEdge{v1}{v4}
		 \myEdge{v2}{v3}
		 \myEdge{v2}{v4}
		 \myEdge{v2}{v5}
		 \myEdge{v3}{v4}
		 \myEdge{v3}{v5}
		 \myEdge{v4}{v5}
		 \myPath{v5}{w5}{black}{left}{0}
    }%
    \renewcommand{\CfN}{}%
    \renewcommand{\CfL}{fig:ExcGraph}%
}%

\newcommand{\allPatterns}{

\renewcommand{\MyScale}{1.10}
\renewcommand\OrdonneeFleche{1.4}

The following patterns involve one special vertex.

{\center
\hspace{8mm}
\RuleCV

}
\begin{itemize}
\item \ncf{\cfCV}: Identical to the elementary partial configuration \ncf{\cfCV} from
\ifthenelse{\equal{\isThesis}{true}}
{Chapter~\ref{ch:ci}.}
{Section~\ref{sec:mce}.}

{\center
\RuleCvp

}
\item \ncf{\cfCvp}: The special vertex $u_1$ has two adjacent remaining neighbors $v_1,v_2$, and the edge $v_1v_2$ belongs to $S$.

\patred
In the reduced graph, we add the edge $v_1v_2$.

\patrec
In $G$, we deviate the color of $v_1v_2$ on $u_1$. The color of $v_1v_2$ in the recoloring of $G$ is given by the subdivision.

\renewcommand{\MyScale}{1.00}
\renewcommand\OrdonneeFleche{1.4}

{\center
\RuleCvTwo

}
\item \ncf{\cfCvTwo}: Identical to the elementary partial configuration \ncf{\cfCvTwo} from
\ifthenelse{\equal{\isThesis}{true}}
{Chapter~\ref{ch:ci}.}
{Section~\ref{sec:mce}.}

\patcol
If the remaining neighbor $v_1$ is even, then the other remaining neighbor $v_2$ of $u_1$ cannot touch the color of $S$ that ends on $u_1$.

{\center
\RuleCvThree

}
\item \ncf{\cfCvThree}: Identical to the elementary partial configuration \ncf{\cfCvThree} from
\ifthenelse{\equal{\isThesis}{true}}
{Chapter~\ref{ch:ci}.}
{Section~\ref{sec:mce}.}

\patcol
None of the remaining neighbors $v_1,v_2$ of $u_1$ can touch the color of $S$ that ends on $u_1$.
\end{itemize}

The following patterns involve two special vertices.

\renewcommand{\MyScale}{1.20}

{\center
\hspace{8mm}
\RuleCU

}
\begin{itemize}
\item \ncf{\cfCU}: The two special vertices $u_1, u_2$ are adjacent but the edge $u_1u_2$ does not belong to $S$. Let $v_1$ (resp. $v_2$) be the remaining neighbor of $u_1$ (resp. $u_2$) distinct from $u_2$ (resp $u_1$). The vertices $v_1,v_2$ are distinct, non-adjacent and disjoint from $U$.

\patred
In the reduced graph, we add the edge $v_1v_2$.

\patrec
In $G$, we deviate the color of $v_1v_2$ on $u_1,u_2$.
\end{itemize}

\renewcommand{\MyScale}{1.3}
\renewcommand\OrdonneeFleche{1.4}

{\center
\RuleCDOnenea

}
\begin{itemize}
\item \ncf{\cfCDOnenea}:
The two special vertices $u_1,u_2$ are adjacent and the edge $u_1,u_2$ belongs to $S$. The special vertices $u_1,u_2$ have precisely one common remaining neighbor $v$, and $u_1,u_2$ have $v_1,v_2$ respectively as their other remaining neighbor. The vertices $v_1,v_2$ are adjacent and both are adjacent to $v$.
The vertices $v,v_1,v_2$ are disjoint from $S$.

\patred
In the reduced graph, we remove the edge $v_1v_2$.

\patrec
In $G$, we deviate the color of $vv_1$ on the edges $vu_2$, $u_2u_1$ and $u_1v_1$. We redirect the path $u_1\sim u_2$ of $S$ through the edges $u_1v$, $vv_1$, $v_1v_2$ and $v_2u_2$.

{\center
\RuleCDOneneb

}
\item \ncf{\cfCDOneneb}:
The two special vertices $u_1,u_2$ are adjacent and the edge $u_1,u_2$ belongs to $S$. The special vertices $u_1,u_2$ have precisely one common remaining neighbor $v$, and $u_1,u_2$ have $v_1,v_2$ respectively as their other remaining neighbor. The vertices $v_1,v_2$ are not adjacent and both are adjacent to $v$.
The vertex $v$ does not belong to $S$.

\patred
In the reduced graph, we add the edge $v_1v_2$.

\patrec
In $G$, we deviate the color of $v_1v_2$ on the edges $v_1u_1$, $u_1u_2$ and $u_2v_2$. We redirect the path $u_1\sim u_2$ of $S$ to make it go through the edges $u_1v$ and $vu_2$.
\end{itemize}

\renewcommand\OrdonneeFleche{1.7}
\renewcommand{\MyScale}{1.37}

{\center
\RuleCTOnea

}
\begin{itemize}
\item \ncf{\cfCTOnea}: The two special vertices $u_1,u_2$ have precisely one remaining neighbor $v$ in common. We denote $v_1,v_2$ the other remaining neighbor of $u_1,u_2$ respectively. Both $v_1$ and $v_2$ are adjacent to $v$. The vertices $v,v_1,v_2$ are disjoint from $U$.
The vertex $v_1$ has an even degree.

\patcol
The vertices $v,v_2$ cannot touch the color of $S$ that ends on $u_1$.

\patred
In the reduced graph, $v_1$ has an odd degree: let $Q$ be a path of the coloring of $G'$ that ends on $v_1$.

\patrec
In $G$, we deviate the color of $vv_2$ on $u_2$, we extend the path $Q$ on the edge $v_1u_1$, and we extend the extra color that ends on $u_1$ on the edges $u_1v$ and $vv_2$.
\end{itemize}

\ifthenelse{\equal{\isThesis}{true}}
{}
{\vfill
\pagebreak}

\renewcommand{\MyScale}{1.27}

{\center
\RuleCTOneb

}
\begin{itemize}
\item \ncf{\cfCTOneb}: The two special vertices $u_1,u_2$ have precisely one remaining neighbor $v$ in common. We denote $v_1,v_2$ the other remaining neighbor of $u_1,u_2$ respectively. Both $v_1$ and $v_2$ are adjacent to $v$. The vertices $v,v_1,v_2$ are disjoint from $U$.
The vertices $v_1, v_2$ both have an odd degree in $G$.

\patcol
The colors ending on $u_1,u_2$ in a $2$-coloring of $S$ must be different. The vertex $v_1$ (resp. $v_2$) cannot touch the color that ends on $u_1$ (resp. $u_2$).

\patred
In the reduced graph, we remove the edges $vv_1$ and $vv_2$. The vertices $v_1,v_2$ keep an odd degree in $G'$: let $Q,R$ be paths of the coloring of $G'$ that end on $v_1,v_2$ respectively.

\patrec
In $G$, we extend the paths $Q,R$ on the edges $v_1u_1$ and $v_2,u_2$ respectively. We extend the extra color ending on $u_1$ on the edges $u_1v$ and $vv_1$, and we extend the extra color ending on $u_2$ on the edges $u_2v$ and $vv_2$.
\end{itemize}

\renewcommand{\MyScale}{1.37}

{\center
\RuleCTTAa

}
\begin{itemize}
\item \ncf{\cfCTTAa}: The two special vertices $u_1,u_2$ have both their remaining neighbors $v,v'$ in common, which are adjacent and disjoint from $U$.
The vertex $v$ has an odd degree in $G$.

\patcol
There is a color of $S$ ending on $u_1$ or $u_2$ (let us say $u_1$) that does not touch $v$ nor $v'$.

\patred
In the reduced graph, $v$ keeps an odd degree: let $Q$ be a path of the coloring of $G'$ that ends on $v$.

\patrec
In $G$, we extend the path $Q$ on the edge $vu_1$, we deviate the color of $vv'$ on $u_2$, and we extend the extra color ending on $u_1$ on the edges $u_1v'$ and $v'v$.
\end{itemize}

\renewcommand{\MyScale}{1.28}

{\center
\RuleCTTAb

}
\begin{itemize}
\item \ncf{\cfCTTAb}: The two special vertices $u_1,u_2$ have both their remaining neighbors $v,v'$ in common, which are adjacent and disjoint from $U$.
Both $v$ and $v'$ have an even degree in $G$.

\patcol
The colors of $S$ that end on $u_1$ and $u_2$ must be different. At least one of $v,v'$ (let us say $v'$) does not touch at least one of the two colors of $S$ (let us say the one ending on $u_1$).

\patred
In the reduced graph, we remove the edge $vv'$. The vertices $v,v'$ have an odd degree in $G'$: let $T,R$ be paths of the coloring of $G'$ that end on $v,v'$ respectively.

\patrec
In $G$, we extend the paths $T,R$ on the edges $vu_1$ and $v'u_2$ respectively. We extend the extra color ending on $u_1$ on the edges $u_1v'$ and $v'v$, and we extend the extra color ending on $u_2$ on the edge $u_2v$.
\end{itemize}

\renewcommand{\MyScale}{1.37}

{\center
\RuleCTTNAa

}
\begin{itemize}
\item \ncf{\cfCTTNAa}: The two special vertices $u_1,u_2$ have both their remaining neighbors $v,v'$ in common, which are not adjacent and are disjoint from $U$.
The vertex $v$ has an even degree in the $G$.

\patcol
One of $v,v'$ (let us say $v'$) does not touch a color of $S$ that ends on $u_1$ or $u_2$ (let us say $u_1$).

\patred
In the reduced graph, we add the edge $vv'$. The vertex $v$ has an odd degree in $G'$: let $Q$ be a path of the coloring of $G'$ that ends on $v$.

\patrec
In $G$, we extend the path $Q$ on the edge $vu_1$, we deviate the color of $vv'$ on $u_2$, and we extend the extra color ending on $u_1$ on the edge $u_1v'$.
\end{itemize}

\renewcommand{\MyScale}{1.28}

{\center
\RuleCTTNAb

}
\begin{itemize}
\item \ncf{\cfCTTNAb}: The two special vertices $u_1,u_2$ have both their remaining neighbors $v,v'$ in common, which are not adjacent and are disjoint from $U$.
Both $v$ and $v'$ have an odd degree in $G$.

\patcol
One of $v,v'$ (let us say $v'$) does not touch the color ending on $u_1$, the other ($v$) does not touch the one ending on $u_2$.

\patred
In the reduced graph, $v,v'$ keep an odd degree: let $T,R$ be paths of the coloring of $G'$ that end on $v,v'$ respectively.

\patrec
In $G$, we extend the paths $T,R$ on the edges $vu_1$ and $v'u_2$ respectively. We extend the extra colors ending on $u_1,u_2$ on the edges $u_1v'$ and $u_2v$ respectively.
\end{itemize}

\medskip

For convenience, we define some aliases which group several patterns together.

\renewcommand{\MyScale}{1.15}
\renewcommand\OrdonneeFleche{1.2}

{\center

\tkcf{\cfCN} %
			  \TexteCentre{$:=\{$}{\OrdonneeFleche}%
			  \tkcf{\cfCV} %
			  \TexteCentre{$\vert$}{\OrdonneeFleche} %
			  \tkcf{\cfCvTwo} %
			  \TexteCentre{$\vert$}{\OrdonneeFleche} %
			  \tkcf{\cfCvThree}%
			  \TexteCentre{$\}$}{\OrdonneeFleche}%

}

\begin{itemize}
\item \ncf{\cfCN}: The special vertex $u_1$ has $2$ remaining neighbors $v_1,v_2$. If $v_1,v_2$ are non-adjacent, this is configuration \ncf{\cfCV}. Otherwise, if one of $v_1,v_2$ has an even degree in $G$, this is configuration \ncf{\cfCvTwo}, and if both have an odd degree in $G$, this is configuration \ncf{\cfCvThree}.

\renewcommand{\MyScale}{1.2}
\renewcommand\OrdonneeFleche{1.7}

{\center

\tkcf{\cfCTOne} %
			  \TexteCentre{$:=\{$}{\OrdonneeFleche}%
			  \tkcf{\cfCTOnea} %
			  \TexteCentre{$\vert$}{\OrdonneeFleche} %
			  \tkcf{\cfCTOneb}%
			  \TexteCentre{$\}$}{\OrdonneeFleche}%

}

\item \ncf{\cfCTOne}: The two special vertices $u_1,u_2$ have one remaining neighbor $v$ in common. We denote $v_1,v_2$ the other remaining neighbor of $u_1,u_2$ respectively. Both $v_1$ and $v_2$ are adjacent to $v$. The vertices $v,v_1,v_2$ are disjoint from $U$.

{\center

\tkcf{\cfCTTA} %
			  \TexteCentre{$:=\{$}{\OrdonneeFleche} %
			  \tkcf{\cfCTTAa} %
			  \TexteCentre{$\vert$}{\OrdonneeFleche} %
			  \tkcf{\cfCTTAb}
			  \TexteCentre{$\}$}{\OrdonneeFleche}%

}

\item \ncf{\cfCTTA}: The two special vertices $u_1,u_2$ have both their remaining neighbors $v,v'$ in common, which are adjacent and disjoint from $U$.

{\center

\tkcf{\cfCTTNA} %
			  \TexteCentre{$:=\{$}{\OrdonneeFleche} %
			  \tkcf{\cfCTTNAa} %
			  \TexteCentre{$\vert$}{\OrdonneeFleche} %
			  \tkcf{\cfCTTNAb}
			  \TexteCentre{$\}$}{\OrdonneeFleche}%

}

\item \ncf{\cfCTTNA}: The two special vertices $u_1,u_2$ have both their remaining neighbors $v,v'$ in common, which are not adjacent and are disjoint from $U$.
\end{itemize}

}

\newcommand{\cfExOne}{%
	\renewcommand\Myspecialnodes{%
        \node[blacknode,label={[label distance=-0.08cm]45:\labelUOne}] (u1) at (0,0) {};%
		\node[blacknode,label={[label distance=-0.08cm](90+45):\labelUTwo}] (u2) at (2,0) {};%
    }%
    \renewcommand\Mynodes{%
        \node[whitenode]
        	(v1) at (0,1) {};%
        \node[whitenode]
        	(v1') at (-1,0) {};%
        \node[whitenode]
        	(v1'') at (0,-1) {};%
        \node[whitenode]
        	(v2) at ($(2,0)+(45:1)$) {};%
        \node[whitenode]
        	(v2') at ($(2,0)+(-45:1)$) {};%
    }%
    \renewcommand\MyEdgesBefore{%
    	\myPath{u1}{u2}{black}{left}{0}%
        \myEdge{u1}{v1}%
        \myEdge{u1}{v1'}%
        \myEdge{u1}{v1''}%
        \myEdge{u2}{v2}%
        \myEdge{u2}{v2'}%
        \myEdge{v1}{v1'}%
        \myEdge{v1'}{v1''}%
        \myDEdge{v2}{v2'}%
    }%
    \renewcommand{\CfN}{}%
}%

\newcommand{\RuleExOne}{%
    \cfExOne
    \renewcommand\MyEdgesRemoved{
        \myCEdge{v1}{v1'}{blue}{$Q$}
        \myCEdge{v2}{v2'}{green}{$R$}
        \myDEdge{v1'}{v1''}%
    }%
    \renewcommand\MyEdgesAfter{%
        \draw[red, thick, decorate, decoration=snake, segment aspect=0, segment amplitude=0.5mm, segment length = 4mm, bend left = 0] (u1) to node[red,nodelabel,midway] {$P$} (u2);
        \myCEdge{u1}{v1}{blue}{$Q$}%
        \myCEdge{u1}{v1'}{blue}{$Q$}%
        \myCEdge{u1}{v1''}{red}{$P$}%
        \myCEdge{u2}{v2}{green}{$R$}%
        \myCEdge{u2}{v2'}{green}{$R$}%
        \myCEdge{v1}{v1'}{red}{$P$}%
        \myCEdge{v1'}{v1''}{red}{$P$}%
        \myDEdge{v2}{v2'}%
    }%
	\drawRule{}{}%
}%

\newcommand{\cfExTwo}{%
	\renewcommand\Myspecialnodes{%
        \node[blacknode,label=below:$u_1$] (u1) at (-3,0) {};%
		\node[blacknode,label=below:$u_2$] (u2) at (3,0) {};%
		\node[blacknode,label=left:$u_3$] (u3) at (0,3) {};%
		\node[blacknode,label=above left:$u_4$] (u4) at (0,6) {};%
    }%
    \renewcommand\Mynodes{%
		\node[whitenode] (v1) at ($(-3,0)+(32:2)$) {};%
        \node[whitenode] (v1') at ($(-3,0)+(3:2)$) {};%
        \node[whitenode] (v2) at ($(3,0)+(180-32:2)$) {};%
        \node[whitenode] (v2') at ($(3,0)+(180-3:2)$) {};%
        \node[whitenode] (v3) at ($(0,3)+(-90-20:2)$) {};%
        \node[whitenode] (v3') at ($(0,3)+(-90+20:2)$) {};%
        \node[whitenode] (v4) at ($(0,6)+(-78:2)$) {};%
        \node[whitenode] (v4') at ($(0,6)+(-56:2)$) {};%
    }%
    \renewcommand\MyEdgesBefore{%
    	\myPath{u1}{u2}{\subColor}{right}{20}
    	\myPath{u1}{u3}{\subColor}{left}{10}
    	\myPath{u1}{u4}{\subColor}{left}{20}
    	\myPath{u3}{u2}{\subColor}{left}{10}
    	\myPath{u2}{u4}{\subColor}{right}{20}
    	\myPath{u3}{u4}{\subColor}{left}{0}
    	
    	\patternCN{u1}{v1}{v1'}{black}
    	\patternCV{u2}{v2}{v2'}{black}
    	\patternCV{u3}{v3}{v3'}{black}
    	\patternCN{u4}{v4}{v4'}{black}
		
    }%
    \renewcommand{\CfN}{}%
}%
\newcommand{\RuleExTwo}{%
    \cfExTwo
    \renewcommand\MyEdgesRemoved{%
		\myEdge{v1}{v1'}
		\myCEdge{v2}{v2'}{\PatternColorII}{$Q_2$}
		\myCEdge{v3}{v3'}{\PatternColorIII}{$Q_3$}
		\myDEdge{v4}{v4'}
		\myHalfEdge{v1'}{-10}{\PatternColorI}{$Q_1$}%
		\myHalfEdge{v4'}{-60}{\PatternColorIV}{$Q_4$}%
    }%
    \renewcommand\MyEdgesAfter{%
    	\myPath{u1}{u2}{\subColorOne}{right}{20}
    	\myPath{u1}{u3}{\subColorTwo}{left}{10}
    	\myPath{u1}{u4}{\subColorTwo}{left}{20}
    	\myPath{u3}{u2}{\subColorOne}{left}{10}
    	\myPath{u2}{u4}{\subColorTwo}{right}{20}
    	\myPath{u3}{u4}{\subColorOne}{left}{0}
    	
		\myBentEdge{u1}{v1}{\subColorOne}{left}{0}
		\myCEdge{u1}{v1'}{\PatternColorI}{$Q_1$}
		\myHalfEdge{v1'}{-10}{\PatternColorI}{}%
		\myEdge{v1'}{v1}%
		\myCEdge{u2}{v2}{\PatternColorII}{$Q_2$}
		\myCEdge{u2}{v2'}{\PatternColorII}{$Q_2$}
		\myDEdge{v2}{v2'}
		\myCEdge{u3}{v3}{\PatternColorIII}{$Q_3$}
		\myCEdge{u3}{v3'}{\PatternColorIII}{$Q_3$}
		\myDEdge{v3}{v3'}
		\myCEdge{u4}{v4}{\subColorOne}{}
		\myCEdge{u4}{v4'}{\PatternColorIV}{$Q_4$}
		\myHalfEdge{v4'}{-60}{\PatternColorIV}{}%
		\myBentEdge{v4'}{v4}{\subColorOne}{left}{0}
    }%
    \drawRule{}{}%
}%

\newcommand{\cfExThree}{%
	\renewcommand\Myspecialnodes{%
        \node[blacknode,label=below:$u_1$] (u1) at (-3,0) {};%
		\node[blacknode,label=below:$u_2$] (u2) at (3,0) {};%
		\node[blacknode,label=left:$u_3$] (u3) at (0,3) {};%
		\node[blacknode,label=above:$u_4$] (u4) at (0,6) {};%
    }%
    \renewcommand\Mynodes{%
        \node[whitenode] (v12) at (0,0.5) {};
    	\node[contactnode] (v4) at (-0.5,4) {};
    	\node[whitenode] (v4') at (-1.2,4.3) {};
    	\node[contactnode] (v3) at (1,4) {};
    	\node[whitenode] (v3') at (2,3.5) {};
    	\node[contactnode] (v2) at (0.5,1.5) {};
    	\node[contactnode] (v1) at (-0.5,1.5) {};
    	\node[contactnode] (w2) at (1.5,-0.5) {};
    	\node[contactnode] (w1) at (-1.5,-0.5) {};
    }%
    \renewcommand\MyEdgesBefore{%
    	\myBentEdge{u1}{w1}{\subColor}{left}{0}
    	\myBentEdge{w2}{u2}{\subColor}{left}{0}
    	\myPath{w1}{w2}{\subColor}{right}{10}
    	\myPath{u1}{u3}{\subColor}{left}{0}
    	\myPath{u1}{u4}{\subColor}{left}{20}
    	\myPath{u3}{u2}{\subColor}{left}{0}
    	\myPath{u2}{v3'}{\subColor}{right}{10}
    	\myPath{v3'}{u4}{\subColor}{right}{10}
    	\myPath{u3}{u4}{\subColor}{left}{0}
    	
		\myEdge{u1}{v12}
		\myEdge{u2}{v12}
		\myEdge{u1}{v1}
		\myEdge{u2}{v2}
		\myEdge{v1}{v12}
		\myEdge{v2}{v12}
		\myEdge{u3}{v3}
		\myEdge{u3}{v3'}
		\myEdge{v3}{v3'}
		\myEdge{u4}{v4}
		\myEdge{u4}{v4'}
		\myEdge{v4}{v4'}
		\myDEdge{v1}{w1}
		\myDEdge{v2}{w2}
    }%
    \renewcommand\MyEdgesAfter{%
    	\myPath{w1}{w2}{black}{right}{10}
    	\myPath{u1}{u3}{\subColorOne}{left}{0}
    	\myPath{u1}{u4}{\subColorTwo}{left}{20}
    	\myPath{u3}{u2}{\subColorTwo}{left}{0}
    	\myPath{u2}{v3'}{\subColorOne}{right}{10}
    	\myPath{v3'}{u4}{\subColorOne}{right}{10}
    	\myPath{u3}{u4}{\subColorOne}{left}{0}
    	
		\myBentEdge{u1}{v12}{\subColorTwo}{left}{0}
		\myBentEdge{v12}{u2}{\subColorTwo}{left}{0}
		\myBentEdge{u1}{v1}{\PatternColorI}{left}{0}
		\myBentEdge{u1}{w1}{\PatternColorI}{left}{0}
		\node at (-2.2,-0.6) {\textcolor{\PatternColorI}{{\CV}}};
		\myBentEdge{u2}{v2}{\PatternColorII}{left}{0}
		\myBentEdge{u2}{w2}{\PatternColorII}{left}{0}
		\node at (2.2,-0.6) {\textcolor{\PatternColorII}{{\CV}}};
		\myEdge{v1}{v12}
		\myEdge{v2}{v12}
		\patternCN{u3}{v3}{v3'}{\subColorTwo}
		\patternCN{u4}{v4}{v4'}{\subColorTwo}
		\myDEdge{v1}{w1}
		\myDEdge{v2}{w2}
    }%
    \renewcommand{\CfN}{}%
	\drawSimpleRule{fig:cii_distclose}{2.8}{A $(C_{II})$ configuration where $u_1,u_2$ form a close problem and $u_3$ a distant problem. The close problem is eliminated by redirection of the subdivision $S$, and the distant problem is inactivated by the $2$-coloring of $S$.}
}%

\newcommand{\cfExFour}{%
    \renewcommand\Mynodes{
        \node[whitenode] (v1) at (-0.5,0) {};
        \node[whitenode] (v2) at (1.5,0) {};
        \node[whitenode] (v3) at (-1,1) {};
        \node[whitenode] (v4) at (2,1) {};
        \node[whitenode] (v) at ($ (0,1) + (60:1) $) {};
    }%
    \renewcommand\Myspecialnodes{%
        \node[blacknode,label={[label distance=-0.15cm]135:\labelUOne}] (u1) at (0,1) {};%
        \node[blacknode,label={[label distance=-0.15cm]20:\labelUTwo}] (u2) at (1,1) {};%
    }%
    \renewcommand\MyEdgesBefore{%
    	\myDEdge{v1}{v2}%
        \myEdge{u1}{v1}
        \myEdge{u1}{u2}
        \myEdge{u2}{v2}
        \myEdge{u1}{v}
        \myEdge{u2}{v}
        \myEdge{u1}{v3}
        \myEdge{u2}{v4}
        \node at (-0.5,-0.78) {}; 
    }%
    \renewcommand{\CfN}{}%
    \renewcommand{\CfL}{}%
}%

\newcommand{\RuleExFour}{%
    \cfExFour
    \renewcommand\MyEdgesRemoved{%
        \myCEdge{v1}{v2}{blue}{$Q$}%
    }%
    \renewcommand\MyEdgesAfter{%
    	\myDEdge{v1}{v2}
		\myCEdge{u1}{v3}{red}{$P$}
        \myCEdge{u1}{v}{red}{$P$}
        \myCEdge{v}{u2}{red}{$P$}
        \myCEdge{u2}{v4}{red}{$P$}%
        \myCEdge{v1}{u1}{blue}{$Q$}
        \myCEdge{u1}{u2}{blue}{$Q$}
        \myCEdge{u2}{v2}{blue}{$Q$}%
    }%
    \drawRule{}{}%
}%

\newcommand{\SemiCFour}{%
    \renewcommand\Mynodes{
        \node[whitenode,label=above:$w$] (w) at (0,1.5) {};
    }%
    \renewcommand\Myspecialnodes{%
		\node[blacknode,label=left:$u_1$] (u1) at (-3+0.5,3) {};
        \node[blacknode,label=above left:$u_2$] (u2) at (-3+0.5,0) {};
        \node[blacknode,label=above right:$u_3$] (u3) at (3-0.5,3) {};
        \node[blacknode,label=above right:$u_4$] (u4) at (3-0.5,0) {};
    }%
    \renewcommand\MyEdgesBefore{%
    	\myPath{u1}{u3}{\subColorOne}{left}{30}%
        \myPath{u1}{w}{\subColorTwo}{left}{0}%
        \myPath{w}{u3}{\subColorTwo}{left}{0}%
		\myPath{u2}{u4}{\subColorTwo}{right}{30}
        \myPath{u2}{u1}{\subColorOne}{right}{0}%
        \myPath{u3}{u4}{\subColorTwo}{left}{0}%
		\myPath{w}{u2}{\subColorOne}{left}{0}
		\myPath{u4}{w}{\subColorOne}{left}{0}
    }%
    \renewcommand{\CfN}{}%
    \renewcommand{\CfL}{}%
	\begin{figure}[ht]%
    	\ifthenelse{\isundefined{\NoFigures}}{%
        \centering%
        \tkcf{} %
        }{}%
        \caption{A $2$-colored semi-$C_{4+}$-subdivision}
        \label{fig:semiCFour}%
    \end{figure}%
}%

\newcommand{\boxlines}[5]{
	\draw[ultra thick, #1] (#2-0.02,#3) -- (#4+0.02,#3);
	\draw[ultra thick, #1] (#4,#3) -- (#4,#5);
	\draw[ultra thick, #1] (#4+0.02,#5) -- (#2-0.02,#5);
	\draw[ultra thick, #1] (#2,#5) -- (#2,#3);
}

\newcommand{\KTTOne}{%
    \renewcommand\Mynodes{
        \node[whitenode,label=left:$w_1$] (w1) at (-2,3.2) {};
        \node[whitenode,label=left:$w_2$] (w2) at (-2,1.8) {};
        \node[whitenode,label=above:$v_1$] (v1) at (-1,3.5) {};
        \node[whitenode,label=below:$v_2$] (v2) at (-1,1.5) {};
        \node[whitenode,label=below right:$v$] (v) at (0.5,2.5) {};
    }%
    \renewcommand\Myspecialnodes{%
		\node[blacknode,label=left:$u_1$] (u1) at (-2,5) {};
        \node[blacknode,label=left:$u_2$] (u2) at (-2,0) {};
        \node[blacknode,label=right:$u_3$] (u3) at (4,5) {};
        \node[blacknode,label=right:$u_4$] (u4) at (4,0) {};
    }%
    \renewcommand\MyEdgesBefore{%
    	\myPath{u1}{u3}{\subColor}{left}{0}%
        \myPath{u2}{u4}{\subColor}{left}{0}%
        \myPath{w1}{w2}{\subColor}{left}{0}%
        \myPath{u4}{u3}{\subColor}{left}{0}%
        \myPath{u1}{u4}{\subColor}{left}{25}%
        \myBentEdge{u1}{v}{black}{left}{30}%
        \myBentEdge{u2}{v}{black}{right}{30}%
        \myBentEdge{u1}{w1}{\subColor}{left}{0}%
        \myBentEdge{u2}{w2}{\subColor}{left}{0}%
        \myEdge{v1}{v}%
        \myEdge{v2}{v}%
        \myEdge{w1}{v1}%
        \myEdge{w2}{v2}%
        \myEdge{u1}{v1}%
        \myEdge{u2}{v2}%
        \myBentEdge{u3}{v}{black}{left}{25}%
        \myBentEdge{u3}{w1}{black}{left}{30}%
        
        \boxlines{red}{-2.8}{5.4}{-1.5}{4.6}
        \boxlines{red}{-2.8}{3.6}{-1.5}{2.8}
        \boxlines{red}{0}{2.8}{1}{2}
        \boxlines{green}{-1.3}{4.2}{-0.7}{3.2}
        \boxlines{green}{-2.8}{0.4}{4.8}{-0.4}
        \boxlines{green}{3.5}{5.4}{4.8}{4.6}
    }%
    \renewcommand{\CfN}{}%
    \renewcommand{\CfL}{}%
	\begin{figure}[ht]%
    	\ifthenelse{\isundefined{\NoFigures}}{%
        \centering%
        \tkcf{} %
        }{}%
        \caption{The $K_{3,3}$-minor formed by $\{u_1,w_1,v\}$ and $\{v_1,u_2=u_4,u_3\}$ if $S$ is a $K_4$-subdivision (the path $u_2\sim u_3$ is not pictured)}
        \label{fig:KTTOne}%
    \end{figure}%
}%

\newcommand{\KTTTwo}{%
    \renewcommand\Mynodes{
        \node[whitenode,label=left:$w_1$] (w1) at (-2,3.2) {};
        \node[whitenode,label=left:$w_2$] (w2) at (-2,1.8) {};
        \node[whitenode,label=above:$v_1$] (v1) at (-1,3.5) {};
        \node[whitenode,label=below:$v_2$] (v2) at (-1,1.5) {};
        \node[whitenode,label=below right:$v$] (v) at (0.5,2.5) {};
    }%
    \renewcommand\Myspecialnodes{%
		\node[blacknode,label=left:$u_1$] (u1) at (-2,5) {};
        \node[blacknode,label=left:$u_2$] (u2) at (-2,0) {};
        \node[blacknode,label=right:$u_3$] (u3) at (4,5) {};
        \node[blacknode,label=right:$u_4$] (u4) at (4,0) {};
    }%
    \renewcommand\MyEdgesBefore{%
    	\myPath{u1}{u3}{\subColor}{left}{0}%
        \myPath{u2}{u4}{\subColor}{left}{0}%
        \myPath{w1}{w2}{\subColor}{left}{0}%
        \myPath{u4}{u3}{\subColor}{left}{0}%
        \myBentEdge{u1}{v}{black}{left}{30}%
        \myBentEdge{u2}{v}{black}{right}{30}%
        \myBentEdge{u1}{w1}{\subColor}{left}{0}%
        \myBentEdge{u2}{w2}{\subColor}{left}{0}%
        \myEdge{v1}{v}%
        \myEdge{v2}{v}%
        \myEdge{w1}{v1}%
        \myEdge{w2}{v2}%
        \myEdge{u1}{v1}%
        \myEdge{u2}{v2}%
        \myBentEdge{u3}{v}{black}{left}{25}%
        \myBentEdge{u3}{w1}{black}{left}{30}%
        \myPath{u2}{u1}{\subColor}{left}{30}%
        \myPath{u3}{u4}{\subColor}{left}{30}%
        
        \boxlines{red}{-2.8}{5.4}{-1.5}{4.6}
        \boxlines{red}{-2.8}{3.6}{-1.5}{2.8}
        \boxlines{red}{0}{2.8}{1}{2}
        \boxlines{green}{-1.3}{4.2}{-0.7}{3.2}
        \boxlines{green}{-2.8}{0.4}{-1.5}{-0.4}
        \boxlines{green}{3.5}{5.4}{4.8}{4.6}
    }%
    \renewcommand{\CfN}{}%
    \renewcommand{\CfL}{}%
	\begin{figure}[ht]%
    	\ifthenelse{\isundefined{\NoFigures}}{%
        \centering%
        \tkcf{} %
        }{}%
        \caption{The $K_{3,3}$-minor formed by $\{u_1,w_1,v\}$ and $\{v_1,u_2,u_3\}$ if $S$ is a $C_{4+}$-subdivision}
        \label{fig:KTTTwo}%
    \end{figure}%
}%

\newcommand{\ObsTwoPict}{%
    \renewcommand\Mynodes{
        \node[whitenode,label=above:$v_1$] (v1) at (-0.5,3.5) {};
        \node[whitenode,label=below:$v_2$] (v2) at (0,1.5) {};
        \node[whitenode,label=left:$v$] (v) at (-1,2.5) {};
    }%
    \renewcommand\Myspecialnodes{%
		\node[blacknode,label=left:$u_1$] (u1) at (-2,5) {};
        \node[blacknode,label=left:$u_2$] (u2) at (-2,0) {};
        \node[blacknode,label=right:$u_3$] (u3) at (4,5) {};
        \node[blacknode,label=right:$u_4$] (u4) at (4,0) {};
    }%
    \renewcommand\MyEdgesBefore{%
        \myPath{u2}{u4}{\subColor}{left}{0}%
        \myPath{u4}{u3}{\subColor}{left}{0}%
        \myPath{u1}{u2}{\subColor}{left}{0}%
        \myEdge{u1}{v}%
        \myEdge{u2}{v}%
        \myEdge{v1}{v}%
        \myEdge{v2}{v}%
        \myEdge{u1}{v1}%
        \myEdge{u2}{v2}%
        \myEdge{u3}{v1}%
        \myPath{u1}{v2}{\subColor}{left}{75}%
        \myPath{v2}{u3}{\subColor}{right}{25}%
        
        \boxlines{red}{-2.8}{5.4}{-1.5}{4.6}
        \boxlines{red}{3.5}{5.4}{4.8}{4.6}
        \boxlines{red}{-1.6}{2.8}{-0.7}{2.2}
        \boxlines{green}{-0.8}{4.1}{-0.2}{3.2}
        \boxlines{green}{-0.3}{1.8}{0.3}{0.9}
        \boxlines{green}{-2.8}{0.4}{-1.5}{-0.4}
        
    }%
    \renewcommand{\CfN}{}%
    \renewcommand{\CfL}{}%
	\begin{figure}[ht]%
    	\ifthenelse{\isundefined{\NoFigures}}{%
        \centering%
        \tkcf{} %
        }{}%
        \caption{The $K_{3,3}$-minor formed by $\{u_1,u_3,v\}$ and $\{u_2,v_1,v_2\}$ in Observation~\ref{obs:red1}}
        \label{fig:ObsTwoPict}%
    \end{figure}%
}%

\newcommand{\ObsThreePict}{%
    \renewcommand\Mynodes{
        \node[whitenode,label=above:$v_1$] (v1) at (0,4) {};
        \node[whitenode,label=below:$v_2$] (v2) at (2,1.5) {};
        \node[whitenode,label=left:$v$] (v) at (-1,3.5) {};
    }%
    \renewcommand\Myspecialnodes{%
		\node[blacknode,label=left:$u_1$] (u1) at (-2,5) {};
        \node[blacknode,label=left:$u_2$] (u2) at (-2,0) {};
        \node[blacknode,label=left:$u_3$] (u3) at (-0.5,2) {};
        \node[blacknode,label=right:$u_4$] (u4) at (-1,1) {};
    }%
    \renewcommand\MyEdgesBefore{%
    	\begin{scope}
		\clip (-2.5,-0.2) rectangle (2.7,5.5);
		
        \myPath{u2}{u4}{\subColor}{left}{0}%
        \myPath{u4}{u3}{\subColor}{left}{0}%
        \myPath{u1}{u2}{\subColor}{left}{0}%
        \myEdge{u1}{v}%
        \myEdge{u2}{v}%
        \myEdge{v1}{v}%
        \myEdge{v2}{v}%
        \myEdge{u1}{v1}%
        \myEdge{u2}{v2}%
        \myEdge{u3}{v}%
        \myPath{u1}{v2}{\subColor}{left}{75}%
        \myPath{v2}{u3}{\subColor}{left}{0}%
        \node[label=above:\textcolor{\subColor}{$Q_{13}$}] at (2.3,3.5) {};
        \end{scope}
    }%
    \renewcommand{\CfN}{}%
    \renewcommand{\CfL}{}%
	\begin{figure}[ht]%
    	\ifthenelse{\isundefined{\NoFigures}}{%
        \centering%
        \tkcf{} %
        }{}%
        \caption{The planar embedding of the graph of Observation~\ref{obs:red2}}
        \label{fig:ObsThreePict}%
    \end{figure}%
}%

\newcommand{\ObsFourPictOne}{%
    \renewcommand\Mynodes{
        \node[whitenode,label=above:$v_1$] (v1) at (-0.5,3.5) {};
        \node[whitenode,label=below:$v_2$] (v2) at (0,1.5) {};
        \node[whitenode,label=left:$v$] (v) at (-1,2.5) {};
        \node[whitenode,label=above:$v_4$] (v4) at (2,2.5) {};
    }%
    \renewcommand\Myspecialnodes{%
		\node[blacknode,label=left:$u_1$] (u1) at (-2,5) {};
        \node[blacknode,label=left:$u_2$] (u2) at (-2,0) {};
        \node[blacknode,label=right:$u_3$] (u3) at (4,5) {};
        \node[blacknode,label=right:$u_4$] (u4) at (4,0) {};
    }%
    \renewcommand\MyEdgesBefore{%
        \myPath{u2}{u4}{\subColor}{left}{0}%
        \myPath{u4}{u3}{\subColor}{left}{0}%
        \myPath{u1}{u2}{\subColor}{left}{0}%
        \myEdge{u1}{v}%
        \myEdge{u2}{v}%
        \myEdge{v1}{v}%
        \myEdge{v2}{v}%
        \myEdge{u1}{v1}%
        \myEdge{u2}{v2}%
        \myEdge{u4}{v1}%
        \myPath{u1}{v2}{\subColor}{left}{75}%
        \myPath{v2}{v4}{\subColor}{right}{15}%
        \myPath{v4}{u3}{\subColor}{right}{15}%
        \myEdge{u4}{v4}
        
        \boxlines{red}{-2.8}{5.4}{-1.5}{4.6}
        \boxlines{red}{3.5}{0.4}{4.8}{-0.4}
        \boxlines{red}{-1.6}{2.8}{-0.7}{2.2}
        \boxlines{green}{-0.8}{4.1}{-0.2}{3.2}
		\draw[ultra thick, green] (-0.6,1.5) -- (2.1,3.35);
		\draw[ultra thick, green] (-0.05,0.7) -- (2.7,2.55);
		\draw[ultra thick, green] (-0.6,1.5) -- (-0.05,0.7);
		\draw[ultra thick, green] (2.1,3.35) -- (2.7,2.55);
        \boxlines{green}{-2.8}{0.4}{-1.5}{-0.4}
        
    }%
    \renewcommand{\CfN}{}%
    \renewcommand{\CfL}{}%
}%

\newcommand{\ObsFourPictTwo}{%
    \renewcommand\Mynodes{
        \node[whitenode,label=above:$v_1$] (v1) at (0,4) {};
        \node[whitenode,label=below:$v_2$] (v2) at (2,1.5) {};
        \node[whitenode,label=left:$v$] (v) at (-1,3.5) {};
    }%
    \renewcommand\Myspecialnodes{%
		\node[blacknode,label=left:$u_1$] (u1) at (-2,5) {};
        \node[blacknode,label=left:$u_2$] (u2) at (-2,0) {};
        \node[blacknode,label=below:$u_3$] (u3) at (0.5,1.8) {};
        \node[blacknode,label=left:$u_4$] (u4) at (-0.5,2) {};
    }%
    \renewcommand\MyEdgesBefore{%
        \myPath{u2}{u4}{\subColor}{left}{0}%
        \myPath{u4}{u3}{\subColor}{left}{0}%
        \myPath{u1}{u2}{\subColor}{left}{0}%
        \myEdge{u1}{v}%
        \myEdge{u2}{v}%
        \myEdge{v1}{v}%
        \myEdge{v2}{v}%
        \myEdge{u1}{v1}%
        \myEdge{u2}{v2}%
        \myEdge{u4}{v}%
        \myPath{u1}{v2}{\subColor}{left}{75}%
        \myPath{v2}{u3}{\subColor}{left}{0}%
        \node[label=above:\textcolor{\subColor}{$Q_{13}$}] at (2.3,3.5) {};
        \node[whitenode,label=below:$v_4$] (v4) at (1.1,1.65) {};
        \myBentEdge{u4}{v4}{black}{left}{50}%
    }%
    \renewcommand{\CfN}{}%
    \renewcommand{\CfL}{}%
}%

\newcommand{\ObsFourPict}{
	\begin{figure}[ht]
	\centering
	\begin{subfigure}[t]{7cm}
		\centering%
		\ObsFourPictOne%
		\tkcf{}%
		\caption{The $K_{3,3}$-minor formed by $\{u_1,u_4,v\}$\\and $\{u_2,v_1,v_2=v_4\}$ in Observation~\ref{obs:red3}\\if $v_4'=v_1$}\label{fig:obsfourpictone}		
	\end{subfigure}
	\begin{subfigure}[t]{5cm}
		\centering%
		\ObsFourPictTwo%
		\tkcf{}%
		\caption{The planar embedding of the graph of Observation~\ref{obs:red3} if $v_4'=v$}\label{fig:obsfourpicttwo}
	\end{subfigure}
	\caption{The $K_{3,3}$-minor and the planar embedding of Observation~\ref{obs:red3}}\label{fig:red3}
\end{figure}
}

\newcommand{\ObsFivePict}{%
    \renewcommand\Mynodes{
        \node[whitenode,label=above:$v_1$] (v1) at (-0.5,3.5) {};
        \node[whitenode,label=below:$v_2$] (v2) at (0,1.5) {};
        \node[whitenode,label=left:$v$] (v) at (-1,2.5) {};
        \node[whitenode,label=above:$w_1$] (w1) at (-0.2,5.2) {};
    }%
    \renewcommand\Myspecialnodes{%
		\node[blacknode,label=left:$u_1$] (u1) at (-2,5) {};
        \node[blacknode,label=left:$u_2$] (u2) at (-2,0) {};
        \node[blacknode,label=right:$u_3$] (u3) at (2,3) {};
        \node[blacknode,label=right:$u_4$] (u4) at (4,0) {};
    }%
    \renewcommand\MyEdgesBefore{%
        \myPath{u2}{u4}{\subColor}{left}{0}%
        \myPath{u4}{u3}{\subColor}{left}{0}%
        \myPath{u1}{u2}{\subColor}{left}{0}%
        \myEdge{u1}{v}%
        \myEdge{u2}{v}%
        \myEdge{v1}{v}%
        \myEdge{v2}{v}%
        \myEdge{u1}{v1}%
        \myEdge{u2}{v2}%
        \myBentEdge{u4}{v}{black}{right}{20}%
        \myBentEdge{u4}{w1}{black}{right}{30}%
        \myBentEdge{u1}{w1}{\subColor}{left}{0}%
        \myPath{w1}{v2}{\subColor}{left}{50}%
        \myPath{v2}{u3}{\subColor}{right}{15}%
        
        \boxlines{red}{-2.8}{5.4}{-1.5}{4.6}
        \boxlines{red}{3.5}{0.4}{4.8}{-0.4}
        \boxlines{green}{-1.6}{2.8}{-0.7}{2.2}
        \boxlines{green}{-0.5}{5.8}{0.1}{4.9}
        \boxlines{red}{-0.3}{1.8}{0.3}{0.9}
        \boxlines{green}{-2.8}{0.4}{-1.5}{-0.4}
        
    }%
    \renewcommand{\CfN}{}%
    \renewcommand{\CfL}{}%
	\begin{figure}[ht]%
    	\ifthenelse{\isundefined{\NoFigures}}{%
        \centering%
        \tkcf{} %
        }{}%
        \caption{The $K_{3,3}$-minor formed by $\{u_1,u_4,v_2\}$ and $\{u_2,w_1,v\}$ in Observation~\ref{obs:red4}}
        \label{fig:ObsFivePict}%
    \end{figure}%
}%
\newcommand\CfPrintNom[1]{\textbf{\underline{Configuration #1{\CfN}}}\\\ } %
\newcommand\CfProperties{} %
\newcommand\Propcf[1]{Properties: #1{\CfProperties}} %
\newcommand\CfPict{} %
\newcommand\Pictcf[1]{#1{\CfPict}} %
\newcommand\CfText{} %
\newcommand\Textcf[1]{#1{\CfText}} %

\definecolor{myRed}{RGB}{228,26,28}
\definecolor{myRed2}{RGB}{227,26,28}
\definecolor{myBlue}{RGB}{55,126,184}
\definecolor{myBlue2}{RGB}{55,125,184}
\definecolor{myGreen}{RGB}{77,175,74}
\definecolor{myPurple}{RGB}{152,78,163}
\definecolor{myOrange}{RGB}{255,130,0}
\definecolor{myYellow}{RGB}{255,205,51}
\definecolor{myBrown}{RGB}{166,86,40}

\newcommand\subColor{myPurple}
\newcommand\subColorOne{myRed}
\newcommand\subColorTwo{myBlue}
\newcommand\PatternColorI{myGreen}
\newcommand\PatternColorII{myOrange}
\newcommand\PatternColorIII{myBrown}
\newcommand\PatternColorIV{myYellow}

\newcommand{\RuOne}{%
   \renewcommand\CfN{\nameRuOne}%
   \renewcommand\CfProperties{
		\begin{itemize}
			\item The graph has a strong $K_4$-subdivision $S$ rooted on $u_1,u_2,u_3,u_4$, with $2$ special vertices involved in a close problem: $u_1,u_2$ share a remaining neighbor $v$
			\item $v\notin S$
			\item $u_1, u_2$ each have another remaining neighbor $v_1,v_2$ respectively, and $v_1 \neq v_2$
			\item \textbf{Remark:} each of $u_3, u_4$ is either settled or causes a distant problem
		\end{itemize}
   }%
   \renewcommand\CfPict{
    
    \renewcommand\Myspecialnodes{
        \node[blacknode,label=below:$u_1$] (u1) at (-3,0) {};
        \node[blacknode,label=below:$u_2$] (u2) at (3,0) {};
        \node[blacknode,label=above left:$u_3$] (u3) at (0,3) {};
        \node[blacknode,label=above:$u_4$] (u4) at (0,6) {};
    }
    \renewcommand\Mynodes{
    	\node[whitenode,label=below:$v$] (v12) at (0,0.5) {};
    	\node[contactnode] (v4) at (-0.5,4) {};
    	\node[whitenode] (v4') at (-1.2,4.3) {};
    	\node[contactnode] (v3) at (1,4.2) {};
    	\node[whitenode] (v3') at (1.5,3.5) {};
    	\node[contactnode,label=above:$v_2$] (v2) at (0.5,1.5) {};
    	\node[contactnode,label=above:$v_1$] (v1) at (-0.5,1.5) {};
    }
    \renewcommand\MyEdgesBefore{
        \myPath{u1}{u2}{\subColor}{right}{20}
        \myPath{u1}{u3}{\subColor}{left}{20}
        \myPath{u2}{u3}{\subColor}{right}{20}
        \myPath{u1}{u4}{\subColor}{left}{20}
        \myPath{u2}{u4}{\subColor}{right}{20}
        \myPath{u3}{u4}{\subColor}{left}{0}
        \myEdge{u4}{v4}
        \myEdge{u4}{v4'}
        \myEdge{u3}{v3}
        \myEdge{u3}{v3'}
        \myEdge{u1}{v12}
        \myEdge{u2}{v12}
        \myEdge{u2}{v2}
        \myEdge{u1}{v1}

    }

    \renewcommand\MyEdgesAfter{
		\myPath{u1}{u2}{\PatternColorIII}{right}{20}
        \myPath{u1}{u3}{\subColor}{left}{20}
        \myPath{u2}{u3}{\subColor}{right}{20}
        \myPath{u1}{u4}{\subColor}{left}{20}
        \myPath{u2}{u4}{\subColor}{right}{20}
        \myPath{u3}{u4}{\subColor}{left}{0}

		\patternCN{u4}{v4}{v4'}{\PatternColorI}
		\patternCN{u3}{v3}{v3'}{\PatternColorII}
		\patternCDOnene{u1}{v1}{v12}{u2}{v2}{\PatternColorIII}
		
    }
    \renewcommand{\CfL}{fig:R2}
    \begin{figure}[h]%
        \ifthenelse{\isundefined{\NoFigures}}{%
        \centering
        \tkcfSansNomFigure{}%
        \flecheACentrer{2.8}%
        \tikzPartRule{\Myspecialnodes \Mynodes \MyEdgesAfter}%
        }{}%
        \caption{Reduction of configuration \CfN. $u_3,u_4$ may cause distant problems}%
    \end{figure}%
	}%
	\renewcommand\CfText{We consider a $2$-coloring of $S$ given by Claim~\ref{clm:inact} (p.~\pageref{clm:inact}) to inactivate the two potential distant problems on $u_3$ and $u_4$.
	If one of $v_1,v_2$ is not adjacent to $v$, then its associated special vertex forms a {\CV} pattern and is thus settled: a contradiction, as $u_1$ and $u_2$ are the ones causing a close problem. Hence $v_1,v_2$ are both adjacent to $v$. The vertices $u_1$ and $u_2$ form a {\CDOne} configuration, hence a {\CDA} or {\CDB} pattern. Note that $v_2$ cannot belong to the path $u_1\sim u_3$ and $v_1$ cannot belong to $u_2\sim u_3$, as this would form a {\redirXThree} configuration, forbidden by the redirection procedure.
	
	The patterns used are {\CDA} or {\CDB}$(u_1,u_2)$, {\CN}$(u_3)$, {\CN}$(u_4)$.
	}%
}

\newcommand{\RuTwo}{%
   \renewcommand\CfN{\nameRuTwo}%
   \renewcommand\CfProperties{
		\begin{itemize}
			\item The graph has a strong $K_4$-subdivision $S$ rooted on $u_1,u_2,u_3,u_4$, without distant problems and such that $2$ special vertices are involved in a close problem: $u_1,u_2$ share two remaining neighbors $v,v'$
			\item $v\notin S$ and $v'\in S$
			\item \textbf{Remark:} each of $u_3, u_4$ is either settled or causes a distant problem
		\end{itemize}
   }%
   \renewcommand\CfPict{
    
    \renewcommand\Myspecialnodes{
        \node[blacknode,label=below:$u_1$] (u1) at (-3,0) {};
        \node[blacknode,label=below:$u_2$] (u2) at (3,0) {};
        \node[blacknode,label=above left:$u_3$] (u3) at (0,3) {};
        \node[blacknode,label=above:$u_4$] (u4) at (0,6) {};
    }
    \renewcommand\Mynodes{
    	\node[whitenode,label=below right:$v$] (v) at (0,-0.5) {};
    	\node[whitenode,label=left:$v'$] (v') at (0,4.5) {};
    	\node[whitenode] (v4) at (-1.5,5.5) {};
    	\node[whitenode] (v4') at (-1.5,6.5) {};
    	\node[contactnode] (v3) at (0,1.2) {};
    	\node[whitenode] (v3') at (1.2,1.5) {};
    }
    \renewcommand\MyEdgesBefore{
		\myEdge{u1}{v}
		\myEdge{u2}{v}
        \myPath{u1}{u3}{\subColor}{left}{20}
        \myPath{u2}{u3}{\subColor}{right}{20}
        \myPath{u1}{u4}{\subColor}{left}{20}
        \myPath{u2}{u4}{\subColor}{right}{20}
        \myPath{u3}{v'}{\subColor}{left}{0}
        \myPath{v'}{u4}{\subColor}{left}{0}
        \myEdge{u4}{v4}
        \myEdge{u4}{v4'}
        \myEdge{u3}{v3}
        \myEdge{u3}{v3'}
		\myPath{u1}{u2}{\subColor}{left}{10}
    	\myBentEdge{u1}{v'}{black}{left}{15}
        \myBentEdge{u2}{v'}{black}{right}{15}

        \path ($(v3) + (-0.35,-0.15)$) edge [->, bend right = 40] ($0.5*(u1)+0.5*(u2)+0.12*(u3) + (-0.55,0.15)$);
    }

	\renewcommand\MyEdgesAfter{
        \myPath{u1}{u3}{\subColorOne}{left}{20}
        \myPath{u2}{u3}{\subColorTwo}{right}{20}
        \myPath{u1}{u4}{\subColorTwo}{left}{20}
        \myPath{u2}{u4}{\subColorTwo}{right}{20}
        \myPath{v'}{u4}{\subColorOne}{left}{0}
        \myPath{u3}{v'}{\subColorOne}{left}{0}
        
        \myBentEdge{u1}{v'}{\PatternColorII}{left}{15}
        \myBentEdge{u2}{v'}{\PatternColorII}{right}{15}
    	\myBentEdge{u1}{v}{\PatternColorII}{left}{0}
        \myBentEdge{u2}{v}{\PatternColorII}{left}{0}

		\myPath{u1}{u2}{\subColorOne}{left}{10}
        
		\patternCN{u4}{v4}{v4'}{\subColorOne}
		\patternCN{u3}{v3}{v3'}{\subColorTwo}
		
		\node at ($(u1)+(1.5,1)$) {\textcolor{\PatternColorII}{{\CTTNA}}};
		\path ($(v3) + (-0.35,-0.15)$) edge [->, bend right = 40] ($0.5*(u1)+0.5*(u2)+0.12*(u3) + (-0.55,0.15)$);
    }
    \renewcommand{\CfL}{fig:R3}
    \drawSimpleRule{fig:R3}{2.8}{Reduction of configuration \CfN}
	}%
	\renewcommand\CfText{
	By planarity and property A, there is at most one distant problem, caused by $u_3$ or $u_4$ on the path $u_1\sim u_2$. If it is the case, we assume w.l.o.g. that this is $u_3$.
	
	Let us color $S$ with this $2$-coloring:
	$\{red = (u_2\rightarrow u_1\rightarrow u_3\rightarrow u_4), blue = (u_1\rightarrow u_4\rightarrow u_2\rightarrow u_3)\}$.
	The colors ending on $u_1,u_2$ are different, so $u_1,u_2$ form a {\CTTNA} pattern that crosses the red path $u_3\sim u_4$. This is authorized by the definition of {\CTTNA}.
	The potential distant problem of $u_3$ is inactive in this coloring of $S$.
	
	The patterns used are {\CTTNA}$(u_1,u_2)$, {\CN}$(u_3)$ (or {\CVp}), {\CN}$(u_4)$.
	}%
}

\newcommand{\RuThree}{%
	\renewcommand\CfN{\nameRuThree}%
	\renewcommand\CfProperties{%
		\begin{itemize}
			\item The graph has a strong $K_4$-subdivision $S$
			\item $u_1, u_2$ each have two remaining neighbors $v_1,v_1'$ and $v_2,v_2'$ respectively
			\item $v_1,v_2$ belong to $u_3\sim u_4$; by convention $u_3\sim u_4 = (u_3,P_1,v_1,P_2,v_2,P_3,u_4)$, with $l(P_1),l(P_3)\geq 1$ and $l(P_2)\geq 0$ (so $v_1$ may equal $v_2$)
			\item If $v_1\neq v_2$, $v_1',v_2'$ are disjoint from $S$
			\item \textbf{Remark:} if $v_1=v_2$, then $v_1',v_2'$ may belong to $u_3\sim u_4$ and $v_1'$ may equal $v_2'$
			\item If $v_1'$ (resp. $v_2'$) does not belong to $S$, then it is adjacent to $v_1$ (resp. $v_2$)
			\item $u_3,u_4$ each have two remaining neighbors $v_3,v_3'$ and $v_4,v_4'$ respectively
			\item The path $u_1\sim u_2$ does not have length $1$
			\item If $u_1\sim u_2$ has length $2$, let $w$ be its middle vertex. Then 
			$w$ has at most $1$ neighbor among $u_3,u_4$, or at least one of $v_1',v_2'$ does not have a neighbor in $\{u_3,u_4\}$
		\end{itemize}
	}%
	\renewcommand\CfPict{%
		\renewcommand\Myspecialnodes{%
			\node[blacknode,label=below:$u_1$] (u1) at (-3,0) {};%
			\node[blacknode,label=below:$u_2$] (u2) at (3,0) {};%
			\node[blacknode,label=above right:$u_3$] (u3) at (0,1.5) {};%
			\node[blacknode,label=above:$u_4$] (u4) at (0,6) {};%
		}%
		\renewcommand\Mynodes{%
			\node[whitenode,label=right:$v_1$] (v1) at (0,3) {};%
			\node[whitenode,label=left:$v_1'$] (v1') at (-1,3.6) {};%
			\node[whitenode,label=left:$v_2$] (v2) at (0,4.2) {};%
			\node[whitenode,label=right:$v_2'$] (v2') at (0.8,4.8) {};%
			\node[whitenode] (v3) at (-0.8,0.4) {};%
			\node[whitenode] (v3') at (0,-0.0) {};%
			\node[whitenode] (v4) at (-1,-2.5) {};%
			\node[whitenode] (w1) at (-1.5,-0.4) {};%
			\node[whitenode] (w2) at (1.5,-0.4) {};%
			\node[whitenode] (v4') at (-0.1,-1.6) {};%
		}%
		\renewcommand\MyEdgesBefore{%
			\begin{scope}
				\clip (-3.2,-3) rectangle (5.2,6.5);
				
				\myBentEdge{u1}{w1}{\subColor}{left}{0}%
				\myBentEdge{u2}{w2}{\subColor}{left}{0}%
				\myPath{w1}{w2}{\subColor}{right}{15}%
				\myPath{u1}{u3}{\subColor}{left}{5}%
				\myPath{u3}{u2}{\subColor}{left}{5}%
				\myPath{u1}{u4}{\subColor}{left}{35}%
				\myPath{u2}{u4}{\subColor}{right}{35}%
				\myPath{u3}{v1}{\subColor}{left}{0}%
				\myPath{v1}{v2}{\subColor}{left}{0}%
				\myPath{v2}{u4}{\subColor}{left}{0}%
				\myBentEdge{u1}{v1}{black}{left}{15}%
				\myBentEdge{u1}{v1'}{black}{left}{15}%
				\myBentEdge{u2}{v2}{black}{right}{15}%
				\myBentEdge{u2}{v2'}{black}{right}{15}%
				\myEdge{v1}{v1'}
				\myEdge{v2}{v2'}
				\myEdge{u3}{v3}%
				\myEdge{u3}{v3'}%
				\myBentEdge{u4}{v4'}{black}{left}{0,looseness=2, out=90+10, in=90-10}%
				\myBentEdge{u4}{v4}{black}{left}{0,looseness=2.2, out=90+18, in=90-10}%
				\path ($(v3') + (-0.3,0.05)$) edge [->, bend right = 20] ($(w1) + (0.5,0.1)$);%
				\path ($(v4') + (0.3,0.15)$) edge [->, bend right = 40] ($(w2) + (-0.5,-0.5)$);%
				
			\end{scope}
		}%
		\renewcommand\MyEdgesAfter{
			\begin{scope}
				\clip (-3.2,-3) rectangle (5.2,6.5);
				
				\myBentEdge{u1}{v1}{\subColorOne}{left}{15}%
				\myPath{w1}{w2}{black}{right}{15}%
				\myPath{u1}{u3}{\subColorTwo}{left}{5}%
				\myPath{u3}{u2}{\subColorTwo}{left}{5}%
				\myPath{u1}{u4}{\subColorOne}{left}{35}%
				\myPath{u2}{u4}{\subColorOne}{right}{35}%
				\myPath{u3}{v1}{\subColorOne}{left}{0}%
				\myPath{v1}{v2}{black}{left}{0}%
				\myPath{v2}{u4}{\subColorTwo}{left}{0}%
				
				\myBentEdge{u2}{v2}{\subColorTwo}{right}{15}%
				
				\myBentEdge{u2}{v2'}{\PatternColorII}{right}{15}%
				\myBentEdge{u2}{w2}{\PatternColorII}{left}{0}%
				\myDEdge{w2}{v2'}
				\node at ($(u2)+(-0.8,-0.6)$) {\textcolor{\PatternColorII}{{\CV}}};%
				
				\myBentEdge{u1}{v1'}{\PatternColorI}{left}{15}%
				\myBentEdge{u1}{w1}{\PatternColorI}{left}{0}%
				\myDEdge{w1}{v1'}
				\node at ($(u1)+(0.8,-0.6)$) {\textcolor{\PatternColorI}{{\CV}}};%
				
				\myBentEdge{v1}{v1'}{black}{left}{0}%
				\myBentEdge{v2}{v2'}{black}{left}{0}%
				
				\patternCN{u3}{v3}{v3'}{\subColorOne}%
				\myBentEdge{u4}{v4'}{myBlue2}{left}{0,looseness=2, out=90+10, in=90-10}%
				\myBentEdge{u4}{v4}{myBlue2}{left}{0,looseness=2.2, out=90+18, in=90-10}%
				\myBentEdge{v4}{v4'}{myBlue2}{left}{0}%
				\node at ($(v4')+(0,-0.6)$) {\textcolor{\subColorTwo}{{\CN}}};%
				\path ($(v3') + (-0.3,0.05)$) edge [->, bend right = 20] ($(w1) + (0.5,0.1)$);%
				\path ($(v4') + (0.3,0.15)$) edge [->, bend right = 40] ($(w2) + (-0.5,-0.5)$);%
			\end{scope}
		}%
		\renewcommand{\CfL}{fig:R599}%
		\drawSimpleRule{fig:R599}{4}{\CfN\ when $v_1\neq v_2$, $v_1'\neq v_2'$}%
	}%
	\renewcommand\CfText{%
		We transform the $K_4$-subdivision $S$ into a $C_{4+}$-semi-subdivision $S'$ by removing the paths $u_1\sim u_2$ and $u_3\sim u_4$, and adding the paths $(u_1,v_1,P_1,u_3)$ and $(u_2,v_2,P_3,u_4)$ if $v_1\neq v_2$ or if $v_1=v_2$ and $v_1'\neq v_2'$. Otherwise, we have $v_1=v_2$ and $v_1'=v_2'$, and we assume $u_3\sim u_4 = (u_3,Q_1,v_1',Q_2,v_1,Q_3,u_4)$. In this case, we add the paths $(u_1,v_1',Q_1,u_3)$ and $(u_2,v_1,Q_3,u_4)$ to $S'$ instead.
		This semi-subdivision has two paths that intersect if $v_1=v_2$ and $v_1'\neq v_2'$ ($l(P_2) = 0$), but it is $2$-colorable with the coloring $\{red = (u_2\rightarrow u_4\rightarrow u_1\rightarrow v_1\rightarrow u_3), blue = (u_1\rightarrow u_3\rightarrow u_2\rightarrow v_2\rightarrow u_4)\}$.		
		
		Since the path $u_1\sim u_2$ of $S$ does not have length $1$, the special vertices $u_1,u_2$ form {\CV} patterns in $S'$.
		
		The special vertices $u_3,u_4$ may form {\CTTNA} patterns in $S'$ with each of $u_1,u_2$. These patterns may only touch each other on the path $u_1\sim u_2$ from $S$, but the last condition of the configuration ensures that this is not the case (this case is treated as configuration \ncf{\JFour}).
		
		Otherwise, the vertices $u_3,u_4$ form {\CN} patterns 
		if they do not have common remaining neighbors. These {\CN} patterns are compatible with the {\CV} patterns of $u_1,u_2$.
		
		\textbf{Remark:} If $v_1'\neq v_2'$ and $u_4$ has both as its remaining neighbors, then $u_4$ forms a {\CV} pattern which touches the {\CV} patterns of $u_1$ and $u_2$. The precise case where $u_4$ has $v_1,v_2$ as remaining neighbors and $u_1\sim u_2$ has length $1$ in $S$ is treated as configuration \ncf{\JThree}.
		
		The special vertices $u_3,u_4$ may also form a {\CTOne}, 
		{\CTTA} or {\CTTNA} pattern, since the colors ending on each vertex are different in any $2$-coloring of $S'$.
		
		The patterns used are {\CV}$(u_1)$, {\CV}$(u_2)$ and {\CTOne}, 
		{\CTTA} or {\CTTNA}$(u_3,u_4)$, or {\CN} for $u_3,u_4$.
	}%
}

\newcommand{\RuFour}{%
   \renewcommand\CfN{\nameRuFour}%
   \renewcommand\CfProperties{
		\begin{itemize}
			\item The graph has a strong $K_4$-subdivision $S$ rooted on $u_1,u_2,u_3,u_4$, with $3$ special vertices involved in a close problem: $u_1, u_2, u_3$.
			\item  The vertices $u_1,u_3$ share a remaining neighbor $v_{13}\notin S$
			\item $u_1,u_2,u_3$ share a remaining neighbor $v_{123}$ (different from $v_{13}$)
			\item $u_2$ has another remaining neighbor $v_2$ adjacent to $v_{123}$
			\item \textbf{Remark:} $u_4$ is either settled or causes a distant problem
		\end{itemize}
	}%
	\renewcommand\CfPict{
    
    \renewcommand\Myspecialnodes{
        \node[blacknode,label=below:$u_1$] (u1) at (-3,0) {};
        \node[blacknode,label=below:$u_2$] (u2) at (3,0) {};
        \node[blacknode,label=above left:$u_3$] (u3) at (0,3) {};
        \node[blacknode,label=above:$u_4$] (u4) at (0,6) {};
    }
    \renewcommand\Mynodes{
    	\node[whitenode,label=below:$v_{123}$] (v123) at (0,0.5) {};
    	\node[contactnode] (v4) at (-0.5,4) {};
    	\node[whitenode] (v4') at (-1.2,4.3) {};
    	\node[whitenode,label=left:$v_{13}$] (v13) at (-1,1) {};
    	\node[whitenode,label=left:$v_2$] (v2) at (1,-0.2) {};
    	\node[whitenode] (w1) at (-2,-1) {};
    	\node[whitenode] (w2) at (2,-1) {};
    }
    \renewcommand\MyEdgesBefore{
    	\myBentEdge{u1}{w1}{\subColor}{left}{0}
        \myPath{w1}{w2}{\subColor}{right}{10}
        \myBentEdge{u2}{w2}{\subColor}{left}{0}
        \myPath{u1}{u3}{\subColor}{left}{20}
        \myPath{u2}{u3}{\subColor}{right}{20}
        \myPath{u1}{u4}{\subColor}{left}{20}
        \myPath{u2}{u4}{\subColor}{right}{20}
        \myPath{u3}{u4}{\subColor}{left}{0}
        \myEdge{u4}{v4}
        \myEdge{u4}{v4'}
        \myBentEdge{u3}{v13}{black}{left}{10}
        \myBentEdge{u3}{v123}{black}{left}{10}
        \myBentEdge{u1}{v13}{black}{right}{10}
        \myBentEdge{u1}{v123}{black}{right}{10}
        \myEdge{u2}{v123}
        \myEdge{u2}{v2}
        \myEdge{v2}{v123}

    }
    \renewcommand\MyEdgesAfter{
        \myPath{w1}{w2}{black}{right}{10}
        \myPath{u1}{u3}{\subColor}{left}{20}
        \myPath{u2}{u3}{\subColor}{right}{20}
        \myPath{u1}{u4}{\subColor}{left}{20}
        \myPath{u2}{u4}{\subColor}{right}{20}
        \myPath{u3}{u4}{\subColor}{left}{0}
        \myBentEdge{u2}{v123}{\subColor}{left}{0}
        \myEdge{v2}{v123}

		\patternCN{u4}{v4}{v4'}{\PatternColorIV}
		
		\myBentEdge{u3}{v13}{\PatternColorIII}{left}{10}
        \myBentEdge{u3}{v123}{\PatternColorIII}{left}{10}
        \myDEdge{v123}{v13}
        \node at ($0.5*(v13) + 0.5*(v123) + (0.2,0.5)$)  {\textcolor{\PatternColorIII}{{\CV}}};
        
        \myBentEdge{u1}{w1}{\PatternColorI}{left}{0}
        \myBentEdge{u1}{v13}{\PatternColorI}{right}{10}
        \myDEdge{w1}{v13}
        \node at ($0.33*(u1) + 0.33*(v13) + 0.33*(w1) + (-0.2,-0.3)$)  {\textcolor{\PatternColorI}{{\CV}}};
        
        \myBentEdge{u1}{v123}{\subColor}{right}{10}
		
		\patternCN{u2}{v2}{w2}{\PatternColorII}	
		
    }
    \renewcommand{\CfL}{fig:R9}
    \drawSimpleRule{fig:R9}{2.8}{\CfN\ when $l(u_1\sim u_2) \geq 2$ in $S$}
	}%
	\renewcommand\CfText{
		By planarity and property A, $v_{13}$ and $v_{123}$ do not belong to $S$.
		The special vertices $u_1,u_3$ thus form a {\CDTwo} configuration, which is therefore a {\CTTNA} pattern: the vertices $v_{13}$ and $v_{123}$ are non-adjacent. However, this pattern is not compatible with a {\CN} pattern applied to $u_2$.
		
		The vertex $u_4$ may cause a distant problem or cause a {\CVp} pattern on the path $u_1\sim u_3$ or (w.l.o.g.) $u_2\sim u_3$. We replace the $K_4$-subdivision $S$ with another $K_4$-subdivision $S'$ by replacing the path $u_1\sim u_2$ with the path $(u_1,v_{123},u_2)$. If $u_4$ causes a distant problem on $u_1\sim u_3$ or $u_2\sim u_3$, we use Claim~\ref{clm:inact} (p.~\pageref{clm:inact}) and consider a $2$-coloring of $S'$ that inactivates it.
		
		By planarity and property A, $v_2$ does not belong to $S$.
		The vertices $u_1$ and $u_2$ form either a {\CV} and a {\CN} in $S'$, or a {\CU} pattern depending on the length of the original $u_1\sim u_2$ in $S$. 		
		
		The pattern used are {\CV}$(u_1)$, {\CN}$(u_2)$, or {\CU}$(u_1,u_2)$, {\CV}$(u_3)$, {\CN}$(u_4)$ (or {\CVp}).
	}%
}

\newcommand{\RuFive}{%
   \renewcommand\CfN{\nameRuFive}%
   \renewcommand\CfProperties{
		\begin{itemize}
			\item The graph has a strong $K_4$-subdivision $S$ rooted on $u_1,u_2,u_3,u_4$, with $3$ vertices involved in a close problem: $u_1, u_2, u_3$.
			\item $u_1,u_3$ share a remaining neighbor $v_{13}$ (which is not a neighbor of $u_2$)
			\item $u_1,u_2$ share a remaining neighbor $v_{12}$ (which is not a neighbor of $u_3$)
			\item $u_2, u_3$ each have another remaining neighbor $v_2,v_3$ respectively, and $v_2\neq v_3$
			\item $v_3,v_{12}$ are not adjacent
			\item None of $v_{13},v_{12},v_2,v_3$ belong to $S$
			\item The graph contains the edges $v_{13}v_3, v_{13}v_{12}, v_{12}v_2$
			\item \textbf{Remark:} $u_4$ is either settled or causes a distant problem
		\end{itemize}
   }%
   \renewcommand\CfPict{
   	\renewcommand\Myspecialnodes{
        \node[blacknode,label=below:$u_1$] (u1) at (-3,0) {};
        \node[blacknode,label=below:$u_2$] (u2) at (3,0) {};
        \node[blacknode,label=above left:$u_3$] (u3) at (0,3) {};
        \node[blacknode,label=above:$u_4$] (u4) at (0,6) {};
    }
    \renewcommand\Mynodes{
    	\node[whitenode,label=left:$v_{13}$] (v13) at (-1,1.4) {};
    	\node[whitenode,label=below:$v_{12}$] (v12) at (0,0.5) {};
    	\node[whitenode,label=right:$v_2$] (v2) at (0.8,1.2) {};
    	\node[whitenode,label=right:$v_3$] (v3) at (0.4,1.8) {};
    	\node[contactnode] (v4) at (-0.5,4) {};
    	\node[whitenode] (v4') at (-1.2,4.3) {}; 
    }
    \renewcommand\MyEdgesBefore{
        \myPath{u1}{u2}{\subColor}{right}{20}
        \myBentEdge{u1}{u3}{\subColor}{left}{20}
        \myPath{u2}{u3}{\subColor}{right}{20}
        \myPath{u1}{u4}{\subColor}{left}{20}
        \myPath{u2}{u4}{\subColor}{right}{20}
        \myPath{u3}{u4}{\subColor}{left}{0}
        \myEdge{u4}{v4}
        \myEdge{u4}{v4'}
        \myEdge{u1}{v13}
        \myEdge{u1}{v12}
        \myEdge{u2}{v2}
        \myEdge{u2}{v12}
        \myEdge{u3}{v13}
        \myEdge{u3}{v3}
        
        \myEdge{v12}{v13}
        \myEdge{v13}{v3}
        \myEdge{v12}{v2}
        
        \myDEdge{v12}{v3}
    }
    \renewcommand\MyEdgesAfter{
		\myPath{u1}{u2}{\subColor}{right}{20}
        \myBentEdge{u1}{u3}{\PatternColorI}{left}{20}
        \myPath{u2}{u3}{\subColor}{right}{20}
        \myPath{u1}{u4}{\subColor}{left}{20}
        \myPath{u2}{u4}{\subColor}{right}{20}
        \myPath{u3}{u4}{\subColor}{left}{0}

		\patternCN{u4}{v4}{v4'}{\PatternColorIII}
		
		\myBentEdge{u1}{v13}{\subColor}{left}{0}
		\myBentEdge{u3}{v13}{\subColor}{left}{0}

		\myBentEdge{u3}{v3}{\PatternColorI}{left}{0}
		\myBentEdge{u1}{v12}{\PatternColorI}{left}{0}
        
        \myEdge{v12}{v13}
        \myEdge{v13}{v3}

		\patternCN{u2}{v2}{v12}{\PatternColorII}
        
        \myDEdge{v12}{v3}
        
        \node at ($0.33*(u1) + 0.33*(v12) + 0.33*(v13)$)  {\textcolor{\PatternColorI}{{\CDB}}};
    }
    \renewcommand{\CfL}{fig:R14}
    \drawSimpleRule{fig:R14}{2.8}{Reduction of configuration \CfN. The special vertex $u_4$ may cause a distant problem on $u_1\sim u_2$ or $u_2\sim u_3$}
    }%
    \renewcommand\CfText{
		At least one edge among $v_2v_{13}, v_3v_{12}$ does not exist by planarity (otherwise $\{(u_1,v_2,v_3),$ $(u_2=u_3,v_{12},v_{13})\}$ form a $K_{3,3}$ minor by contracting the path $u_2\sim u_3$). Assume w.l.o.g. that the edge $v_3v_{12}$ is absent from the graph.
		
		The vertices $u_1,u_3$ form a {\CDOne} configuration, and since $v_3v_{12}$ does not exist, they form a {\CDB} pattern. By planarity and property A, $v_2$ does not belong to $S$. Hence, $u_2$ can be treated as a {\CN} pattern that is compatible with the {\CDB} pattern, since it touches only $v_{12}$. We use Claim~\ref{clm:inact} (p.~\pageref{clm:inact}) and consider a $2$-coloring of $G$ that inactivates the potential distant problem caused by $u_4$.
		
		The patterns used are {\CDB}$(u_1,u_3)$, {\CN}$(u_2)$, {\CN}$(u_4)$.
	}%
}

\newcommand{\RuSix}{%
   \renewcommand\CfN{\nameRuSix}%
   \renewcommand\CfProperties{
		\begin{itemize}
			\item The graph has a strong $K_4$-subdivision $S$ rooted on $u_1,u_2,u_3,u_4$, with all $4$ special vertices involved in close problems
			\item $u_1,u_2$ share a remaining neighbor $v_{12} \notin S$
			\item $u_3,u_4$ share a remaining neighbor $v_{34} \notin S$
			\item $u_1,u_2,u_3,u_4$ each have another remaining neighbor $v_1,v_2,v_3,v_4$ respectively
			\item $v_1,v_2$ are disjoint from $v_3,v_4$
			\item \textbf{Remark:} it may be that $v_1=v_2$ or $v_3=v_4$
			\item $v_1,v_2,v_3,v_4$ are disjoint from $S$ 
		\end{itemize}
	}%
	\renewcommand\CfPict{
    
    \renewcommand\Myspecialnodes{
        \node[blacknode,label=below:$u_1$] (u1) at (-3,0) {};
        \node[blacknode,label=below:$u_2$] (u2) at (3,0) {};
        \node[blacknode,label=above left:$u_3$] (u3) at (0,3) {};
        \node[blacknode,label=above:$u_4$] (u4) at (0,6) {};
    }
    \renewcommand\Mynodes{
    	\node[whitenode,label=below:$v$] (v12) at (0,0.5) {};
    	\node[contactnode] (v4) at (0.9,4.7) {};
    	\node[contactnode] (v3) at (0.9,3.5) {};
    	\node[whitenode] (v34) at (0.5,4.2) {};
    	\node[contactnode] (v2) at (0.5,1.5) {};
    	\node[contactnode] (v1) at (-0.5,1.5) {};
    }
    \renewcommand\MyEdgesBefore{
        \myPath{u1}{u2}{\subColor}{right}{20}
        \myPath{u1}{u3}{\subColor}{left}{20}
        \myPath{u2}{u3}{\subColor}{right}{20}
        \myPath{u1}{u4}{\subColor}{left}{20}
        \myPath{u2}{u4}{\subColor}{right}{20}
        \myPath{u3}{u4}{\subColor}{left}{0}
        \myEdge{u4}{v4}
        \myEdge{u3}{v3}
        \myEdge{u3}{v34}
        \myEdge{u1}{v12}
        \myEdge{u2}{v12}
        \myEdge{u2}{v2}
        \myEdge{u1}{v1}
        \myEdge{u4}{v34}

    }

    \renewcommand\MyEdgesAfter{
		\myPath{u1}{u2}{\PatternColorI}{right}{20}
        \myPath{u1}{u3}{\subColor}{left}{20}
        \myPath{u2}{u3}{\subColor}{right}{20}
        \myPath{u1}{u4}{\subColor}{left}{20}
        \myPath{u2}{u4}{\subColor}{right}{20}
        \myPath{u3}{u4}{\PatternColorII}{left}{0}
		
		\myBentEdge{u3}{v3}{\PatternColorII}{left}{0}
		\myBentEdge{u3}{v34}{\PatternColorII}{left}{0}
		\myBentEdge{v3}{v34}{\PatternColorII}{left}{0}
		\myBentEdge{u4}{v34}{\PatternColorII}{left}{0}
		\myBentEdge{u4}{v4}{\PatternColorII}{left}{0}
		\myBentEdge{v4}{v34}{\PatternColorII}{left}{0}
		\node at ($(v3)+(0,-0.4)$)  {\textcolor{\PatternColorII}{{\CDOne}}};
		
		\patternCDOnene{u1}{v1}{v12}{u2}{v2}{\PatternColorI}
		
		\path ($(v1) + (0.2,0.2)$) edge [->, bend left = 40] ($(v2) + (-0.2,0.2)$);
		\path ($(v3) + (0.2,0.2)$) edge [->, bend right = 40] ($(v4) + (0.2,-0.2)$);
    }
    \renewcommand{\CfL}{fig:R32}
    \drawSimpleRule{fig:R32}{2.8}{\CfN\ when $u_1,u_2$ and $u_3,u_4$ form {\CDOne} configurations}
	}%
	\renewcommand\CfText{
		This case is straightforward: each of $(u_1,u_2)$ and $(u_3,u_4)$ forms a {\CDOne} or {\CDTwo} 
		configuration, which can be a {\CDA}, {\CDB} or {\CTTNA} 
		pattern. By definition, these two patterns are disjoint.
	}%
}

\newcommand{\RuSeven}{%
   \renewcommand\CfN{\nameRuSeven}%
   \renewcommand\CfProperties{
		\begin{itemize}
			\item The graph has a strong $K_4$-subdivision $S$ rooted on $u_1,u_2,u_3,u_4$, with all $4$ special vertices involved in one close problem
			\item $u_1,u_2,u_3$ share a remaining neighbor $v_{123} \notin S$
			\item $u_1,u_2,u_4$ share a remaining neighbor $v_{124}\notin S$
			\item $u_3,u_4$ do not share a remaining neighbor
			\item $u_3,u_4$ each have another remaining neighbor $v_3,v_4$ respectively, adjacent to $v_{123},v_{124}$ respectively
		\end{itemize}
	}%
	\renewcommand\CfPict{
    
    \renewcommand\Myspecialnodes{
        \node[blacknode,label=below:$u_1$] (u1) at (-3,0) {};
        \node[blacknode,label=below:$u_2$] (u2) at (3,0) {};
        \node[blacknode,label=left:$u_3$] (u3) at (0,3) {};
        \node[blacknode,label=above:$u_4$] (u4) at (0,6) {};
    }
    \renewcommand\Mynodes{
    	\node[contactnode] (v4) at (0.5,4.5) {};
    	\node[contactnode] (v3) at (-0.8,1.2) {};
    	\node[whitenode,label=right:$w_2$] (w2) at (2.8,1.3) {};
    	\node[whitenode,label=below left:$v_{124}$] (v12) at (0,-0.8) {};
    	\node[contactnode,label=below:$v_{123}$] (v12') at (0,0.8) {};
    	\node[whitenode,label=right:$w_4$] (w4) at (1.4,5.2) {};
    }
    \renewcommand\MyEdgesBefore{
    	\begin{scope}
		\clip (-3.1,-1.4) rectangle (4.3,6.2);
		
        \myPath{u1}{u2}{\subColor}{left}{0}
        \myPath{u1}{u3}{\subColor}{left}{0}
        \myPath{u2}{u3}{\subColor}{right}{0}
        \myPath{u1}{u4}{\subColor}{left}{20}
        \myBentEdge{u2}{w2}{\subColor}{left}{0}
        \myPath{w2}{w4}{\subColor}{right}{10}
        \myBentEdge{u3}{v12'}{black}{left}{0}
        \myBentEdge{w4}{u4}{\subColor}{left}{0}
        \myPath{u3}{u4}{\subColor}{left}{0}
        
        \myEdge{u3}{v3}
        \myEdge{v3}{v12'}
        \myEdge{u4}{v4}
        
		\myBentEdge{u4}{v12}{black}{left}{0, out=90+10, in=90-30, looseness=2.2}
        
        \myBentEdge{u1}{v12}{black}{left}{0}
        \myBentEdge{u2}{v12}{black}{right}{0}
        \myBentEdge{u1}{v12'}{black}{left}{0}
        \myBentEdge{u2}{v12'}{black}{right}{0}
        
        \end{scope}
    }%
    \renewcommand\MyEdgesAfter{
    	\begin{scope}
		\clip (-3.1,-1.4) rectangle (4.3,6.2);
		
		\myDEdge{v12'}{v12}
		
        \myPath{u1}{u2}{\subColor}{left}{0}
        \myPath{u1}{u3}{\subColor}{left}{0}
        \myPath{u2}{u3}{\subColor}{right}{0}
        \myPath{u1}{u4}{\subColor}{left}{20}
        \myPath{w2}{w4}{black}{right}{10}
        \myPath{u3}{u4}{\subColor}{left}{0}
        
        \patternCN{u3}{v3}{v12'}{\PatternColorIII}
		\myBentEdge{u1}{v12'}{\PatternColorI}{left}{0}
		\myBentEdge{u1}{v12}{\PatternColorI}{left}{0}
		\node at ($(u1)+(1,-0.6)$)  {\textcolor{\PatternColorI}{{\CV}}};
		
        \patternCV{u2}{w2}{v12'}{\PatternColorII}
        \patternCN{u4}{w4}{v4}{\PatternColorIV}
        
		\myBentEdge{u4}{v12}{\subColor}{left}{0, out=90+10, in=90-30, looseness=2.2}
        
        \myBentEdge{u2}{v12}{\subColor}{right}{0}
		
		\end{scope}
    }
    \renewcommand{\CfL}{fig:R38}
    \drawSimpleRule{fig:R38}{4}{\CfN\ when $l(u_2\sim u_4) \geq 2$ in $S$}
	}%
	\renewcommand\CfText{
		We first claim that $v_4$ cannot be adjacent to $v_{123}$. If it is the case, then $\{(u_1,u_2,v_4),$ $(u_4,v_{123},v_{124})\}$ form a $K_{3,3}$-minor, a contradiction with the planarity of $G$.	
	
		We replace the $K_4$-subdivision $S$ with another $K_4$-subdivision $S'$ by removing the path $u_2\sim u_4$ and adding the path $(u_2,v_{124},u_4)$.
		
		The special vertex $u_1$ forms a {\CV} pattern and $u_3$ a {\CN} pattern disjoint from $S'$ by property A and planarity.
		$u_2$ forms a {\CV} and $u_4$ a {\CN} pattern, unless the length of $u_2\sim u_4$ in $S$ is $1$.
		In this case $u_2,u_4$ form a {\CU} pattern, since $v_4,v_{123}$ are non-adjacent.
		The {\CN} patterns are disjoint and may only touch {\CV} patterns, hence the mapping is compatible.
		
		The patterns used are {\CV}$(u_1)$, {\CN}$(u_3)$ and either {\CV}$(u_2)$ and {\CN}$(u_4)$ or {\CU}$(u_2,u_4)$. 
	}%
}

\newcommand{\RuEight}{%
   \renewcommand\CfN{\nameRuEight}%
   \renewcommand\CfProperties{
		\begin{itemize}
			\item The graph has a strong $C_{4+}^*$-subdivision $S$ rooted on $u_1,u_2,u_3,u_4$, such that $u_1,u_2$ are $1$-linked and $u_1,u_3$ are $2$-linked
			\item There is at most one distant problem, caused by $u_3$ on a $(u_2,u_4)$-path of $S$ if there is one
			\item $u_1,u_2$ share a remaining neighbor
			\item There is no remaining neighbor in common between one of $u_1,u_2$ and one of $u_3,u_4$
			\item \textbf{Remark:} if there is no distant problem, $u_3$ and $u_4$ may share a remaining neighbor
		\end{itemize}
	}%
\renewcommand\CfPict{
	
   	\renewcommand\Myspecialnodes{
        \node[blacknode,label=left:$u_1$] (u1) at (-3+0.5,3) {};
        \node[blacknode,label=left:$u_2$] (u2) at (-3+0.5,0) {};
        \node[blacknode,label=right:$u_3$] (u3) at (3-0.5,3) {};
        \node[blacknode,label=below right:$u_4$] (u4) at (3-0.5,0) {};
    }
    \renewcommand\Mynodes{
    	
    	\node[whitenode] (v1) at (-1.2,2) {};
    	\node[whitenode] (v1') at (-2,1.5) {};
    	\node[whitenode] (v2) at (-1.2,1) {};
    	\node[whitenode] (v4) at (3.7,0.6) {};
    	\node[whitenode] (v4') at (3,1.4) {};
    	\node[whitenode] (v3) at (0.8,1) {};
    	\node[whitenode] (v3') at (1.2,0) {};
    }
	\renewcommand\MyEdgesBefore{
        \myPath{u1}{u3}{\subColor}{left}{60}%
        \myPath{u1}{u3}{\subColor}{left}{0}%
        \myPath{u2}{v3'}{\subColor}{left}{0}%
        \myPath{v3'}{u4}{\subColor}{left}{0}%
        \myPath{u2}{u4}{\subColor}{right}{60}%
        \myPath{u1}{u2}{\subColor}{right}{10}%
        \myPath{u4}{u3}{\subColor}{left}{0}%
        
        \myEdge{u1}{v1}
        \myEdge{u1}{v1'}
        \myEdge{v1}{v1'}
        \myEdge{u2}{v2}
        \myEdge{u2}{v1'}
        \myEdge{v1'}{v2}
        \myEdge{u3}{v3}
        \myEdge{u3}{v3'}
        \myEdge{u4}{v4}
        \myEdge{u4}{v4'}
        \myEdge{v4}{v4'}
        \myEdge{v3}{v3'}
    }%
    \renewcommand\MyEdgesAfter{
        \myPath{u1}{u3}{\subColorTwo}{left}{60}%
        \myPath{u1}{u3}{\subColorOne}{left}{0}%
        \myPath{u2}{v3'}{\subColorOne}{left}{0}%
        \myPath{v3'}{u4}{\subColorOne}{left}{0}%
        \myPath{u2}{u4}{\subColorTwo}{right}{60}%
        \myPath{u1}{u2}{\subColorTwo}{right}{10}%
        \myPath{u4}{u3}{\subColorOne}{left}{0}%
       
		\myBentEdge{u1}{v1}{\PatternColorI}{left}{0}
		\myBentEdge{u1}{v1'}{\PatternColorI}{left}{0}
		\myBentEdge{v1}{v1'}{\PatternColorI}{left}{0}
		\myBentEdge{u2}{v2}{\PatternColorI}{left}{0}
		\myBentEdge{u2}{v1'}{\PatternColorI}{left}{0}
		\myBentEdge{v2}{v1'}{\PatternColorI}{left}{0}
		\node at ($(v1')+(0.2,0.6)$) {\textcolor{\PatternColorI}{{\CDOne}}};
		
		\patternCN{u3}{v3}{v3'}{\subColorTwo}
		\patternCN{u4}{v4}{v4'}{\subColorTwo}
    }%
    \renewcommand{\CfL}{fig:R44}
    \drawSimpleRule{fig:R44}{2.8}{\CfN\ when $u_3$ causes a distant problem and $u_1,u_2$ form a {\CDOne} configuration. Example of a $2$-coloring of $S$}
	}%
	\renewcommand\CfText{	
		If $u_3$ causes a distant problem, it is necessarily on a parallel $(u_2,u_4)$-path of $S$ by property ``$1$-linked'' of $S$ and property A, and then we may swap the colors of the two $(u_2,u_4)$-paths in a $2$-coloring of $S$ to inactivate this distant problem.
	
		By the last property of this configuration, $u_1,u_2$ form a {\CDOne} or {\CDTwo} pattern, and $u_3,u_4$ as well if $u_3$ does not cause a distant problem.
		
		The patterns used are thus {\CDA}, {\CDB} or {\CTTNA} for $(u_1,u_2)$ and maybe for $(u_3,u_4)$, or {\CN}$(u_3)$, {\CN}$(u_4)$.
	}%
}
\newcommand{\RuNine}{%
   \renewcommand\CfN{\nameRuNine}%
   \renewcommand\CfProperties{
		\begin{itemize}
			\item The graph has a strong $C_{4+}^*$-subdivision $S$ rooted on $u_1,u_2,u_3,u_4$, such that $u_1,u_2$ are $1$-linked and $u_1,u_3$ are $2$-linked
			\item $u_2,u_4$ share exactly remaining neighbor $v_{24}$, and it belongs to a parallel $(u_1,u_3)$-path $P_{13}$ of $S$
			\item The other remaining neighbors $v_2,v_4$ of $u_2,u_4$ respectively are both adjacent to $v_{24}$
			\item The remaining neighbors of $u_1,u_3$ are disjoint from $S$
			\item \textbf{Remark:} there is no distant problem
		\end{itemize}
	}%
	\renewcommand\CfPict{
	
   	\renewcommand\Myspecialnodes{
        \node[blacknode,label=left:$u_1$] (u1) at (-3+0.5,3) {};
        \node[blacknode,label=above left:$u_2$] (u2) at (-3+0.5,0) {};
        \node[blacknode,label=above right:$u_3$] (u3) at (3-0.5,3) {};
        \node[blacknode,label=above right:$u_4$] (u4) at (3-0.5,0) {};
    }
    \renewcommand\Mynodes{
    	\node[whitenode,label=above:$v_{24}$] (a2) at (0,3) {};

    	\node[whitenode] (v1) at (-3.0,1.6) {};
    	\node[whitenode] (v1') at (-3.8,2) {};
    	\node[whitenode] (v2) at (-1.6,2) {};
    	\node[whitenode] (v3) at (3.0,1.6) {};
    	\node[whitenode] (v3') at (3.8,2) {};
    	\node[whitenode] (v4) at (1.6,2) {};
    	
    	\node[whitenode] (w2) at ($(u2) + (0.5,-0.7)$) {};
    	\node[whitenode] (w4) at ($(u4) + (-0.5,-0.7)$) {};
    }
	\renewcommand\MyEdgesBefore{
		\myPath{u1}{u3}{\subColor}{left}{60}%
        \myPath{u1}{a2}{\subColor}{left}{0}%
		\myPath{u2}{u4}{\subColor}{left}{0}%
		\myBentEdge{u2}{w2}{\subColor}{right}{0}%
        \myPath{a2}{u3}{\subColor}{left}{0}%
        \myBentEdge{u4}{w4}{\subColor}{left}{0}%
		\myPath{w2}{w4}{\subColor}{right}{30}
        \myPath{u1}{u2}{\subColor}{right}{0}%
        \myPath{u3}{u4}{\subColor}{left}{0}%
		\myEdge{u1}{v1}
		\myEdge{u1}{v1'}
        \myEdge{u2}{v2}
		\myEdge{u2}{a2}
		\myEdge{u3}{v3}
		\myEdge{u3}{v3'}
        \myEdge{u4}{v4}
		\myEdge{u4}{a2}
        \myEdge{v4}{a2}
		\myEdge{v2}{a2}
    }
    \renewcommand\MyEdgesAfter{
		\myPath{u1}{u3}{\subColorOne}{left}{60}%
        \myPath{u1}{a2}{\subColorTwo}{left}{0}%
		\myPath{u2}{u4}{\subColorTwo}{left}{0}%
        \myPath{a2}{u3}{\subColorOne}{left}{0}%
		\myPath{w2}{w4}{black}{right}{30}
        \myPath{u1}{u2}{\subColorOne}{right}{0}%
        \myPath{u3}{u4}{\subColorTwo}{left}{0}%
		\patternCN{u1}{v1}{v1'}{\subColorTwo}
		\patternCN{u3}{v3}{v3'}{\subColorTwo}
		\myBentEdge{a2}{u2}{\subColorTwo}{left}{0}
		\myBentEdge{u4}{a2}{\subColorOne}{left}{0}
        \myEdge{v4}{a2}
		\myEdge{v2}{a2}
		
		\myBentEdge{u2}{w2}{\PatternColorII}{left}{0}
		\myBentEdge{u2}{v2}{\PatternColorII}{left}{0}
		\myDEdge{v2}{w2}
		\node at ($(u2)+(-0.2,-0.4)$)  {\textcolor{\PatternColorII}{{\CV}}};
		\myBentEdge{u4}{w4}{\PatternColorIV}{left}{0}
		\myBentEdge{u4}{v4}{\PatternColorIV}{left}{0}
		\myDEdge{v4}{w4}
		\node at ($(u4)+(0.2,-0.4)$)  {\textcolor{\PatternColorIV}{{\CV}}};
		
    }
    \renewcommand{\CfL}{fig:R53}
    \drawSimpleRule{fig:R53}{3}{Reduction of configuration \CfN}
	}%
	\renewcommand\CfText{
		Let us write $P_{13} = (u_1,Q_1,v_{24},Q_2,u_3)$.
		We transform the $C_{4+}$-subdivision $S$ into another $C_{4+}$-(semi-)subdivision $S'$ by removing the paths $P_{13}$ and any of the two $(u_2,u_4)$-paths of $S$, $P_{24}$, and adding $(u_1,Q_1,v_{24},u_2)$ and $(u_3,Q_2,v_{24},u_4)$. This semi-subdivision has a contact between the two new paths on $v_{24}$, so in a $2$-coloring of $S'$ we may swap the colors of the two $(u_1,u_2)$-paths to have different colors on the new paths.
		
		The special vertices $u_2,u_4$ form {\CV} patterns in $S'$ by planarity, 
		and since the remaining neighbors of $u_1,u_3$ are disjoint from $S$, none of $(u_1,u_2)$ or $(u_3,u_4)$ form {\CTTNA} patterns.
		
		The patterns used are {\CN}$(u_1)$, {\CV}$(u_2)$, {\CN}$(u_3)$, {\CV}$(u_4)$, or maybe
		{\CTTNA}$(u_1,u_2)$ or {\CTTNA}$(u_3,u_4)$.
	}%
}
\newcommand{\DOne}{%
   \renewcommand\CfN{\nameDOne}%
   \renewcommand\CfProperties{
		\begin{itemize}
			\item The graph has a strong $K_4$-subdivision $S$ rooted on $u_1,u_2,u_3,u_4$
			\item $u_1$ causes a distant problem on the path $u_2\sim u_4$: it has a remaining neighbor $v_1$ such that $u_2\sim u_4 = (P_2^1,v_1,P_4^1)$
			\item $u_2$ causes a distant problem on the path $u_3\sim u_4$: it has a remaining neighbor $v_2$ such that $u_3\sim u_4 = (P_3^2,v_2,P_4^2)$
			\item $u_3$ causes a distant problem on the path $u_1\sim u_4$: it has a remaining neighbor $v_3$ such that $u_1\sim u_4 = (P_1^3,v_3,P_4^3)$
		\end{itemize}
	}%
	\renewcommand\CfPict{
	
   	\renewcommand\Myspecialnodes{
        \node[blacknode,label=below:$u_1$] (u1) at (-3,0) {};
        \node[blacknode,label=right:$u_2$] (u2) at (3,0) {};
        \node[blacknode,label=below:$u_3$] (u3) at (0,3) {};
        \node[blacknode,label=below left:$u_4$] (u4) at (0,6) {};
    }
    \renewcommand\Mynodes{
    	\node[whitenode] (v1) at (0,-2) {};
    	\node[whitenode] (v1') at (-0.5,-2.5) {};
    	\node[whitenode] (v2) at (1.5,3) {};
    	\node[whitenode] (v2') at (2.3,3.3) {};
    	\node[whitenode] (v3) at (-1.8,3.8) {};
    	\node[whitenode] (v3') at (-1.2,4.6) {};
    	\node[whitenode] (v4) at (-0.5,7) {};
    	\node[whitenode] (v4') at (0.5,7) {};
    	
    }
    \renewcommand\MyEdgesBefore{
    	\begin{scope}
    	\clip (-3.3,-2.7) rectangle (4.3,7);
    	
         \myPath{u1}{u2}{\subColor}{right}{20}
         \myPath{u1}{u3}{\subColor}{left}{0}
         \myPath{u2}{u3}{\subColor}{left}{0}
         
         \myPath{u1}{v3}{\subColor}{left}{30}
         \myPath{v3}{u4}{\subColor}{left}{50}
         
         \myPath{v1}{u2}{\subColor}{right}{30}
		 \myPath{u4}{v1}{\subColor}{left}{0, out=90-10, in=90-30, looseness=1.8}
         
         \myPath{u3}{v2}{\subColor}{left}{0}
         \myPath{v2}{u4}{\subColor}{left}{0}
         
		 \myEdge{u1}{v1}
         \myEdge{u1}{v1'}
         \myEdge{u2}{v2}
         \myEdge{u2}{v2'}
         \myEdge{u3}{v3}
         \myEdge{u3}{v3'}
         \myEdge{u4}{v4}
         \myEdge{u4}{v4'}
         \myEdge{v1}{v1'}
         \myEdge{v2}{v2'}
         \myEdge{v3}{v3'}
       \end{scope}
    }%
    \renewcommand\MyEdgesAfter{
		\begin{scope}
    	\clip (-3.3,-2.7) rectangle (4.3,7);    
    
         \myPath{u1}{u2}{\subColor}{right}{20}
         \myPath{u1}{u3}{\subColor}{left}{0}
         \myPath{u2}{u3}{\subColor}{left}{0}
         
         \myPath{u1}{v3}{black}{left}{30}
         \myPath{v3}{u4}{\subColor}{left}{50}
         
         \myPath{v1}{u2}{black}{right}{30}
		 \myPath{u4}{v1}{\subColor}{left}{0, out=90-10, in=90-30, looseness=1.8}
         
         \myPath{u3}{v2}{black}{left}{0}
         \myPath{v2}{u4}{\subColor}{left}{0}
         
		 \myCEdge{u1}{v1}{\subColor}{$P$}
         \myEdge{u1}{v1'}
         \myCEdge{u2}{v2}{\subColor}{$P$}
         \myEdge{u2}{v2'}
         \myCEdge{u3}{v3}{\subColor}{$P$}
         \myEdge{u3}{v3'}      
         \myEdge{u4}{v4}
         \myEdge{u4}{v4'}
         \myEdge{v1}{v1'}
         \myEdge{v2}{v2'}
         \myEdge{v3}{v3'}
       	\end{scope}
    }%
    \renewcommand{\CfL}{\Ncf}%
    \ifthenelse{\equal{\isThesis}{true}}%
	{\drawSimpleRule{fig:D1}{2.8}{Semi-subdivision of \CfN}}%
	{\begin{figure}[ht]
        \ifthenelse{\isundefined{\NoFigures}}{%
        \centering
        \tkcfSansNomFigure{}%
        \flecheACentrer{2.8}%
        \tikzPartRule{\Myspecialnodes \Mynodes \MyEdgesAfter}%
        }{}%
        \caption{Semi-subdivision of \CfN}%
    \end{figure}%
	}%
	}%
    \renewcommand\CfText{
		We transform the $K_4$-subdivision $S$ into another $K_4$-subdivision $S'$, by removing the paths $u_1\sim u_4$, $u_2\sim u_4$, $u_3\sim u_4$, and adding the paths $(u_1,v_1,P_4^1)$, $(u_2,v_2,P_4^2)$, $(u_3,v_3,P_4^3)$.
		
		After the routing operation is applied, all unsettled 
		special vertices are turned into {\CV} patterns.
		
		\textbf{Remark:} $u_4$ could not form a {\CVp} pattern in $S$ by property A and planarity, hence it stays lone-settled in $S'$.
	}%
}%
\newcommand{\DTwo}{%
   \renewcommand\CfN{\nameDTwo}%
   \renewcommand\CfProperties{
		\begin{itemize}
			\item The graph has a strong $K_4$-subdivision $S$ rooted on $u_1,u_2,u_3,u_4$
			\item $u_2$ causes a distant problem on the path $u_1\sim u_3$: it has a remaining neighbor $v_2$ such that $u_1\sim u_3 = (P_1^2,v_2,P_3^2)$
			\item $u_3$ causes a distant problem on the path $u_2\sim u_4$: it has a remaining neighbor $v_3$ such that $u_2\sim u_4 = (P_2^3,v_3,P_4^3)$
			\item $u_4$ causes a distant problem on the path $u_1\sim u_2$: it has a remaining neighbor $v_4$ such that $u_1\sim u_2 = (P_1^4,v_4,P_2^4)$
		\end{itemize}
	}%
	\renewcommand\CfPict{
	
   	\renewcommand\Myspecialnodes{
        \node[blacknode,label=below:$u_1$] (u1) at (-3,0) {};
        \node[blacknode,label=right:$u_2$] (u2) at (3,0) {};
        \node[blacknode,label=left:$u_3$] (u3) at (0,3) {};
        \node[blacknode,label=above left:$u_4$] (u4) at (0,6) {};
    }
    \renewcommand\Mynodes{
    	\node[whitenode] (v1) at (-1.6,-0.2) {};
    	\node[whitenode] (v1') at (-2,1) {};
    	\node[whitenode] (v2) at (0,0.5) {};
    	\node[whitenode] (v2') at (0.5,1.2) {};
    	\node[whitenode] (v3) at (1.8,3.8) {};
    	\node[whitenode] (v3') at (1.4,4.2) {};
    	\node[whitenode] (v4) at (2.5,7) {};
    	\node[whitenode] (v4') at (2.2,7.5) {};
    	
    }
    \renewcommand\MyEdgesBefore{
		 \begin{scope}
    	 \clip (-3.1,-1.75) rectangle (5.8,7.5); 
    
    	 \myPath{u1}{u4}{\subColor}{left}{20}
    	 \myPath{u3}{u4}{\subColor}{left}{0}
         \myPath{u2}{u3}{\subColor}{right}{20}
    	 
		 \myPath{u1}{v4}{\subColor}{left}{0, out=-90, in=-90+10, looseness=2.2}
         \myPath{v4}{u2}{\subColor}{left}{30}
         
         \myPath{u1}{v2}{\subColor}{left}{10}
         \myPath{v2}{u3}{\subColor}{left}{10}
         
         \myPath{u2}{v3}{\subColor}{right}{30}
         \myPath{v3}{u4}{\subColor}{right}{40}
         
		 \myEdge{u1}{v1}
         \myEdge{u1}{v1'}
         \myEdge{u2}{v2}
         \myEdge{u2}{v2'}
         \myEdge{u3}{v3}
         \myEdge{u3}{v3'}      
         \myEdge{u4}{v4}
         \myEdge{u4}{v4'}
         \myEdge{v4}{v4'}
         \myEdge{v2}{v2'}
         \myEdge{v3}{v3'}
         
       \end{scope}
    }%
    \renewcommand\MyEdgesAfter{
		\begin{scope}
    	\clip (-3.1,-1.75) rectangle (5.8,7.5);     
    
         \myPath{u1}{u4}{\subColor}{left}{20}
    	 \myPath{u3}{u4}{\subColor}{left}{0}
         \myPath{u2}{u3}{\subColor}{right}{20}
    	 
		 \myPath{u1}{v4}{\subColor}{left}{0, out=-90, in=-90+10, looseness=2.2}
         \myPath{v4}{u2}{black}{left}{30}
         
         \myPath{u1}{v2}{\subColor}{left}{10}
         \myPath{v2}{u3}{black}{left}{10}
         
         \myPath{u2}{v3}{\subColor}{right}{30}
         \myPath{v3}{u4}{black}{right}{40}
         
		 \myEdge{u1}{v1}
         \myEdge{u1}{v1'}
         \myCEdge{u2}{v2}{\subColor}{}
         \myEdge{u2}{v2'}
         \myCEdge{u3}{v3}{\subColor}{}
         \myEdge{u3}{v3'}      
         \myCEdge{u4}{v4}{\subColor}{}
         \myEdge{u4}{v4'}
         \myEdge{v4}{v4'}
         \myEdge{v2}{v2'}
         \myEdge{v3}{v3'}
      \end{scope}
    }%
    \renewcommand{\CfL}{\Ncf}%
    \drawSimpleRule{fig:D2}{2.8}{Semi-subdivision of \CfN}%
    }%
    \renewcommand\CfText{
		We transform the $K_4$-subdivision $S$ into a $C_{4+}$-subdivision $S'$, by removing the paths $u_1\sim u_2$, $u_1\sim u_3$ and $u_2\sim u_4$, and adding the paths $(u_2,v_2,P_1^2,u_1)$, $(u_3,v_3,P_2^3,u_2)$, $(u_4,v_4,P_1^4,u_1)$.
		
		After the routing operation is applied, all unsettled 
		special vertices are turned into {\CV} patterns.
		
		\textbf{Remark:} $u_1$ could only form a {\CVp} pattern in $S$ on the path $u_3\sim u_4$ by property A and planarity, and this path was not modified in $S'$. Hence, $u_1$ remains lone-settled in $S'$.
	}%
}%
\newcommand{\DFour}{%
   \renewcommand\CfN{\nameDFour}%
   \renewcommand\CfProperties{
		\begin{itemize}
			\item The graph has a strong $K_4$-subdivision $S$ rooted on $u_1,u_2,u_3,u_4$
			\item $u_1$ causes a distant problem on the path $u_2\sim u_4$: it has a remaining neighbor $v_1$ such that $u_2\sim u_4 = (P_2^1,v_1,P_4^1)$
			\item $u_2$ causes a distant problem on the path $u_1\sim u_3$: it has a remaining neighbor $v_2$ such that $u_1\sim u_3 = (P_1^2,v_2,P_3^2)$
			\item $u_3$ causes a distant problem on the path $u_1\sim u_4$: it has a remaining neighbor $v_3$ such that $u_1\sim u_4 = (P_1^3,v_3,P_4^3)$
			\item $u_4$ causes a distant problem on the path $u_2\sim u_3$: it has a remaining neighbor $v_4$ such that $u_2\sim u_3 = (P_2^4,v_4,P_3^4)$
		\end{itemize}
	}%
	\renewcommand\CfPict{
	
   	\renewcommand\Myspecialnodes{
        \node[blacknode,label=below:$u_1$] (u1) at (-3,0) {};
        \node[blacknode,label=right:$u_2$] (u2) at (3,0) {};
        \node[blacknode,label=below right:$u_3$] (u3) at (0,3) {};
        \node[blacknode,label=below left:$u_4$] (u4) at (0,6) {};
    }
    \renewcommand\Mynodes{
    	\node[whitenode] (v1) at (0,-2) {};
    	\node[whitenode] (v1') at (-0.5,-2.5) {};
    	\node[whitenode] (v2) at (-0.2,1) {};
    	\node[whitenode] (v2') at (0.2,1.6) {};
    	\node[whitenode] (v3) at (-1.8,3.8) {};
    	\node[whitenode] (v3') at (-1.2,4.6) {};
    	\node[whitenode] (v4) at (1.5,3.6) {};
    	\node[whitenode] (v4') at (2,4) {};
    	
    }
    \renewcommand\MyEdgesBefore{
    	\begin{scope}
    	\clip (-3.3,-2.7) rectangle (4.3,6.2);
    	
         \myPath{u1}{u2}{\subColor}{right}{20}
         \myPath{u2}{v4}{\subColor}{left}{0}
         \myPath{v4}{u3}{\subColor}{left}{0}
         
         \myPath{u4}{u3}{\subColor}{left}{0}
         
         \myPath{u1}{v3}{\subColor}{left}{30}
         \myPath{v3}{u4}{\subColor}{left}{50}
         
         \myPath{v1}{u2}{\subColor}{right}{30}
         \myPath{u4}{v1}{\subColor}{left}{0, out=90-10, in=90-30, looseness=1.8}
         
         \myPath{u3}{v2}{\subColor}{right}{10}
         \myPath{v2}{u1}{\subColor}{right}{10}
         
		 \myEdge{u1}{v1}
         \myEdge{u1}{v1'}
         \myEdge{u2}{v2}
         \myEdge{u2}{v2'}
         \myEdge{u3}{v3}
         \myEdge{u3}{v3'}      
         \myEdge{u4}{v4}
         \myEdge{u4}{v4'}
         \myEdge{v1}{v1'}
         \myEdge{v2}{v2'}
         \myEdge{v3}{v3'}
         \myEdge{v4}{v4'}
       \end{scope}
    }%
    \renewcommand\MyEdgesAfter{
		\begin{scope}
    	\clip (-3.3,-2.7) rectangle (4.3,6.2);    
    
         \myPath{u1}{u2}{\subColor}{right}{20}
         \myPath{u2}{v4}{\subColor}{left}{0}
         \myPath{v4}{u3}{black}{left}{0}
         
         \myPath{u4}{u3}{\subColor}{left}{0}
         
         \myPath{u1}{v3}{\subColor}{left}{30}
         \myPath{v3}{u4}{black}{left}{50}
         
         \myPath{v1}{u2}{black}{right}{30}
         \myPath{u4}{v1}{\subColor}{left}{0, out=90-10, in=90-30, looseness=1.8}
         
         \myPath{u3}{v2}{\subColor}{right}{10}
         \myPath{v2}{u1}{black}{right}{10}
         
		 \myCEdge{u1}{v1}{\subColor}{}
         \myEdge{u1}{v1'}
         \myCEdge{u2}{v2}{\subColor}{}
         \myEdge{u2}{v2'}
         \myCEdge{u3}{v3}{\subColor}{}
         \myEdge{u3}{v3'}      
         \myCEdge{u4}{v4}{\subColor}{}
         \myEdge{u4}{v4'}
         \myEdge{v1}{v1'}
         \myEdge{v2}{v2'}
         \myEdge{v3}{v3'}
         \myEdge{v4}{v4'}
         
         \end{scope}
    }%
    \renewcommand{\CfL}{\Ncf}%
    \drawSimpleRule{fig:D5}{2.8}{Semi-subdivision of \CfN}%
    }%
    \renewcommand\CfText{
		We transform the $K_4$-subdivision $S$ into another $K_4$-subdivision $S'$ by keeping the paths $u_1\sim u_2$ and $u_3\sim u_4$ from $S$ and adding the paths $(u_1,v_1,P_4^1,u_4)$, $(u_2,v_2,P_3^2,u_3)$, $(u_3,v_3,P_1^3,u_1)$, $(u_4,v_4,P_2^4,u_2)$.
		
		By planarity and definition of distant problem, the routing operation does not need to be applied for all the 
		special vertices to be turned into {\CV} patterns. 
	}%
}%

\newcommand{\subdK}{%
	\renewcommand\CfN{$K_4$}%
	\renewcommand\CfPict{%
   	\renewcommand\Myspecialnodes{%
        \node[blacknode,label=below:{\large $u_1$}] (u1) at (-3,0) {};
        \node[blacknode,label=below:{\large $u_2$}] (u2) at (3,0) {};
        \node[blacknode,label=left:{\large $u_3$}] (u3) at (0,3) {};
        \node[blacknode,label=left:{\large $u_4$}] (u4) at (0,6) {};
    }%
    \renewcommand\Mynodes{%
    	\node[whitenode] (v1) at (-3.8,0.5) {};
    	\node[whitenode] (v1') at (-3.8,-0.5) {};
    	\node[whitenode] (v2) at (3.8,0.5) {};
    	\node[whitenode] (v2') at (3.8,-0.5) {};
    	\node[whitenode] (v3) at (-0.5,4) {};
    	\node[whitenode] (v3') at (0.5,4) {};
    	\node[whitenode] (v4) at (-0.5,7) {};
    	\node[whitenode] (v4') at (0.5,7) {};
    }%
    \renewcommand\MyEdgesBefore{%
         \myPath{u1}{u2}{\subColor}{right}{20}
         \myPath{u1}{u3}{\subColor}{left}{0}
         \myPath{u2}{u3}{\subColor}{left}{0}
         \myPath{u1}{u4}{\subColor}{left}{30}
         \myPath{u2}{u4}{\subColor}{right}{30}
         \myPath{u3}{u4}{\subColor}{left}{0}
		 \myEdge{u1}{v1}
         \myEdge{u1}{v1'}
         \myEdge{u2}{v2}
         \myEdge{u2}{v2'}
         \myEdge{u3}{v3}
         \myEdge{u3}{v3'}      
         \myEdge{u4}{v4}
         \myEdge{u4}{v4'}
    }%
    \renewcommand{\CfL}{\Ncf}%
    \centering%
    \tkcfSansNomFigure%
    }%
}

\newcommand{\subdCFP}{%
   \renewcommand\CfN{$C_{4+}$}%
	\renewcommand\CfPict{
	
   	\renewcommand\Myspecialnodes{
        \node[blacknode,label=above left:$u_1$] (u1) at (-3+0.5,3) {};
        \node[blacknode,label=below left:$u_2$] (u2) at (-3+0.5,0) {};
        \node[blacknode,label=above right:$u_3$] (u3) at (3-0.5,3) {};
        \node[blacknode,label=below right:$u_4$] (u4) at (3-0.5,0) {};
    }
    \renewcommand\Mynodes{    	
    	\node[whitenode] (v1) at (-3.3,1.5) {};
    	\node[whitenode] (v1') at (-3.8,2) {};
    	\node[whitenode] (v2) at (-2.5+0.8,1.5) {};
    	\node[whitenode] (v2') at (-2.5+1.3,1) {};
    	\node[whitenode] (v3) at (3.3,1.5) {};
    	\node[whitenode] (v3') at (3.8,2) {};
    	\node[whitenode] (v4) at (2.5-0.8,1.5) {};
    	\node[whitenode] (v4') at (2.5-1.3,1) {};
    }
	\renewcommand\MyEdgesBefore{

		\myPath{u1}{u3}{\subColor}{left}{50}%
        \myPath{u1}{u3}{\subColor}{left}{0}%
		\myPath{u2}{u4}{\subColor}{left}{0}%

		\myPath{u2}{u4}{\subColor}{right}{50}
        \myPath{u1}{u2}{\subColor}{right}{0}%
        \myPath{u3}{u4}{\subColor}{left}{0}%

		\myEdge{u1}{v1}
		\myEdge{u1}{v1'}
        \myEdge{u2}{v2}
		\myEdge{u2}{v2'}
		\myEdge{u3}{v3}
		\myEdge{u3}{v3'}
        \myEdge{u4}{v4}
		\myEdge{u4}{v4'}
    }
    \renewcommand{\CfL}{\Ncf}%
    \centering%
    \tkcfSansNomFigure%
	}%
}

\newcommand{\propA}{%
	\renewcommand\CfN{A-chord}%
	\renewcommand\Mynodes{ 
        \node[whitenode,label=left:$v_1$] (v1) at (0,0.5) {}; 
        \node[whitenode,label=left:$v_2$] (v2) at (0,2) {};
    }%
    \renewcommand\Myspecialnodes{%
        \node[blacknode,label=left:$u$] (u1) at (0,0-0.5) {}; 
        \node[blacknode,label=left:$u'$] (u2) at (0,3) {};
    }%
    \renewcommand\MyEdgesBefore{
		\myCEdge{u1}{v1}{sedgemixed}{$S$}
        \myBentEdge{u1}{v2}{black}{right}{60}
        \myPath{v2}{v1}{\subColor}{left}{0}
        \myPath{u2}{v2}{\subColor}{left}{0} 
		
		\myHalfEdge{u1}{-90+30}{sedgemixed}{$S$}%
        \myHalfEdge{u1}{-90-30}{sedgemixed}{$S$}%
    }%
    \renewcommand{\CfL}{\Ncf}%
    \centering%
    \tkcfSansNomFigure%
}%

\newcommand{\propB}{%
	\renewcommand\CfN{B-chord}%
	\renewcommand\Mynodes{ 
        \node[whitenode,label=left:$v_1$] (v1) at (-1,1-0.5) {}; 
        \node[whitenode,label=right:$v_2$] (v2) at (1,1-0.5) {};
    }%
    \renewcommand\Myspecialnodes{%
        \node[blacknode,label=left:$u$] (u1) at (0,0-0.5) {}; 
        \node[blacknode,label=left:$u'$] (u2) at (-2,2.5) {}; 
        \node[blacknode,label=right:$u''$] (u'') at (2,2.5) {};
    }%
    \renewcommand\MyEdgesBefore{
		\myEdge{u1}{v1}
        \myEdge{u1}{v2}
        \myEdge{v1}{v2}
        \myPath{u2}{v1}{\subColor}{left}{0}
        \myPath{v1}{v2}{\subColor}{left}{70}
		\myPath{v2}{u''}{\subColor}{left}{0}  
		
		\myHalfEdge{u1}{-90}{sedgemixed}{$S$}%
		\myHalfEdge{u1}{-90+40}{sedgemixed}{$S$}%
        \myHalfEdge{u1}{-90-40}{sedgemixed}{$S$}%
    }%
    \renewcommand{\CfL}{\Ncf}%
    \centering%
    \tkcfSansNomFigure%
}%

\newcommand{\segm}[3]{\myCEdge{$(-2.6+1.25*0.3*#1,0.5+1*0.3*#1)+#2*(-1,1.25)$}{$(-2.6+1.25*0.3*#1,0.5+1*0.3*#1)+#3*(1,-1.25)$}{lightgray}{}}
\newcommand{\NOneGr}{
	\renewcommand\CfN{$N_1$-graph}
	\renewcommand\CfPict{
	
   	\renewcommand\Myspecialnodes{
        \node[blacknode,label=above:{\LARGE $u_1$}] (u1) at (-2.5,5.5) {};
        \node[blacknode,label=above right:{\LARGE $u_2$}] (u2) at (3,6) {};
        \node[blacknode,label=above:{\LARGE $u_3$}] (u3) at (2.5,0.5) {};
        \node[blacknode,label=left:{\LARGE $u_4$}] (u4) at (-3,0) {};
    }
    \renewcommand\Mynodes{
    	\node[whitenode] (v1) at (-3.5,5.5) {};
    	\node[whitenode] (v1') at (-2,6.5) {};
    	\node[whitenode] (v1'') at (-1,5) {};
    	\node[whitenode] (v1''') at (-2.5,4) {};
    	\node[whitenode] (v3) at (2.2,-0.5) {};
    	\node[whitenode] (v3') at (3.5,0.5) {};
    	\node[whitenode] (v3'') at (2.5,2) {};
    	\node[whitenode] (v3''') at (1,0.8) {};
    }
    \renewcommand\MyEdgesBefore{
         \myPath{v1'}{u2}{lightgray}{left}{20}
         \myPath{u2}{v3'}{lightgray}{left}{15}
         \myPath{v3}{u4}{lightgray}{left}{20}
         \myPath{u4}{v1}{lightgray}{left}{15}
		 
		 \myPath{v1}{v1'}{lightgray}{left}{10}
		 \myBentEdge{v1}{v1'''}{lightgray}{right}{15}
		 \myBentEdge{v1}{v1'''}{lightgray}{left}{15}
		 \myBentEdge{v1''}{v1'''}{lightgray}{right}{15}
		 \myBentEdge{v1''}{v1'''}{lightgray}{left}{15}
		 \myBentEdge{v1'}{v1''}{lightgray}{right}{15}
		 \myBentEdge{v1'}{v1''}{lightgray}{left}{15}
		 
		 \myPath{v3}{v3'}{lightgray}{right}{10}
		 \myBentEdge{v3}{v3'''}{lightgray}{right}{15}
		 \myBentEdge{v3}{v3'''}{lightgray}{left}{15}
		 \myBentEdge{v3''}{v3'''}{lightgray}{right}{15}
		 \myBentEdge{v3''}{v3'''}{lightgray}{left}{15}
		 \myBentEdge{v3'}{v3''}{lightgray}{right}{15}
		 \myBentEdge{v3'}{v3''}{lightgray}{left}{15}
		 
		\segm{0}{0.4}{0.4}
		\segm{1}{1}{1}
		\segm{2}{1.5}{1.3}
		\segm{3}{1.9}{1.6}
		\segm{4}{2}{1.8}
		\segm{5}{1.4}{2}
		\segm{6}{1.3}{1.1}
		\segm{7}{1.3}{1}
		\segm{8}{1.4}{1}
		\segm{9}{1.5}{1}
		\segm{10}{2.3}{1}
		\segm{11}{2.2}{1.9}
		\segm{12}{2}{1.4}
		\segm{13}{1.7}{1}
		\segm{14}{1.3}{0.6}
		\segm{15}{0.8}{0}
		\segm{6}{2.7}{-1.85}
		\segm{7}{2.7}{-1.85}
		\segm{8}{2.6}{-1.8}
		\segm{7}{-2.2}{1.45}
		\segm{8}{-2.3}{1.5}
		\segm{9}{-2.3}{1.5}
		
		\node[blacknode,label=above:{\LARGE $u_1$}] at (-2.5,5.5) {};
        \node[blacknode,label=above right:{\LARGE $u_2$}] at (3,6) {};
        \node[blacknode,label=above:{\LARGE $u_3$}] at (2.5,0.5) {};
        \node[blacknode,label=left:{\LARGE $u_4$}] at (-3,0) {};
    }%
    \renewcommand{\CfL}{\Ncf}%
    \centering%
    \tkcfSansNomFigure%
    }%
}

\newcommand{\NTwoGr}{
	\renewcommand\CfN{$N_2$-graph}
	\renewcommand\CfPict{
	
   	\renewcommand\Myspecialnodes{
        \node[blacknode,label=left:{\LARGE $u_1$}] (u1) at (1+-3+0.5,3) {};
        \node[blacknode,label=left:{\LARGE $u_2$}] (u2) at (1+-3+0.5,0) {};
        \node[blacknode,label=right:{\LARGE $u_3$}] (u3) at (-1+3-0.5+6,3) {};
        \node[blacknode,label=right:{\LARGE $u_4$}] (u4) at (-1+3-0.5+6,0) {};
    }
    \renewcommand\Mynodes{
    	\node[whitenode,label=above right:{\Large $a_1\ \ \ \dots$}] (a1) at (0,4.5) {};
    	\node[whitenode,label=above right:{\Large $a_2\ \ \ \dots$}] (a2) at (0,3) {};
    	\node[whitenode,label=above right:{\Large $a_3\ \ \ \dots$}] (a3) at (0,0) {};
    	\node[whitenode,label=above right:{\Large $a_4\ \ \ \dots$}] (a4) at (0,-1.5) {};
    	
    	\node[whitenode,label=above:{\Large $t_1$}] (t1) at (0+3,4.5) {};
    	\node[whitenode,label=above:{\Large $t_2$}] (t2) at (0+3,3) {};
    	\node[whitenode,label=above:{\Large $t_3$}] (t3) at (0+3,0) {};
    	\node[whitenode,label=above:{\Large $t_4$}] (t4) at (0+3,-1.5) {};
    	
    	\node[whitenode,label=above left:{\Large $\dots\ \ \ \ b_1$}] (b1) at (0+6,4.5) {};
    	\node[whitenode,label=above left:{\Large $\dots\ \ \ \ b_2$}] (b2) at (0+6,3) {};
    	\node[whitenode,label=above left:{\Large $\dots\ \ \ \ b_3$}] (b3) at (0+6,0) {};
    	\node[whitenode,label=above left:{\Large $\dots\ \ \ \ b_4$}] (b4) at (0+6,-1.5) {};
    	
    }
    \renewcommand\MyEdgesBefore{
    	\begin{scope}
		\clip (-3.8,-2.1) rectangle (9.8,5.2);
		\myBentEdge{a1}{a4}{lightgray}{left}{0, out=-90-32, in=-90+32, looseness=2.4}
    	\myBentEdge{b1}{b4}{lightgray}{left}{0, out=90+32, in=90-32, looseness=2.4}
        \myBentEdge{b1}{b2}{lightgray}{left}{40}%
        \myBentEdge{a1}{a2}{lightgray}{right}{40}%
        \myBentEdge{b2}{b3}{lightgray}{left}{24}%
        \myBentEdge{a2}{a3}{lightgray}{right}{24}%
        \myBentEdge{b3}{b4}{lightgray}{left}{40}%
        \myBentEdge{a3}{a4}{lightgray}{right}{40}%
        
        \myPath{a1}{t1}{black}{left}{0}
        \myPath{a2}{t2}{black}{left}{0}
        \myPath{a3}{t3}{black}{left}{0}
        \myPath{a4}{t4}{black}{left}{0}
        \myPath{t1}{b1}{black}{left}{0}
        \myPath{t2}{b2}{black}{left}{0}
        \myPath{t3}{b3}{black}{left}{0}
        \myPath{t4}{b4}{black}{left}{0}
        
        \segm{-2}{0}{1.35}
        \segm{-1}{0.55}{1.6}
        \segm{0}{0.95}{1.7}
        \segm{1}{1.2}{1.7}
        \segm{2}{1.45}{1.3}
        \segm{3}{1.7}{1.1}
        \segm{4}{1.8}{0.6}
        \segm{5}{1.85}{0.2}
        \segm{6}{1.8}{-0.1}
        \segm{7}{1.7}{-0.3}
        \segm{8}{1.5}{-0.85}
        \segm{9}{1.1}{-1}
        \segm{12}{-4.6}{4.4}
        \segm{13}{-5}{4.2}
        \segm{14}{-5.2}{3.8}
        \segm{15}{-5.3}{3.5}
        \segm{16}{-5.3}{3.2}
        \segm{17}{-5.2}{2.8}
        \segm{18}{-5}{2.2}
        \segm{19}{-4.8}{2}
        \segm{20}{-4.6}{1.6}
        \segm{21}{-4.2}{1.7}
        \segm{22}{-3.8}{1.85}
        \segm{23}{-3.3}{2.1}
        
		\myPath{u1}{a1}{black}{left}{0}
        \myPath{u1}{a2}{black}{left}{0}
        \myPath{u1}{u2}{black}{left}{0}
        \myPath{u2}{a3}{black}{left}{0}
        \myPath{u2}{a4}{black}{left}{0}
        \myPath{u3}{b1}{black}{left}{0}
        \myPath{u3}{b2}{black}{left}{0}
        \myPath{u3}{u4}{black}{left}{0}
        \myPath{u4}{b3}{black}{left}{0}
        \myPath{u4}{b4}{black}{left}{0}        
        
        \node[blacknode,label=left:{\LARGE $u_1$}] at (1+-3+0.5,3) {};
        \node[blacknode,label=left:{\LARGE $u_2$}] at (1+-3+0.5,0) {};
        \node[blacknode,label=right:{\LARGE $u_3$}] at (-1+3-0.5+6,3) {};
        \node[blacknode,label=right:{\LARGE $u_4$}] at (-1+3-0.5+6,0) {};
		\end{scope}
    }%
    \renewcommand{\CfL}{\Ncf}%
    \centering%
    \tkcfSansNomFigure%
    }%
}

\newcommand{\DThree}{%
   \renewcommand\CfN{\nameDThree}%
   \renewcommand\CfProperties{%
		\begin{itemize}
			\item The graph has a strong $K_4$-subdivision $S$ rooted on $u_1,u_2,u_3,u_4$
			\item $u_1$ and $u_2$ cause distant problems on the path $u_3\sim u_4$: they have remaining neighbors $v_1,v_2$ respectively, such that $u_3\sim u_4 = (P_3^{12},v_1,P^{12},v_2,P_4^{12})$
			\item $u_3$ causes a distant problem on the path $u_1\sim u_2$: it has a remaining neighbor $v_3$ such that $u_1\sim u_2 = (P_1^3,v_3,P_2^3)$
			\item \textbf{Remark:} $u_4$ may cause a distant problem on $u_1\sim u_2$; it has a remaining neighbor $v_4$ that may belong to $P_1^3$ or $P_2^3$
		\end{itemize}
	}%
	\renewcommand\CfPict{%
   	\renewcommand\Myspecialnodes{%
        \node[blacknode,label=below:$u_1$] (u1) at (-3,0) {};%
        \node[blacknode,label=below:$u_2$] (u2) at (3,0) {};%
        \node[blacknode,label=above right:$u_3$] (u3) at (0,1.5) {};%
        \node[blacknode,label=above:$u_4$] (u4) at (0,6) {};%
    }%
    \renewcommand\Mynodes{%
    	\node[whitenode,label=right:$v_1$] (v1) at (0,3.6) {};%
    	\node[whitenode,label=left:$v_1'$] (v1') at (-1,3.6) {};%
    	\node[whitenode,label=left:$v_2$] (v2) at (0,4.8) {};%
    	\node[whitenode,label=right:$v_2'$] (v2') at (0.8,4.8) {};%
    	\node[whitenode] (v3) at (-0.8,0) {};%
    	\node[whitenode] (v3') at (0,-0.6) {};%
    	\node[whitenode] (v4) at (-1,-2.5) {};%
    	\node[whitenode] (w2) at (1.5,-0.4) {};%
    	\node[whitenode] (v4') at (-0.1,-1.6) {};%
    	\node[whitenode] (w3) at (0,2.8) {};%
    }%
    \renewcommand\MyEdgesBefore{%
    	\begin{scope}
		\clip (-3.2,-3) rectangle (5.2,6.5);
    	
    	\myBentEdge{u2}{w2}{\subColor}{left}{0}%
        \myPath{u1}{v3'}{\subColor}{right}{15}%
        \myPath{v3'}{w2}{\subColor}{right}{10}%
        \myPath{u1}{u3}{\subColor}{left}{5}%
        \myPath{u3}{u2}{\subColor}{right}{5}%
        \myPath{u1}{u4}{\subColor}{left}{35}%
        \myPath{u2}{u4}{\subColor}{right}{35}%
        \myBentEdge{u3}{w3}{\subColor}{left}{0}%
        \myPath{w3}{v1}{\subColor}{left}{0}%
        \myPath{v1}{v2}{\subColor}{left}{0}%
        \myPath{v2}{u4}{\subColor}{left}{0}%
		\myBentEdge{u1}{v1}{black}{left}{15}%
        \myBentEdge{u1}{v1'}{black}{left}{15}%
        \myBentEdge{u2}{v2}{black}{right}{15}%
        \myBentEdge{u2}{v2'}{black}{right}{15}%
        \myEdge{v1}{v1'}
        \myEdge{v2}{v2'}
        \myEdge{u3}{v3}%
        \myEdge{u3}{v3'}%
        \myEdge{v3}{v3'}
		\myBentEdge{u4}{v4'}{black}{left}{0,looseness=2, out=90+10, in=90-10}%
        \myBentEdge{u4}{v4}{black}{left}{0,looseness=2.2, out=90+18, in=90-10}%
       	\path ($(v4') + (-0.3,0.15)$) edge [->, bend left = 40] ($(-1.7,-0.4) + (0.5,-0.5)$);%
       	\path ($(v4') + (0.3,0.15)$) edge [->, bend right = 40] ($(w2) + (-0.5,-0.5)$);%
       	
       \end{scope}
    }%
    \renewcommand\MyEdgesAfter{
    	\begin{scope}
		\clip (-3.2,-3) rectangle (5.2,6.5);
		
		\myPath{u1}{v3'}{\subColorTwo}{right}{15}%
        \myPath{v3'}{w2}{black}{right}{10}%
        \myPath{u1}{u3}{\subColorOne}{left}{5}%
        \myPath{u3}{u2}{\subColorOne}{right}{5}%
        \myPath{u1}{u4}{\subColorTwo}{left}{35}%
        \myPath{u2}{u4}{\subColorOne}{right}{35}%
        \myPath{w3}{v1}{black}{left}{0}%
        \myPath{v1}{v2}{black}{left}{0}%
        \myPath{v2}{u4}{\subColorTwo}{left}{0}%
		
        \myBentEdge{u2}{v2}{\subColorTwo}{right}{15}%
        
        \myBentEdge{u2}{v2'}{\PatternColorII}{right}{15}%
        \myBentEdge{u2}{w2}{\PatternColorII}{left}{0}%
        \myDEdge{w2}{v2'}
        \node at ($(u2)+(-0.8,-0.6)$) {\textcolor{\PatternColorII}{{\CV}}};%
        
        \myBentEdge{u1}{v1}{myRed2}{left}{15}%
        \myBentEdge{u1}{v1'}{myRed2}{left}{15}%
		\myBentEdge{v1}{v1'}{myRed2}{left}{0}%
        \node at ($(v1')+(0,-0.4)$) {\textcolor{myRed2}{{\CN}}};%
        
        \myBentEdge{v2}{v2'}{black}{left}{0}%
        
		\myBentEdge{u3}{v3'}{\subColorTwo}{left}{0}%
		\myBentEdge{u3}{v3}{\PatternColorI}{left}{0}%
		\myEdge{v3}{v3'}
		\myBentEdge{u3}{w3}{\PatternColorI}{left}{0}%
		\myDEdge{v3}{w3}
		\node at ($(v3)+(-0.4,0)$) {\textcolor{\PatternColorI}{{\CV}}};%
		\myBentEdge{u4}{v4'}{myRed2}{left}{0,looseness=2, out=90+10, in=90-10}%
        \myBentEdge{u4}{v4}{myRed2}{left}{0,looseness=2.2, out=90+18, in=90-10}%
        \myBentEdge{v4}{v4'}{myRed2}{left}{0}%
        \node at ($(v4')+(0,-0.6)$) {\textcolor{myRed2}{{\CN}}};%
       	\path ($(v4') + (-0.3,0.15)$) edge [->, bend left = 40] ($(-1.7,-0.4) + (0.5,-0.5)$);%
       	\path ($(v4') + (0.3,0.15)$) edge [->, bend right = 40] ($(w2) + (-0.5,-0.5)$);%
        \end{scope}
    }%
    \renewcommand{\CfL}{fig:D43}%
    \drawSimpleRule{fig:D43}{4}{Reduction of configuration \CfN}%
	}%
	\renewcommand\CfText{%
		We transform the $K_4$-subdivision $S$ into a $C_{4+}$-subdivision $S'$ by removing the paths $u_1\sim u_2$ and $u_3\sim u_4$, and adding the paths $(u_1,P_1^3,v_3,u_3)$ and $(u_2,v_2,P_4^{12},u_4)$. We consider the following $2$-coloring of $S'$: $\{red = (u_1\rightarrow u_3\rightarrow u_2\rightarrow u_4), blue = (u_2 \rightarrow v_2\rightarrow u_4\rightarrow u_1\rightarrow v_3\rightarrow u_3)\}$. There is no need to apply the routing operation.
		
		The special vertex $u_3$ is turned into a {\CV} pattern. 
		The special vertices $u_1,u_4$ are treated as {\CN} patterns. 
		The {\CN} pattern of $u_4$ may cause a distant problem on the new path $(u_1,P_1^3,v_3,u_3)$, but this path is colored blue and $u_4$ uses the color red, hence the distant problem is inactive.
		
		The patterns used are {\CN}$(u_1)$, {\CV}$(u_2)$, {\CN}$(u_3)$, 
		{\CN}$(u_4)$.
	}%
}

\newcommand{\JOne}{%
   \renewcommand\CfN{\nameJOne}%
   \renewcommand\CfProperties{
		\begin{itemize}
			\item The graph has a strong $K_4$-subdivision $S$ rooted on $u_1,u_2,u_3,u_4$
			\item At most $2$ special vertices cause distant problems
			\item If there are two distant problems, they are not on the same path of $S$
			\item The special vertices that do not cause distant problems are settled
		\end{itemize}
   }%
	\renewcommand\CfPict{
    \renewcommand\Myspecialnodes{
        \node[blacknode,label=below:$u_1$] (u1) at (-3,0) {};
        \node[blacknode,label=below:$u_2$] (u2) at (3,0) {};
        \node[blacknode,label=left:$u_3$] (u3) at (0,3) {};
        \node[blacknode,label=above left:$u_4$] (u4) at (0,6) {};
    }
    \renewcommand\Mynodes{
    	\node[whitenode] (v1) at (1,2) {};
    	\node[whitenode] (v1') at (0.8,1) {};
    	\node[whitenode] (v2) at (3.8,0.5) {};
    	\node[whitenode] (v2') at (3.8,-0.5) {};
    	\node[whitenode] (v3) at (2.2,4) {};
    	\node[whitenode] (v3') at (1,4.2) {};
    	\node[whitenode] (v4) at (-0.5,7) {};
    	\node[whitenode] (v4') at (0.5,7) {};
    }
    \renewcommand\MyEdgesBefore{
         \myPath{u1}{u2}{\subColor}{right}{20}
         \myPath{u1}{u3}{\subColor}{left}{0}
         \myPath{u2}{v1}{\subColor}{left}{0}
         \myPath{u3}{v1}{\subColor}{left}{0}
         \myPath{u1}{u4}{\subColor}{left}{30}
         \myPath{u2}{v3}{\subColor}{right}{10}
         \myPath{v3}{u4}{\subColor}{right}{10}
         \myPath{u3}{u4}{\subColor}{left}{0}
		 \myEdge{u1}{v1}
         \myEdge{u1}{v1'}
         \myEdge{v1}{v1'}
         \myEdge{u2}{v2}
         \myEdge{u2}{v2'}
         \myEdge{u3}{v3}
         \myEdge{u3}{v3'}
         \myEdge{v3}{v3'}     
         \myEdge{u4}{v4}
         \myEdge{u4}{v4'}
    }%
    \renewcommand\MyEdgesAfter{
    	 \myPath{u1}{u2}{\subColorTwo}{right}{20}
         \myPath{u1}{u3}{\subColorOne}{left}{0}
         \myPath{u2}{v1}{\subColorTwo}{left}{0}
         \myPath{u3}{v1}{\subColorTwo}{left}{0}
         \myPath{u1}{u4}{\subColorTwo}{left}{30}
         \myPath{u2}{v3}{\subColorOne}{right}{10}
         \myPath{v3}{u4}{\subColorOne}{right}{10}
         \myPath{u3}{u4}{\subColorOne}{left}{0}
		 \patternCN{u1}{v1}{v1'}{\subColorOne}
		 \patternCN{u2}{v2}{v2'}{\subColorOne}
		 \patternCN{u3}{v3}{v3'}{\subColorTwo}
		 \patternCN{u4}{v4}{v4'}{\subColorTwo}
    }
    \renewcommand{\CfL}{fig:J1}
    \drawSimpleRule{fig:J1}{2.8}{\CfN\ in a case where $u_1$ and $u_3$ cause distant problems}
	}%
	\renewcommand\CfText{We consider a $2$-coloring of $S$ given by Claim~\ref{clm:inact} (p.~\pageref{clm:inact}) to inactivate the two potential distant problems.
	
	The patterns used are {\CN} for all special vertices.
	}%
}

\newcommand{\JTwo}{%
   \renewcommand\CfN{\nameJTwo}%
   \renewcommand\CfProperties{
		\begin{itemize}
			\item The graph has a strong $K_4$-subdivision $S$ rooted on $u_1,u_2,u_3,u_4$
			\item $u_1$ and $u_2$ cause distant problems on the path $u_3\sim u_4$: the path $u_3\sim u_4 = (u_3,P_1,v_1,P_2,v_2,P_3,u_4)$, with $v_1,v_2$ being remaining neighbors of $u_1,u_2$ respectively, and $l(P_1),l(P_2),l(P_3)\geq 1$; we denote $v_1',v_2'$ the other remaining neighbor of $u_1,u_2$ respectively
			\item $u_3, u_4$ are settled
			\item Either $l(u_1\sim u_2)\geq 2$ or neither $u_3$ nor $u_4$ has $v_1',v_2'$ as remaining neighbors
			\item If $u_1\sim u_2$ has length $2$ ($u_1,w,u_2$), then $w$ has at most $1$ neighbor among $u_3,u_4$, or at least one of $v_1',v_2'$ does not have a neighbor in $\{u_3,u_4\}$
		\end{itemize}
   }%
	\renewcommand\CfPict{
    \renewcommand\Myspecialnodes{
        \node[blacknode,label=below:$u_1$] (u1) at (-3,0) {};
        \node[blacknode,label=below:$u_2$] (u2) at (3,0) {};
        \node[blacknode,label=left:$u_3$] (u3) at (0,2) {};
        \node[blacknode,label=above left:$u_4$] (u4) at (0,6) {};
    }
    \renewcommand\Mynodes{
    	\node[whitenode] (v1) at (0,3) {};
    	\node[whitenode] (v1') at (-1,3.5) {};
    	\node[whitenode] (v2) at (0,4.5) {};
    	\node[whitenode] (v2') at (0.5,5) {};
    	\node[whitenode] (v3) at (-0.5,1) {};
    	\node[whitenode] (v3') at (0.5,1) {};
    	\node[whitenode] (v4) at (-0.5,7) {};
    	\node[whitenode] (v4') at (0.5,7) {};
    	
    	\node[whitenode] (w1) at (-1.5,-0.5) {};
    	\node[whitenode] (w2) at (1.5,-0.5) {};
    }
    \renewcommand\MyEdgesBefore{
    	 \myBentEdge{u1}{w1}{\subColor}{left}{0}
    	 \myBentEdge{u2}{w2}{\subColor}{left}{0}
         \myPath{w1}{w2}{\subColor}{right}{10}
         \myPath{u1}{u3}{\subColor}{left}{0}
         \myPath{u2}{u3}{\subColor}{left}{0}
         \myPath{u1}{u4}{\subColor}{left}{30}
         \myPath{u2}{u4}{\subColor}{right}{30}
         \myPath{u3}{v1}{\subColor}{left}{0}
         \myPath{v1}{v2}{\subColor}{left}{0}
         \myPath{v2}{u4}{\subColor}{left}{0}
		 \myBentEdge{u1}{v1}{black}{left}{15}
         \myBentEdge{u1}{v1'}{black}{left}{15}
         \myBentEdge{u2}{v2}{black}{right}{15}
         \myBentEdge{u2}{v2'}{black}{right}{15}
         \myEdge{v1}{v1'}
         \myEdge{v2}{v2'}
         \myEdge{u3}{v3}
         \myEdge{u3}{v3'}      
         \myEdge{u4}{v4}
         \myEdge{u4}{v4'}
    }
    \renewcommand\MyEdgesAfter{
    	 \myPath{w1}{w2}{black}{right}{10}
         \myPath{u1}{u3}{\subColor}{left}{0}
         \myPath{u2}{u3}{\subColor}{left}{0}
         \myPath{u1}{u4}{\subColor}{left}{30}
         \myPath{u2}{u4}{\subColor}{right}{30}
         \myPath{u3}{v1}{\subColor}{left}{0}
         \myPath{v1}{v2}{black}{left}{0}
         \myPath{v2}{u4}{\subColor}{left}{0}
		 \myBentEdge{u1}{v1}{\subColor}{left}{15}
         \myBentEdge{u2}{v2}{\subColor}{right}{15}
         \myEdge{v1}{v1'}
         \myEdge{v2}{v2'}
         \myEdge{u3}{v3}
         \myEdge{u3}{v3'}      
         \myEdge{u4}{v4}
         \myEdge{u4}{v4'}
         
         \myBentEdge{u1}{v1'}{\PatternColorI}{left}{15}
         \myBentEdge{u1}{w1}{\PatternColorI}{left}{0}
         \myDEdge{v1'}{w1}
         \node at (-1.8,0.2) {\textcolor{\PatternColorI}{{\CV}}};
         
         \myBentEdge{u2}{v2'}{\PatternColorII}{right}{15}
         \myBentEdge{u2}{w2}{\PatternColorII}{left}{0}
         \myDEdge{v2'}{w2}
         \node at (1.8,0.2) {\textcolor{\PatternColorII}{{\CV}}};
    }
    \renewcommand{\CfL}{fig:J2}
    \drawSimpleRule{fig:J2}{2.8}{\CfN\ in a case where the length of $u_1\sim u_2$ is at least $2$}
	}%
	\renewcommand\CfText{We transform the $K_4$-subdivision $S$ into a $C_{4+}$-subdivision $S'$, by removing the paths $u_1\sim u_2$ and $u_3\sim u_4$ from $S$, and adding the paths $(u_1,v_1,P_1,u_3)$ and $(u_2,v_2,P_3,u_4)$.
	
	The special vertices $u_1,u_2$ are thus turned into {\CV} patterns, unless the path $u_1\sim u_2$ from $S$ has length $1$, in which case $u_1,u_2$ form a {\CU} pattern. By the fourth condition of this configuration, neither $u_3$ nor $u_4$ has $v_1',v_2'$ as remaining neighbors, thus they remain settled (the case where one has $v_1',v_2'$ as remaining neighbors is treated as \ncf{\JThree}). 
	
	Instead of {\CV} patterns, $u_1$ or $u_2$ may form {\CTTNA} patterns with $u_3$ or $u_4$. If there are two such patterns, for instance {\CTTNA}$(u_1,u_3)$ and {\CTTNA}$(u_2,u_4)$, they may only intersect if they have a common vertex in the path $u_1\sim u_2$ of $S$. By the last condition of the configuration, this is not the case (this case is treated as \ncf{\JFour}). 
	
	The patterns used are {\CV}$(u_1)$, {\CV}$(u_2)$, or {\CTTNA}$(u_i,u_j)$ for $i\in \{1,2\}$ and $j\in \{3,4\}$, or {\CU}$(u_1,u_2)$, {\CN}$(u_3)$, {\CN}$(u_4)$.
	}%
}
\newcommand{\JThree}{%
   \renewcommand\CfN{\nameJThree}%
   \renewcommand\CfProperties{
		\begin{itemize}
			\item The graph has a strong $K_4$-subdivision $S$ rooted on $u_1,u_2,u_3,u_4$
			\item $u_1$ and $u_2$ cause distant problems on the path $u_3\sim u_4$: the path $u_3\sim u_4 = (u_3,P_1,v_1,P_2,v_2,P_3,u_4)$, with $v_1,v_2$ being remaining neighbors of $u_1,u_2$ respectively, and $l(P_1),l(P_2),l(P_3)\geq 1$; we denote $v_1',v_2'$ the other remaining neighbor of $u_1,u_2$ respectively
			\item $l(u_1\sim u_2) = 1$
			\item $u_4$ has $v_1',v_2'$ as remaining neighbors
			\item $u_3$ is settled
		\end{itemize}
   }%
	\renewcommand\CfPict{
    \renewcommand\Myspecialnodes{
        \node[blacknode,label=below:$u_1$] (u1) at (-3,0) {};
        \node[blacknode,label=below:$u_2$] (u2) at (3,0) {};
        \node[blacknode,label=left:$u_3$] (u3) at (0,2) {};
        \node[blacknode,label=above:$u_4$] (u4) at (0,6) {};
    }%
    \renewcommand\Mynodes{
    	\node[whitenode,label=right:$v_1$] (v1) at (0,3) {};
    	\node[whitenode,label=left:$v_1'$] (v1') at (-0.8,4) {};
    	\node[whitenode,label=left:$v_2$] (v2) at (0,3.8) {};
    	\node[whitenode,label=right:$v_2'$] (v2') at (1.2,4) {};
    	\node[whitenode] (w4) at (0,4.6) {};
    	\node[whitenode] (v3) at (0,0.2) {};
    	\node[whitenode] (v3') at (0.8,0.6) {};
    }%
    \renewcommand\MyEdgesBefore{
    	\myBentEdge{u1}{u2}{\subColor}{right}{20}
        \myPath{u1}{u3}{\subColor}{left}{0}
        \myPath{u2}{u3}{\subColor}{left}{0}
        \myPath{u1}{u4}{\subColor}{left}{30}
        \myPath{u2}{u4}{\subColor}{right}{30}
        \myPath{u3}{v1}{\subColor}{left}{0}
        \myPath{v1}{v2}{\subColor}{left}{0}
        \myPath{v2}{w4}{\subColor}{left}{0}
        \myBentEdge{w4}{u4}{\subColor}{left}{0}
		\myBentEdge{u1}{v1}{black}{left}{15}
        \myBentEdge{u1}{v1'}{black}{left}{15}
        \myBentEdge{u2}{v2}{black}{right}{15}
        \myBentEdge{u2}{v2'}{black}{right}{15}
        \myEdge{v1}{v1'}
        \myEdge{v2}{v2'}
        \myEdge{u3}{v3}
        \myEdge{u3}{v3'}

		\myEdge{u4}{v2'}
		\myEdge{u4}{v1'}
    }%
    \renewcommand\MyEdgesAfter{
        \myPath{u1}{u3}{\subColor}{left}{0}
        \myPath{u2}{u3}{\subColor}{left}{0}
        \myPath{u1}{u4}{\subColor}{left}{30}
        \myPath{u2}{u4}{\subColor}{right}{30}
        \myPath{u3}{v1}{\subColor}{left}{0}
        \myPath{v1}{v2}{black}{left}{0}
        \myPath{v2}{w4}{black}{left}{0}
		\myBentEdge{u1}{v1}{\subColor}{left}{15}
        \myBentEdge{u2}{v2}{\subColor}{right}{15}
        \myEdge{v1}{v1'}
        \myBentEdge{v2}{v2'}{\subColor}{left}{0}
        \myBentEdge{u4}{v2'}{\subColor}{left}{0}
        
        \patternCN{u3}{v3}{v3'}{\PatternColorIII}
        
        \myBentEdge{u1}{u2}{\PatternColorI}{right}{20}
        \myBentEdge{u1}{v1'}{\PatternColorI}{left}{15}
        \myBentEdge{u2}{v2'}{\PatternColorI}{right}{15}
        \node at ($(u1)+(1.5,0)$)  {\textcolor{\PatternColorI}{{\CU}}};
		\patternCN{u4}{w4}{v1'}{\PatternColorII}
    }%
    \renewcommand{\CfL}{fig:J3}
    \drawSimpleRule{fig:J3}{2.8}{Reduction of configuration~\CfN}
	}%
	\renewcommand\CfText{\textbf{Remark:} $u_4$ is initially settled, but the rule that follows changes its {\CV} nature into a {\CN} one.
	
	We transform the $K_4$-subdivision $S$ into a $C_{4+}$-subdivision $S'$, by removing the paths $u_1\sim u_2$ and $u_3\sim u_4$ from $S$, and adding the paths $(u_1,v_1,P_1,u_3)$ and $(u_2,v_2,v_2',u_4)$.
	
	The special vertices $u_1,u_2$ are thus turned into a {\CU} pattern, while $u_4$ is turned into a {\CN}. If $v_2$ is adjacent to $u_4$, it becomes one of its remaining neighbors in $S'$, and in this case $u_4$ causes a distant problem. We inactivate this problem by maybe swapping the colors of the paths $u_2\sim u_4$ and $(u_2,v_2,v_2',u_4)$ in a $2$-coloring of $S'$.
	
	The patterns used are {\CU}$(u_1,u_2)$, {\CN}$(u_3)$, {\CN}$(u_4)$.
	}%
}
\newcommand{\JFour}{%
   \renewcommand\CfN{\nameJFour}%
   \renewcommand\CfProperties{%
		\begin{itemize}
			\item The graph has a strong $K_4$-subdivision $S$ rooted on $u_1,u_2,u_3,u_4$
			\item $u_1$ has two adjacent remaining neighbors $v_1,v_1'$, with $v_1\in u_3\sim u_4$
			\item $u_2$ has two adjacent remaining neighbors $v_2,v_2'$ with $v_2\in u_3\sim u_4$
			\item $v_1,v_2$ may be equal, or come in any order on $u_3\sim u_4$
			\item $u_1\sim u_2$ has length $2$: call the third vertex $w$
			\item $u_3$ has $v_2',w$ as remaining neighbors
			\item $u_4$ has $v_1',w$ as remaining neighbors
		\end{itemize}
	}%
	\renewcommand\CfPict{
	
   	\renewcommand\Myspecialnodes{
        \node[blacknode,label=below:$u_1$] (u1) at (-3,0) {};
        \node[blacknode,label=below:$u_2$] (u2) at (3,0) {};
        \node[blacknode,label=left:$u_3$] (u3) at (0,2) {};
        \node[blacknode,label=above:$u_4$] (u4) at (0,6) {};
    }
    \renewcommand\Mynodes{
    	\node[whitenode,label=left:$v_1$] (v1) at (0,3.8) {};
    	\node[whitenode,label=left:$v_1'$] (v1') at (-0.8,4.6) {};
    	\node[whitenode,label=left:$v_2$] (v2) at (0,4.5) {};
    	\node[whitenode,label=right:$v_2'$] (v2') at (0.5,3) {};
    	\node[whitenode,label=left:$w_3$] (w3) at (0,3) {};
    	\node[whitenode,label=below:$w$] (w) at (0,-0.5) {};
    }
    \renewcommand\MyEdgesBefore{
    	\begin{scope}
    	\clip (-3.1,-1.4) rectangle (4.3,6.2);
    	
    	\myBentEdge{u2}{w}{\subColor}{left}{0}
    	\myBentEdge{w}{u1}{\subColor}{left}{0}
        \myPath{u1}{u3}{\subColor}{left}{0}
        \myPath{u2}{u3}{\subColor}{left}{0}
        \myPath{u1}{u4}{\subColor}{left}{30}
        \myPath{u2}{u4}{\subColor}{right}{30}
        \myBentEdge{u3}{w3}{\subColor}{left}{0}
        \myPath{w3}{v1}{\subColor}{left}{0}
        \myPath{v1}{v2}{\subColor}{left}{0}
        \myPath{v2}{u4}{\subColor}{left}{0}
		\myBentEdge{u1}{v1}{black}{left}{15}
        \myBentEdge{u1}{v1'}{black}{left}{15}
        \myBentEdge{u2}{v2}{black}{right}{25}
        \myBentEdge{u2}{v2'}{black}{right}{15}
        \myEdge{v1}{v1'}
        \myEdge{v2}{v2'}
        \myEdge{u3}{w}
        
		\myBentEdge{u4}{w}{black}{left}{0, out=90-5, in=90-25, looseness=2.2}

		\myEdge{u3}{v2'}
		\myEdge{u4}{v1'}
        
        \end{scope}
    }%
    \renewcommand\MyEdgesAfter{
    	\begin{scope}
    	\clip (-3.1,-1.4) rectangle (4.3,6.2);
    	
        \myBentEdge{u2}{w}{\subColorOne}{left}{0}
    	\myBentEdge{w}{u1}{\subColorOne}{left}{0}
        \myPath{u1}{u3}{\subColorOne}{left}{0}
        \myPath{u2}{u3}{\subColorTwo}{left}{0}
        \myPath{u1}{u4}{\subColorTwo}{left}{30}
        \myPath{u2}{u4}{\subColorTwo}{right}{30}
        \myBentEdge{u3}{w3}{\PatternColorIII}{left}{0}
        \myPath{w3}{v1}{black}{left}{0}
        \myPath{v1}{v2}{black}{left}{0}
        \myPath{v2}{u4}{\subColorOne}{left}{0}
		\draw[\subColorTwo, thick, segment aspect=0, bend left = 15] (u1) to (v1);
		\draw[\subColorTwo, thick, segment aspect=0, bend left = 15] (u1) to (v1');
        \myBentEdge{u2}{v2}{\PatternColorII}{right}{25}
        \myBentEdge{u2}{v2'}{\PatternColorII}{right}{15}
        \myCEdge{v1}{v1'}{\subColorTwo}{}
        \myBentEdge{v2}{v2'}{\subColorOne}{left}{0}
		\myCEdge{u3}{v2'}{\subColorOne}{}
		\node[nodelabel] at ($(v2')+(0,-0.6)$)  {\textcolor{\subColorOne}{$P_1$}};
        \myBentEdge{u4}{v1'}{\PatternColorI}{left}{0}
        \myBentEdge{u3}{w}{\PatternColorIII}{left}{0}
        
		\myBentEdge{u4}{w}{\PatternColorI}{left}{0, out=90-5, in=90-25, looseness=2.2}

		\path (w) edge[black, dashed, thick, bend right = 20] (w3);

        \node at ($(u2)+(-0.8,1)$)  {\textcolor{\PatternColorII}{{\CVp}}};
        \node at ($(u1)+(2,3.6)$)  {\textcolor{\subColorTwo}{{\CN}}};
        \node at ($(u3)+(-0.3,-0.8)$)  {\textcolor{\PatternColorIII}{{\CV}}};
        \node at ($(u4)+(1.6,0)$)  {\textcolor{\PatternColorI}{{\CV}}};
        
        \end{scope}
    }
    \renewcommand{\CfL}{fig:J4}
    \drawSimpleRule{fig:J4}{3}{Reduction of configuration \CfN. Example of a $2$-coloring of $S'$}
	}%
	\renewcommand\CfText{
		We transform the $K_4$-subdivision $S$ into another $K_4$-subdivision $S'$ by replacing the path $u_3\sim u_4$ by the path $(u_3,v_2',v_2,\dots,u_4)$. The vertices $u_3,u_4$ are turned into {\CV} patterns and $u_2$ into a {\CVp} pattern. Depending on the order of $v_1,v_2$ on the path $u_3\sim u_4$ of $S$, $u_1$ forms a {\CN} that may cause a distant problem on the new path $(u_3,v_2',v_2,\dots,u_4)$. We consider a coloring of $S'$ given by Claim~\ref{clm:inact} (p.~\pageref{clm:inact}) to inactivate it.
		
		The patterns used are {\CN}$(u_1)$, {\CVp}$(u_2)$, {\CV}$(u_3)$, {\CV}$(u_4)$.
	}%
}
\newcommand{\JFive}{%
    \renewcommand\CfN{\nameJFive}%
    \renewcommand\CfProperties{
		\begin{itemize}
			\item The graph has a strong $C_{4+}^*$-subdivision $S$ rooted on $u_1,u_2,u_3,u_4$, such that $u_1,u_2$ are $1$-linked and $u_1,u_3$ are $2$-linked
			\item There are at most $2$ distant problems: if there is at least one, we may assume w.l.o.g. that $u_1$ causes a distant problem on a $(u_2,u_4)$-path $P_{24}$
			\item $u_2$ is settled
			\item $u_3$ is settled or causes a distant problem on the $(u_2,u_4)$-path $P_{24}'$ of $S$ different from $P_{24}$
		\end{itemize}
    }%
	\renewcommand\CfPict{
	
   	\renewcommand\Myspecialnodes{
        \node[blacknode,label=left:$u_1$] (u1) at (-3+0.5,3) {};
        \node[blacknode,label=left:$u_2$] (u2) at (-3+0.5,0) {};
        \node[blacknode,label=right:$u_3$] (u3) at (3-0.5,3) {};
        \node[blacknode,label=right:$u_4$] (u4) at (3-0.5,0) {};
    }
    \renewcommand\Mynodes{ 	
    	\node[whitenode] (v1) at (-0.8,1) {};
    	\node[whitenode] (v1') at (-1.2,0) {};
    	\node[whitenode] (v2) at (-3.5,0.6) {};
    	\node[whitenode] (v2') at (-2.8,1.2) {};
    	\node[whitenode] (v3) at (3.5,3-0.6) {};
    	\node[whitenode] (v3') at (2.8,3-1.2) {};
    	\node[whitenode] (v4) at (0.8,3-1) {};
    	\node[whitenode] (v4') at (1.2,3-0) {};
    }
    \renewcommand\MyEdgesBefore{
        \myPath{u1}{u3}{\subColor}{left}{60}%
        \myPath{u1}{v4'}{\subColor}{left}{0}%
        \myPath{u2}{v1'}{\subColor}{left}{0}%
        \myPath{v1'}{u4}{\subColor}{left}{0}%
        \myPath{u2}{u4}{\subColor}{right}{60}%
        \myPath{v4'}{u3}{\subColor}{left}{0}%
        \myPath{u1}{u2}{\subColor}{left}{0}%
        \myPath{u3}{u4}{\subColor}{left}{0}%
        
        \myEdge{u1}{v1}
        \myEdge{u1}{v1'}
        \myEdge{v1}{v1'}
        \myEdge{u2}{v2}
        \myEdge{u2}{v2'}
        \myEdge{u3}{v3}
        \myEdge{u3}{v3'}
        \myEdge{u4}{v4}
        \myEdge{u4}{v4'}
        \myEdge{v4}{v4'}
    }%
    \renewcommand\MyEdgesAfter{
        \myPath{u1}{u3}{\subColorOne}{left}{60}%
        \myPath{u1}{v4'}{\subColorTwo}{left}{0}%
        \myPath{u2}{v1'}{\subColorOne}{left}{0}%
        \myPath{v1'}{u4}{\subColorOne}{left}{0}%
        \myPath{u2}{u4}{\subColorTwo}{right}{60}%
        \myPath{v4'}{u3}{\subColorTwo}{left}{0}%
        \myPath{u1}{u2}{\subColorOne}{left}{0}%
        \myPath{u3}{u4}{\subColorTwo}{left}{0}%
       
		\patternCN{u1}{v1}{v1'}{\subColorTwo}
		\patternCN{u2}{v2}{v2'}{\subColorTwo}
		\patternCN{u3}{v3}{v3'}{\subColorOne}
		\patternCN{u4}{v4}{v4'}{\subColorOne}
    }%
	\renewcommand{\CfL}{fig:J5}
    \drawSimpleRule{fig:J5}{2.8}{\CfN\ when $u_1$ and $u_4$ cause distant problems}
	}%
	\renewcommand\CfText{We consider a $2$-coloring of $S$ that inactivates the distant problems: $\{red = (u_3\rightarrow u_1\rightarrow u_2\rightarrow u_4), blue = (u_1\rightarrow u_3\rightarrow u_4\rightarrow u_2)\}$ in such a way that $P_{24}$ receives the color red. The distant problem of $u_1$ is thus inactivated. Since the colors ending on $u_1$ and $u_3$ are different, and since the colors of $P_{24}$ and $P_{24}'$ are different, the potential distant problem of $u_3$ is inactivated. If $u_4$ causes a distant problem instead, we inactivate it by maybe swapping the colors of the two paths between $u_1$ and $u_3$.
		
	The patterns used are {\CN} for all the special vertices, or possibly {\CTTNA}$(u_1,u_2)$ and {\CTTNA}$(u_3,u_4)$.
	}%
}
\newcommand{\JSix}{%
    \renewcommand\CfN{\nameJSix}%
    \renewcommand\CfProperties{
		\begin{itemize}
			\item The graph has a strong $C_{4+}^*$-subdivision $S$ rooted on $u_1,u_2,u_3,u_4$, such that $u_1,u_2$ are $1$-linked and $u_1,u_3$ are $2$-linked
			\item $u_1$ and $u_3$ cause distant problems on a path $P_{24}$ between $u_2$ and $u_4$
			\item $u_2$ and $u_4$ are settled and their remaining neighbors are disjoint from $S$
		\end{itemize}
    }%
	\renewcommand\CfPict{
	
   	\renewcommand\Myspecialnodes{
        \node[blacknode,label=left:$u_1$] (u1) at (-3+0.5,3) {};
        \node[blacknode,label=left:$u_2$] (u2) at (-3+0.5,0) {};
        \node[blacknode,label=right:$u_3$] (u3) at (3-0.5,3) {};
        \node[blacknode,label=right:$u_4$] (u4) at (3-0.5,0) {};
    }
    \renewcommand\Mynodes{ 	
    	\node[whitenode] (v1) at (-0.8,1) {};
    	\node[whitenode] (v1') at (-1.2,0) {};
    	\node[whitenode] (v2) at (-3.3,0.6) {};
    	\node[whitenode] (v2') at (-2.8,1.2) {};
		\node[whitenode] (v3) at (0.8,1) {};
    	\node[whitenode] (v3') at (1.2,0) {};
    	\node[whitenode] (v4) at (3.3,0.6) {};
    	\node[whitenode] (v4') at (2.8,1.2) {};
    	
    	\node[whitenode] (w1) at (-1.6,4) {};
    	\node[whitenode] (w3) at (1.6,4) {};
    }
    \renewcommand\MyEdgesBefore{
    	\myBentEdge{u1}{w1}{\subColor}{left}{0}
    	\myBentEdge{u3}{w3}{\subColor}{left}{0}
        \myPath{w1}{w3}{\subColor}{left}{20}%
        \myPath{u1}{u3}{\subColor}{left}{0}%
        \myPath{u2}{v1'}{\subColor}{left}{0}%
        \myPath{v1'}{v3'}{\subColor}{left}{0}%
        \myPath{u2}{u4}{\subColor}{right}{60}%
        \myPath{v3'}{u4}{\subColor}{left}{0}%
        \myPath{u1}{u2}{\subColor}{left}{0}%
        \myPath{u3}{u4}{\subColor}{left}{0}%
        
        \myEdge{u1}{v1}
        \myEdge{u1}{v1'}
        \myEdge{v1}{v1'}
        \myEdge{u2}{v2}
        \myEdge{u2}{v2'}
        \myEdge{u3}{v3}
        \myEdge{u3}{v3'}
        \myEdge{v3}{v3'}
        \myEdge{u4}{v4}
        \myEdge{u4}{v4'}
    }%
    \renewcommand\MyEdgesAfter{
        \myPath{w1}{w3}{black}{left}{20}%
        \myPath{u1}{u3}{\subColor}{left}{0}%
        \myPath{u2}{v1'}{\subColor}{left}{0}%
        \myPath{v1'}{v3'}{black}{left}{0}%
        \myPath{u2}{u4}{\subColor}{right}{60}%
        \myPath{v3'}{u4}{\subColor}{left}{0}%
        \myPath{u1}{u2}{\subColor}{left}{0}%
        \myPath{u3}{u4}{\subColor}{left}{0}%
       
       	\patternCV{u1}{w1}{v1}{\PatternColorI}
		\patternCV{u3}{w3}{v3}{\PatternColorIII}
		\patternCN{u2}{v2}{v2'}{\PatternColorII}
		\patternCN{u4}{v4}{v4'}{\PatternColorIV}
		
        \myBentEdge{u1}{v1'}{\subColor}{left}{0}
        \myEdge{v1}{v1'}
        \myBentEdge{u3}{v3'}{\subColor}{left}{0}
        \myEdge{v3}{v3'}
    }%
	\renewcommand{\CfL}{fig:J6}
    \drawSimpleRule{fig:J6}{2.8}{Reduction of configuration \CfN}
	}%
	\renewcommand\CfText{Assume w.l.o.g. that the path $P_{24} = (u_2,Q_1,v_1,Q_2,v_3,Q_3,u_4)$, where $v_1,v_3$ are remaining neighbors of $u_1,u_3$ respectively, and $l(Q_1),l(Q_2),l(Q_3)\geq 1$. Each of $u_1,u_3$ has another remaining neighbor $v_1',v_3'$ respectively, adjacent to $v_1,v_3$ respectively. The vertices $v_1',v_3'$ belong to a region of the graph delimited by the four paths $P_{24}$, $u_1\sim u_2$, $u_3\sim u_4$ and a path $P_{13}$ of $S$ between $u_1$ and $u_3$. Let $P_{13}'$ be the other path of $S$ between $u_1$ and $u_3$.
	We transform the $C_{4+}$-subdivision $S$ into another $C_{4+}$-subdivision $S'$, by removing the paths $P_{13}'$ and $P_{24}$, and adding the paths $(u_1,v_1,Q_1,u_2)$ and $(u_3,v_3,Q_3,u_4)$.
	
	The remaining neighbors of $u_1$ (resp. $u_3$) w.r.t. $S'$ are non-adjacent, and since the remaining neighbors of $u_2$ (resp. $u_4$) are disjoint from $S$, $(u_1,u_2)$ (resp. $(u_3,u_4)$) cannot from a {\CTTNA} pattern.
	
	The special vertices $u_1,u_3$ are thus turned into {\CV} patterns, unless the path $P_{13}'$ has length $1$, in which case they form a {\CU} pattern. In the latter case, by property ``$0$-linked'', none of $u_2,u_4$ can have both $v_1',v_3'$ as remaining neighbors, and by property ``$2$-linked'', the remaining neighbors of $u_2,u_4$ are disjoint, so $u_2$ and $u_4$ remain settled w.r.t. $S'$.

	The patterns used are {\CV}$(u_1)$, {\CN}$(u_2)$, {\CV}$(u_3)$, {\CN}$(u_4)$, or {\CU}$(u_1,u_3)$.
	}%
}

\newcommand\AllDistantConf{

\renewcommand{\MyScale}{0.9}

\ifthenelse{\equal{\isThesis}{true}}
{\CfPrintNom{\DOne}\\
\Propcf{\DOne}
\Pictcf{\DOne}
\Textcf{\DOne}\\
}
{\Pictcf{\DOne}
\CfPrintNom{\DOne}\\
\Propcf{\DOne}
\Textcf{\DOne}\\
}

\vfill
\pagebreak

\renewcommand{\MyScale}{0.8}

\CfPrintNom{\DTwo}\\
\Propcf{\DTwo}
\Pictcf{\DTwo}
\Textcf{\DTwo}\\

\vfill
\pagebreak

\renewcommand{\MyScale}{0.9}

\CfPrintNom{\DThree}\\
\Propcf{\DThree}
\Pictcf{\DThree}
\Textcf{\DThree}\\

\vfill
\pagebreak

\renewcommand{\MyScale}{0.9}

\CfPrintNom{\DFour}\\
\Propcf{\DFour}
\Pictcf{\DFour}
\Textcf{\DFour}\\

}

\newcommand\AllCloseConf{

\renewcommand{\MyScale}{0.9}

\Pictcf{\RuOne}
\CfPrintNom{\RuOne}\\
\Propcf{\RuOne}
\Textcf{\RuOne}\\

\vfill
\pagebreak

\CfPrintNom{\RuTwo}\\
\Propcf{\RuTwo}
\Pictcf{\RuTwo}
\Textcf{\RuTwo}\\

\renewcommand{\MyScale}{0.85}

\CfPrintNom{\RuThree}\\
\Propcf{\RuThree}
\Pictcf{\RuThree}
\vfill
\pagebreak
\Textcf{\RuThree}\\

\vfill
\pagebreak

\renewcommand{\MyScale}{0.9}

\CfPrintNom{\RuFour}\\
\Propcf{\RuFour}
\Pictcf{\RuFour}
\Textcf{\RuFour}\\

\vfill
\pagebreak

\CfPrintNom{\RuFive}\\
\Propcf{\RuFive}
\Pictcf{\RuFive}
\Textcf{\RuFive}\\

\vfill
\pagebreak

\renewcommand{\MyScale}{0.9}

\CfPrintNom{\RuSix}\\
\Propcf{\RuSix}
\Pictcf{\RuSix}
\Textcf{\RuSix}\\

\vfill
\pagebreak

\renewcommand{\MyScale}{0.9}

\CfPrintNom{\RuSeven}\\
\Propcf{\RuSeven}
\Pictcf{\RuSeven}
\Textcf{\RuSeven}\\

\renewcommand{\MyScale}{1}

\vfill
\pagebreak

\CfPrintNom{\RuEight}\\
\Propcf{\RuEight}
\Pictcf{\RuEight}
\Textcf{\RuEight}\\

\renewcommand{\MyScale}{0.9}

\CfPrintNom{\RuNine}\\
\Propcf{\RuNine}
\Pictcf{\RuNine}
\Textcf{\RuNine}\\

}

\newcommand\AllSemiDistantConf{

\renewcommand{\MyScale}{0.9}

\ \\

\CfPrintNom{\JOne}\\
\Propcf{\JOne}
\Pictcf{\JOne}
\Textcf{\JOne}\\

\vfill
\pagebreak

\CfPrintNom{\JTwo}\\
\Propcf{\JTwo}
\Pictcf{\JTwo}
\Textcf{\JTwo}\\

\vfill
\pagebreak

\CfPrintNom{\JThree}\\
\Propcf{\JThree}
\Pictcf{\JThree}
\Textcf{\JThree}\\

\vfill
\pagebreak

\renewcommand{\MyScale}{0.9}

\CfPrintNom{\JFour}\\
\Propcf{\JFour}
\Pictcf{\JFour}
\Textcf{\JFour}\\

\renewcommand{\MyScale}{1}

\vfill
\pagebreak

\CfPrintNom{\JFive}\\
\Propcf{\JFive}
\Pictcf{\JFive}
\Textcf{\JFive}\\

\vfill
\pagebreak

\CfPrintNom{\JSix}\\
\Propcf{\JSix}
\Pictcf{\JSix}
\Textcf{\JSix}\\

\vfill
\pagebreak
}

\usepackage[sorting=none,backend=bibtex,maxnames=50]{biblatex}
\newcounter{obscounter}
\newtheorem{thm}{Theorem}[section]
\newtheorem{lem}[thm]{Lemma}
\newtheorem{defn}[thm]{Definition}
\newtheorem{obs}[obscounter]{Observation}
\newtheorem{claim}[thm]{Claim}

\newcounter{claimb}

\def\claimb{$$\vcenter\bgroup\advance\hsize by -8em\noindent
\refstepcounter{claimb}\ignorespaces\it}
\makeatletter
\def\endclaimb{\rm\egroup\leqno(\theclaimb)$$\global\@ignoretrue}
\makeatother

    {\noindent \emph{Proof.} {}{#1}{}}{\hfill
    $\Diamond$\vspace{1em}}

\newcommand{\C}{\mathcal{C}}

\newcommand\Boxcf[1]{\fbox{\Ncf{#1}}}

\title{Gallai's path decomposition in planar graphs}

\author[1]{Alexandre Blanch\'e}

\author[1]{Marthe Bonamy}

\author[1]{Nicolas Bonichon}
\affil[1]{CNRS, LaBRI, Université de Bordeaux, Bordeaux, France}

\newcommand{\isThesis}{false}
\newcommand{\chapsec}{section}
\newcommand{\chCI}{Section~\ref{sec:mce}}
\newcommand{\chCII}{Section~\ref{sec:cii}}

\begin{document}
\maketitle

\begin{abstract}
In 1968, Gallai conjectured that the edges of any connected graph with $n$ vertices can be partitioned into $\lceil \frac{n}{2} \rceil$ paths. We show that this conjecture is true for every planar graph. More precisely, we show that every connected planar graph except $K_3$ and $K_5^-$ ($K_5$ minus one edge) can be decomposed into $\lfloor \frac{n}{2} \rfloor$ paths.
\end{abstract}

\section{Introduction}
Given a finite undirected connected graph $G$, a \emph{$k$-path decomposition} of $G$ is a partition of the edges of $G$ into $k$ paths.
In 1968, Gallai stated this simple but surprising conjecture~\cite{Lovasz-covering}: every graph on $n$ vertices admits a $\lceil \frac{n}{2} \rceil$-path decomposition. Gallai's conjecture is still unsolved as of today, and has only been confirmed on very specific graph classes:
graphs in which each vertex has degree $2$ or $4$~\cite{favaron},
graphs whose vertices of even degree induce a forest~\cite{PYBER1996152}, graphs for which each block of the subgraph induced by the vertices of even degree is triangle-free with maximum degree at most $3$~\cite{FAN2005117, botler2019gallais}, series-parallel graphs~\cite{series_parallel}, or planar $3$-trees~\cite{Botler_2020}.
Bonamy and Perrett confirmed the conjecture on graphs with maximum degree at most $5$~\cite{BP19}, and Chu, Fan and Liu for graphs of maximum degree $6$, under the condition that the vertices of degree $6$ form an independent set~\cite{CHU2021112212}.
Recently, Botler, Jim\'enez and Sambinelli proved that the conjecture is true in the case of triangle-free planar graphs~\cite{Botler_2018_triangle_free}. 

We confirm the conjecture for the whole class of planar graphs.

\begin{thm}
Every connected planar graph on $n$ vertices can be decomposed into $\lceil \frac{n}{2} \rceil$ paths.
\end{thm}

Our result is actually a little more precise.
An \emph{odd semi-clique} is a graph obtained from a clique on $2k+1$ vertices by deleting at most $k-1$ edges.
In 2016, Bonamy and Perrett asked the following question~\cite[Question 1.1]{BP19}:
does every connected graph on $n$ vertices that is not an odd semi-clique admit a $\lfloor \frac{n}{2} \rfloor$-path decomposition?
In other words, is it possible to save one path in the decomposition when $n$ is odd?

We answer this question positively for planar graphs. Only two odd semi-cliques are planar: the triangle $K_3$ and $K_5$ minus one edge, which we denote by $K_5^-$ (see Figure~\ref{fig:K3K5}). We can therefore state our result as follows:

\begin{thm}\label{thm:gallai}
Every connected planar graph on $n$ vertices, except $K_3$ and $K_5^-$, can be decomposed into $\lfloor \frac{n}{2} \rfloor$ paths.
\end{thm}

\begin{figure}[htb]
    \centering
    \begin{tikzpicture}[scale=1,auto]
        \node[blacknode] (u1) at (0,0) {};
        \node[blacknode] (u2) at (0,2) {};
        \node[blacknode] (u3) at (30:2) {};
        \myCEdge{u1}{u2}{blue}{2}
        \myCEdge{u3}{u2}{red}{1}
        \myCEdge{u3}{u1}{red}{1}
        \node[ghost] (u000) at (2,0) {};
    \end{tikzpicture}
    \begin{tikzpicture}[scale=1,auto]
        \node[blacknode] (u1) at (0,0) {};
        \node[blacknode] (u2) at (0,2) {};
        \node[blacknode] (u3) at (30:2) {};
        \node[blacknode] (u4) at (0.5,1) {};
        \node[blacknode] (u5) at (3,1) {};
        \myCEdge{u1}{u2}{red}{1}
        \myCEdge{u2}{u4}{red}{1}
        \myCEdge{u4}{u3}{red}{1}
        \myCEdge{u3}{u5}{red}{1}

        \myCEdge{u4}{u1}{green}{2}
        \myCEdge{u1}{u3}{green}{2}
        \myCEdge{u3}{u2}{green}{2}
        \myCEdge{u2}{u5}{green}{2}

        \myCEdge{u1}{u5}{blue}{3}

        \node[ghost] (u000) at (-1,0) {};
    \end{tikzpicture}
        \caption{A $2$-path decomposition of $K_3$ (left) and a $3$-path decomposition of $K_5^-$ (right)}
    \label{fig:K3K5}
\end{figure}
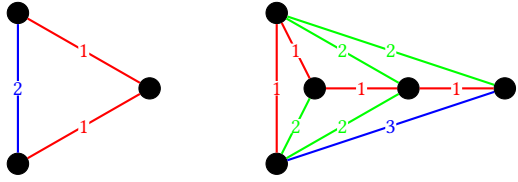


To prove this result, we proceed with a standard approach for coloring problems, by considering a planar graph that is a counterexample to our theorem and is vertex-minimum with respect to this property.
We prove that such a \emph{minimum counterexample} (MCE) does not contain a certain set of configurations, by providing for each of these configurations a reduction rule that takes advantage of the properties of the MCE and yields a contradiction. This technique is widely used in the literature on graph coloring and on Gallai's conjecture (\cite{BP19,CHU2021112212,Botler_2018_triangle_free,Botler2017}).
More precisely, these reducible configurations deal with vertices of small degree (at most $5$), and after showing that our MCE cannot contain any of these configurations (Lemma~\ref{lem:mce}, p.~\pageref{lem:mce}, in Section~\ref{sec:mce}), we know that it has mostly vertices of degree at least $6$.
We finally use Euler's formula and structural arguments to prove that there is no such graph (Lemma~\ref{lem:nomce}, p.~\pageref{lem:mce}, in Section~\ref{sec:nomce}).

\section{Preliminaries}

The graphs we consider are undirected, finite and without loops or multi-edges.
In a graph $G = (V,E)$ and given a subset $X \subseteq V$ of vertices, we denote $N(X) = \{v\in V\ \vert\ \exists u\in X, uv \in E\}$ the \emph{neighborhood} of $X$. The subgraph of $G$ \textbf{induced by} $X$ is the graph with vertices $X$ and edges $\{uv\in E\ \vert\ u,v\in X\}$, denoted $G[X]$. An \emph{$i$-vertex} is a vertex $v$ of degree $i$ in $G$, and we denote $i=d_G(v)$.

A \emph{path} of a graph $G$ is a finite sequence of edges $(v_1v_2,v_2v_3,\dots,v_{k-1}v_k)$ of $G$, denoted $P = (v_1,v_2,\dots,v_k)$, and such that all $v_1,\dots,v_k$ are pairwise distinct. We denote $V(P) = \{v_1,\dots,v_k\}$. We say that a vertex \textbf{touches} $P$ if it belongs to $V(P)$. The \emph{internal vertices} of $P$ are $v_2,\dots,v_{k-1}$ and its \emph{ends} are $v_1,v_k$. We say that two paths are \emph{internally disjoint} if they have no internal vertex in common. We also call $P$ a \emph{$(v_1,v_k)$-path}. A \emph{section} $Q = (v_i,\dots,v_j)$ of a path $P = (v_1,\dots,v_k)$, $1\leq i < j\leq k$ is a subsequence of consecutive edges of $P$. The \emph{length} $l(P)$ of a path $P = (v_1,\dots,v_k)$ is $k$. We say that two vertices $u,v$ are \emph{at distance} $k$ if the minimum length of a $(u,v)$-path is $k$.
For simplicity, we denote some paths by a sequence of subpaths: if $(Q_i)_{i=1..k}$ is a family of edge-disjoint paths, we may denote the path $P = (Q_1,Q_2,\dots,Q_k)$; we may also denote it by an alternation of vertices and subpaths: $P = (v_1,Q_1,v_2,\dots,v_{k-1},Q_k,v_k)$.

A path $P = (v_1,\dots,v_k)$ of a graph $G$ has a \emph{chord} if there is an edge $v_iv_j\in E(G)$ such that $v_i$ and $v_j$ belong to $V(P)$ but are not consecutive in $P$. We say that $P$ is \emph{chordless} if it has no chord.

A \emph{cycle} is a finite sequence of edges $(v_1v_2,\dots,v_kv_1)$, denoted $C = (v_1,v_2,\dots,v_k)$. Its \emph{length} is $k$.

\label{prelim:ksubd}
Given two graphs $G,K$, a \emph{$K$-subdivision} in $G$ is a subgraph of $G$ that is isomorphic to the graph obtained by replacing the edges of $K$ by internally disjoint paths of positive length. The \emph{roots} of the subdivision are the images in $G$ of the vertices of $K$, and the \emph{paths} of the subdivision are the images in $G$ of the edges of $K$. We say that a $K$-subdivision is \emph{chordless} if its paths are chordless. Given a set $U$ of vertices of a graph $G$, we say that a $K$-subdivision of $G$ is \emph{rooted on $U$} if its roots are exactly the vertices of $U$. Given a $4$-family $U = \{u_1,u_2,u_3,u_4\}$, the subdivisions we use in this paper are $K_4$-subdivisions and $C_{4+}$-subdivisions rooted on $U$, where $K_4$ is the clique on $4$ vertices and $C_{4+}$ is the graph made up of a cycle on $4$ vertices $U = \{x_1,x_2,x_3,x_4\}$ and two additional parallel edges $x_1x_3, x_2x_4$.
We say that two subdivisions $S,S'$ have the same \emph{type} if $S,S'$ are both $K_4$-subdivisions or both $C_{4+}$-subdivision.
For $u_i,u_j\in U$, we denote $u_i\sim u_j$ the $(u_i,u_j)$-path of a $K_4$- or $C_{4+}$-subdivision when there is no ambiguity.

To talk about path decompositions, the framework we adopt deals with paths as colors. We say that a ($k$-)\emph{coloring} of a graph $G = (V,E)$ is a function $c : E \rightarrow \llbracket 1,k \rrbracket$, with $k\in\mathbb{N}\setminus\{0\}$, such that the edges with the same color form a path. We denote $\vert c\vert = k$ the number of colors used in the coloring. We say that a color $x$ \emph{induces} a path $P$ if $E(P) = c^{-1}(\{x\})$. A \emph{good coloring} of a connected graph $G$ with $n$ vertices is a $\left \lfloor{\frac{n}{2}}\right \rfloor$-coloring.
A color $x$ \emph{ends} on a vertex $v$ if the path induced by the color $x$ ends on $v$.

To describe the $2$-coloring of a $K_4$-subdivision or a $C_{4+}$-subdivision $S$ rooted on $\{u_1,u_2,u_3,u_4\}$, we use the notation $\{red = (u_{i_1} \rightarrow u_{i_2} \rightarrow u_{i_3} \rightarrow u_{i_4}), blue = (u_{j_1} \rightarrow u_{j_2} \rightarrow u_{j_3} \rightarrow u_{j_4})\}$ for $i_1,i_2,i_3,i_4$ and $j_1,j_2,j_3,j_4$ two permutations of $1,2,3,4$. This notation means that we decompose $S$ into two paths $P_{red} = (u_{i_1}\sim u_{i_2}, u_{i_2}\sim u_{i_3}, u_{i_3}\sim u_{i_4})$ and $P_{blue} = (u_{j_1}\sim u_{j_2}, u_{j_2}\sim u_{j_3}, u_{j_3}\sim u_{j_4})$. The decompositions of this kind that we consider throughout the paper feature each edge of $S$ exactly once.
We sometimes insert vertices that are not roots in between the vertices from $U$ to describe the paths we take more precisely: the notation $(~\dots$~$\rightarrow u_i \rightarrow v \rightarrow u_j \rightarrow \dots)$ means that the path $u_i\sim u_j$ considered is the $(u_i,u_j)$-path of the subdivision that contains the vertex $v$.

We call \emph{minimum counterexample} (MCE) a planar graph that is distinct from $K_3$ or $K_5^-$, which does not admit a good coloring and is vertex minimum with respect to this property.

Given a planar graph $G$, a \emph{$2$-family} is a set $U$ of two vertices of $G$ of degree at most $4$.
A \emph{$4$-family} is a set of four $5$-vertices.
We say that a graph $G$ with a $4$-family $U$ is \emph{almost $4$-connected} w.r.t. $U$ if it does not contain a $3$-cut $A = \{a_1,a_2,a_3\}\subseteq V(G)$ 
that separates two vertices $u,u'\in U$ or two neighbors of some vertex $u\in U\cap A$.


\section{Minimal counterexamples properties}\label{sec:mce}

A \emph{configuration} $C$ is a property satisfied by a graph.
The configurations we consider are usually defined locally as specifications on the neighborhoods of some vertices in the graph, and the possible presence (or absence) of some paths between them.
By abuse of language, we say that a graph $G$ \emph{contains} a configuration $C$ when $G$ satisfies the properties of $C$.

The following lemma is the main result of this section and the next, and helps us prove the main theorem in Section~\ref{sec:nomce}. We prove the first property as Lemma~\ref{lem:ci} at the end of this section and the second property as Lemma~\ref{lem:cii} (p.~\pageref{lem:cii}) in Section~\ref{sec:cii}.

\begin{lem}\label{lem:mce} An MCE does not contain any of the following configurations:
    \begin{itemize}
        \item $(C_I)$: a $2$-family;
        \item $(C_{II})$: a $4$-family with respect to which the MCE is almost $4$-connected.
    \end{itemize}
\end{lem}

\ifthenelse{\equal{\isThesis}{true}}
{In this {\chapsec}, we prove Lemma~\ref{lem:ci} (p.~\pageref{lem:ci}), which constitutes the first property of the main lemma of our proof, Lemma~\ref{lem:mce} in the previous {\chapsec}.

\medskip}
{}

To prove that an MCE $G$ does not contain a configuration $\C$, we proceed by contradiction: we assume that $G$ does contain the configuration, then use its nature of MCE to build a good coloring of $G$. The only configurations $\C$ we consider are specifications of $(C_I)$ and $(C_{II})$. To build a good coloring of $G$, we start by removing the two or four \emph{special vertices} forming the $2$- or $4$-family of $\C$, along with their incident edges. Depending on the case, we add or remove some edges from the obtained graph, and define the resulting graph as the \emph{reduced graph} $G'$. Ideally, this graph is connected and not a $K_3$ or $K_5^-$: in this case, since it is smaller than the minimum counterexample $G$, it admits a good coloring, that we call \emph{pre-coloring}. We then adapt this pre-coloring to the original graph $G$, and since $G$ has two or four more vertices than $G'$, we may use $1$ additional color (for $(C_I)$) or $2$ (for $(C_{II})$) to color $G$. The cases where $G'$ is disconnected are equally easy, unless $G'$ has some $K_3$ or $K_5^-$ connected components, in which case a complementary method is used to combine a ``bad'' coloring of these components with the rest of the coloring.

In the present {\chapsec},
we split the configuration $(C_I)$ into simpler, specified cases, and for each of them provide a \emph{reduction rule} describing the adaptation of the pre-coloring to a good coloring of $G$.
The general method consists in fixing a shortest path $P$ between the two special vertices. The edges of $P$ are removed alongside the special vertices and possibly some additional edges, in order to obtain the reduced graph $G'$. In $G$, the edges of $P$ are then colored with the extra color. This has the advantage of having an end of this extra color on each of the two special vertices, and the path induced by the extra color can be conveniently extended to help cover all the edges of $G$ that were missing in $G'$.

This is the method used when the two special vertices are sufficiently distant from each other, and in this case the adapting methods for each are independent and can be defined separately. Figure~\ref{fig:ci_half} depicts a $(C_I)$ configuration where the two special vertices $u_1,u_2$ are at distance at least $3$, and their neighborhoods are taken care of with two independent \emph{elementary partial rules} that when combined form a complete reduction rule. The extra color induces the path $P$ in red. These rules are defined in \ifthenelse{\equal{\isThesis}{true}}
{Section~\ref{subsec:red_ci}}
{Subsection~\ref{subsec:red_ci}}
 on page~\pageref{half_ci_rules}.

\renewcommand{\MyScale}{1}
\renewcommand{\OrdonneeFleche}{0.7}

\begin{figure}[ht]
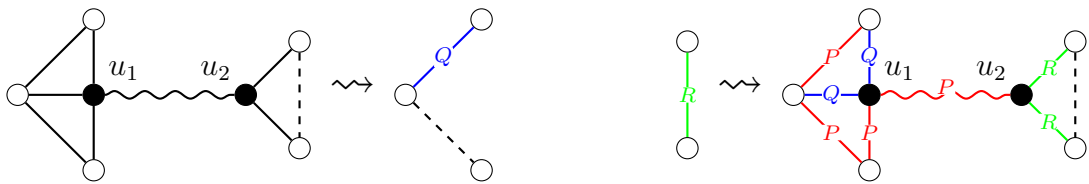

\centering
\RuleExOne
\caption{A $(C_I)$ reduction rule made up of two elementary partial rules}
\label{fig:ci_half}
\end{figure}

When instead the special vertices are too close to each other and share some neighbors, two elementary partial rules would interfere with each other and possibly create some cycles in the decomposition. In these cases, we discard the shortest path and instead use a custom reduction rule to treat the neighborhoods of both special vertices at once. These rules correspond to rules \ncf{\cfCZa}, \ncf{\cfXXa}, $\dots$, \ncf{\cfXXup} in
\ifthenelse{\equal{\isThesis}{true}}
{Section~\ref{subsec:red_ci}}
{Subsection~\ref{subsec:red_ci}}
 on page~\pageref{full_ci_rules}. Figure~\ref{fig:ci_full} features an example of such a rule, and the extra color is again represented in red as the path $P$.

\renewcommand{\OrdonneeFleche}{1.4}

\begin{figure}[ht]
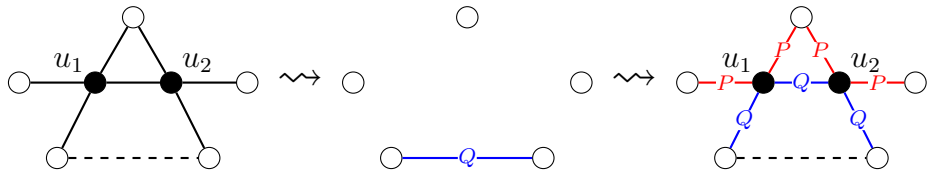

\centering
\RuleExFour
\vspace{-0.5cm}
\caption{A $(C_I)$ reduction rule treating close special vertices}
\label{fig:ci_full}
\end{figure}

The next
\ifthenelse{\equal{\isThesis}{true}}
{section}
{subsection}
 introduces the reduction rules we need to treat the $(C_I)$ configurations. The generalization to this method for $(C_{II})$ configurations is presented at the beginning of 
{\chCII}, on 
\ifthenelse{\equal{\isThesis}{true}}
{page~\pageref{ch:cii}.}
{page~\pageref{sec:cii}.}


\ifthenelse{\equal{\isThesis}{true}}
{\section{Reduction rules for $(C_I)$ configurations}\label{subsec:red_ci}}
{\subsection{Reduction rules for $(C_I)$ configurations}\label{subsec:red_ci}}
This notion of reducible subgraph has been used in previous
\ifthenelse{\equal{\isThesis}{true}}
{works~\cite{BP19,Botler2017,Botler_2018_triangle_free,botler2019gallais,Botler_2020}.}
{works~\cite{botler2019gallais,Botler_2020,BP19,Botler_2018_triangle_free,Botler2017}.}
We present here a formalism appropriate for our subgraphs.

A \emph{reduction rule} $\mathcal{R} = (\C, f^r, f^c)$ is composed of a configuration $\C$, a \emph{reduction function} $f^r$ and a \emph{recoloring function} $f^c$. The configuration distinguishes a $2$-family or $4$-family $U$, that we call \textbf{special vertices}. Given a planar graph $G$ that contains the configuration $\C$, we call $f^r(G)$ the \emph{reduced graph} $G'$, whose vertex set is $V(G') = V(G)\setminus U$ and some of its edges were added or removed from $G$.

Given $G$ and a coloring $pc$ (called \emph{pre-coloring}) of the reduced graph $G' = f^r(G)$, the recoloring function $f^c(G,pc)$ gives a coloring of $G$.

For instance, let us consider the rule shown in Figure~\ref{fig:exRule}, whose formalism will be fully described in the next
\ifthenelse{\equal{\isThesis}{true}}
{section. }
{subsection. }
The configuration of the rule is the following: the graph contains at least 5 vertices $u_1, u_2, v, v_1, v_2$ and potentially a vertex $v_3$ such that $u_1$ is a $3$-vertex adjacent to $v_1, v$ and $u_2$; the vertex $u_2$ is adjacent to $u_1, v, v_2$ and potentially $v_3$, but not to other vertices; and $v_2$ is a vertex of even degree. The reduction function consists in removing every edge incident to $u_1$ and $u_2$. As $v_2$ has an even degree in $G$, it has an odd degree in the reduced graph $G'$. Hence it is the end of a path $Q$. The recoloring function is the following: color the edges $(v_2,u_2)$ and $(u_2,u_1)$ with the color of $Q$, use a new color to color the edges $(v_1,u_1),(u_1,v),(v,u_2)$ and $(u_2,v_3)$ if $u_2$ is a 4-vertex, to form a new path $P$. For all other edges, use colors of $pc$.

\begin{figure}[thb]
\centering
\renewcommand{\MyScale}{1}
\renewcommand\OrdonneeFleche{1.2}
\RuleXXa
{\ifthenelse{\equal{\isThesis}{true}}
{\caption{Example of a reduction rule. The formalism used to describe reduction rules is given in Section~\ref{subsec:red_ci}.} \label{fig:exRule}}
{\caption{Example of a reduction rule. The formalism used to describe reduction rules is given in Subsection~\ref{subsec:red_ci}.} \label{fig:exRule}}
}
\end{figure}

A reduction rule $\mathcal{R} = (\C,f^r,f^c)$ is \emph{valid} if for any planar graph $G$ that contains the configuration $\C$, then the reduced graph $G' = f^r(G)$ is planar; for any path coloring $pc$ of $G'$, $f^c(G,pc)$ is a path coloring; and for any coloring $pc$ of $G'$, $\vert f^c(G,pc)\vert - \vert pc\vert \leq \lfloor\frac{\vert V(G)\vert - \vert V(G')\vert}{2}\rfloor$.
One can easily check that the rule of Figure~\ref{fig:exRule} is valid.

The rest of 
this {\chapsec}
and the entirety of 
{\chCII}
are dedicated to describing a set of configurations that cover all the cases of Lemma~\ref{lem:mce}, and providing a resolution rule for each of these configurations.
For each rule we provide, we justify that it is valid.
The existence of a valid rule $\mathcal{R}=(\C,f^r,f^c)$ is in itself not enough to guarantee that an MCE cannot contain the configuration $\C$. Indeed, the reduced graph could contain a $K_3$ or $K_5^-$ connected component, and therefore not admit a good coloring. However, we argue in Lemma~\ref{lem:safety_ci} (on page~\pageref{lem:safety_ci}, for rules associated with $(C_I)$ configurations) and Lemma~\ref{lem:safety_cii} (on page~\pageref{lem:safety_cii}, for rules associated with $(C_{II})$ configurations) that our rules are sufficient to build a good coloring of the MCE regardless of the presence of $K_3$ or $K_5^-$ components in the reduced graph.


\ \\
\ifthenelse{\equal{\isThesis}{true}}
{}
{\vfill\pagebreak}

The rest of 
this {\chapsec}
is dedicated to proving the first part of Lemma~\ref{lem:mce} (p.~\pageref{lem:mce}), which we reformulate as the following lemma.

\begin{lem}
\label{lem:ci}
An MCE does not contain a configuration $(C_I)$.
\end{lem}

To prove this result let us introduce a first set of reduction rules, defined over local conditions. We show that each reduction rule is valid in Lemmas~\ref{lem:valid1}, p.~\pageref{lem:valid1}, and~\ref{lem:valid2}, p.~\pageref{lem:valid2}.
We then prove that the application of each of these rules on an MCE $G$ is sufficient to provide a good coloring of $G$ (Lemma~\ref{lem:safety_ci}, p.~\pageref{lem:safety_ci}).
Finally, we show (Lemma~\ref{lem:conf_CI}, p.~\pageref{lem:conf_CI}) that the configuration $(C_I)$ is of the form of at least one of the reducible configurations that we list below.

We define the rules with both a graphical and a textual formalism, each of them being self-sufficient.

\paragraph{\textbf{Graphical formalism}}
Each rule $(\C,f^r,f^c)$ is formally defined by a triplet of drawings (see for instance Figure~\ref{fig:exRule}).
The first drawing defines the configuration $\C$ of the rule, the reduction function $f^r$ of defined by the difference between the first two drawings, and finally the third drawing defines the recoloring function $f^c$. Let us first describe the graphical formalism used to define the configuration of a rule.

\renewcommand{\MyScale}{0.8}

The vertices involved in a configuration are represented by circles (%
\tkSymbol{\node[whitenode] (wn) at (0,0) {};} %
 or
\tkSymbol{\node[blacknode] (wn) at (0,0) {};}%
), diamonds (%
\tkSymbol{\node[oddnode] (wn) at (0,0) {};}%
) or squares (%
\tkSymbol{\node[evennode] (wn) at (0,0) {};}%
). The existence of an edge is materialized by a solid line (%
\tkSymbol{\node[blacknode] (d1) at (0,9) {};%
\node[blacknode] (d2) at (1,9) {};%
\draw (d1) -- (d2);}%
) between the vertices. The absence of an edge is materialized by a dashed line (%
\tkSymbol{\node[blacknode] (c1) at (0,8) {};%
\node[blacknode] (c2) at (1,8) {};%
\draw[dashed] (c1) -- (c2);}%
). A waved line (%
\tkSymbol{\node[blacknode] (b1) at (0,7) {};%
\node[blacknode] (b2) at (1,7) {};%
\draw[path] (b1) -- (b2);}%
) represents a path between two vertices that avoids other represented vertices (unless specified).
When an edge is represented by a dash-dotted edge (%
\tkSymbol{\node[blacknode] (u1) at (0,0) {};%
        \node[blacknode] (u2) at (1,0) {};%
        \myCEdge{u1}{u2}{optedge,black}{}%
        }%
), this means that we consider the cases when this edge is present and when it is absent. A solid line with a gray vertex in the middle (%
\tkSymbol{\node[blacknode] (u1) at (0,0) {};%
        \node[blacknode] (u2) at (1,0) {};%
        \node[graynode] (uvm) at (0.5,0) {};%
        \myCEdge{u1}{uvm}{black}{}%
        \myCEdge{uvm}{u2}{black}{}%
        }%
) represents a path of length $1$ or $2$. If it is of length $2$, the middle vertex is distinct from the other vertices on the figure.

When all the incident edges of a vertex are explicitly drawn (with solid or dash-dotted edges), the vertex is represented by a black circle (%
\tkSymbol{\node[blacknode] (wn) at (0,0) {};}%
). Vertices of odd (resp. even) degree are represented by diamonds (%
\tkSymbol{\node[oddnode] (wn) at (0,0) {};}%
) (resp. squares (%
\tkSymbol{\node[evennode] (wn) at (0,0) {};}%
)). A dashed waved line (%
\tkSymbol{\node[blacknode] (b1) at (0,7) {};%
\node[blacknode] (b2) at (1,7) {};%
\draw[path,dashed] (b1) -- (b2);}%
) means that the graph does not contain a path between its endpoints that avoids every vertex represented.

For the second drawing of the triplet, we need a few other conventions, because it also encodes information on the pre-coloring.
Letters $Q,R$ denote paths from the pre-coloring, and the two letters may represent the same path.
If an edge is colored $Q$ and another is colored $\overline{Q}$, this means that they have different colors.
If an edge stays black in the second drawing, it means that the edge keeps its color in the recoloring of $G$.
An incoming arrow colored $Q$ at a vertex means that this vertex is the end of a path colored $Q$. If this arrow is on a half-edge, this means that the last edge of the path is not determined by the figure, 
and one of the drawn vertices could be the other end of the edge.

A black solid (resp. dashed) arrow between a vertex and a path means that the vertex is (resp. is not) on the path (see for example the cases {\nameXXgFive} and {\nameXXgSix} of rule \ncf{\cfXXg}).

The definition of the reduction function is quite straightforward. Every edge that is in the first (resp. second) drawing but not in the second (resp. first) is deleted (resp. added) by the reduction function and both vertices $u_1, u_2$ are deleted (together with their incident edges).

For rules~\ncf{\cfXXg} and \ncf{\cfXXr}, a case analysis is needed to define the recoloring function: the second drawing is split into several copies of the same graph, depicting the different possible properties of the pre-coloring of the reduced graph.

The third graph encodes directly the recoloring function by giving colors explicitly to the edges of $G \setminus G'$ from the pre-coloring of the reduced graph $G'$.


\paragraph{\textbf{Textual formalism}}
%
For each rule $(\C,f^r,f^c)$, we first define textually the configuration $\C$, over local conditions around the special vertices. We then define the reduction function $f^r$ by specifying the edges that are added or removed to form the reduced graph $G'$. The special vertices $u_1,u_2$ and their incident edges are removed in each rule and omitted in the descriptions. Finally, we define the recoloring function $f^c$ by describing the operations applied to a coloring of $G'$ to color all the edges of $G$.

We call \emph{deviation} the recoloring operation that consists in replacing a color inducing a path $P'$ in $G'$ by a color inducing a path $P$ in $G$, such that $P$ is obtained from $P'$ by replacing an edge $vv'$ with a section $(v,u,v')$ of length $2$ or $(v,u,u',v')$ of length $3$, using only special vertices $u,u'$ as internal vertices. An example of deviation is the path $Q$ in rule \ncf{\cfCZa}.

We call \emph{extension} the recoloring operation that extends a path $Q$ induced by a color of $G'$ on several additional edges in the coloring of $G$, those edges having only the endpoints of $Q$ and special vertices as ends.
In particular, when a non-special vertex has an odd degree in $G'$, a color must end on it and we may extend this color in $G$. The rules are frequently defined so as to ``force'' some vertices to have an odd degree in the reduced graph.


When the rule involves $2$ special vertices (which is the case for all of them except \ncf{\cfCZa}), we may use an \emph{extra color} (inducing the red path $P$ on the drawings) to color the graph $G$, as by definition of valid rule, one ($\left\lfloor\frac{\vert V(G)\vert - \vert V(G')\vert}{2}\right\rfloor$) additional color is allowed when adapting the pre-coloring of $G'$ to a coloring of $G$.

\ifthenelse{\equal{\isThesis}{true}}
{\vfill
\pagebreak}
{\medskip}

With all this formalism in mind, we can introduce our first set of rules: \ncf{\cfCZa}, \ncf{\cfXXa}, $\dots$, \ncf{\cfXXup}.
Note that we justify the planarity of the reduced graph and the validity of all the rules in Lemma~\ref{lem:valid1} after the definitions.\\

\ifthenelse{\equal{\isThesis}{true}}
{}
{\vfill
\pagebreak}

\textbf{List of the rules \ncf{\cfCZa}, \ncf{\cfXXa}, $\dots$, \ncf{\cfXXup}:}
\label{full_ci_rules}

\AllRulesDesc

\medskip

  \begin{lem}\label{lem:valid1}
      The reduction rules of configurations \ncf{\cfCZa}, \ncf{\cfXXa}, $\dots$, \ncf{\cfXXup} are valid.
  \end{lem}

  \begin{proof}
	For each rule, we need to check three properties: when the rule applied to a planar graph $G$, the reduced graph $G'$ produced is planar; the recoloring function yields a path coloring (i.e. does not introduce cycles); and the number of additional colors used in the recoloring function is at most $\left\lfloor\frac{\vert V(G)\vert - \vert V(G')\vert}{2}\right\rfloor$.

	For the first property, observe that in all considered rules except rule \ncf{\cfXXg}, an edge $ab$ added by the reduction function replaces a deleted path of length $2$ or $3$ between $a$ and $b$ that goes through $1$ or $2$ special vertices. Hence the planarity is preserved in these cases. Now, let us consider rule \ncf{\cfXXg}. Let $G''$ be the graph obtained from $G$ by removing the vertices $u_1$ and $u_2$ and their incident edges. When removing $u_1$ and its incident edges, we obtain a face $f_1$ incident with $v_3$, $v_1$ and $v_2$. For the same reason, $v_4$, $v_1$ and $v_2$ have a common face $f_2$. If $f_1 \neq f_2$, since $v_1,v_2$ is a separating pair that separates $v_3$ from $v_4$, there exists a planar embedding of $G''$ such that $v_3$ and $v_4$ are on the same face. Hence in both cases there is an embedding of $G''$ such that $v_3$ and $v_4$ are incident with a common face. Hence we can add the edge $v_3v_4$ to obtain $G'$ while preserving the planarity.

	The third property is easy to check, since the rule \ncf{\cfCZa} does not introduce any new color, and all other rules have $2$ special vertices and introduce exactly $1$ new color (the color red on the figures, inducing the new path $P$).

	Finally, let us check the second property. We can easily check on each rule that when the recoloring function $f^c$ is applied on a planar graph $G$ and a coloring $pc$ of the reduced graph $G'$, the coloring $f^c(G,pc)$ provides a color for each edge of $G$ (as a reminder, the edges drawn in black in the second and third drawings of a rule keep their color from the pre-coloring $pc$ in $G$).

	In $G$, the only edges of the new path $P$ are the ones represented in the third drawing of each rules. One can easily check for each rule that these edges form a path.
  For colors used in the pre-coloring $pc$, except in rules \ncf{\cfXXg} and \ncf{\cfXXr}, we only perform two types of modification, deviations and extensions. 
  Since each special vertex is only involved in one such modification, each color of $pc$ induces a path in $G$.

  In the first case of rule~\ncf{\cfXXr}, we do a simple path extension as before. In the second case, we first split an existing path into two paths, using the new color (the path $P$ in red). Then we extend these two subpaths toward special vertices, hence we obtain a path coloring.

  For rule~\ncf{\cfXXg}, heavier modifications are made on the path $Q$ in $G'$. In rule {\nameXXgOne}, the edge $v_3v_4$ is replaced by the path $(v_3,u_1,v_2,u_2,v_4)$. Since the vertex $v_2$ explicitly avoids the path $Q$, the resulting coloring does not introduce cycles.
 In rule {\nameXXgTwo}, the vertex $v_5$ avoids the subpath $Q_2$ of $Q$, i.e. the section after $v_4$, and $v_6$ avoids $Q_1$, i.e. the section before $v_2$. Hence in the final coloring, $P$ and $Q'$ are paths.
 In rule {\nameXXgThree},
 we can easily check that $Q'$ is a path.
 Moreover, $v_5$ avoids the section $R_1$ of $Q$, i.e. the section between $v_2$ and $v_3$. Since by planarity $v_6$ cannot touch $R_1$, then $P$ is a path. The same argument applies to rule {\nameXXgFour} with both $v_5$ and $v_6$ avoiding the section $R_1''$, the former by hypothesis and the latter by planarity.
 In rule {\nameXXgFive}, $Q'$ is a path since $v_5$ cannot touch the section $R_1'$ of $Q$ by hypothesis, and the section $Q_2$ by planarity. Similarly, $P$ is a path since $v_6$ cannot touch the section $R_1$ of $Q$ by hypothesis, and the section $R_1''$ by planarity.
 In rule {\nameXXgSix}, we can easily check that $Q'$ is a path. Since both $v_5$ and $v_6$ avoid the section $R_1'$ of $Q$, then $P$ is a path.
 This completes the examination of all cases.
  \end{proof}

The second group of rules we consider are given by the neighborhood of two vertices $u_1$ and $u_2$ of degree at most $4$ together with a shortest path $S$ joining them, such that 
$\vert N(u_1)\ \cap\ N(u_2)\vert \leq 1$. Note that this includes the case where $u_1$ and $u_2$ are adjacent, with no common neighbors.
These rules can be described as the product of two so-called \emph{elementary partial rules}, that specifies the behavior of the rule around each endpoint of the path. In 
{\chCII},
we will introduce more general rules and partial rules to deal with $(C_{II})$ configurations.

More formally, an \emph{elementary partial configuration} $\C_i$ is a configuration defined over the neighborhood of one \emph{special vertex} $u_i$, with one identified incident edge, called \emph{subdivision edge}. We say that a neighbor $v$ of $u_i$ is a \emph{remaining neighbor} of $u_i$ if $u_iv$ is not the subdivision edge.

Given a graph $G$, two vertices $u_1,u_2$ of $G$, a shortest $(u_1,u_2)$-path $S$, and two elementary partial configurations $\C_1$ and $\C_2$, the \emph{path composite configuration} 
$(\{\C_1(u_1),\C_2(u_2)\}, S)$ 
is defined as the following configuration: $u_1$ (resp. $u_2$) satisfies the partial configuration $\C_1$ (resp. $\C_2$), the path $S$ contains the subdivision edges of $\C_1$ and $\C_2$ and does not touch the other neighbors of $u_1$ and $u_2$. For ease of notation, we may simply write $\C_1\oplus\C_2$.

An \emph{elementary partial rule} is a rule $\mathcal{R}_i = (\C_i,f_i^r,f_i^c)$ associated with an elementary partial configuration $\C_i$, a partial reduction function $f_i^r$ and a partial recoloring function $f_i^c$. The partial reduction function $f_i^r$ of the rule is encoded by a set $\mathcal{O}_i \subseteq \{\text{\textit{add}}, \text{\textit{remove}}\}\times E(\C_i)$ with straightforward semantics.
In particular, we can identify all the vertices of $V(G)\setminus N(u_i)$ between $G$ and $f_i^r(G)$. The partial recoloring function defines the coloring of the edges, based on existing colors plus an extra color (represented as red on the figures) used in part for the edges of $S$.




If $\mathcal{R}_1=(\C_1,f_1^r,f_1^c)$ and $\mathcal{R}_2=(\C_2,f_2^r,f_2^c)$ are two elementary partial rules, $u_1,u_2$ two vertices and $S$ a shortest $(u_1,u_2)$-path, the \emph{path composite rule} 
$(\{\mathcal{R}_1(u_1),\mathcal{R}_2(u_2)\}, S)$ 
is the reduction rule $(\C_c, f_c^r, f_c^c)$ associated with the path composite configuration\\$(\{\C_1(u_1),\C_2(u_2)\}, S)$, and is defined as follows. Let $U = \{u_1,u_2\}$. 
The reduction function $f_c^r$ is defined by $f_c^r(G) = (f_1^r \circ f_2^r(G))\setminus (U \cup E(S))$, i.e. the successive application of the operations in $\mathcal{O}_2$ and $\mathcal{O}_1$ and the removal of the special vertices and the edges of $S$ to form the reduced graph $G'$.


Let $pc$ be a coloring of $G' = f_c^r(G)$, and $c_S$ a $1$-coloring of the path $S$. The recoloring function $f_c^c$ is defined by $f_c^c(G,pc) = f_2^c(G, f_1^c(f_2^r(G), pc\cup c_S))$; in other words, the path $S$ is added to $G'$ and colored with $c_S$, then the reduction of $\C_1$ is undone, the edges in the neighborhood of $u_1$ are colored according to the partial recoloring function $f_1^c$, and finally the reduction of $\C_2$ is undone (to obtain $G$) and the edges in the neighborhood of $u_2$ are colored according to the partial recoloring function $f_2^c$.\\


Let us present the list of elementary partial rules that we consider in this 
{\chapsec}.
We extend our graphical formalism by representing the subdivision edge as a red edge with a double arrow (\tkSymbol{\node[blacknode] (f1) at (1,11) {};  \myHalfEdge{f1}{0}{sedge}{};\draw (f1) node[right] {};}). 

Note again that we justify the planarity of the reduced graph and the validity of the rules formed by two elementary partial rules in Lemma~\ref{lem:valid2} after the definitions.

\ \\

\textbf{List of the elementary partial configurations:}
\label{half_ci_rules}

\ifthenelse{\equal{\isThesis}{true}}
{}
{\ \\}

\AllSemiRulesTwoDesc

\ \\
For convenience, we define some aliases which group several elementary partial configurations together. Note that these aliases are not disjoint: the configuration \ncf{\cfCV} appears in both \ncf{\cfCVPlus} and \ncf{\cfCN}, and these two aliases are particular cases of \ncf{\cfCNP}.

\AllAliasDesc

\medskip

Since two partial rules must be applied on the same graph, we need to make sure that they are compatible and do not interfere with each other. For example, if two \ncf{\cfCV} rules were to be applied on the same non-edge $v_1v_2$, the color of $v_1v_2$ in the reduced graph would have to be deviated to two different special vertices, creating a cycle in the path decomposition. Alternatively, if two \ncf{\cfCvTwo} or \ncf{\cfCvThree} configurations share a remaining neighbor $v$, then the extra color may be extended to $v$ from both remaining neighbors, again creating a cycle in the decomposition. The following lemma provides sufficient conditions of compatibility between partial rules, which are satisfied in the composite configurations given in the proof of Lemma~\ref{lem:conf_CI} (p.~\pageref{lem:conf_CI}).

\begin{lem}[Sufficient conditions of compatibility between partial rules]
\label{lem:valid2}
Let $u_1, u_2$ be two special vertices, and $S$ a shortest $(u_1,u_2)$-path, 
let $\C_a\in$\ \ncf{\cfCNP} and $\C_b\in$\ \ncf{\cfCNP}, and let $\mathcal{R}_a,\mathcal{R}_b$ be the elementary partial rules associated with $\C_a,\C_b$ respectively.
%
Then the path composite rule 
$(\{\mathcal{R}_a(u_1),\mathcal{R}_b(u_2)\}, S)$
is valid if the following conditions are satisfied:
\begin{itemize}
\item If none of $\C_a,\C_b$ are \ncf{\cfCV} or \ncf{\cfCFoura} configurations, then $u_1,u_2$ share no remaining neighbors;
\item If $\C_a$ is a configuration \ncf{\cfCV} or \ncf{\cfCFoura}, then $u_1$ shares at most one remaining neighbor $v$ with $u_2$, and $v$ is non-adjacent to at least one other remaining neighbor of $u_1$ (i.e. $v$ is not $v_3$ in the definition of \ncf{\cfCFoura}).
\end{itemize}
\end{lem}


\begin{proof}
We need to check three properties: when the path composite rule is applied on a planar graph $G$, the reduced graph $G'$ produced is planar; the recoloring function does not introduce cycles; and the number of additional colors used in the recoloring function is at most $\lfloor\frac{\vert V(G)\vert - \vert V(G')\vert}{2}\rfloor$.

The first property is easy to check, as in all considered elementary partial rules, an edge $v_1v_2$ added by the reduction function replaces a deleted path of length $2$ between $v_1$ and $v_2$ that goes through a special vertex, and each special vertex is involved in at most one such operation.

For the the third property, observe that the rule 
$(\{\mathcal{R}_a(u_1),\mathcal{R}_b(u_2)\}, S)$
involves two special vertices. The recoloring function only uses colors from the pre-coloring, plus $1$ new color ($c_S$ in the definition, the path $P$ drawn in red on the figures) which is exactly $\lfloor\frac{\vert V(G)\vert - \vert V(G')\vert}{2}\rfloor$.

For colors used in the pre-coloring, we only perform two types of modification: deviations and extensions.
Since each special vertex is involved in at most one such modification, the recoloring function does not introduce cycles involving these colors.

It remains to check that the set $P$ of edges colored with the new color in $c_S$ is indeed a path. It is made up of a shortest path $S$ between the two special vertices, and maybe some edges incident with only remaining neighbors and special vertices. Let $P_a$ (resp. $P_b$) be the edges of $P$ in the recoloring of $\mathcal{R}_a$ (resp. $\mathcal{R}_b$). So $P = S\cup P_a\cup P_b$. We can check that each elementary partial rule does not introduce cycles within the elementary partial configuration, i.e. $S\cup P_a$ and $S\cup P_b$ are paths.
If $\C_a$ and $\C_b$ do not share remaining neighbors, then $P_a$ and $P_b$ are vertex-disjoint, and thus $P$ induces a path.
Now assume that $\C_a$ is a configuration \ncf{\cfCV} or \ncf{\cfCFoura}, that (w.l.o.g.) the vertex $v_1$ of $\C_a$ (w.r.t. the notations in the definition of \ncf{\cfCV} and \ncf{\cfCFoura}) also belongs to $\C_b$, and that the other remaining neighbors of $\C_a$ and $\C_b$ are disjoint. Since $v_1$ does not touch the edges of $P_a$, then $P_a$ and $P_b$ are vertex-disjoint and $P$ is a path. This concludes the proof.
\end{proof}



\ifthenelse{\equal{\isThesis}{true}}
{\section{Sufficiency of the $(C_I)$ rules}}
{\subsection{Sufficiency of the $(C_I)$ rules}}

We prove with Lemma~\ref{lem:safety_ci} that, because of the reduction rules associated with configurations \ncf{\cfCZa}, \ncf{\cfXXa}, $\dots$, \ncf{\cfXXup},
\ncf{\cfCNP}~$\oplus$~\ncf{\cfCNP}, an MCE $G$ cannot contain any of them.
We first show that it is easy to derive a contradiction if the reduced graph $G'$ does not contain $K_3$ or $K_5^-$ components, by building a good coloring of $G$. Then, we assume that $G'$ contains some $K_3$ or $K_5^-$ connected components, and we build once again a good (path-)coloring of $G$, with the help of an intermediate coloring of $G'$ with paths and cycles.
The proof then combines cycles and paths together to save colors. Let us introduce a relevant lemma from \cite{CHU2021112212}. The \emph{exceptional graph} is a graph consisting of a cycle $C$ of length $5$ and a path $P$, such that $V(C)\subseteq V(P)$ and $V(C)$ induces a $K_5^-$. We restate here Lemma 2.1. from \cite{CHU2021112212}.

\begin{figure}[ht]
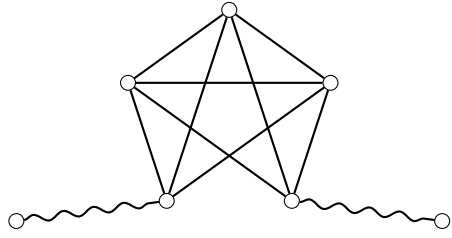

\renewcommand{\MyScale}{0.7}
\cfExcGraph
\tkcf{}
\caption{The exceptional graph}
\label{fig:ExcGraph}
\end{figure}

\begin{lem}[\cite{CHU2021112212}]
\label{L21}
Let $C$ be a cycle and $P$ a path, such that $C$ and $P$ are edge-disjoint. If $1\leq \vert V(C)\cap V(P)\vert \leq 5$, then $E(C)\cup E(P)$ can be decomposed into $2$ paths, unless $C\cup P$ is the exceptional graph.
\end{lem}

\textbf{Remark:} Lemma~\ref{L21} cannot be applied to $K_5^-$, as in any decomposition of $K_5^-$ into a cycle of length $5$ and a path of length $4$, the path and the cycle form the exceptional graph.

There are several possible decompositions of the exceptional graph into a path and a cycle. The following observation states that in any such decomposition, the path and the cycle satisfy the properties of the definition.

\begin{obs}
Let $C$ be a cycle and $P$ a path, such that $C\cup P$ forms the exceptional graph. Then $C$ has length $5$, $V(C) \subseteq V(P)$ and $(C\cup P)[V(C)]$ (the subgraph of $C\cup P$ induced by the vertices of $C$) is a $K_5^-$.
\label{obs1}
\end{obs}

\begin{proof}
There are exactly $5$ vertices of degree at least $3$ in $C\cup P$. Since the vertices of $(C\setminus P) \cup (P\setminus C)$ have degree $1$ or $2$, we deduce that $\vert C\cap P\vert = 5$.

$(C\cup P)[V(C\cap P)]$ is a $K_5^-$, so if there is an edge $e\in C$ that does not belong to $E(C\cap P)$, then there is a $2$-path decomposition of $K_5^-$, a contradiction. So $E(C)\subseteq E(C\cap P)$, hence $V(C)=V(C\cap P)$. We deduce that $V(C)\subseteq V(P)$, $C$ has length $5$ and $(C\cup P)[V(C)]$ is a $K_5^-$.
\end{proof}


\begin{lem}
An MCE does not contain any of the configurations \ncf{\cfCZa}, \ncf{\cfXXa}, $\dots$, \ncf{\cfXXup}, %
and does not contain a path composite configuration \ncf{\cfCNP}~$\oplus$~\ncf{\cfCNP} that satisfies the conditions of Lemma~\ref{lem:valid2} (p.~\pageref{lem:valid2}).
\label{lem:safety_ci}
\end{lem}

\begin{proof}

Let us consider such a configuration $X$ 
and let $\mathcal{R}_X = (X,f_X^r,f_X^c)$ be its associated reduction rule. By Lemmas~\ref{lem:valid1} (p.~\pageref{lem:valid1}) and \ref{lem:valid2} (p.~\pageref{lem:valid2}), $\mathcal{R}_X$ is valid.
Let $G$ be an MCE containing the configuration $X$, and $G' = f^r(G)$ its reduced graph.

Let $n'_1,\dots,n'_p$ be the sizes of the connected components $G'_1,\dots, G'_p$ of $G'$. Observe that $\sum_{j\leq k} n'_j = |V(G')|$.
Observe that each $G_i'$ that is not a $K_3$ nor a $K_5^-$ component is a connected planar graph that is smaller than the minimum counterexample $G$, hence these $G_i'$ each admit a good coloring, using $\lfloor \frac{n'_i}{2} \rfloor$ colors colors. Each $K_3$ component can be decomposed into $1$ cycle and each $K_5^-$ component into $1$ path and $1$ cycle.
Then $G'$ admits a coloring $pc$ into paths and cycles with $\sum_{i=1}^p \lfloor \frac{n'_i}{2} \rfloor \leq \lfloor \frac{|V(G')|}{2} \rfloor$ colors. Since the reduction rule is valid, $G$ admits a coloring $c_0$ into paths and cycles with at most $\lfloor \frac{|V(G')|}{2} \rfloor + \lfloor\frac{\vert V(G)\vert - \vert V(G')\vert}{2}\rfloor \leq \lfloor \frac{|V(G)|}{2} \rfloor$ colors, such that no new cycles are created.
%


Observe that all the cycles in $c_0$ are vertex-disjoint. Indeed, the cycles in $pc$ are vertex-disjoint because they belong to different connected components of $G'$; the cycles may have been deviated into longer cycles in $G$, but since the internal vertices of the deviated sections are all special vertices, and since each special vertex is involved in at most one deviation, then no vertex of $G$ can belong to the intersection of two cycles of $c_0$.

We build iteratively a good coloring $c$ of $G$, by starting from $c_0$ and using Lemma~\ref{L21} (p.~\pageref{L21}) at each iteration to replace a cycle and a path of $c$ by two new paths that decompose the same set of edges. We first consider the case where $X\neq$ \ncf{\cfCZa} and $X\neq$ \ncf{\cfXXg}.


We successively treat the $K_5^-$ and $K_3$ components in $G'$.
First, let us consider a component $K$ in $G'$ that is a $K_5^-$, colored with a cycle $C'$ of length $5$ and a path $P'$ of length $4$ in $pc$. 
$C'$ is turned into a cycle $C_0$ from $c_0$ after possibly some deviations, and Lemma~\ref{L21} has not been applied to it yet, so $C_0$ is also induced by a color of $c$.
$P'$ may have been extended (such as with path $Q$ in rule \ncf{\cfCvTwo}) and deviated to special vertices, into a path $P_0$ of $G$ induced by a color of $c_0$. So $V(P_0)\subseteq V(P')\cup U$, and since each special vertex is involved in at most one deviation or one extension, then $P_0$ is disjoint from the cycles of $c_0$ 
different from $C_0$, so Lemma~\ref{L21} has not been applied to it in previous iterations, and thus $P_0$ is also induced by a color of $c$.
Finally, let $\widehat{P}$ be a path induced by a color of $c$ such that $V(\widehat{P}) \cap V(C_0)\neq \emptyset$ and $\widehat{P} \neq P_0$.
Such a path $\widehat{P}$ exists because $G$ is connected and the cycles in $c$ are disjoint.

Observe that $V(P_0)\cap V(C_0) = V(K)$, so $\vert V(P_0)\cap V(C_0) \vert = 5$.
If there is at least one deviation on $C'$, then by Observation~\ref{obs1} (p.~\pageref{obs1}), $C\cup P$ does not form the exceptional graph. We may then apply Lemma~\ref{L21} (p.~\pageref{L21}) on $C$ and $P$.
Otherwise there is no deviation on $C'$, so $C_0 = C'$ (of length $5$). At least $2$ edges of $G[V(C_0)]$ belong to $P_0$, and thus by Observation~\ref{obs1} (p.~\pageref{obs1}), $C_0\cup \widehat{P}$ does not form the exceptional graph. We can apply Lemma~\ref{L21} (p.~\pageref{L21}) on $C_0$ and $\widehat{P}$, as $\vert V(C_0) \cap V(\widehat{P})\vert \leq \vert V(C_0)\vert = 5$.

Now consider the case where $K$ is a $K_3$, colored by a cycle $C'$ of length $3$ in $G'$. Again let $C$ be the cycle in $G$ that is induced by the same color in $c$ as $C'$ in $pc$, after possibly some deviations, and let $\widehat{P}$ be a path induced by a color of $c$ such that $V(\widehat{P}) \cap V(C)\neq \emptyset$. At most two deviations were performed on $C'$, so $\vert V(C)\vert \leq 5$. We apply Lemma~\ref{L21} (p.~\pageref{L21}) on $C$ and $\widehat{P}$, unless they form the exceptional graph. Assume it is the case and $C$ is disjoint from any other path in the coloring of $G$.
By Observation~\ref{obs1} (p.~\pageref{obs1}), $C$ has length $5$ and thus contains both special vertices.
Since $G$ is not a $K_5^-$, $\widehat{P}$ has at least one other edge, necessarily from one special vertex to a vertex outside $V(C)$.
$G$ fits the description of the configuration \ncf{\cfXXlp}. In the rule associated with \ncf{\cfXXlp}, if the reduced graph contains a $K_3$ component, it contains only the vertex $v_2$ (in the definition of the rule), and its cycle is not deviated to the special vertices, a contradiction.

In all cases, we were able to find a cycle and a path that do not form the exceptional graph. We apply to them Lemma~\ref{L21} (p.~\pageref{L21}) and obtain two new paths that decompose them. We replace the cycle and the path in $c$ by the two new paths, to obtain a coloring of $G$ that uses the same number of colors and contains one less cycle. After this method has been successively applied to all $K_3$ and $K_5^-$ appearing in $G'$, $c$ that has the right number of colors and contains only paths, a contradiction with $G$ being a counterexample.

\medskip
We are left with rules~\ncf{\cfCZa} and \ncf{\cfXXg}.
If $X =$ \ncf{\cfCZa}, observe that the reduced graph $G'$ is connected. If $G' = K_3$ or $G' = K_5^-$, 
then $G$ is the same graph with one subdivided edge, so we can easily find a good coloring of $G$.
Finally, let us consider the case $X =$ \ncf{\cfXXg}. First note that if a $K_3$ appears in $G'$, it contains only $v_5$ or only $v_6$ among the vertices represented in the figures. It cannot contain both as there would be a path from $v_3$ to $v_4$ in $G$, a contradiction with the definition of \ncf{\cfXXg}. A $K_5^-$ in $G'$ contains only $v_5$, only $v_6$, or only $v_1,v_2,v_3,v_4$ among the vertices represented.
In the latter case, the $K_5^-$ must be on vertices $v_1,v_2,v_3,v_4$ and another vertex $v_7$. If there is a non-edge between $v_1$ and $v_2$, then $v_7$ is adjacent to $v_1,v_2,v_3,v_4$. Then the path $v_3,v_7,v_4$ contradicts the definition of \ncf{\cfXXg}. Thus, there is an edge $v_1v_2$, and we may assume w.l.o.g. that $v_7$ is adjacent to $v_1,v_2,v_3$.

If a $K_3$ (resp. $K_5^-$) component appears in $G'$ and contains (only) $v_5$ or $v_6$, then the red path $P$ or the blue path $Q$ in $f_X^c(G,pc)$ can easily be extended to color both the path and the $K_3$ (resp. $K_5^-$) with $2$ colors (resp. $3$ colors) (which is equivalent to applying Lemma~\ref{L21}, p.~\pageref{L21}).
Now assume that $v_1,v_2,v_3,v_4,v_7$ form a $K_5^-$ in $G'$. We color it in $G'$ with the path $R = (v_7,v_3,v_1,v_2,v_4)$ and a cycle $Q = (v_7,v_1,v_4,v_3,v_2)$. We thus place ourselves in case 3 of \ncf{\cfXXg}, treated with rule {\nameXXgThree}. When applying the recoloring function, we change the color of the edge $u_2v_4$ from $Q$ (blue) to $R$ (green), to turn $Q$ into a path. We thus built a good coloring of $G$ in all cases, a contradiction.
\end{proof}


\ifthenelse{\equal{\isThesis}{true}}
{\section{The rules cover all cases}}
{\subsection{The rules cover all cases}}

\begin{lem}\label{lem:conf_CI}
If a graph contains a configuration $(C_I)$ and is different from $K_3$, then it contains at least one configuration 
among \ncf{\cfCZa}, \ncf{\cfXXa}, $\dots$, \ncf{\cfXXup},
\ncf{\cfCNP}~$\oplus$~\ncf{\cfCNP}, \ncf{\cfCVPlus}~$\oplus$~\ncf{\cfCNP}, in which case the conditions of Lemma~\ref{lem:valid2} (p.~\pageref{lem:valid2}) are satisfied.
\end{lem}
\begin{proof}
We make a case analysis, described by the tree of Figure~\ref{fig:treeCI}.
For each inner node, we define the configurations of its children, and we show that if a graph contains the configuration described by the inner node, then it contains (at least) one of its children configurations. The configurations drawn inside a frame are the leaves of the tree, and the others are inner nodes.

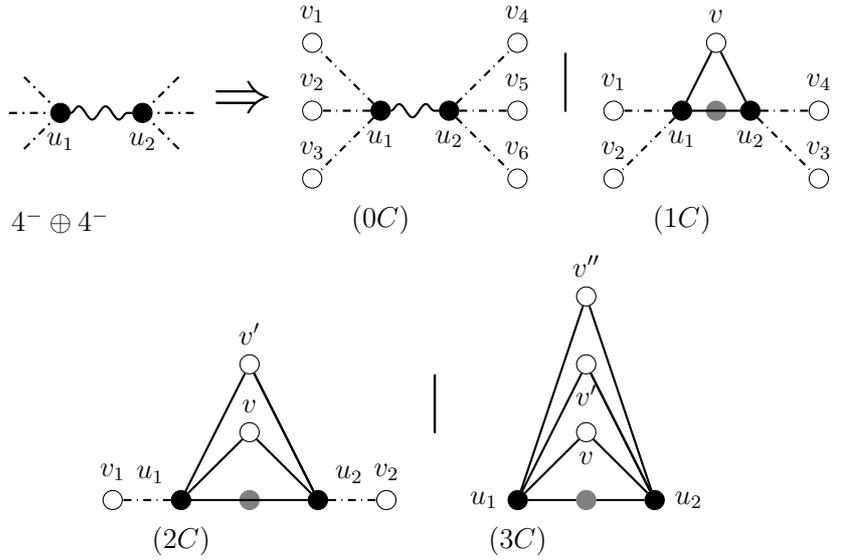
\renewcommand{\MyScale}{0.5}

	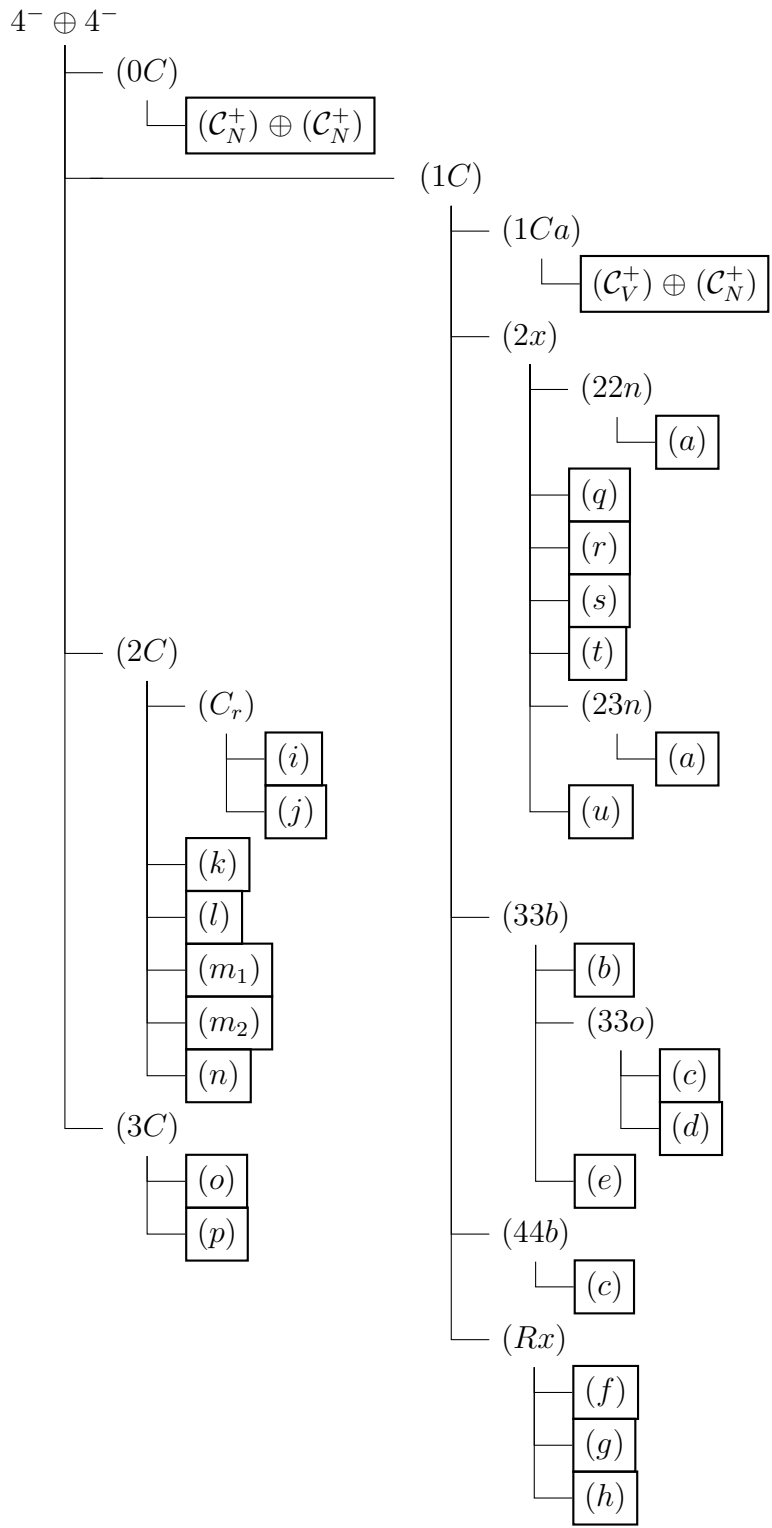
\begin{figure}[htb]
		\tikzstyle{every node}=[thick,anchor=west,align=left]
		\begin{tikzpicture}
		  [grow via three points={one child at (0.5,-0.7) and
		  two children at (0.5,-0.7) and (0.5,-1.4)},
		  edge from parent path={(\tikzparentnode.south) |-(\tikzchildnode.west)}]
		  \node [] {\Ncf{\cfCFmFm}}
			child { node [] {{\Ncf{\cfCZC}}}
				child { node [draw] {\NcfP{\cfCNP}{\cfCNP}}
				}
			  }
			child [missing] {}
			child { node  [] {{\hspace{4cm}\Ncf{\cfCOC}\hspace{4cm}}}
				child { node [] {\Ncf{\cfCOCa}}
				  child { node [draw] {\NcfP{\cfCVPlus}{\cfCNP}} }
				}
				child [missing] {}
				child { node [] {\Ncf{\cfTwoAny}}
					child { node [] {\Ncf{\cfTwoTwoN}}
						child { node [draw] {\Ncf{\cfCZa}}}
					}
					child [missing] {}
					child { node [draw] {\Ncf{\cfXXp}}}
					child { node [draw] {\Ncf{\cfXXr}}}
					child { node [draw] {\Ncf{\cfXXs}}}
					child { node [draw] {\Ncf{\cfXXt}}}
					child { node [] {\Ncf{\cfXXu}} 
						child { node [draw] {\Ncf{\cfCZa}}}
						}
					child [missing] {}
					child { node [draw] {\Ncf{\cfXXup}}} 
				}
				child [missing] {}
				child [missing] {}
				child [missing] {}
				child [missing] {}
				child [missing] {}
				child [missing] {}
				child [missing] {}
				child [missing] {}
				child [missing] {}
				child [missing] {}
				child { node [] {\Ncf{\cfTTb}}
				  child { node [draw] {\Ncf{\cfXXa}} }
                    child { node [] {\Ncf{\cfXXb}}
                      child { node [draw] {\Ncf{\cfXXd}}
                      }
                      child { node [draw] {\Ncf{\cfXXbb}}
                      }
                    }
                    child [missing] {}
                    child [missing] {}
				  child { node [draw] {\Ncf{\cfXXc}} }
				}
			  child [missing] {}
                child [missing] {}
                child [missing] {}
			  child [missing] {}
			  child [missing] {}
			  child { node [] {\Ncf{\cfFFb}}
				child { node [draw] {\Ncf{\cfXXd}}
				}
			  }
			  child [missing] {}
			  child { node [] {\Ncf{\cfRx}}
				child { node [draw] {\Ncf{\cfXXe}} }
				child { node [draw] {\Ncf{\cfXXf}}
				}
				child { node [draw] {\Ncf{\cfXXg}}
				}
			  }
			}
			child [missing] {}
			child [missing] {}
			child [missing] {}
			child [missing] {}
			child [missing] {}
			child [missing] {}
			child [missing] {}
			child [missing] {}
			child { node [] {{\Ncf{\cfCTC}}}
				child { node [] {\Ncf{\cfCrr}}
				  child { node [draw] {\Ncf{\cfXXh}}
				  }
				  child { node [draw] {\Ncf{\cfXXi}}
				  }
				}
				child [missing] {}
				child [missing] {}
					child { node [draw] {\Ncf{\cfXXj}}
				}
				child { node [draw] {\Ncf{\cfXXk}}
				}
				child { node [draw] {\Ncf{\cfXXl}}
				}
				child { node [draw] {\Ncf{\cfXXlp}}
				}
				child { node [draw] {\Ncf{\cfXXm}}
				}
			}
			child [missing] {}
			child [missing] {}
			child [missing] {}
			child [missing] {}
			child [missing] {}
			child [missing] {}
			child [missing] {}
			child [missing] {}
			child { node [] {{\Ncf{\cfCFC}}}
			  child { node [draw] {\Ncf{\cfXXn}}
			  }
			  child { node [draw] {\Ncf{\cfXXo}}
			  }
		};
		\draw[] ($(1.2,-2.100)$) -- ($(5.2,-2.100)$);
		\end{tikzpicture}
		\caption{Tree of implications between configurations.}
    \label{fig:treeCI}
		\end{figure}
    \clearpage

\renewcommand\MyScale{0.9}
\renewcommand\OrdonneeFleche{1.5}

\begin{center}
\hspace{1cm}
\tkcf{\cfCFmFm} %
			  \TexteCentre{$\Rightarrow$}{\OrdonneeFleche} %
			  \tkcf{\cfCZC} %
			  \TexteCentre{$\vert$}{\OrdonneeFleche} %
			  \tkcf{\cfCOC} %
			  \TexteCentre{$\vert$}{\OrdonneeFleche} %
			  \\
			  \tkcf{\cfCTC} %
			  \TexteCentre{$\vert$}{\OrdonneeFleche} %
			  \tkcf{\cfCFC}%
\end{center}

\begin{itemize}
		\item \ncf{\cfCFmFm}:
		In the following configurations, let $u_1,u_2$ be two special vertices that form a configuration $(C_I)$ and let $P$ be a shortest path between $u_1$ and $u_2$.
		\begin{itemize}
		\item \ncf{\cfCZC} ($0$ common remaining neighbor): the vertices $u_1$ and $u_2$ have degree $1$, $2$, $3$ or $4$; the vertices $u_1$ and $u_2$ have no common remaining neighbor.
		\item \ncf{\cfCOC} ($1$ common remaining neighbor): the special vertices have degree $2$, $3$ or $4$; the path $P$ has length $1$ or $2$; the special vertices have exactly one common remaining neighbor $v$.
		\item \ncf{\cfCTC} ($2$ common remaining neighbors): the special vertices have degree $3$ or $4$; the path $P$ has length $1$ or $2$; the special vertices have exactly two common remaining neighbors $v$ and $v'$.
		\item \ncf{\cfCFC} ($3$ common remaining neighbors): the special vertices have degree $4$; the path $P$ has length $1$ or $2$; the special vertices have exactly three common remaining neighbors $v$, $v'$ and $v''$.
		\end{itemize}
		Depending on the number of common remaining neighbors between $u_1$ and $u_2$, we have one of the configurations \ncf{\cfCZC}, \ncf{\cfCOC}, \ncf{\cfCTC} or \ncf{\cfCFC}.

		Since $P$ is a shortest path between $u_1$ and $u_2$, when $u_1$ and $u_2$ have a common neighbor, $P$ has length at most $2$, hence all cases are covered.\\
		
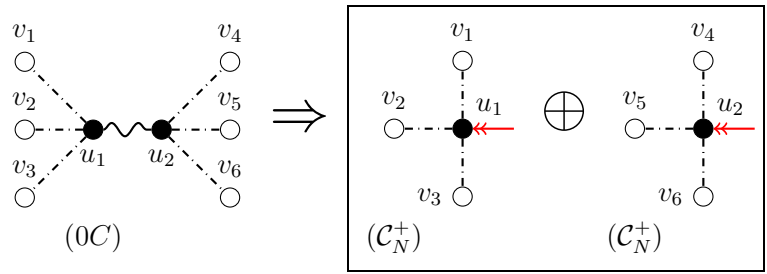
		\begin{center}
		\tkcf{\cfCZC} %
			   \TexteCentre{$\Rightarrow$}{\OrdonneeFleche} %
			   \ensuremath{\boxed{\relabeledCNP{$u_1$}{$v_1$}{$v_2$}{$v_3$}{} \ \ \TexteCentre{$\oplus$}{\OrdonneeFleche} \relabeledCNP{$u_2$}{$v_4$}{$v_5$}{$v_6$}{}}}
		\end{center}

		 \item \ncf{\cfCZC}:
		 The neighborhoods of both special vertices match the elementary partial configuration \ncf{\cfCNP}. Since $u_1$ and $u_2$ have no common remaining neighbors, we have a path composite configuration \ncf{\cfCNP}~$\oplus$~\ncf{\cfCNP} that satisfies the conditions of Lemma~\ref{lem:valid2} (p.~\pageref{lem:valid2}).\\

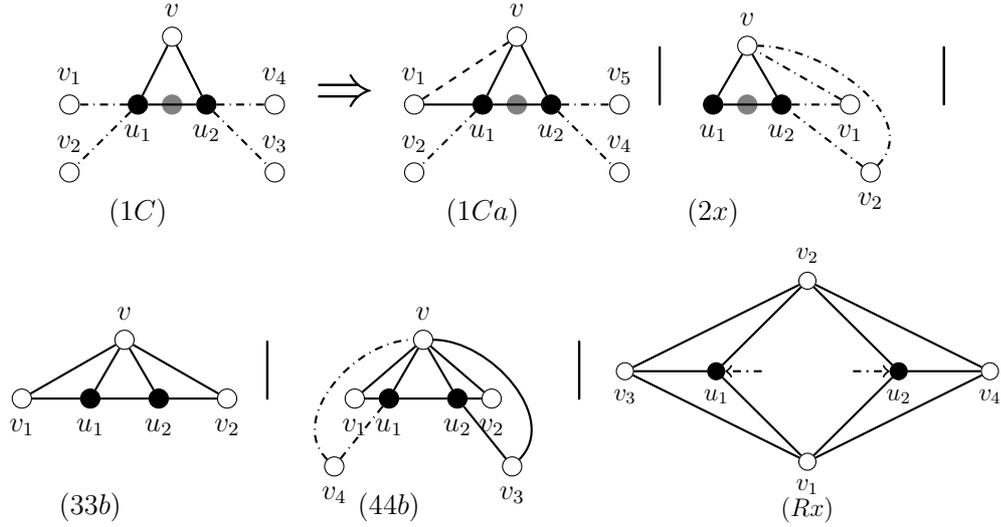
	 \begin{center}
	 \tkcf{\cfCOC} %
			  \TexteCentre{$\Rightarrow$}{\OrdonneeFleche} %
			  \tkcf{\cfCOCa}%
			  \TexteCentre{$\vert$}{\OrdonneeFleche}%
			  \tkcf{\cfTwoAny}%
			  \TexteCentre{$\vert$}{\OrdonneeFleche} %
			  \\
			  \tkcf{\cfTTb}%
			  \TexteCentre{$\vert$}{\OrdonneeFleche}%
			  \tkcf{\cfFFb}%
			  \TexteCentre{$\vert$}{\OrdonneeFleche} %
			  \renewcommand\MyScale{0.8}%
			  \tkcf{\cfRx}%
			  \renewcommand\MyScale{0.9}%
		\end{center}

		\item \ncf{\cfCOC}:

		\begin{itemize}
			\item \ncf{\cfCOCa}: The special vertex $u_1$ has degree $3$ or $4$ and the special vertex $u_2$ has degree $2$, $3$ or $4$; the special vertices have one common remaining neighbor $v$ and $u_1$ and $u_2$ are linked by a path $P$ of length $1$ or $2$; moreover $u_1$ has a remaining neighbor $v_1$ that is not adjacent to $v$.
			\item \ncf{\cfTwoAny}: The special vertex $u_1$ has degree $2$ and the special vertex $u_2$ has degree $2$, $3$ or $4$; the special vertices have one common remaining neighbor $v$ and $u_1$ and $u_2$ are linked by a path $P$ of length $1$ or $2$; if the degree of $u_2$ is greater than $2$, let $v_1$ and possibly $v_2$ be its other remaining neighbors and these neighbors are adjacent to $v$.
			\item \ncf{\cfRx}: The special vertices $u_1$ and $u_2$ have degree $3$ or $4$; $u_1$ and $u_2$ have $2$ common neighbors $v_1$ and $v_2$; each special vertex has a remaining neighbor ($v_3$ and $v_4$ respectively) that is adjacent to $v_1$ and $v_2$.
			\item \ncf{\cfTTb}: The special vertices $u_1$ and $u_2$ have degree $3$; $u_1$ is adjacent to $u_2$ and they have a common neighbor $v$.
			\item \ncf{\cfFFb}: The special vertex $u_1$ has degree $3$ or $4$ and the special vertex $u_2$ has degree $4$; $u_1$ is adjacent to $u_2$ and they have a common neighbor $v$; the remaining neighbors of $u_1$ are adjacent to $v$.
		\end{itemize}

 If there is a remaining neighbor of a special vertex that is not adjacent with the common remaining neighbor, then we have the configuration \ncf{\cfCOCa}. Otherwise, every remaining neighbor is adjacent to the common neighbor $v$.
 If at least one of the special vertices has degree $2$, then we have the configuration \ncf{\cfTwoAny}. Otherwise, every special vertex has degree $3$ or $4$.
 If the distance between the special vertices is $2$, then $u_1$ and $u_2$ have $2$ common neighbors (let us call them $v$ and $v'$). In this case, all the other neighbors are adjacent to both of them (otherwise we are back in case \ncf{\cfCOCa}, possibly changing the role of $v$ and $v'$). In this case, we have the configuration \ncf{\cfRx}. Otherwise, the special vertices are adjacent. If both vertices have degree $2$, we have the configuration \ncf{\cfTTb}, otherwise we have the configuration \ncf{\cfFFb}.\\


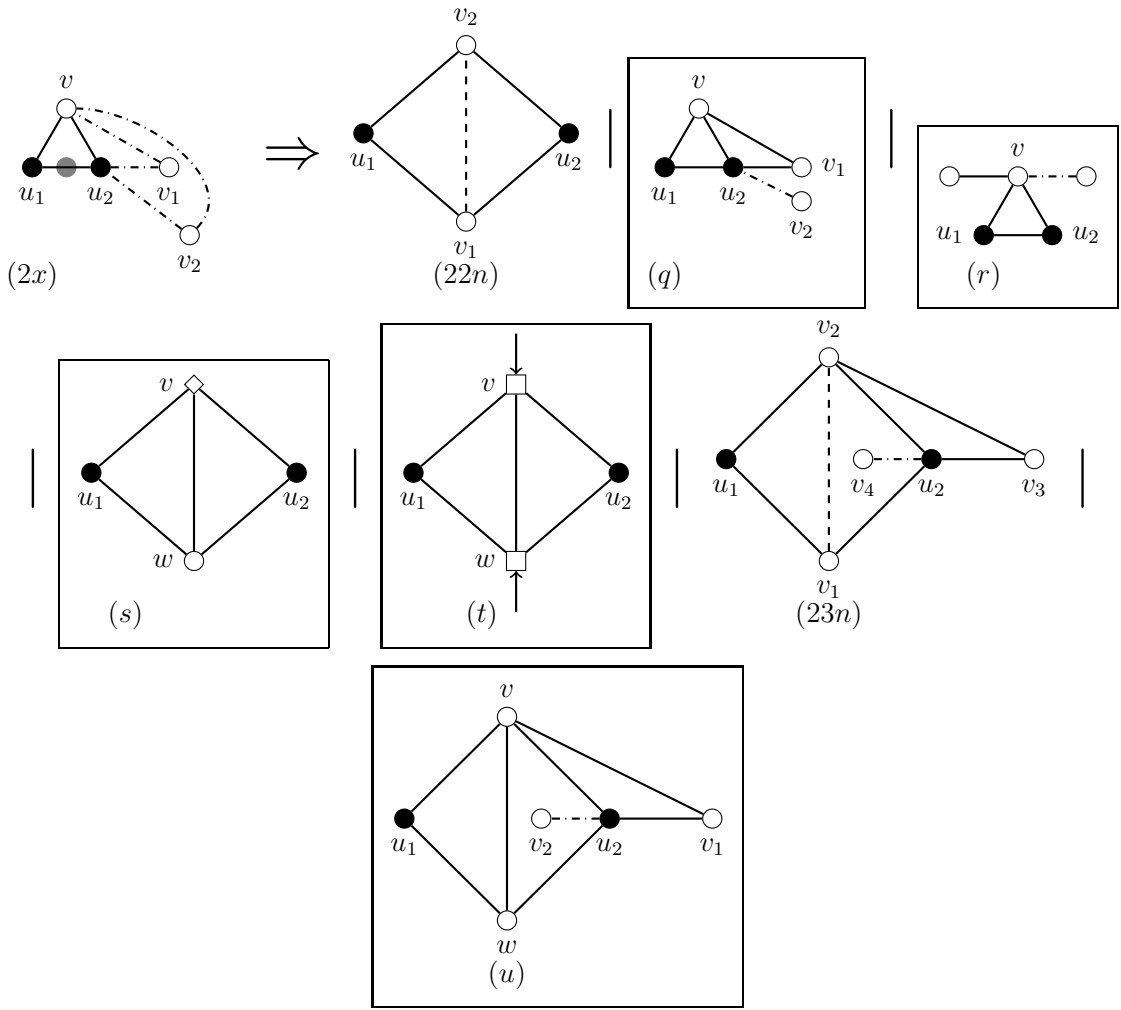
		\begin{center}
		\tkcf{\cfTwoAny} %
			  \TexteCentre{$\Rightarrow$}{\OrdonneeFleche} %
			  \tkcf{\cfTwoTwoN}%
			  \TexteCentre{$\vert$}{\OrdonneeFleche}%
			  \boxed{\relabeledXXp{$u_1$}{$u_2$}{$v$}{$v_1$}{$v_2$}} %
			  \TexteCentre{$\vert$}{\OrdonneeFleche} %
			  \boxed{\relabeledXXr{$u_1$}{$u_2$}{$v$}{}{}} %
			  \\
			  \TexteCentre{$\vert$}{\OrdonneeFleche} %
			  \boxed{\relabeledXXs{$u_1$}{$u_2$}{$w$}{$v$}} %
			  \TexteCentre{$\vert$}{\OrdonneeFleche} %
			  \boxed{\relabeledXXt{$u_1$}{$u_2$}{$w$}{$v$}} %
			  \TexteCentre{$\vert$}{\OrdonneeFleche} %
			  \tkcf{\cfXXu} %
			  \TexteCentre{$\vert$}{\OrdonneeFleche} %
			  \\
			  \vspace{0.2cm}
			  \boxed{\relabeledXXup{$u_1$}{$u_2$}{$w$}{$v$}{$v_1$}{$v_2$}}%
		\end{center}		 

			\item \ncf{\cfTwoAny}:
    		\begin{itemize}
    			\item \ncf{\cfTwoTwoN}: The two special vertices $u_1$ and $u_2$ have degree $2$; they have two common neighbors $v_1$ and $v_2$; moreover $v_1$ and $v_2$ are not adjacent.
    			\item \ncf{\cfXXu}: The special vertex $u_1$ has degree $2$ and $u_2$ has degree $3$ or $4$; $u_1$ and $u_2$ have two common neighbors $v_1$ and $v_2$; $v_1$ and $v_2$ are not adjacent; $u_2$ has a neighbor $v_3$ that is adjacent to $v_2$.
			\end{itemize}
			\ \\
			In this case, $u_1$ has degree $2$, and $u_1,u_2$ are either adjacent or share a neighbor $w$ in addition to their common remaining neighbor $v$. First let us assume that $u_2$ has degree $2$. If $u_1,u_2$ are adjacent, and since the graph is different from $K_3$, then $v$ has degree at least $3$ and this is configuration \ncf{\cfXXr}. Otherwise, $u_1,u_2$ are non-adjacent. If $v,w$ are non-adjacent, this is configuration \ncf{\cfTwoTwoN}, and otherwise this is configuration \ncf{\cfXXs} if one among $v,w$ has an odd degree, and configuration \ncf{\cfXXt} if both have an even degree.

			Now assume $u_2$ has degree at least $3$. If $u_1,u_2$ are adjacent this is configuration \ncf{\cfXXp}, and otherwise this is configuration 
			%
			\ncf{\cfXXu} if $v,w$ are non-adjacent, or \ncf{\cfXXup} otherwise.\\


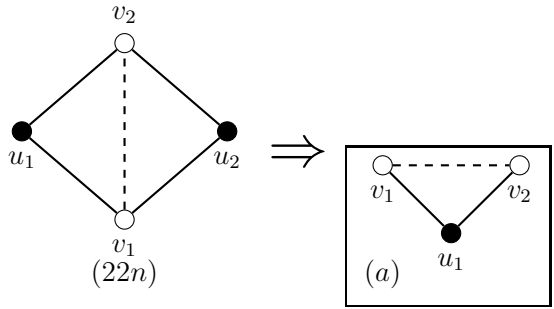
		\begin{center}
		\tkcf{\cfTwoTwoN} %
			  \TexteCentre{$\Rightarrow$}{\OrdonneeFleche} %
			  \boxed{\tkcf{\cfCZa}}%
		\end{center}		 

		\item \ncf{\cfTwoTwoN}: In this case have the configuration \ncf{\cfCZa} around special vertex $u_1$.\\

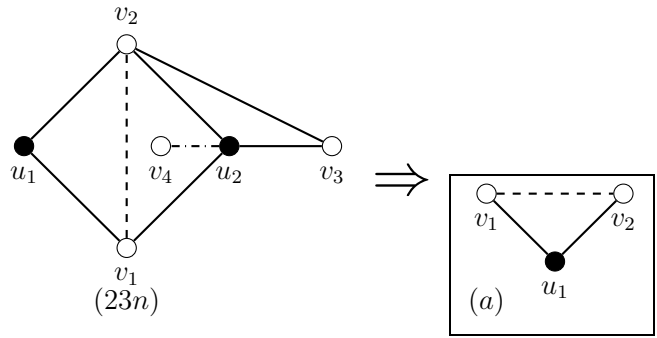
		\begin{center}
		\tkcf{\cfXXu} 
			  \TexteCentre{$\Rightarrow$}{\OrdonneeFleche} %
			  \boxed{\tkcf{\cfCZa}}%
		\end{center}

			\item \ncf{\cfXXu}: In this case have the configuration \ncf{\cfCZa} around special vertex $u_1$.\\

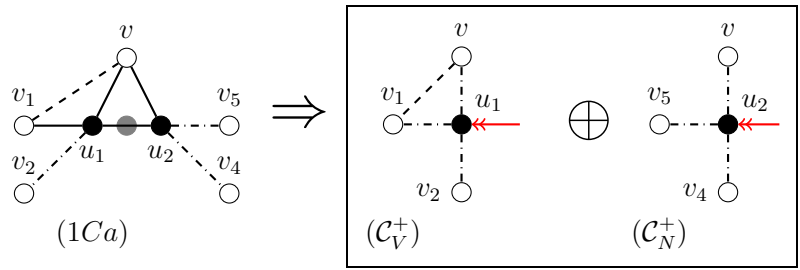
		\begin{center}
		\tkcf{\cfCOCa} %
			  \TexteCentre{$\Rightarrow$}{\OrdonneeFleche} %
			  \boxed{ 
			  \relabeledCVPlus{$u_1$}{$v$}{$v_1$}{$v_2$}{} \ \ \TexteCentre{$\oplus$}{\OrdonneeFleche-0.2}\relabeledCNP{$u_2$}{$v$}{$v_5$}{$v_4$}{} \ 
			  }%
		\end{center}
		
		\item \ncf{\cfCOCa}:
		The neighborhood of the first special vertex matches the elementary partial configuration \ncf{\cfCVPlus} and the neighborhood of the second one matches the elementary partial configuration \ncf{\cfCNP}. Moreover both special vertices have only one remaining neighbor in common, $v$, which is non-adjacent to $v_1$. Hence the conditions of Lemma~\ref{lem:valid2} (p.~\pageref{lem:valid2}) are satisfied.\\

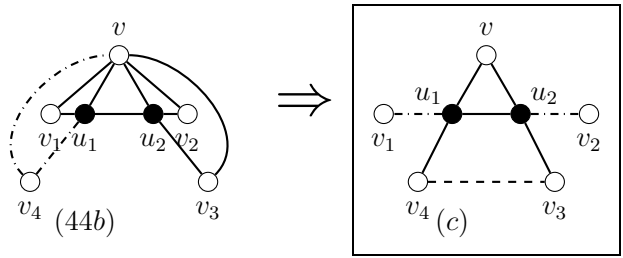
		\begin{center} \tkcf{\cfFFb} %
			  \TexteCentre{$\Rightarrow$}{\OrdonneeFleche} %
			  \boxed{\relabeledXXd{$u_1$}{$u_2$}{$v$}{$v_4$}{$v_3$}{$v_1$}{$v_2$}}%
		\end{center}

		\item \ncf{\cfFFb}:
		Since the graph is planar, at least one of the remaining neighbors of $u_2$ is not adjacent to the remaining neighbor of $u_1$. Hence we have the configuration \ncf{\cfXXd}.\\
		
		\vfill
		\pagebreak
		
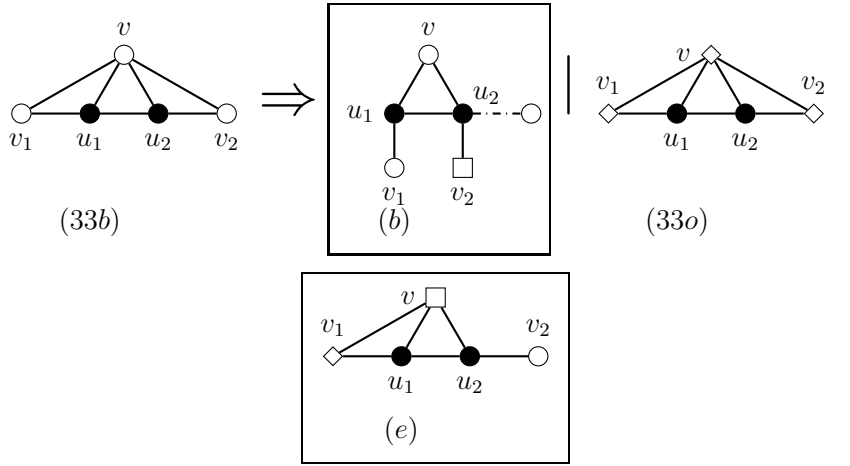
		\begin{center}
		\tkcf{\cfTTb}%
			  \TexteCentre{$\Rightarrow$}{\OrdonneeFleche}%
			  \boxed{\relabeledXXa{$u_1$}{$u_2$}{$v$}{$v_1$}{$v_2$}{}}%
			  \TexteCentre{$\vert$}{\OrdonneeFleche}%
			  \tkcf{\cfXXb}%
			  \TexteCentre{$\vert$}{\OrdonneeFleche} %
			  \\
			  \vspace{0.2cm}
			  \boxed{\tkcf{\cfXXc}}%
		\end{center}
		\item \ncf{\cfTTb}:
		\begin{itemize}
			\item \ncf{\cfXXb}: The special vertices $u_1$ and $u_2$ have degree 3; they are adjacent and they have a common neighbor $v$ of odd degree; each special vertex has another neighbor of odd degree ($v_1$ and $v_2$ respectively) that is adjacent to $v$.
		\end{itemize}

		If there is a special vertex with an even non-common remaining neighbor, we are in the configuration \ncf{\cfXXa}. Otherwise, depending of the parity of the common neighbor we are in configuration \ncf{\cfXXb} (odd) or in configuration \ncf{\cfXXc} (even).\\

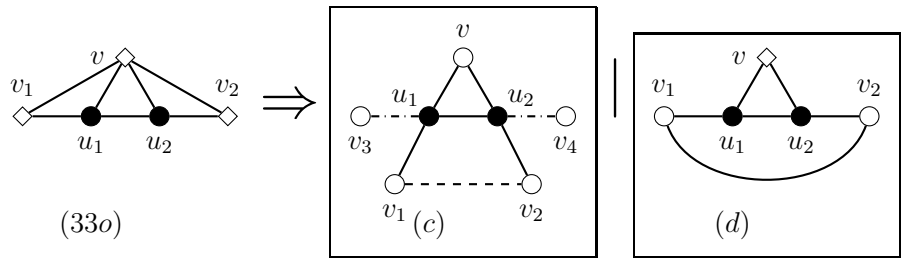
        \begin{center}
        \tkcf{\cfXXb}%
			  \TexteCentre{$\Rightarrow$}{\OrdonneeFleche}%
			  \boxed{\tkcf{\cfXXd}}
			  \TexteCentre{$\vert$}{\OrdonneeFleche}%
			  \boxed{\tkcf{\cfXXbb}}%
		\end{center}         

        \item \ncf{\cfXXb}: We split the case depending on whether there is an edge between $v_1$ and $v_2$, the associated rules are \ncf{\cfXXd} 
        if there is not and \ncf{\cfXXbb} otherwise.\\


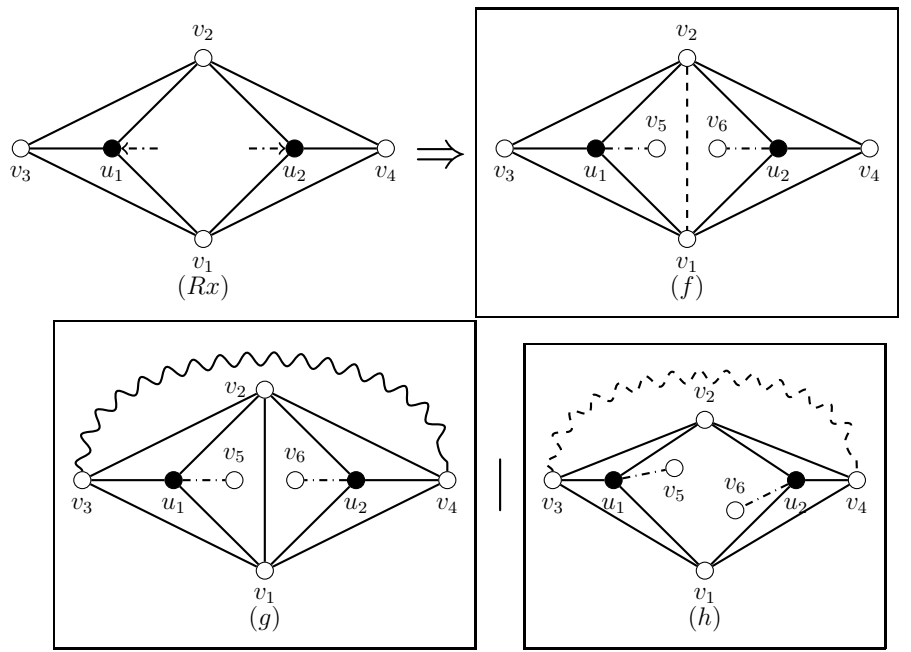
		\begin{center}
		\renewcommand\MyScale{0.8}%
			  \tkcf{\cfRx}%
			  \TexteCentre{$\Rightarrow$}{\OrdonneeFleche+0.4}%
			  \boxed{\tkcf{\cfXXe}} %
			  \TexteCentre{$\vert$}{\OrdonneeFleche+0.2} %
			  \\
			  \boxed{\tkcf{\cfXXf}} %
			  \TexteCentre{$\vert$}{\OrdonneeFleche+0.2} %
			  \boxed{\tkcf{\cfXXg}}%
			  \renewcommand\MyScale{0.9}%
		\end{center}		 
		
		\item \ncf{\cfRx}:
		If the two common neighbors are non-adjacent we are in configuration \ncf{\cfXXe}, Otherwise, if the two common neighbors form a separating pair, then we are in configuration \ncf{\cfXXg}, otherwise we are in configuration \ncf{\cfXXf}.\\


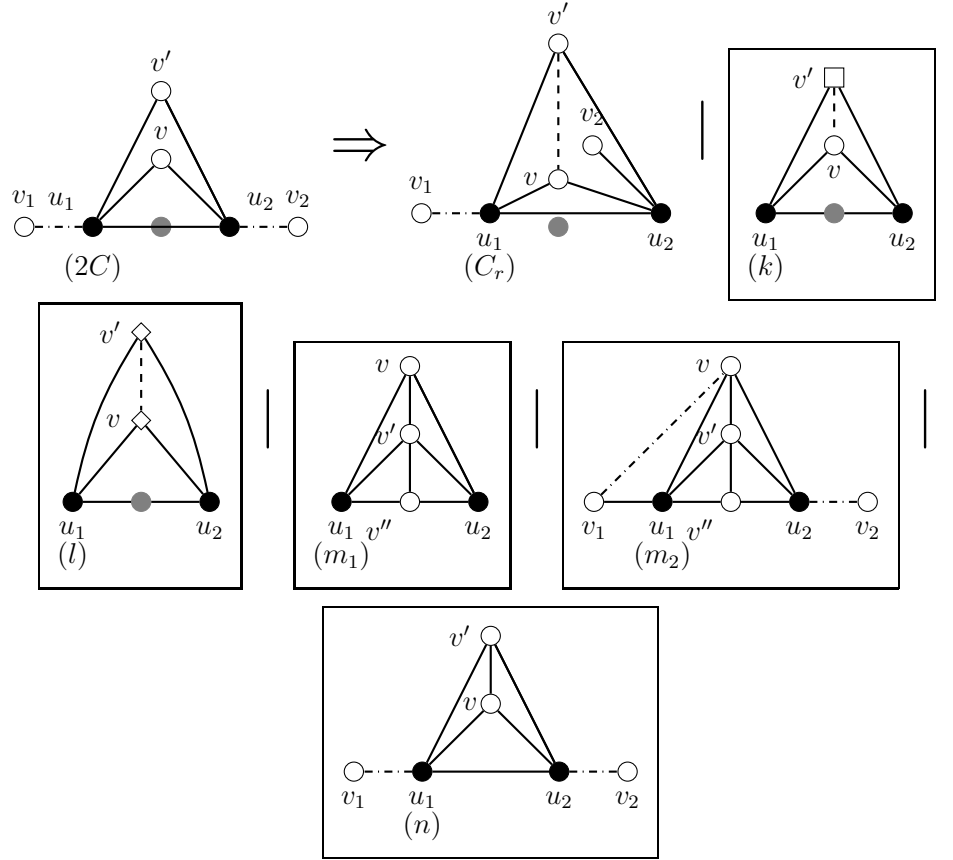
		\begin{center}
		\tkcf{\cfCTC}%
			  \TexteCentre{$\Rightarrow$}{\OrdonneeFleche}%
			  \tkcf{\cfCrr}%
			  \TexteCentre{$\vert$}{\OrdonneeFleche} %
			  \boxed{\tkcf{\cfXXj}} %
			  \TexteCentre{$\vert$}{\OrdonneeFleche} %
			  \\
			  \boxed{\tkcf{\cfXXk}} %
			  \TexteCentre{$\vert$}{\OrdonneeFleche} %
			  \boxed{\tkcf{\cfXXl}} %
			  \TexteCentre{$\vert$}{\OrdonneeFleche} %
			  \boxed{\tkcf{\cfXXlp}} %
			  \TexteCentre{$\vert$}{\OrdonneeFleche} %
			  \\
			  \vspace{0.2cm}
			  \boxed{\tkcf{\cfXXm}}%
		\end{center}		 
		
		\item \ncf{\cfCTC}:
		\begin{itemize}
			\item  \ncf{\cfCrr}: the special vertex $u_1$ has degree $3$ or $4$ and the special vertex $u_2$ has degree $4$; $u_1$ and $u_2$ are linked by a path $P$ of length $1$ or $2$; they have two common remaining neighbors $v$ and $v'$; moreover, $v$ and $v'$ are not adjacent.
		\end{itemize}

		If there are two common neighbors that are non-adjacent, then we are in configuration \ncf{\cfCrr} if at least one special vertex has degree $4$, and otherwise in configuration \ncf{\cfXXj} or \ncf{\cfXXk} depending on whether one of these neighbors has an even degree.
		Otherwise, all common neighbors are pairwise adjacent. If the special vertices are adjacent, then we are in configuration \ncf{\cfXXm}. Otherwise, the path $P$ has length $2$ and a middle-vertex $v''$. If both special vertices have degree $3$, then we are in configuration \ncf{\cfXXl}, and if at least one has degree $4$, we are in configuration \ncf{\cfXXlp}.\\
		\textbf{Remark}: Note that if the graph is a $K_5^-$, then it is a configuration \ncf{\cfXXl}, but in this case the associated rule defines a coloring of the graph with $3$ colors. This case does not occur if the graph is an MCE.

		
		\begin{center}
		\renewcommand\MyScale{0.81}%
			  \tkcf{\cfCrr}%
			  \TexteCentre{$\Rightarrow$}{\OrdonneeFleche}%
			  \boxed{\tkcf{\cfXXh}}%
			  \TexteCentre{$\vert$}{\OrdonneeFleche}%
			  \boxed{\tkcf{\cfXXi}}%
			  \renewcommand\MyScale{0.9}%

		\end{center}
		
		\item \ncf{\cfCrr}:	 
		If there is a non-common remaining neighbor that is not adjacent to a common remaining neighbor we are in configuration \ncf{\cfXXh}, otherwise at least one remaining neighbor is adjacent to both common remaining neighbors and we are in configuration \ncf{\cfXXi}.\\

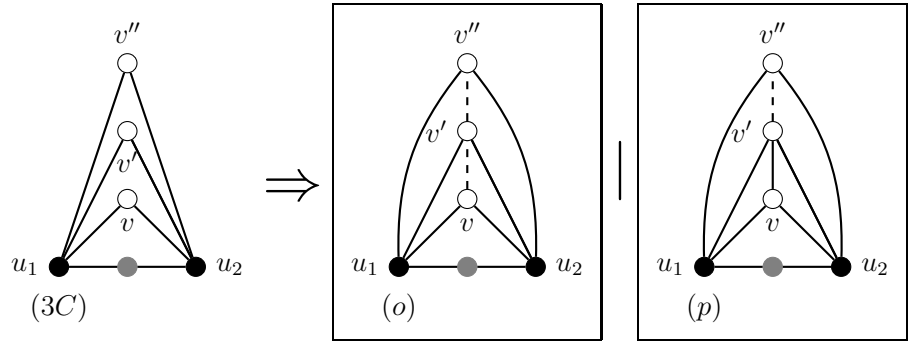
		\begin{center}
		\tkcf{\cfCFC}%
			  \TexteCentre{$\Rightarrow$}{\OrdonneeFleche}%
			  \boxed{\tkcf{\cfXXn}}%
			  \TexteCentre{$\vert$}{\OrdonneeFleche}%
			  \boxed{\tkcf{\cfXXo}}%
		\end{center}		 

		\item \ncf{\cfCFC}:
		Since the graph is planar, there are at least two non-adjacent common remaining neighbors. If none of them are adjacent we are in configuration \ncf{\cfXXn}, otherwise we are in configuration \ncf{\cfXXo}.
\end{itemize}

\end{proof}

The proof of the main lemma of this 
{\chapsec}
is now straightforward.

\begin{proof}[Proof of Lemma~\ref{lem:ci} (p.~\pageref{lem:ci})]
Let $G$ be an MCE, and assume it contains a configuration $(C_I)$. By Lemma~\ref{lem:conf_CI}, $G$ contains a configuration among \ncf{\cfCZa}, \ncf{\cfXXa}, $\dots$, \ncf{\cfXXup}, or a path composite configuration \ncf{\cfCVPlus} $\oplus$ \ncf{\cfCNP} or \ncf{\cfCNP} $\oplus$ \ncf{\cfCNP}.
Lemma~\ref{lem:safety_ci} (p.~\pageref{lem:safety_ci}) provides a contradiction.
\end{proof}

\vfill
\pagebreak

\section{Configuration $(C_{II})$}
\label{sec:cii}


In this {\chapsec}, we prove the following lemma, which constitutes the second property of Lemma~\ref{lem:mce}.

\begin{lem}\label{lem:cii}
An MCE does not contain a configuration $(C_{II})$.
\end{lem}

As a reminder, a planar graph $G$ has a configuration $(C_{II})$ if it is almost $4$-connected w.r.t. a $4$-family $U$, i.e. a set of four vertices of degree $5$; we say that $G$ has a $(C_{II})$ configuration w.r.t. $U$.
The proof of Lemma~\ref{lem:cii} is given at the end of this 
{\chapsec},
and uses a method similar to the proof of Lemma~\ref{lem:ci} (p.~\pageref{lem:ci}) in the previous 
{\chapsec}.



In this 
{\chapsec},
we take care of our four special vertices by generalizing the tools of the previous 
{\chapsec}.
Instead of considering a shortest path like in $(C_I)$ rules, we use a \emph{subdivision}, either a $K_4$ or a $C_{4+}$-subdivision (see Figure~\ref{fig:subdK}), and again remove it in the reduction, then color it with our extra colors. These structures have the same convenient properties as the shortest path of the $(C_I)$ rules: they can be colored with $2$ extra colors, and there is an end of an extra color on each of the four special vertices, which is again helpful to take care of all the missing edges. 
We generalize the concept of elementary partial rule and consider \emph{patterns} that recolor the neighborhoods of one or two special vertices at once. Similarly to the distant special vertices in the $(C_I)$ rules, when the remaining neighbors of the special vertices are disjoint, we combine four ``normal'' patterns (called {\CN} like in 
{\chCI}
)
to form a complete reduction rule. Figure~\ref{fig:cii_CN} shows four special vertices $u_1,u_2,u_3,u_4$ linked by a $K_4$-subdivision $S$, with disjoint remaining neighbors and thus treated with the pattern {\CN}.

%

\renewcommand{\OrdonneeFleche}{2}
\renewcommand{\MyScale}{0.8}

\begin{figure}[h]
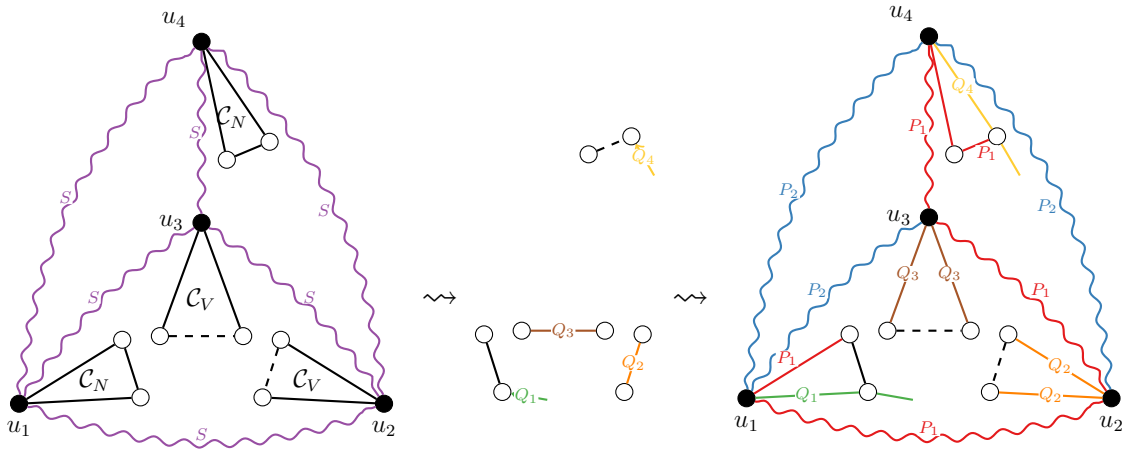

\RuleExTwo
\caption{A $(C_{II})$ reducible configuration featuring {\CN} patterns}
\label{fig:cii_CN}
\end{figure}

Several problems can occur however, which we classify in two types, \emph{distant} and \emph{close} problems. A special vertex forms a distant problem when its remaining neighbors are adjacent and touch the subdivision. If left untreated, a distant problem may cause a {\CN} pattern to create a cycle in the final coloring. The distant problems are eliminated in two ways: either by modifying the subdivision to assign new remaining neighbors to the problematic special vertices, or by \textbf{inactivating} them by finding a $2$-coloring of the subdivision that is compatible with the {\CN} patterns treating the remaining distant problems. Figure~\ref{fig:cii_distclose} features a $(C_{II})$ configuration with a $K_4$-subdivision $S$, where the special vertex $u_3$ causes a distant problem on a $(u_2,u_4)$-path of $S$. This problem is inactivated by a carefully chosen $2$-coloring of $S$.

The close problems occur when two special vertices share some remaining neighbors, with these remaining neighbors possibly touching the subdivision as well. These problems are treated in a custom manner, by redirecting the subdivision to assign new remaining neighbors to the special vertices and finding a compatible set of patterns to treat all four special vertices. In Figure~\ref{fig:cii_distclose}, $u_1$ and $u_2$ initially form a close problem, which is eliminated by a redirection of the subdivision. The special vertices $u_1, u_2$ are then both treated with the {\CV} pattern.

\renewcommand{\MyScale}{0.885}

\cfExThree


More precisely, we first treat the cases with at least $3$ distant problems (and no close ones), in the \textit{distant lemma} (Lemma~\ref{lem:dist1}, p.~\pageref{lem:dist1}), then the cases with at most $2$ distant problems and no close ones in the \textit{semi-distant lemma} (Lemma~\ref{lem:dist2}, p.~\pageref{lem:dist2}), and finally the cases with at most $2$ distant problems and some close problems in the \textit{close lemma} (Lemma~\ref{lem:close}, p.~\pageref{lem:close}).\\


The following claim is a corollary of the properties of almost $4$-connectivity, and is useful in various proofs of this 
{\chapsec}.

\begin{claim}
Let $G$ be a planar graph that is almost $4$-connected w.r.t. a $4$-family $U$. Then $G$ does not have a special vertex $u\in U$ that forms a $K_4$ with three of its neighbors.
\label{clm:inducedK4}
\end{claim}

\begin{proof}
Let $v_1,v_2,v_3\in N(u)$, such that $\{u,v_1,v_2,v_3\}$ form an induced $K_4$ in $G$. Since $d(u) = 5$, $u$ has a neighbor $v_4$ distinct from $v_1,v_2,v_3$. W.l.o.g., $v_4$ belongs to the face delimited by $\{u,v_1,v_2\}$. Then $\{u,v_1,v_2\}$ is a $3$-cut that separates $v_3$ from $v_4$, a contradiction to the definition of almost $4$-connectivity.
\end{proof}

%
%
%

\ifthenelse{\equal{\isThesis}{true}}
{\section{$K_4$, $C_{4+}$-subdivisions}}
{\subsection{$K_4$, $C_{4+}$-subdivisions}}

In 
{\chCI},
the composite rules that we considered were associated with a path between the two special vertices. In order to apply a similar method to the $(C_{II})$ configurations made up of $4$ special vertices, we consider more complex structures: $K_4$-subdivisions and $C_{4+}$-subdivisions (defined in the preliminaries {\chapsec}).

\begin{figure}[htb]
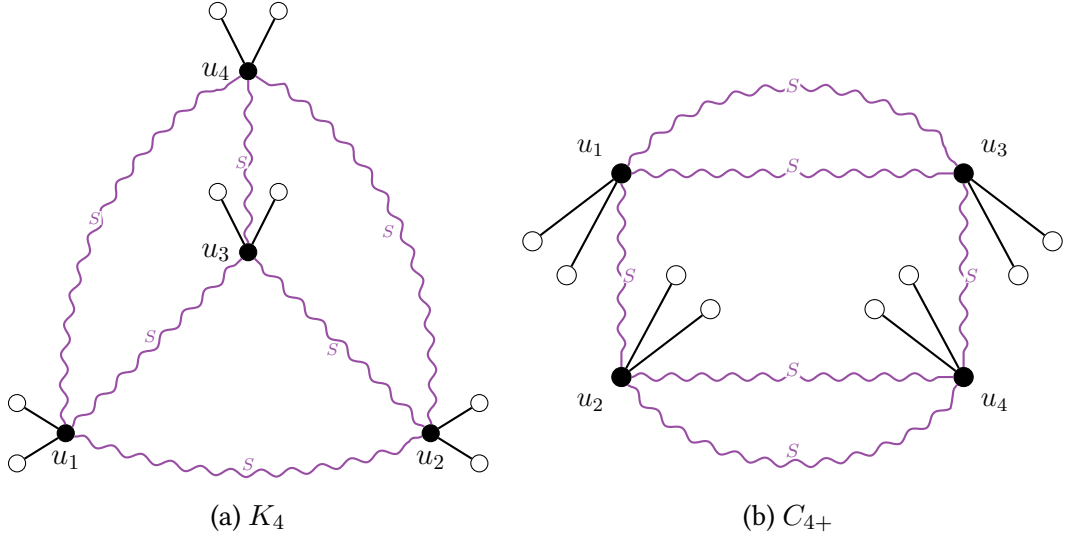

    \centering

    \begin{subfigure}[b]{7cm}
    	\renewcommand{\MyScale}{0.8}
        \centering
        \Pictcf{\subdK}
        \caption{\Ncf{\subdK}}
        \label{K4pic}
    \end{subfigure}
    \begin{subfigure}[b]{7cm}
    	\renewcommand{\MyScale}{0.9}
        \centering
        \Pictcf{\subdCFP}
        \caption{\Ncf{\subdCFP}}
        \label{C4pic}
    \end{subfigure}
    \caption{The two subdivisions considered in the proof of Lemma~\ref{lem:k4c4p}: $K_4$-subdivision and $C_{4+}$-subdivision}
    \label{fig:subdK}
\end{figure}

The general idea is to color the edges of the subdivision with $2$ new ``extra'' colors such that each special vertex is the endpoint of one path. We can then proceed as in the case of the configuration $(C_I)$ to color the neighborhood of the special vertices.
However, not only are the remaining neighbors of the special vertices still not necessarily disjoint, but they can now be touched by paths of the subdivision, which forces us to consider a large number of subcases.

A result by Yu~\cite{YU199810} gives us a $K_4$-subdivision (see Figure~\ref{K4pic}) under the condition of almost $4$-connectivity, with two exceptions. We show in Lemma~\ref{lem:k4c4p} below that we are able to extract from these two exceptions a $C_{4+}$-subdivision (see Figure~\ref{C4pic}) with some additional properties, which we call $C_{4+}^*$-subdivision.

In a subdivision $S$ rooted on a $4$-family $U$, let us say that two special vertices $u_i,u_j$ are \emph{$k$-linked}, $k\in \{0,1,2\}$, if there are $k$ $(u_i,u_j)$-paths in $S$ with no special vertex as an internal vertex. In a $K_4$-subdivision, all special vertices $u_i,u_j$ are pairwise $1$-linked, while in a $C_{4+}$-subdivision there are two pairs of $0$-linked, two pairs of $1$-linked, and two pairs of $2$-linked special vertices. Note that it is sufficient to specify one pair of $1$-linked and one pair of $2$-linked special vertices to deduce the link of all pairs.
If $u_i,u_j$ are $1$-linked, we call the $(u_i,u_j)$-path a \emph{solo path} of $S$. If $u_i,u_j$ are $2$-linked, we call the two $(u_i,u_j)$-paths \emph{parallel paths} of $S$.

Just like in the previous 
{\chapsec},
if $u_i\in U$ is a special vertex and $v$ one of its neighbors, we say that $v$ is a \emph{remaining neighbor} of $u_i$ if the edge $u_iv$ does not belong to $S$.

\begin{defn}[$C_{4+}^*$-subdivision]
Let $G$ be a planar graph with a $C_{4+}$-subdivision $S$ rooted on a $4$-family $U$.
$S$ is a $C_{4+}^*$ if its satisfies the following three conditions:
\begin{itemize}
\item \emph{Property ``$0$-linked'':} Two $0$-linked special vertices have no common remaining neighbor;
\item \emph{Property ``$1$-linked'':} No internal vertex of a solo $(u_i,u_j)$-path of $S$ is a remaining neighbor of some $u_k\in U\setminus \{u_i,u_j\}$;
\item \emph{Property ``$2$-linked'':} If $u_i,u_j\in U$ are $2$-linked, then $u_i,u_j$ have at most one common remaining neighbor, and it belongs to a parallel path of $S$ that is not incident with $u_i,u_j$.
\end{itemize}
\end{defn}

We say that $S$ is a \emph{$\mathcal{K}$-subdivision} if it is a $K_4$-subdivision or a $C_{4+}^*$-subdivision.
In the next lemma, we show how we find such a subdivision in a planar graph $G$ with a $(C_{II})$ configuration $U$. 
In order to reduce the number of cases in the rest of the proof, we want to guarantee the additional property of chordlessness of the subdivision (as defined in the Preliminaries {\chapsec}).

\begin{lem}\label{lem:k4c4p}
Let $H$ be a planar graph that is almost $4$-connected w.r.t. a $4$-family $U = \{u_1,u_2,u_3,u_4\}$.
Then $H$ contains a chordless 
$\mathcal{K}$-subdivision rooted on $U$.
\end{lem}

In order to prove this lemma, we use a result by Yu~\cite{YU199810}, which deals with graphs with two types of structure constraints. Let us introduce them, as $N_1$-graphs and $N_2$-graphs. The following definitions are taken straight from the beginning of Section $4$ of \cite{YU199810}, as the two \textit{obstructions}, pictured in Figure $7$ of \cite{YU199810} on page $36$.

An $N_1$-graph (Figure~\ref{fig:none}) is a planar graph $H$ that has a $4$-family $U = \{u_1,u_2,u_3,u_4\}$ and a facial cycle $C$ (that we assume is the outer cycle), such that for each $i\in \{1,2,3,4\}$ either $u_i\in C$ or $H$ has a $4$-cut $X_i$ separating $u_i$ from $U\setminus \{u_i\}$ (so $u_j\notin X_j$ for $j\neq i$), and $\vert X_i\cap C\vert = 2$. Moreover, if $H_i$ is the component of $H\setminus X_i$ containing $u_i$, then the components $H_i$ for $i\in\{1,\dots,4\}$ are disjoint.

\begin{figure}[ht]
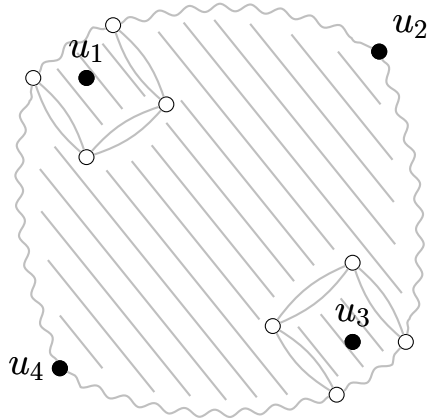

\center
\renewcommand{\MyScale}{0.7}
\Pictcf{\NOneGr}
\caption{An $N_1$-graph, with $u_2,u_4$ on the outer cycle, and $u_1,u_3$ surrounded by $4$-cuts}
\label{fig:none}
\end{figure}

An $N_2$-graph (Figure~\ref{fig:ntwo}) is a planar graph $H$ with a $4$-family $U = \{u_1,u_2,u_3,u_4\}$ and distinct (but not necessarily disjoint) $4$-cuts $T_i$, $i\in\{1,\dots,m\}$. The $4$-cuts are such that each $T_i$ separates two vertices of $U$, say $\{u_1,u_2\}$, from $T_{i+1}-T_i\neq \emptyset$; $T_1 = \{a_1,a_2,a_3,a_4\}$, $T_m = \{b_1,b_2,b_3,b_4\}$ and $H$ contains $4$ disjoint paths $S_i$ from $a_i$ to $b_i$ for $i\in\{1,\dots,4\}$ respectively. Additionally, $H$ has no $4$-cut $T$ separating $T_i\setminus T\neq \emptyset$ from $T_{i+1}\setminus T\neq \emptyset$, or separating $\{u_1,u_2\}$ from $T_1\setminus T\neq \emptyset$, or separating $\{u_3,u_4\}$ from $T_m\setminus T\neq \emptyset$; and either $T_i\cap T_{i+1}\neq \emptyset$ or two vertices of $T_i$ and two vertices of $T_{i+1}$ are cofacial in $H$. 
Finally, the ten following paths exist and are internally disjoint: a $(u_1,u_2)$-path $P_{12}$, a $(u_1,a_1)$-path $Q_1$, a $(u_1,a_2)$-path $Q_2$, a $(u_2,a_3)$-path $Q_3$, a $(u_2,a_4)$-path $Q_4$, a $(u_3,u_4)$-path $P_{34}$, a $(u_3,b_1)$-path $Q_1'$, a $(u_3,b_2)$-path $Q_2'$, a $(u_4,b_3)$-path $Q_3'$, and a $(u_4,b_4)$-path $Q_4'$. This last property is not part of the exact definition from \cite{YU199810}, but deduced from the first remark in Section $3$ on page $20$ of \cite{YU199810}.

\begin{figure}[ht]
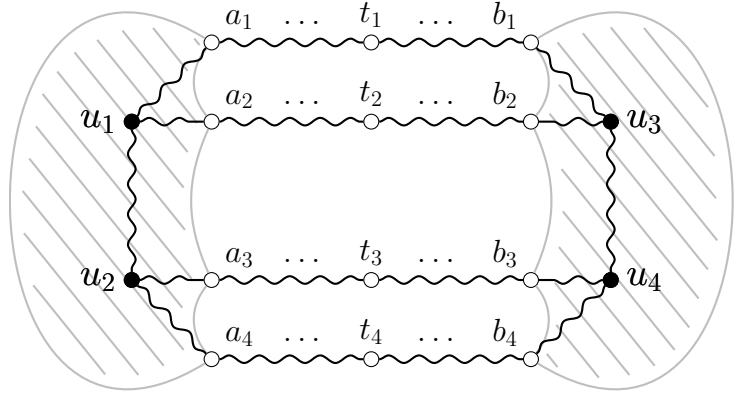

\center
\renewcommand{\MyScale}{0.7}
\Pictcf{\NTwoGr}
\caption{An $N_2$-graph, with three $4$-cuts represented: $\{a_1,a_2,a_3,a_4\}$, $\{t_1,t_2,t_3,t_4\}$, $\{b_1,b_2,b_3,b_4\}$}
\label{fig:ntwo}
\end{figure}

Let us restate Theorem $4.2$ from \cite{YU199810} in the new formalism. The definitions in Yu's paper have stronger constaints and the theorem is an equivalence, but in the present paper we only need one implication. Our definitions of $N_1$-graph and $N_2$-graph are slightly simplified and the theorem is restated as an implication.

\begin{thm}[Theorem $4.2$ of \cite{YU199810}]
Let $G$ be a $3$-connected planar graph and $U = \{u_1,u_2,u_3,u_4\}\subseteq V(G)$ be such that $G$ has no $3$-cut separating two vertices in $U$. Then $G$ has a $K_4$-subdivision rooted on $U$, or $G$ is an $N_1$-graph or an $N_2$-graph.
\label{thm:YU42}
\end{thm}

We argue that Theorem~\ref{thm:YU42} does not in fact require the $3$-connectivity assumption.

\begin{thm}
Let $G$ be a planar graph and $U = \{u_1,u_2,u_3,u_4\}\subseteq V(G)$ be such that $G$ has no $3$-cut separating two vertices in $U$. Then $G$ has a $K_4$-subdivision rooted on $U$, or $G$ is an $N_1$-graph or an $N_2$-graph.
\label{thm:YU42bis}
\end{thm}

\begin{proof}[Proof of Theorem~\ref{thm:YU42bis}]
We proceed by induction on the size of $G$. Assume that $G$ is not $3$-connected. Then there is a cut $X$ of size at most $2$. We may assume that $X$ is a minimal cut. We consider the connected components $C_1,\ldots,C_p$ of $G\setminus X$, and observe that all special vertices in $U$ belong to the same one, say $C_1$. If $|X|=1$, we consider $G'=G \setminus \{C_2 \cup \ldots \cup C_p\}$. If $|X|=2$, we consider $G'=G \setminus \{C_2 \cup \ldots \cup C_p\}+xy$, where $x$ and $y$ are the two vertices in $X$. If $xy\in E(G)$ then adding the edge $xy$ does not create a double edge. 

In both cases, $G'$ has fewer vertices than $G$, and is almost $4$-connected with respect to $U$. By Theorem~\ref{thm:YU42}, $G'$ has a $K_4$-subdivision rooted on $U$, or $G$ is an $N_1$-graph or an $N_2$-graph. Note that each of those three properties extend to all of $G$ (up to modifying the embedding, in the case of an $N_1$-graph with $|X|=1$ so as to maintain that the cycle is facial).
\end{proof}

\ifthenelse{\equal{\isThesis}{true}}
{\vfill
\pagebreak}
{}

We can now tackle the proof of Lemma~\ref{lem:k4c4p} (p.~\pageref{lem:k4c4p}).

\begin{proof}[Proof of Lemma~\ref{lem:k4c4p}]
First, let us define the operation that turns a subdivision into a chordless one.
Given a $K_4$- or $C_{4+}$-subdivision $S$ rooted on a $4$-family $U$, we say that we \emph{eliminate the chords} of $S$ if we apply the following operation exhaustively:
if a path $P=(v_1,\dots,v_k)$ of $S$ has a chord $v_iv_j$, we replace $P$ in $S$ with the path $P' = (v_1,\dots,v_i,v_j,\dots,v_k)$, except if $v_i,v_j\in U$ and the edge $v_iv_j$ already constitutes a path of $S$. Since each elimination decreases the number of edges in $S$, the process terminates, and the obtained subdivision $S'$ is well-defined and chordless. Since the vertices of $S'$ are a subset of the vertices of $S$ and the ends of the paths are preserved, $S'$ has the same type ($K_4$ or $C_{4+}$) as $S$.\\

If $G$ has a $K_4$-subdivision, eliminating its chords gives us the result. Using Theorem~\ref{thm:YU42bis} (p.~\pageref{thm:YU42bis}), it now suffices to show that any $N_1$-graph and any $N_2$-graph contain a chordless $C_{4+}^*$-subdivision to prove our lemma.\\

Let us first take care of the case where $H$ is an $N_2$-graph, defined as in the definition above.
Let $S$ be the union of the paths $P_{12}$, $Q_1$, $Q_2$, $Q_3$, $Q_4$, $P_{34}$, $Q_1'$, $Q_2'$, $Q_3'$, $Q_4'$, $S_1$, $S_2$, $S_3$, $S_4$. $S$ is a $C_{4+}$-subdivision rooted on $U$, so let $S'$ be a chordless $C_{4+}$-subdivision obtained by eliminating the chords of $S$. Let us prove that $S'$ is a $C_{4+}^*$-subdivision.

By definition, the $4$-cut $T_1$ is a subset of $V(S)$, and we show that $T_1\subseteq V(S')$. Let us assume for contradiction that there is a vertex $t\in T_1$ which belongs to a path $P$ of $S$ on which a chord elimination was performed: there are two vertices $v,v'$ on $P$, such that the edge $vv'\in E(H)$ is not in $P$ and $t$ is on the $(v,v')$-section $P'$ of $P$. But then the path $(P\setminus P')\cup \{vv'\}$ is a path between two $1$-linked special vertices in $H\setminus T_1$, which is a contradiction; so $T_1\subseteq V(S')$.

Assume for contradiction that two $0$-linked special vertices $u_i,u_j$ have a common remaining neighbor $v$ w.r.t. $S'$. If $v$ belongs to $V(S')$, then the edge $u_iv$ or $u_jv$ forms a chord of $S'$ since all paths of $S'$ are incident with $u_i$ or $u_j$. It is a contradiction, so $v$ does not belong to $V(S')$. However, since $u_i$ and $u_j$ are $0$-linked, they are separated by the $4$-cut $T_1$, and so $v$ is a vertex of $T_1$, hence in $S$. This is a contradiction, which gives us the property ``$0$-linked'' of $S'$.

To obtain property ``$1$-linked'', observe that no internal vertex of a solo $(u_i,u_j)$-path is a remaining neighbor of a $u_k\in U\setminus \{u_i,u_j\}$, since this would create a path between $0$-linked special vertices that would be disjoint from the $4$-cut $T_1$, a contradiction.

Let us now prove the property ``$2$-linked''.
By definition of the cut $T_1$, the common remaining neighbors of two $2$-linked special vertices $u_i,u_j$ belong to parallel paths of $S'$. If a common remaining neighbor $v$ of $u_i,u_j$ belongs to a $(u_i,u_j)$-path of $S'$, the two edges $u_iv, u_jv$ form chords in $S'$, a contradiction, so the common remaining neighbors of $u_i,u_j$ belong to parallel $(u_k,u_l)$-paths that are not incident with $u_i,u_j$.
If $u_1,u_3$ ($2$-linked) have two common remaining neighbors $v,v'$, they belong to different parallel $(u_2,u_4)$-paths by planarity (otherwise $\{u_1,u_4,v\}$ and $\{u_2,u_3,v'\}$ form a $K_{3,3}$-minor in $H$ if $v,v'$ belong to a path $P=(u_2,\dots,v,\dots,v',\dots,u_4)$). Then we claim that $\{u_1,v,v'\}$ is a $3$-cut of $H$ that separates $u_2$ from $u_3$. If not, there is a $(u_2,u_3)$-path $P_{23}$ in $H$ that is vertex-disjoint from $u_1,v,v'$, and by definition of $T_1$, $P_{23}$ contains a vertex $t\in T_1$ that is on a $(u_1,u_3)$-path of $S'$, and the $(u_2,t)$-section of $P_{23}$ does not contain $u_4$. Then $\{u_1,u_2,v,t,u_4=v'\}$ induce a $K_5$-minor in $H$ by contracting the $(u_4,v')$-path of $S'$ into a vertex. Hence $\{u_1,v,v'\}$ is a $3$-cut of $H$ that separates $u_2$ from $u_3$, which is a contradiction with the almost $4$-connectivity of $H$ w.r.t. $U$.
Property ``$2$-linked'' follows.\\

Now let us consider the case where $H$ is an $N_1$-graph with outer cycle $C$. We build a $C_{4+}$-subdivision $S$ rooted on $U$ as follows.
First, for each $u_i\notin C$, we find four internally-disjoint paths $p_1^i,p_2^i,p_3^i,p_4^i$ from $u_i$ to the four vertices of its associated $4$-cut $X_i$.
Such paths exist since there is no $3$-cut separating $u_i$ from $U\setminus \{u_i\}$ in $H$.
We assume $p_1^i$ and $p_2^i$ each have an end on the outer cycle $C$. We add $(E(C)\setminus \bigcup_{u_i \notin C} E(H_i)) \cup \bigcup_{u_i\notin C} (E(p_1^i) \cup E(p_2^i))$ to $S$, where $H_i$ is the component of $H\setminus X_i$ containing $u_i$.
To obtain the two remaining paths of $S$, we consider the graph $H'$ formed by removing $(V(C) \cup \bigcup_{u_i \notin C} V(H_i))\setminus U$ from $H$, and add back $p_3^i$ and $p_4^i$ to $H'$. We look at the outer face of $H'$. Let $P_{13}$ be the outer $(u_1,u_3)$-path and $P_{24}$ the outer $(u_2,u_4)$-path of $H'$. We claim that these paths are vertex-disjoint. 
To see it, observe that in $H$ there are at least $4$ internally-disjoint paths from $u_1$ to $u_3$. At least two of them are disjoint from $C$, hence belong to $H'$. Therefore, $H'$ cannot contain a $1$-cut separating $\{u_1,u_2\}$ from $\{u_3,u_4\}$.
Therefore, let us add $P_{13}$ and $P_{24}$ to $S$ to form our $C_{4+}$-subdivision $S$.

Let $S'$ be a chordless $C_{4+}$-subdivision obtained by eliminating the chords of $S$.
Let us now prove that $S'$ is a $C_{4+}^*$-subdivision.

We first check property ``$1$-linked''. Let $u_i,u_j$ be $1$-linked special vertices with a path $P_{ij}$ of $S'$, and $u_k\in U\setminus \{u_i,u_j\}$. If there is an internal vertex $v$ of $P_{ij}$ that is a remaining neighbor of $u_k$, then $\{u_k,v\}$ is a $2$-cut of $H$ if $u_k$ is on $C$, otherwise there is a vertex $x$ in the $4$-cut $X_k$ of $u_k$ such that $\{x,v\}$ is a $2$-cut of $H$, contradicting its $3$-connectivity.

To check properties ``$0$-linked'' and ``$2$-linked'', we 
%
show that there is no remaining neighbor in common between $u_1$ and $\{u_3,u_4\}$ (respectively $2$-linked and $0$-linked to $u_1$), because of the properties of almost $4$-connectivity of $H$. Each special vertex $u_i$ either belongs to $C$ or there is a $4$-cut $X_i = \{x_1,x_2,x_3,x_4\}$ separating $H$ into a component $H_i$ containing $u_i$ and a $H\setminus (X_i\cup H_i)$ containing $U\setminus \{u_i\}$. If $u_1,u_k$, $k\in\{3,4\}$, belong to $C$ and share a remaining neighbor $v$, then $\{u_1,u_k,v\}$ is a $3$-cut that separates two neighbors of $u_1$ (if $k=3$) or separates $u_2$ from $u_3$ (if $k=4$).
If $u_1$ belongs to $C$ and $u_k$ has a $4$-cut $X_k$, then their common remaining neighbor is the only vertex $x\in X_k$ that does not belong to $S'$. Then there is a vertex $x'\in X_k\cap C$, such that $\{u_1,x,x'\}$ is a $3$-cut that separates $u_2$ from $u_3$ (whether $k=3$ or $4$).
If both $u_1,u_k$ have $4$-cuts $X_1,X_k$, then their common remaining neighbor is again the only $x\in X_1\cap X_k$ that does not belong to $S'$, and there are $x_1'\in X_1\cap C$ and $x_k'\in X_k\cap C$ such that $\{x_1',x_k',x\}$ is a $3$-cut that separates $u_2$ from $u_3$.
In all cases, we obtain a contradiction with the almost $4$-connectivity of $H$.
Properties ``$0$-linked'' and ``$2$-linked'' follow, which completes the proof.
\end{proof}

\ifthenelse{\equal{\isThesis}{true}}
{\section{Patterns}}
{\subsection{Patterns}}



Although almost all the subdivisions that we consider throughout the paper are regular $K_4$-subdivision or $C_{4+}$-subdivisions, 
we occasionally consider a more convenient structure, which we call \emph{semi-$C_{4+}$-subdivision}, that consists in a $C_{4+}$-subdivision where two parallel paths with disjoint ends intersect on one vertex.

Let $W_4$ be the wheel graph on $5$ vertices $u_1,u_2,u_3,u_4,w$, i.e. the graph where $u_1,u_2,u_3,u_4$ form a cycle and $w$ is adjacent to the other four vertices.

\begin{defn}[Semi-$C_{4+}$-subdivision]
A \emph{semi-$C_{4+}$-subdivision rooted on a $4$-family $U$} in a graph $G$ is a $W_4$-subdivision rooted on $U\cup \{w\}$, where $w$ is a vertex of $G$.
\end{defn}

By abuse of notation and by analogy with the $C_{4+}$-subdivision, we arbitrarily pick two pairs of special vertices $(u_i,u_j),(u_k,u_l)$ and we view the union of the $(u_i,w)$-path and the $(u_j,w)$-path as a $(u_i,u_j)$-path $P_{ij}$, and the union of the $(u_k,w)$-path and the $(u_l,w)$-path as a $(u_k,u_l)$-path $P_{kl}$. We say that there is a \emph{contact} between $P_{ij}$ and $P_{kl}$.

Observe that a semi-$C_{4+}$-subdivision is always $2$-colorable, since we can simply swap the colors of two parallel paths in a $2$-coloring, in order to give different colors to the two paths in contact (see Figure~\ref{fig:semiCFour}).

\SemiCFour

The following definition regroups the different kinds of structures that we consider for our reduction rules.

\begin{defn}[Semi-subdivision]
A \emph{semi-subdivision} $S$ in a graph $G$ is a $K_4$-subdivision, a $C_{4+}$-subdivision or a semi-$C_{4+}$-subdivision rooted on a $4$-family $U$.
\end{defn}



Our overall goal is to find an invidual reduction rule for each of the four special vertices, and combine them into a general rule for the whole configuration. We define these rules by extending the formalism of 
{\chCI}.

\ifthenelse{\equal{\isThesis}{true}}
{}
{\vfill
\pagebreak}

A \emph{subdivision partial configuration} (or \textbf{pattern}) $\C_i$ is a configuration
defined over the neighborhood of one special vertex $u_j$ or two special vertices $u_j,u_l$, with three identified incident edges per special vertex, called \emph{subdivision edges}.
We denote it by $\C_i(u_j)$ if it involves one special vertex, and $\C_i(u_j,u_k)$ otherwise.

Observe that an elementary partial configuration can be turned into a subdivision partial configuration by adding two subdivision edges. 
The subdivision partial configuration \ncf{\cfCU} defined below is an example of partial configuration that involves two special vertices.

We say that a set of patterns $\mathcal{M} = \{\C_1,\dots,\C_k\}$, $k\in\{2,3,4\}$, is a \emph{mapping} of a $4$-family $U$, w.r.t. a semi-subdivision $S$ rooted on $U$, if in this semi-subdivision 
there is a bijection between the special vertices of $\C_1,\dots,\C_k$ and the special vertices of $U$; i.e. each special vertex $u$ of $\C_1,\dots,\C_k$ can be associated with a special vertex $u'\in U$, and the neighborhoods of $u$ and $u'$ are isomorphic. 
If $u'\in U$ is associated to the special vertex $u$ of a pattern $\C_i$, we say that $u'$ \emph{forms} a pattern $\C_i$ w.r.t. $S$.

Given a mapping $\mathcal{M} = \{\C_1,\dots,\C_k\}$ and $\C_i\in \mathcal{M}$, we denote by $V(\C_i)$ the set containing the special vertices associated with $\C_i$ in $\mathcal{M}$ and their remaining neighbors.
We say that a pattern $\C_i$ \emph{touches} another pattern $\C_j$ if 
$V(\C_i)\cap V(\C_j)\neq \emptyset$.
We say that $\C_i$ touches $S$ if at least one non-special vertex of $V(\C_i)$ belongs to a path of $S$.

A \emph{subdivision composite configuration} $(\mathcal{M},S)$ is the following configuration: the graph contains a $4$-family $U$ and a semi-subdivision $S$ rooted on $U$, while $\mathcal{M}$ is a mapping of $U$ w.r.t. $S$.

In the previous 
{\chapsec},
we defined elementary partial rules over elementary partial configurations. This definition can be directly extended to define \emph{subdivision partial rules} over subdivision partial configurations, i.e. as a rule $\mathcal{R}_i = (\C_i,f_i^r,f_i^c)$ associated with a pattern $\C_i$, a partial reduction function encoded by a set $\mathcal{O}_i\subseteq \{add,remove\}\times E(\C_i)$ and the partial recoloring function $f_i^c$.

Let $\mathcal{M} = \{\C_1,\dots,\C_k\}$, $k\in\{2,3,4\}$, be a mapping of a $4$-family $U = \{u_1,u_2,u_3,u_4\}$, w.r.t. a semi-subdivision $S$ rooted on $U$. For each $i\in\{1,\dots,k\}$, let $\mathcal{R}_i=(\C_i,f_i^r,f_i^c)$ be a subdivision partial rule associated with the pattern $\C_i$.
Let $c_S$ be a $2$-coloring of $S$.
The \emph{subdivision composite rule} $\mathcal{R}_c = (C_c, f_c^r, f_c^c)$, denoted by $(\{\mathcal{R}_1,\dots, \mathcal{R}_k\},S,c_S)$, is the reduction rule associated with the subdivision composite configuration $C_c = (\mathcal{M},S)$ and is defined as follows. The reduction function $f_c^r$ is defined by $f_c^r(G) = (f_1^r \circ \dots \circ f_k^r(G))\setminus (U \cup E(S))$, i.e. the successive application of the operations in $\mathcal{O}_i$ in reverse order and the removal of the special vertices $U$ and the edges of the semi-subdivision $S$, to form the reduced graph $G'$.

In order to provide a semantics of $f_c^c$, we define the \emph{intermediate graphs} $G_{\text{int}}^i = f_{i+1}^r \circ \dots \circ f_k^r(G)$, for $0\leq i \leq k$. We use these graphs to define a sequence of colorings $c_{\text{int}}^i$ that lead to a coloring $c$ of $G$. Let $pc$ be a coloring of $G'$. Let $c_{\text{int}}^0 = pc \cup c_S$ be a coloring of $G_{\text{int}}^0 = G'\cup (U\cup E(S))$. We define $c_{\text{int}}^i = f_i^c(G_{\text{int}}^i, c_{\text{int}}^{i-1})$ for $i\in\{1,\dots,k\}$. We finally define $f_c^c(G,pc,c_S) = c_{\text{int}}^k$ for any planar graph $G$, coloring $pc$ of $f_c^r(G)$ and good coloring $c_S$ of $S$.
In other words, the semi-subdivision $S$ is added to $G'$ and colored with $c_S$, then for each pattern $\C_i$ considered in ascending order, the reduction of $\C_i$ is undone and the edges in the neighborhood of $\C_i$ are colored according to the partial recoloring function $f_i^c$.
This definition is motivated by the fact that whenever $pc$ is a good coloring of the reduced graph $G'$, and the partial rules $\mathcal{R}_i$ are valid and do not interfere with each other, each intermediate coloring $c_{\text{int}}^i$ is a good coloring of the intermediate graph $G_{\text{int}}^i$, which allows to build step by step a good coloring of $G$.
The $2$-coloring $c_S$ of $S$ is specified only when necessary.


%
%
%
%
%

\renewcommand\MyScale{0.8}

\ifthenelse{\equal{\isThesis}{true}}
{\medskip}
{\vfill
\pagebreak}

In the figures, the two paths $P_1, P_2$ induced by the $2$-coloring of $S$ are represented in red and blue. The purple color is used to color the whole subdivision when its $2$-coloring is not specified.
An edge represented in black does not belong to the subdivision.
A red vertex (\tkSymbol{\node[rednode] (wn) at (0,0) {};}) (resp. blue \tkSymbol{\node[bluenode] (wn) at (0,0) {};}) represents a vertex that may be touched by a red (resp. blue) subdivision path. A purple vertex (\tkSymbol{\node[purplenode] (wn) at (0,0) {};}) may be touched by either a red or a blue subdivision path. When a path of the subdivision ends on a special vertex, it is represented by
\tkSymbol{\node[blacknode] (g1) at (0,0) {}; \mySubRedEdge{g1}{0}{}{};},
\tkSymbol{\node[blacknode] (g1) at (0,0) {}; \mySubBlueEdge{g1}{0}{}{};}
or
\tkSymbol{\node[blacknode] (g1) at (0,0) {}; \mySubEdge{g1}{0}{}{};},
respectively if it is the red path, the blue path, or if the color is not specified.

Let us now introduce the patterns we use in the rest of the proof. For each pattern, we describe the associated partial configuration, as well as the conditions on the colors of a $2$-coloring of the associated subdivision $S$. We then provide a definition of the partial reduction and recoloring functions. The patterns \ncf{\cfCV}, \ncf{\cfCvTwo}, \ncf{\cfCvThree} are taken from 
{\chCI}
and their definitions are omitted.


\ifthenelse{\equal{\isThesis}{true}}
{\vfill
\pagebreak}
{\ \\}

\textbf{List of the patterns:}\\

\allPatterns

\medskip

From now on, when we talk about patterns, we refer exclusively to patterns from this list.


Obviously the partial rules associated with the patterns of the considered mapping may conflict with each other. We now address the conditions of compatibility between patterns.

\begin{defn}[Compatible patterns]
Let $G$ be a planar graph with a semi-subdivision $S$ rooted on a a $4$-family $U$.

Let $\C_i, \C_j$ be two patterns on $U$ w.r.t. $S$. $\C_i, \C_j$ are \emph{compatible} if:
\begin{itemize}
\item $\C_i$ or $\C_j$ is a {\CV}, {\CVp} or {\CU}, and $\vert V(\C_i) \cap V(\C_j)\vert \leq 1$; or
\item $\C_i,\C_j$ are among {\CN}, {\CTOne}, {\CTTA}, {\CTTNA}, {\CDA}, {\CDB}, and $V(\C_i) \cap V(\C_j) = \emptyset$.
\end{itemize}

Let $\mathcal{M} = \{\C_1,\dots,\C_k\}$ be a mapping of $U$ w.r.t. a semi-subdivision $S$, and $2\leq k \leq 4$.\\
We say that $\mathcal{M}$ is a \emph{compatible mapping} w.r.t. $S$ if it satisfies the following conditions:
\begin{itemize}
\item The $\C_i$ patterns in $\mathcal{M}$ are pairwise compatible;
\item There exists a $2$-coloring $c_S$ of $S$ that fits the color requirements in the definition of each $\C_i\in \mathcal{M}$.
\end{itemize}
\end{defn}

We justify this notion of compatible patterns and compatible mapping with the following claim.

\begin{claim}
Let $G$ be a planar graph with a $(C_{II})$ configuration w.r.t. a $4$-family $U$.
Let $\mathcal{M} = \{\C_1,\dots,\C_k\}$ be a mapping of $U$ w.r.t. a semi-subdivision $S$, with $2\leq k \leq 4$, and let $\mathcal{R}_i$ be a subdivision partial rule associated with $\C_i$ for $i\in\{1,\dots,k\}$.

If $\mathcal{M}$ is a compatible mapping w.r.t. $S$, then 
the $2$-coloring $c_S$ of $S$ associated with $\mathcal{M}$ is such that the subdivision composite rule $(\{\mathcal{R}_1,\dots,\mathcal{R}_k\},S,c_S)$ associated with $(\mathcal{M},S)$ is valid.
\label{clm:compmapping}
\end{claim}

\begin{proof}
Each pattern $\C_i \in \{${\CV}$,$ {\CVp}$,$ {\CU}$\}$ has a deviated edge, $v_1v_2$ in the previous definitions. Since $\C_i$ is compatible with all other patterns, it shares at most one vertex with each of them, thus the edge $v_1v_2$ cannot be used as a deviated edge by another pattern. Hence the resolution rules of the {\CV}, {\CVp}, {\CU} patterns can be applied independently.

A pattern $\C_i$ in $\{${\CN}$,$ {\CTOne}$,$ {\CTTA}$,$ {\CTTNA}$,$ {\CDA}$,$ {\CDB}$\}$ can only share at most one vertex with each {\CV}, {\CVp} or {\CU} pattern, as they do not prevent the resolution rule of $\C_i$ from being applied.

We emphasize that the parities involved in the resolution rules of {\CN}, {\CTOne}, {\CTTA}, {\CTTNA} are preserved no matter how many {\CV}, {\CVp} or {\CU} patterns touch them. Say we have a pattern $\C_i$ in $\{${\CN}$,$ {\CTOne}$,$ {\CTTA}$,$ {\CTTNA}$\}$, and a non-special vertex $v$ of $\C_i$. In the descriptions of the patterns and their resolution rules, we may specify the parity of $v$ in the graph $G$, then which edges we add or remove to obtain 
that $v$ has an odd degree in the reduced graph $G'$.
These definitions do not take into account the {\CV}, {\CVp}, {\CU} patterns or a path from $S$ that may touch $v$, but we argue that they do not interfere with the parity of $v$ in the reduced graph $G'$.

If a path from $S$ touches $v$ in $G$ and does not form a {\CVp} pattern, then the reduction from $G$ to $G'$ removes two edges incident with $v$, which preserves the parity of $v$. If a {\CV} or {\CU} pattern touches $v$ in $G$, then one edge $vu'$ ($u'\in U$) is removed and one edge $vv'$ ($v'\notin U$) is added, which preserves the parity of $v$. Finally if a {\CVp} pattern $\{u',v,v'\}$ touches $v$, then the edge $vu'$ is removed, as well as an edge $vw$ from the path of $S$ that contains the edge $vv'$. The edge $vv'$ is kept in the reduced graph. Thus, $v$ has lost two incident edges, and so its degree is preserved.

In conclusion, we may apply the resolution rules of compatible patterns in any order.

Since the definitions of patterns do not create cycles, do not use additional colors, and since the $2$-coloring $c_S$ of $S$ fits the color requirements of all patterns in $\mathcal{M}$, the subdivision composite rule $(\{\mathcal{R}_1,\dots,\mathcal{R}_k\},S,c_S)$ associated with $(\mathcal{M},S)$ is valid.
\end{proof}



Let us introduce the notion of \emph{settled} special vertex, to characterize the special vertices that are already compatible with the rest of the configuration and whose remaining neighbors do not need to be further altered.

\begin{defn}[Settled vertices]
Let $G$ be a planar graph with a $(C_{II})$ configuration w.r.t. a $4$-family $U$.
Let $S$ be a semi-subdivision rooted on $U$.

We say that a special vertex $u$ is \emph{lone-settled} w.r.t. $S$ if:
\begin{itemize}
\item $u$ forms a {\CV} or {\CVp} pattern and shares at most one remaining neighbor with each of the other special vertices, and the two remaining neighbors of $u$ are not the two remaining neighbors of a {\CU} pattern; or
\item $u$ forms a {\CN} pattern and its remaining neighbors are disjoint from $S$ and from the remaining neighbors of other special vertices.
\end{itemize}


A special vertex is \emph{settled} if it is lone-settled or forms a {\CTTNA} pattern $\{u,u',v,v'\}$ with another special vertex $u'$, such that $v,v'$ are disjoint from $S$ and from the remaining neighbors of other special vertices.
\end{defn}

Note that in this definition, the {\CV} pattern formed by a lone-settled special vertex can touch the subdivision $S$. By Claim~\ref{clm:compmapping} (p.~\pageref{clm:compmapping}), we deduce immediately that if all four vertices of $U$ are settled w.r.t. $S$, there exists a mapping $\mathcal{M}$ of $U$ compatible w.r.t. $S$.

\begin{figure}[ht]
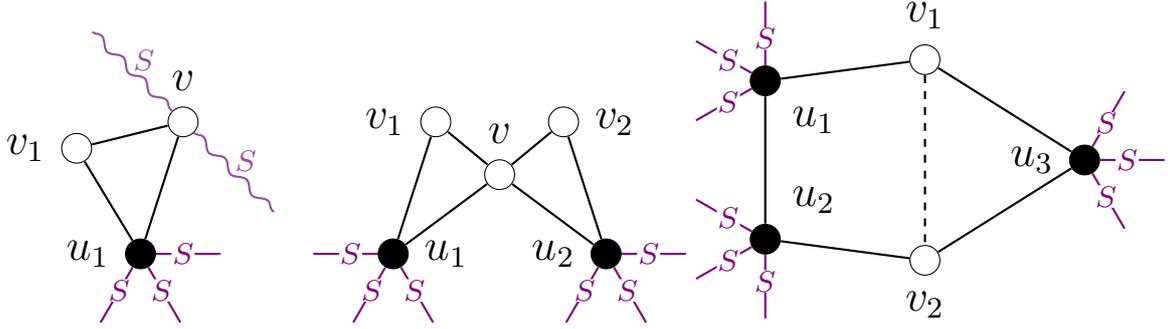

\renewcommand{\MyScale}{1.4}

\cfSettledA
\tkcf{}
\cfSettledB
\tkcf{}
\cfSettledC
\tkcf{}
\caption{Examples of \textbf{unsettled} vertices}
\label{fig:unsettled}
\end{figure}

Figure~\ref{fig:unsettled} provides a few examples of unsettled special vertices: on the left a {\CN} pattern touching the subdivision, in the middle two {\CN} patterns sharing a remaining neighbor, and on the right a {\CU} and a {\CV} pattern sharing both remaining neighbors. None of the depicted special vertices are settled.

\vfill
\pagebreak

\ifthenelse{\equal{\isThesis}{true}}
{\section{Redirection procedure}}
{\subsection{Redirection procedure}}
\label{subs:redir}



To further decrease the number of problematic cases to consider in the rest of the proof, we consider a set of local transformations that, when applied exhaustively to a $\mathcal{K}$-subdivision $S$ rooted on a $4$-family $U$, return another $\mathcal{K}$-subdivision $S'$ rooted on $U$, of the same type and which does not contain some inconvenient configurations. The special vertices of $U$ can be more easily mapped to patterns in $S'$ than in $S$.

This procedure does not preserve the chordlessness of a subdivision it is applied to, so let us consider the following weaker properties that are preserved by the procedure (as is proven in Claim~\ref{clm:redirAB}, p.~\pageref{clm:redirAB}).

Given a $\mathcal{K}$-subdivision $S$ rooted on $U$, an \emph{A-chord} of $S$ is a chord on a path of $S$ incident with a special vertex $u\in U$ that is unsettled or forms a {\CVp} pattern w.r.t. $S$ (see Figure~\ref{fig:chordA}); and a \emph{B-chord} is a chord between two remaining neighbors of $u\in U$ on a path of $S$ that is \textbf{not} incident with $u$ (see Figure~\ref{fig:chordB}). We say respectively that $u$ \emph{has} an A-chord, a B-chord.

A $\mathcal{K}$-subdivision satisfies \emph{property A} (resp. \emph{property B}) if it does not have an A-chord (resp. a B-chord). A chordless $\mathcal{K}$-subdivision obviously satisfies properties A and B.

\renewcommand{\MyScale}{1.3}

\begin{figure}[htb]
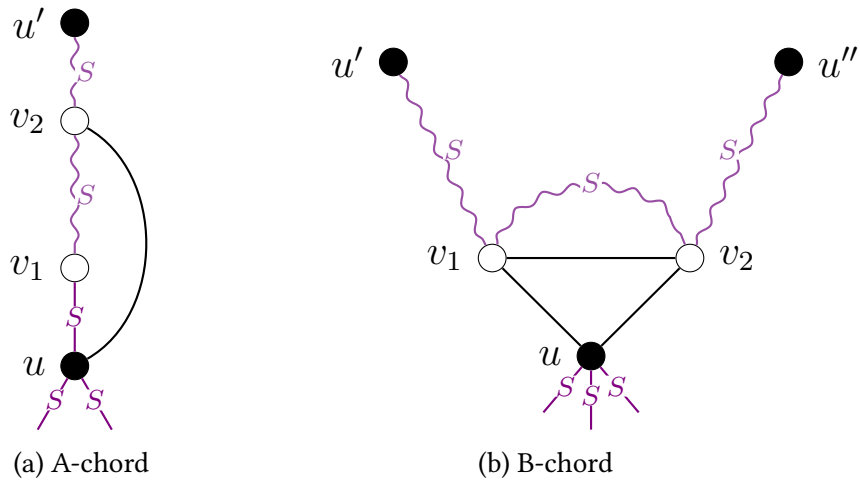

    \centering

    \begin{subfigure}[b]{6cm}
        \centering
        \Pictcf{\propA}
        \caption{\CfN}
        \label{fig:chordA}
    \end{subfigure}
    \begin{subfigure}[b]{6cm}
        \centering
        \Pictcf{\propB}
        \caption{\CfN}
        \label{fig:chordB}
    \end{subfigure}
    \caption{The configurations that are avoided by properties A and B}
    \label{fig:chords}
\end{figure}

Let us now define the procedure that helps take care of problematic cases in a $\mathcal{K}$-subdivision.

\begin{defn}[Redirection procedure]
Let $G$ be a planar graph with a $(C_{II})$ configuration w.r.t. a $4$-family $U$, and $S$ be a $\mathcal{K}$-subdivision rooted on $U$, such that $S$ satisfies properties A and B.
The \emph{redirection procedure} consists in applying as many times as possible the
\emph{redirection} operations
{\redirXOne}, {\redirXTwo}, {\redirXThree} and {\redirXFour}
to $S$.
\end{defn}

Note that these configurations are defined on $u_1,u_2$ but may be in contact with unspecified remaining neighbors of $u_3,u_4$, in which case we apply the redirections anyway. The contacts with paths from $S$ that would prevent us from applying the redirections are specified.

\textbf{Remark:} the {\CV} patterns on the drawings illustrating the redirections could turn out to be {\CTTNA} patterns, and are only featured as an indication.

\ifthenelse{\equal{\isThesis}{true}}
{\vfill
\pagebreak}
{}

\renewcommand\MyScale{1.4}

{\center
\RuleCDOnenec

Redirection {\redirXOne}, when $w_1=w_2$

\
}

\begin{itemize}

\item {\redirXOne}: The vertices $u_1,u_2$ are linked by a path $u_1\sim u_2$ of $S$ of the form $(u_1,w_1,Q,w_2,u_2)$, with $w_1=w_2$ if $l(Q)=0$. The vertices $u_1,u_2$ have exactly one common remaining neighbor $v$ and have another remaining neighbor $v_1,v_2$ respectively, both adjacent to $v$. The vertices $v_1$ and $w_1$ are non-adjacent.
No path from $S$ touches $v$, $v_1$ or $v_2$.

\textbf{Redirection protocol:} We replace the path $u_1\sim u_2$ in $S$ with the path $(u_1,v,u_2)$.

{\center
\RuleCDOnened

Redirection {\redirXTwo}, when $w_1=w_2$

\
}

\item {\redirXTwo}: The vertices $u_1,u_2$ are linked by a path $u_1\sim u_2$ of $S$ of the form $(u_1,w_1,Q,w_2,u_2)$, with $w_1=w_2$ if $l(Q)=0$. The vertices $u_1,u_2$ have exactly one common remaining neighbor $v$ and have another remaining neighbor $v_1,v_2$ respectively, both adjacent to $v$. The vertices $v_1$ and $w_1$ (resp. $v_2$ and $w_2$) are adjacent.
No path from $S$ touches $v$, $v_1$ or $v_2$.

\textbf{Redirection protocol:} We replace the path $u_1\sim u_2$ in $S$ with the path $(u_1,v_1,v,u_2)$. The vertices $v,w_1$ are the new remaining neighbors of $u_1$ and are not adjacent, otherwise $\{u_1,v,v_1,w_1\}$ would form a $K_4$, contradicting Claim~\ref{clm:inducedK4} (p.~\pageref{clm:inducedK4}).

\vfill
\pagebreak

{\center
\RuleCDOneneeBis

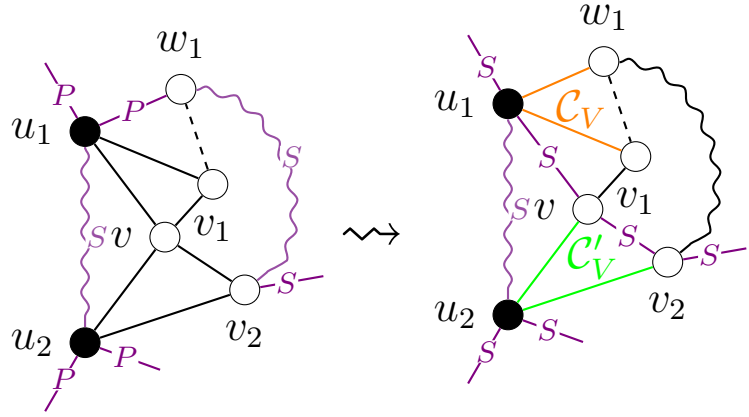
Redirection {\redirXThree}, when $w_1,v_1$ are not adjacent

\vspace{0.2cm}

\RuleCDOnenee

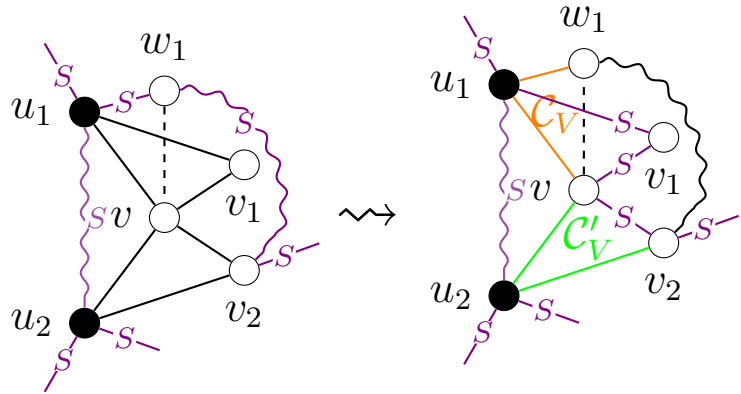
Redirection {\redirXThree} when $w_1,v$ are not adjacent

\
}

\item {\redirXThree}: The vertices $u_1,u_2$ are linked by a path $u_1\sim u_2$ of $S$,
they have exactly one remaining neighbor $v$ in common, and $u_1,u_2$ have another remaining neighbor $v_1,v_2$ respectively. $v_1,v_2$ are both adjacent to $v$.
There is another special vertex $u_3$ such that there is a path $u_1\sim u_3$ in $S$ of the form $P = (u_1,w_1,Q,v_2,Q',u_3)$, with $l(Q)\geq 0$ (so $w_1$ may be equal to $v_2$) and $l(Q')\geq 1$.
No path from $S$ touches $v$ nor $v_1$.

\textbf{Redirection protocol:} The vertex $w_1$ cannot be adjacent to both $v$ and $v_1$, otherwise $\{u_1,v,v_1,w_1\}$ would induce a $K_4$ in the graph, a contradiction by Claim~\ref{clm:inducedK4} (p.~\pageref{clm:inducedK4}). We distinguish between two cases:
\begin{itemize}
\item If $w_1,v_1$ are \textbf{not adjacent:} we replace the path $P$ in $S$ with the path $P' = (u_1,v,v_2,Q',u_3)$;
\item If $w_1,v_1$ are \textbf{adjacent:} then necessarily $w_1,v$ are not, and in this case we replace $P$ in $S$ with the path $P' = (u_1,v_1,v,v_2,Q',u_3)$.
\end{itemize}

In both cases, the two new remaining neighbors of $u_1$ are not adjacent.

%
%
%
%

\vfill
\pagebreak

{\center
\RuleCDTwowod

Redirection {\redirXFour}

\
}

\item {\redirXFour}: The vertices $u_1,u_2$ are linked by a path $u_1\sim u_2$ of the form $(u_1,w_1,Q,w_2,u_2)$, with $l(Q)\geq 0$ (with $w_1=w_2$ if $l(Q)=0$). The vertices $u_1,u_2$ have two adjacent remaining neighbors $v,v'$ in common. 
No path from $S$ touches $v$ nor $v'$.

\textbf{Redirection protocol:} 
There must be a non-edge $e$ among $vw_1, v'w_1$, otherwise $\{u_1,v,v',w_1\}$ form an induced $K_4$, 
which contradicts the fact that $G$ has a $(C_{II})$ configuration by Claim~\ref{clm:inducedK4} (p.~\pageref{clm:inducedK4}). Let us say that $e = vw_1$.
We replace the path $u_1\sim u_2$ in $S$ with the path $(u_1,v',u_2)$.

\end{itemize}

\ \\

In all cases, the special vertex $u_1$ is given a new set of remaining neighbors that are not adjacent. The procedure terminates, as each redirection requires two special vertices that have adjacent remaining neighbors to be applied, and increases the number of special vertices with non-adjacent remaining neighbors.

\medskip

We define the associated property: a subdivision satisfies \textbf{Property C} if no redirection can be applied.
We say that a subdivision is \textbf{strong} if it satisfies properties A, B and C.

\medskip

To justify the choice of this procedure and the notion of strong subdivision, we prove 
that it preserves properties A and B and the structure of $\mathcal{K}$-subdivision.

%
%
%

\begin{claim}
The redirection procedure preserves properties A and B.
\label{clm:redirAB}
\end{claim}

\begin{proof}
Let $G$ be a planar graph with a $4$-family $U$ and a $\mathcal{K}$-subdivision $S$ rooted on $U$, such that $S$ satisfies properties A and B. Let $S'$ be the subdivision obtained by applying the redirection procedure to $S$. Let us prove that $S'$ satisfies properties A and B as well.

\begin{itemize}
\item {\redirXOne}, 
{\redirXFour}: Redirection configurations {\redirXOne} 
and {\redirXFour} feature two special vertices $u_1,u_2$ and only modify one $(u_1,u_2)$-path in the subdivision, by replacing it with another of length $2$. No path of length $2$ has a B-chord, and since an edge $u_1u_2\in E(G)$ already constitutes a path of $S$, the new path does not have an A-chord either. Since the paths of $S$ are internally disjoint, the vertices $w_1,w_2$ do not belong to $S'$ and thus cannot be part of an A-chord or a B-chord.

\item {\redirXTwo}: In redirection configuration {\redirXTwo}, the new $(u_1,u_2)$-path $P'$ has length $3$, hence does not contain B-chords. Since the remaining neighbors of $u_1,u_2$ w.r.t. $S'$ do not belong to $S'$ except one ($v$), the other paths of $S'$ do not contain B-chords either. The remaining neighbors of $u_2$ are disjoint from $P'$, hence $u_2$ does not have an A-chord, but $u_1$ does however. We claim that property A is still satisfied because $u_1$ is (lone-)settled w.r.t. $S'$. The remaining neighbors of $u_1$ are not adjacent, hence $u_1$ is unsettled only if it forms a {\CTTNA} pattern with $u_3$ or $u_4$ (this pattern would then touch $S'$), since $u_3,u_4$ do not form a {\CU} pattern as $S$ satisfies property A.

Let us first consider the case where $S$ is a $K_4$-subdivision, and assume $u_3$ has $w_1,v$ as its remaining neighbors w.r.t. $S'$ (it has the same remaining neighbors w.r.t. $S$, since only the remaining neighbors of $u_1,u_2$ were modified). Then $\{u_1,w_1,v\}$ and $\{v_1,u_2=u_4,u_3\}$ induce a $K_{3,3}$-minor of $G$ (by contracting the path $u_2\sim u_4$ to a vertex, see Figure~\ref{fig:KTTOne}), contradicting the planarity of $G$.

\renewcommand{\MyScale}{0.85}
\KTTOne

Now let us take a look at the case where $S$ is a $C_{4+}$-subdivision and assume that $u_3$ or $u_4$ has $w_1,v$ as its remaining neighbors w.r.t. $S'$ (thus w.r.t. $S$). If $u_1,u_2$ are $1$-linked, $u_3$ or $u_4$ has a remaining neighbor in the solo $(u_1,u_2)$-path $P$, contradicting the property ``$1$-linked'' of $S$. If $u_1,u_2$ are $2$-linked, then $\{u_1,w_1,v\}$ and $\{v_1,u_2,u_3\}$ induce a $K_{3,3}$-minor (see Figure~\ref{fig:KTTTwo}), 
again a contradiction. Thus, $u_1$ cannot form a {\CTTNA} pattern, so forms a {\CV} pattern and is lone-settled.

\KTTTwo


\item {\redirXThree}:
The remaining neighbors $v,v_2$ of $u_2$ w.r.t. $S'$ belong to the new $(u_1,u_3)$-path $P'$ of $S'$ and the edge $vv_2$ belongs to $P'$, so $u_2$ does not have an A-chord nor a B-chord.

The remaining neighbors of $u_3,u_4$ are not modified by the redirection. The paths of $S$ incident with $u_4$ are not modified, so $u_4$ does not have an A-chord, and a B-chord of $u_4$ could only belong to the new path $P'$ of $S'$.
An A-chord of $u_3$ could only belong to $P'$, as its other incident paths of $S$ were not modified, and $u_3$ does not have a B-chord, as its non-incident paths of $S$ were not modified.
We claim that none of $u_3$ and $u_4$ have an A-chord or B-chord on $P'$ in $S'$ 
if they did not in $S$.

We examine the case of $u_1$ at the end of this proof.
Let us make several observations that will help us prove our claims.

Whether $S$ is a $K_4$- or a $C_{4+}$-subdivision, it contains a $(u_1,u_2)$-path $P_{12}$, a $(u_2,u_4)$-path $P_{24}$, a $(u_3,u_4)$-path $P_{34}$, as well as the path $P = P_{13}$, split into a $(u_1,v_2)$-section $Q_{13}$ and a $(v_2,u_3)$-section $T_{13}$, all these paths having no special vertex as internal vertex.

\begin{obs}
The special vertex $u_3$ does not have $v_1$ as a remaining neighbor w.r.t. $S$.
\label{obs:red1}
\end{obs}
\begin{proof}
If it were the case, then $G$ would contain a $K_{3,3}$-minor induced by $\{u_1,u_3,v\}$ and $\{u_2,v_1,v_2\}$, obtained by contracting $P_{12}$, $P_{34}$, $Q_{13}$ and $T_{13}$ to one edge, and $P_{24}$ to one vertex (see Figure~\ref{fig:ObsTwoPict}). This would contradict the planarity of $G$.
\end{proof}

\ObsTwoPict

\begin{obs}
The special vertex $u_3$ does not have $v$ as a remaining neighbor w.r.t. $S$.
\label{obs:red2}
\end{obs}
\begin{proof}
Because the paths $P_{24}$ and $P_{34}$ are disjoint from $v,v_1,v_2$, the vertices $u_1,u_3$ would then belong to two different regions of the plane delimited by the three edges $u_2v$, $u_2v_2$, $vv_2$ (see Figure~\ref{fig:ObsThreePict}). This is a contradiction with the almost $4$-connectivity of $G$.
\end{proof}

\ObsThreePict

\begin{obs}
The special vertex $u_4$ does not have two remaining neighbors (w.r.t. $S$) $v_4$ in $T_{13}$, and $v_4' \in \{v,v_1\}$.
\label{obs:red3}
\end{obs}
\begin{proof}
If $v_4'=v_1$, contracting the $(v_2,v_4)$-section of $T_{13}$ gives us a $K_{3,3}$-minor of $G$ induced by $\{u_1,v,u_4\}$ and $\{u_2,v_1,v_2\}$, a contradiction (see Figure~\ref{fig:obsfourpictone}). If $v_4'=v$, then by planarity $u_1,u_4$ must belong to two different regions of the plane delimited by the three edges $u_2v$, $vv_2$ and $u_2v_2$ (see Figure~\ref{fig:obsfourpicttwo}), again a contradiction to the almost $4$-connectivity of $G$.
\end{proof}

\ObsFourPict


\begin{obs}
The special vertex $u_4$ does not have $w_1,v$ as its remaining neighbors w.r.t. $S$ if $w_1,v$ are non-adjacent.
\label{obs:red4}
\end{obs}
\begin{proof}
If it were the case, there would be a $K_{3,3}$-minor in $G$ induced by $\{u_1,u_4,v_2\}$ and $\{u_2,w_1,v\}$ (see Figure~\ref{fig:ObsFivePict}), a contradiction with the planarity of $G$ (note that by assumption $w_1,v$ are non-adjacent, hence $w_1\neq v_2$).
\end{proof}

\ObsFivePict

Since only the $(u_1,u_3)$-path $P$ changes into the path $P'$, and the remaining neighbors of $u_3,u_4$ are not changed by the redirection,
$u_3$ (resp. $u_4$) may only have an A-chord (resp. B-chord) on the new path $P'$.
This $(u_1,u_3)$-path in incident with $u_3$ so $u_3$ does not have a B-chord on it, and no A-chord either by Observations~\ref{obs:red1} and \ref{obs:red2}.
Note that if $u_3$ has an A-chord in $S$, then $u_3$ forms a {\CV} pattern w.r.t. $S$, and the same pattern w.r.t. $S'$.

$P'$ is not incident with $u_4$, so $u_4$ does not have an A-chord on it, and Observation~\ref{obs:red3} tells us that it does not have a B-chord either.

As for $u_1$, note that in the version with $w_1,v_1$ non-adjdacent, the new remaining neighbors of $u_1$ are disjoint from $S$, thus $u_1$ forms a {\CV} or {\CTTNA} pattern disjoint from $S$ (since no pair of special vertices form a {\CU} pattern by property A), hence is settled. In the version with $w_1,v$ non-adjacent, the edge $u_1v$ is an A-chord of $u_1$, but Observations~\ref{obs:red2} and~\ref{obs:red4} tell us that neither $u_3$ nor $u_4$ can form a {\CTTNA} pattern with $u_1$. Again, since no pair of special vertices form a {\CU} pattern by property A of $S$, $u_1$ is left lone-settled by the redirection, and property A is satisfied by $S'$.
\end{itemize}
\end{proof}

\begin{claim}
%

Let $G$ be a planar graph with a $(C_{II})$ configuration w.r.t. a $4$-family $U$. Then $G$ contains a strong $\mathcal{K}$-subdivision rooted on $U$.
\label{clm:redir3r}
\end{claim}

\begin{proof}
%

By Lemma~\ref{lem:k4c4p} (p.~\pageref{lem:k4c4p}), $G$ contains a chordless $\mathcal{K}$-subdivision $S$ rooted on $U$. 
In particular, $S$ satisfies properties A and B. 
Since by Claim~\ref{clm:redirAB} (p.~\pageref{clm:redirAB}) the redirection procedure preserves properties A and B, and since the type of subdivision ($K_4$ or $C_{4+}$) is preserved by each redirection operation, then the result follows if $S$ is a $K_4$-subdivision.

Now assume that $S$ is a $C_{4+}^*$-subdivision.
%
We prove that after application of any redirection operation to $S$, the obtained subdivision $S'$ is a $C_{4+}^*$-subdivision, and the result follows by induction.
By Claim~\ref{clm:redirAB} (p.~\pageref{clm:redirAB}), $S'$ satisfies properties $A$ and $B$. 
Observe that the redirection operations preserve the ends of the paths of $S$; in particular, vertices that are $k$-linked stay $k$-linked after application of an operation.
Observe that in each redirection configuration, the two special vertices $u_1,u_2$ involved have a common remaining neighbor that is disjoint from $S$. Therefore, by property ``$2$-linked'', the two special vertices $u_1,u_2$ involved in the configuration cannot be $2$-linked: they are necessarily $1$-linked.

Let us prove the three properties of $C_{4+}^*$ of $S'$ in order.

\begin{itemize}
\item \emph{Property ``$0$-linked'':} The only special vertices whose remaining neighbors are modified by a redirection operation are the special vertices $u_1$ and $u_2$ involved in the redirection configuration. Therefore, if $u_j,u_k$ are $0$-linked special vertices, then their remaining neighbors w.r.t. $S'$ are the same as w.r.t. $S$, and property ``$0$-linked'' of $S'$ is implied by the same property of $S$.

\item \emph{Property ``$1$-linked'':} Let $u_i,u_j$ be $1$-linked special vertices associated with a path $P_{ij}$ of $S$, and $u_k\in U\setminus\{u_i,u_j\}$. Assume for contradiction that $u_k$ has a remaining neighbor $v_k$ w.r.t. $S'$ that is an internal vertex of the $(u_i,u_j)$-path $P_{ij}'$ of $S'$.
If $v_k$ is not a remaining neighbor of $u_k$ w.r.t. $S$, then the operation applied to $S$ involves $u_k$, and the edge $u_kv_k$ belongs to $S$. However, we can check in all redirection operations that when an edge $uv$ incident with a special vertex $u$ is removed from the subdivision, then $v$ does not belong to the subdivision after the operation is applied. This is a contradiction with the definition of $v_k$, therefore $v_k$ is indeed a remaining neighbor of $u_k$ w.r.t. $S$.
Thus, $P_{ij}'\neq P_{ij}$.

Since $u_k$ is $0$-linked with one of $u_i,u_j$, the vertex $v_k$ cannot be a remaining neighbor of both w.r.t. $S$, by ``$0$-linked'' property of $S$. So the operation applied to $S$ cannot be {\redirXOne} 
or {\redirXFour}, as their new path has length $2$.

The operation cannot be {\redirXThree} either, since in this case $u_i,u_j$ must be the vertices $u_1,u_3$ in the definition of this configuration (so that $P_{ij}'$ is the new path); $u_i,u_j$ are $1$-linked, and the operation is applied on the vertices $u_1,u_2$, which are also $1$-linked as mentioned above. Thus, $u_i$ or $u_j$ is $1$-linked to two different special vertices, a contradiction.

If the operation is {\redirXTwo}, then $v_k$ cannot be the common remaining neighbor $v$ of $u_1,u_2$ in the definition, since $u_k$ is $0$-linked to one of them.
If $v_k$ is the $v_1$ of {\redirXTwo}, 
then $\{u_k,v,w_1\}$ and $\{u_1,u_2,v_1\}$ induce a $K_{3,3}$-minor in $G$ (by contracting the $(w_1,u_2)$-section of $u_1\sim u_2$ to an edge, and contracting a $(u_k,u_2)$-path of $S$ or a $(u_k,u_m)$-path and a $(u_m,u_2)$-path of $S$ to an edge, for $u_m\in U\setminus \{u_1,u_2,u_k\}$), a contradiction to the planarity of $G$.


\item \emph{Property ``$2$-linked'':} Let $u_i,u_j$ be two $2$-linked special vertices.
Since the operations are applied on $1$-linked special vertices, exactly one of these two special vertices, say $u_i$, is involved in the operation (as $u_1$ or $u_2$ in the definitions) applied to $S$.
Let $v$ be a common remaining neighbor of $u_i,u_j$ w.r.t. $S'$.
If $v$ was not a common remaining neighbor of $u_i,u_j$ w.r.t. $S$, then it was a remaining neighbor of $u_j$ and not $u_i$, as $u_j$ is not involved in the operation: the edge $u_iv$ belongs to $S$, and more precisely belongs to a solo $(u_i,u_k)$-path $P_{ik}$ of $S$ between the two $1$-linked special vertices $u_i$ and $u_k\in U\setminus \{u_i,u_j\}$.
Thus, $u_j$ has a remaining neighbor (w.r.t. $S$) that is an internal vertex of a solo path of $S$, 
a contradiction with the ``$1$-linked'' property of $S$.

Therefore, the potential common remaining neighbors $v,v'$ of $u_i,u_j$ w.r.t. $S'$ are also their common remaining neighbors w.r.t. $S$. 
So by property ``$2$-linked'' of $S$, $u_i,u_j$ have at most one remaining neighbor $v$, and it belongs to a parallel $(u_k,u_l)$-path $P_{kl}$ of $S$ that is not incident with $u_i,u_j$.
If the path $P_{kl}$ was not modified by the operation, the result follows. If $P_{kl}$ was modified into a path $P_{kl}'$ of $S'$, then the operation is necessarily {\redirXThree} (as in the others, 
only a solo path is modified). The special vertex $u_i$ is the $u_2$ in the definition, as it keeps the same special neighbor $v_2$ in $S$ and $S'$. This vertex belongs to both $P_{kl}$ and $P_{kl}'$, and the result follows.
\end{itemize}

This proves that $S'$ is a $C_{4+}^*$-subdivision, and the result follows by induction.
\end{proof}

\ifthenelse{\equal{\isThesis}{true}}
{\section{Sufficiency of the $(C_{II})$ rules}}
{\subsection{Sufficiency of the $(C_{II})$ rules}}

As with the $(C_I)$ configurations, we need to make sure the composite rules we define in this 
{\chapsec}
can indeed be applied, and yield a contradiction with the existence of their associated configuration in an MCE.
Let us first prove that the rules we will define throughout this 
{\chapsec}
allow us to find a good coloring of an MCE, thus a contradiction, similarly to Lemma~\ref{lem:safety_ci} (p.~\pageref{lem:safety_ci}).

\begin{lem}
An MCE with a $4$-family $U$ does not contain a subdivision composite configuration made up of a semi-subdivision $S$ rooted on $U$ and a 
mapping $\mathcal{M}$ of $U$ compatible w.r.t. $S$.
\label{lem:safety_cii}
\end{lem}


\begin{proof}
Let $G$ be such an MCE, and assume it contains the composite configuration $X = (\{\C_1,\dots,\C_k\},S)$ where $\{\C_1,\dots,\C_k\}$ is a compatible mapping of $U$. The associated composite rule $\mathcal{R}_X = (X,f_X^r,f_X^c)$ is thus valid.
%
We build a good coloring $c$ of $G$ to show a contradiction.

Let us first build a coloring $pc$ of $G'$ using the right number of colors. Similarly to the proof of Lemma~\ref{lem:safety_ci} (p.~\pageref{lem:safety_ci}), let us color each $K_3$ component of $G'$ with a cycle of length $3$ and each $K_5^-$ component with a cycle of length $5$ and a path of length $4$. We color each other components with a good coloring. Thus, $pc$ uses the right number of colors ($\lfloor \frac{\vert V(G')\vert}{2} \rfloor$), with a mix of cycles and paths. Let $c_0 = f_X^c(G,pc)$ be the coloring of $G$ obtained from $pc$ by $\mathcal{R}_X$. Since $\mathcal{R}_X$ is valid, $c_0$ has the right number of colors ($\lfloor \frac{\vert V(G)\vert}{2} \rfloor$) again with some cycles instead of paths.
We build iteratively a good coloring $c$ of $G$, by starting from $c_0$ and using Lemma~\ref{L21} (p.~\pageref{L21}) to successively replace a pair of colors, inducing a path and a cycle, by another pair of colors inducing two paths.

Observe that the cycles induced by colors of $c_0$ are disjoint in $G$. Indeed, the cycles in $pc$ are disjoint because they belong to different connected components; the cycles of $pc$ may have been deviated into longer cycles in $c_0$, but since the internal vertices of the deviated sections are all special vertices, and since each special vertex is involved in at most one deviation, then no vertex of $G$ can belong to the intersecion of two cycles in $c_0$.

First let us prove that no $K_3$ connected component can appear in $G'$. Let $K$ be such a component, colored with a cycle in $pc$, and let $C$ be the cycle of $G$ induced by the same color in $c$, after some deviations. 
%
Observe that if one pattern involved in the deviations is {\CTOne}, {\CTTA} or {\CTTNA}, one vertex of $K$ is supposed to have odd degree in $G'$, which is impossible since $K$ is a $K_3$.
%
Since $K$ is a connected component of $G'$, the vertices of $C$ on $V(K)$ are only incident with edges of $G[V(K)]$, edges between special vertices and their remaining neighbors, and edges from the subdivision $S$. A semi-subdivision has at most one pair of intersecting paths, sharing exactly one vertex. Hence, there are at least two vertices in $V(K)$ which are incident with $2$ edges of $C$ and $0$ or $2$ edges of $S$, and so these two vertices have a degree of at most $4$ in $G$. This contradicts Lemma~\ref{lem:ci}, so no $K_3$ component can be created in $G'$.\\ 


Now let $K$ be a $K_5^-$ component of $G'$, and let $C',P'$ be the cycle and path coloring it in $G'$.
Let $C,P$ be the cycle and path induced in $c_0$ by the same colors as $C',P'$. $C$ has not been treated yet and is thus induced by the same color in $c$ as in $c_0$. If $P = P'$, then it is disjoint from the cycles treated in previous iterations of $c$, and thus has not been involved in a replacement of a pair of colors. Otherwise, $P$ results from deviations of $P'$ on special vertices, or an extension of $P'$ by at most two edges (one for each of its endpoints) to special vertices. If $u$ is a special vertex touching $P$, since $u$ belongs to exactly one pattern of the mapping of $X$, then $u$ does not touch a cycle treated in any previous iteration of $c$. Thus $P$ is again disjoint from the cycles treated in the previous iterations. In both cases, $P$ is induced by the same color in $c$ as in $c_0$. 

Observe that $V(C)\cap V(P) \subseteq V(K)$, so $\vert V(C)\cap V(P)\vert \leq 5$. We distinguish between three cases:
\begin{itemize}
\item \textbf{$C$ results from a deviation of $C'$:} then its length is different from $5$. By Observation~\ref{obs1} (p.~\pageref{obs1}), $C\cup P$ does not form the exceptional graph.
\item \textbf{$C'$ has not been deviated, but $P'$ has:} then $V(C) = V(K)$, and there is an edge of $K$ that does not belong to $(C\cup P)[V(C)]$. Thus $(C\cup P)[V(C)] \varsubsetneq (C'\cup P')[V(C')] = K$, and so $(C\cup P)[V(C)]$ does not form a $K_5^-$, and by Observation~\ref{obs1} (p.~\pageref{obs1}), $C\cup P$ does not form the exceptional graph.
\item \textbf{Neither $C'$ nor $P'$ have been deviated:} then $(C\cup P)[V(C)] = K$, which is a $K_5^-$, and so $E(K)\subseteq E(G)$.
The edges of $K$ split $G$ into $6$ regions bounded by triangles.
Since the graph is almost $4$-connected w.r.t. all the special vertices, they belong to the same region of the graph.
Hence at most $3$ vertices from $K$ are neighbors of special vertices. Since there is at most one vertex of degree at most $4$ in $G$, all vertices from $K$, except maybe one, have their degree changed between $G$ and $G'$. Only $3$ of them can belong in patterns, thus at least one of them is touched by a path $Q$ induced by a color of $c$ and different from $P$. Since $C\cup P$ forms the exceptional graph, by Observation~\ref{obs1} (p.~\pageref{obs1}) $C\cup Q$ does not, and since $C'$ has not been deviated, $\vert C\vert = 5$, and so $\vert V(C)\cap V(Q)\vert \leq 5$.
\end{itemize}

In all three cases, Lemma~\ref{L21} (p.~\pageref{L21}) gives us a decomposition of $C\cup P$ or $C\cup Q$ into two new paths $Q',Q''$. We replace $C,P$ or $C,Q$ in $c$, depending on the case, with $Q',Q''$.

The coloring $c$ contains the same number of colors as $c_0$ in all iterations, with one less cycle at each iteration.
When all $K_3$ and $K_5^-$ components have been treated, the resulting coloring $c$ is a good path decomposition of $G$, a contradiction.
\end{proof}


In order to show a contradiction, we now need to prove that an MCE containing a configuration $(C_{II})$ indeed contains a subdivision composite configuration made up of a semi-subdivision rooted on its $4$-family and a compatible mapping.

\begin{lem}\label{lem:confred}

Let $G$ be a planar graph with a $(C_{II})$ configuration w.r.t. a $4$-family $U$, that admits a strong $\mathcal{K}$-subdivision rooted on $U$. Then $G$ contains a subdivision composite configuration made up of a semi-subdivision $S$ rooted on $U$ and a compatible mapping w.r.t. $S$.
\end{lem}

Note that the subdivision is a semi-subdivision and so it may have missing edges or having two paths crossing, but it is $2$-colorable, and therefore sufficient to produce a good coloring of $G$ if $G$ is an MCE, and therefore show a contradiction.

The rest of this 
{\chapsec}
constitutes a proof of this lemma, before the conclusion in which we show that Lemma~\ref{lem:cii} (p.~\pageref{lem:cii}) ensues.

\ifthenelse{\equal{\isThesis}{true}}
{\section{Distant problems}}
{\subsection{Distant problems}}

In the following, we prove that the graph admits a composite configuration made up of a subdivision and a set of compatible patterns.
We will distinguish two types of ``problems'' that could occur and prevent us from applying directly a reduction 
$\{${\CN}$,${\CN}$,${\CN}$,${\CN}$\}$%
. First, a {\CN} pattern could cause a ``distant problem'' by touching a path of the subdivision, and the associated reduction rule could create a cycle in the coloring. Then, some special vertices from $U$ could cause a ``close problem'' by sharing some of their remaining neighbors and the {\CN} patterns would not be compatible. We first treat the cases with at least $3$ distant problems (Lemma~\ref{lem:dist1}, p.~\pageref{lem:dist1}), then the cases with at most $2$ distant problems and no close problems (Lemma~\ref{lem:dist2}, p.~\pageref{lem:dist2}) and 
finally the cases with at most $2$ distant problems and some close problems (Lemma~\ref{lem:close}, p.~\pageref{lem:close}).


\begin{defn}[Distant problem]
Let $G$ be a planar graph with a $4$-family $U$ and let $S$ be a strong $\mathcal{K}$-subdivision rooted on $U$. Let $u\in U$ be a special vertex and $P$ be a path of $S$ that is not incident with $u$. We say that $u$ causes a \emph{distant problem} on $P$ if the three conditions are satisfied:
\begin{itemize}
\item $u$ has two adjacent remaining neighbors $v,v'$ that are disjoint from $U$;
\item exactly one of its remaining neighbors belongs to $P$;
\item if some other special vertex $u'$ has one of $v,v'$ as a remaining neighbor, then 
$u'$ is settled.
\end{itemize}
\end{defn}

\renewcommand{\MyScale}{1.1}

\PbDist


Figure~\ref{fig:pbdist} provides an example of distant problem. Only the path $P$ from the definition is represented.

This definition is motivated by the fact that because of property B, an unsettled special vertex $u$ cannot have both its remaining neighbors belong to a path $P$ not incident with $u$.

Note that if three special vertices cause distant problems, the fourth one is lone-settled, because of the last condition in the definition of distant problem and by properties A and B.

Once we $2$-color the subdivision $S$ such that each vertex of $U$ is at the end of a color, a distant problem is \emph{active} if the color that ends on $u$ is the same as the color of the path $P$. A distant problem is otherwise \emph{inactive}. We can treat inactive distant problems as {\CN} patterns, as 
the coloring fits the color requirements of the pattern (i.e. the color requirements of the two patterns \ncf{\cfCvTwo} and \ncf{\cfCvThree}).
Since the distant problems are caused by special vertices $u$ which have their remaining neighbors $v,v'$ adjacent, we refer to $\{u,v,v'\}$ as the \emph{triangle of $u$}. 


We prove here two results that are convenient for the rest of the proof.

\begin{claim}
A $K_4$-subdivision can be $2$-colored so as to inactivate up to $2$ distant problems that are on different paths of $S$.
\label{clm:inact}
\end{claim}
\begin{proof}
Let $S$ be a $K_4$-subdivision rooted on $\{u_1,u_2,u_3,u_4\}$. Let us assume w.l.o.g. that $u_1$ causes a distant problem on the path $u_2\sim u_3$. The following $2$-coloring of $S$ inactivates the distant problem of $u_1$: $\{(u_3\rightarrow u_2\rightarrow u_1\rightarrow u_4), (u_1\rightarrow u_3\rightarrow u_4\rightarrow u_2)\}$.

If $S$ has a second distant problem on a different path, then there are several possible cases. Either $u_3$ causes a distant problem on $u_2\sim u_4$ (case A) or on $u_1\sim u_4$ (case B). These cases are symmetric with the ones where $u_2$ causes a distant problem on $u_3\sim u_4$ and $u_1\sim u_4$ respectively. If instead the second distant problem is caused by $u_4$ on (w.l.o.g.) $u_1\sim u_3$, then this is equivalent to case A: just replace $(1,2,3,4)$ with $(3,4,2,1)$ to obtain case A. The coloring of the previous case inactivates the two distant problems of case A, and the coloring $\{(u_4\rightarrow u_1\rightarrow u_3\rightarrow u_2),(u_1\rightarrow u_2\rightarrow u_4\rightarrow u_3)\}$ of $S$ inactivates those of case B.
\end{proof}


\begin{claim}
Let $G$ be a planar graph with a $(C_{II})$ configuration w.r.t. a $4$-family $U$ and a strong $C_{4+}^*$-subdivision rooted on $U$, such that $G$ does not have a $K_4$-subdivision rooted on $U$.
If $u_i,u_j\in U$ are $1$-linked, then at most one of them has a remaining neighbor belonging to $S$.
\label{clm:noCH}
\end{claim}

\begin{proof}
By property ``$1$-linked'' of $S$ and property A, the remaining neighbors of $u_i,u_j$ are on parallel paths of $S$. Let us assume for contradiction that the $1$-linked special vertices $u_1,u_2$ are such that $u_1$ has a remaining neighbor $v_1$ on a $(u_2,u_4)$-path $P_2$ of $S$ and $u_2$ has a remaining neighbor $v_2$ on a $(u_1,u_3)$-path $P_1$ of $S$: let $P_1 = (u_1,Q_1,v_2,Q_3,u_3)$ and $P_2 = (u_2,Q_2,v_1,Q_4,u_4)$.

Let $S'$ be the set of paths of $S$ different from $P_1,P_2$, and let $P_1' = (u_1,v_1,Q_4,u_4)$ and $P_2' = (u_2,v_2,Q_3,u_3)$, as depicted on Figure~\ref{fig:CH}. These $(u_1,u_4)$-path and $(u_2,u_3)$-path are internally disjoint from the paths of $S'$, hence $S'\cup \{P_1',P_2'\}$ forms a $K_4$-subdivision of $G$ rooted on $U$, a contradiction. Hence, two $1$-linked special vertices cannot both have a remaining neighbor in $S$.

\renewcommand\MyScale{1}
\begin{figure}[ht]
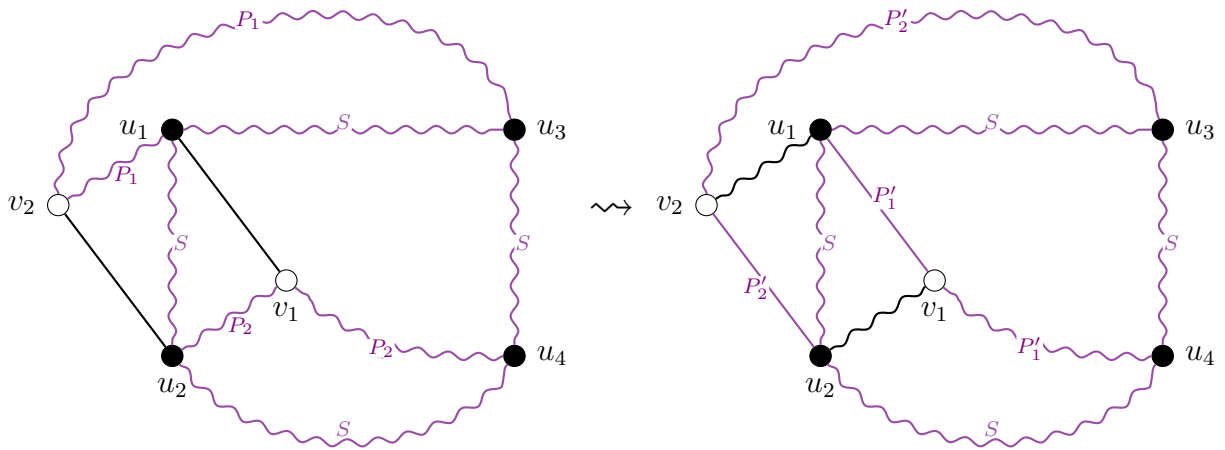

\center
\RuleCH
\caption{Display of the underlying $K_4$-subdivision of $G$ rooted on $U$}
\label{fig:CH}
\end{figure}

Finally, $S$ cannot have three distant problems or more, as two of them would be caused by a pair of $1$-linked special vertices.
\end{proof}


We define below the ``distant configurations'' that correspond to subdivisions with at least $3$ distant problems. We first introduce a \textit{routing operation} that helps us take care of these distant problems. For each distant configuration, we perform a redirection of the subdivision $S$ into a subdivision $S'$, to turn each special vertex
causing a distant problem into a settled one. When $S$ has a distant problem caused by $u_i$ and $S'$ has a new path of the form $(u_i,v_i,Q,u_j)$, the goal is to turn the distant problem on $u_i$ into a {\CV} pattern.
To do so, we use the following operation.

\begin{defn}[Routing operation]
Let $G$ be a planar graph with a $4$-family $U$ and a strong $\mathcal{K}$-subdivision $S$ rooted on $U$.
Let $u\in U$ have two adjacent remaining neighbors $v_1,v_2$ w.r.t. $S$, and let $w$ be a neighbor of $u$ that belongs to $S$. Assume that $u$ causes a distant problem, with $v_1$ touching a path of $S$. One of $v_1,v_2$, say $v'$, is not adjacent to $w$, otherwise $\{u,v_1,v_2,w\}$ would form an induced $K_4$ in the graph, which contradicts the fact that $G$ has a $(C_{II})$ configuration by Claim~\ref{clm:inducedK4} (p.~\pageref{clm:inducedK4}).

The paths of $S$ are redirected to create a new subdivision $S'$, containing a path $P' = (u,v_1,Q,u')$, and such that $w$ does not belong to $S'$. Applying the \emph{routing operation} on $u$ consists in replacing the edge $uv_1$ in $P'$ by the edges $uv_2,v_2v_1$ if $v'=v_1$, and leaving $P'$ as it is otherwise.
\label{def:routing}
\end{defn}

\renewcommand\MyScale{1.1}
\RuleRouting

\ifthenelse{\equal{\isThesis}{true}}
{\vfill
\pagebreak}
{}

This routing operation ensures that the two remaining neighbors of $u$ in $S'$ are $w,v'$ and are thus non-adjacent. 
The vertex $u$ now forms a {\CV} pattern w.r.t. $S'$ and 
we justify for each application of the routing operation that $u$ is left settled.
For all cases, we provide a subdivision $S$ and describe a mapping of compatible patterns that settles all vertices.\\


\ifthenelse{\equal{\isThesis}{true}}
{}
{\vfill
\pagebreak}

\textbf{List of the distant configurations:}

\textit{The \emph{distant configurations} are the configurations \ncf{\DOne}, \ncf{\DTwo}, 
\ncf{\DThree}, \ncf{\DFour} listed below.}


\textit{Each configuration describes a $4$-family $U$ and a strong $\mathcal{K}$-subdivision $S$, such that at least $3$ special vertices of $U$ cause a distant problem on $S$. For each configuration, we describe a new semi-subdivision $S'$. The routing operation is applied to $S'$ for all unsettled special vertices.}

\textit{We provide for each configuration a subdivision composite rule 
made up of {\CV} and {\CN} patterns.
We justify for each rule that the mapping is compatible w.r.t. $S'$.}

\ifthenelse{\equal{\isThesis}{true}}
{\vfill
\pagebreak}
{\ \\}

\AllDistantConf

The following lemma shows how a planar graph with $(C_{II})$ configuration can be treated with one of the distant configurations if the associated subdivision has at least three distant problems.

\begin{lem}[Distant lemma]
Let $G$ be a planar graph with a $(C_{II})$ configuration w.r.t. a $4$-family $U$, with a strong $\mathcal{K}$-subdivision $S$ rooted on $U$. If $G$ has at least $3$ distant problems w.r.t. $S$, then $G$ contains a configuration among $\{$\ncf{\DOne}, \ncf{\DTwo}, \ncf{\DThree}, \ncf{\DFour}$\}$.
\label{lem:dist1}
\end{lem}

\begin{proof}
If $S$ is a $C_{4+}$-subdivision, we may assume that $G$ does not have a $K_4$-subdivision rooted on $U$. Then Claim~\ref{clm:noCH} (p.~\pageref{clm:noCH}) tells us that $S$ cannot have $3$ distant problems or more, which contradicts our hypothesis. Hence $S$ is a $K_4$-subdivision.

Let us call \emph{$i$-path} a path of $S$ that touches exactly $i$ triangles of special vertices. We consider the three quantities, for $i\in\{0,1,2\}$, $p_i := \vert\{i$-paths$\}\vert$.
By property A, a special vertex $u$ can only cause a distant problem on one of the three paths of $S$ that are not incident with it: we call these paths the \emph{potential paths} of $u$. For the same reasons, a path of $S$ can touch at most $2$ triangles of special vertices.

We have $p_0+p_1+p_2=6$ and $p_1+2\cdot p_2 =$ number of distant problems $= 3$ or $4$. Hence we need to consider five cases, depending on whether there are $3$ or $4$ distant problems and on the values the $p_i$ parameters.

\begin{itemize}
\item \textbf{$3$ distant problems, $p_1=3$, $p_2=0$}.
	We have $p_0 = 3$. If all three $0$-paths are incident with say $u_4$, then there cannot be three $1$-paths. Indeed, $u_3$ would have to cause a distant problem on $u_1\sim u_2$, then by planarity none of $u_1,u_2$ could cause a distant problem on $u_2\sim u_3, u_1\sim u_3$ respectively.

	Now let us assume that the three $0$-paths form a subdivision of a triangle rooted on $\{u_1,u_2,u_3\}$. Thus $u_4$ cannot cause any distant problem. The three $1$-paths are $u_1\sim u_4, u_2\sim u_4, u_3\sim u_4$. Let us say w.l.o.g. that $u_3$ causes a distant problem on $u_1\sim u_4$. Thus by planarity $u_2$ causes a distant problem on $u_3\sim u_4$, and $u_1$ causes a distant problem on $u_2\sim u_4$. This is the configuration \ncf{\DOne}.

	Let us finally assume that the three $0$-paths form a subdivision of a path on three edges, rooted on $U$. Let us say that the $0$-paths are $u_1\sim u_4, u_3\sim u_4, u_2\sim u_3$. The path $u_2\sim u_4$ touches the triangle of either $u_1$ or $u_3$.
	If it touches the triangle of $u_3$, then $u_1$ cannot cause any distant problem (as its two other potential paths are $0$-paths), thus in this case both $u_2$ and $u_4$ cause a distant problem.
	There is only one possibility for $u_2$: it causes a distant problem on $u_1\sim u_3$; then there is only one possibility for $u_4$, the path $u_1\sim u_2$. This is the configuration \ncf{\DTwo}. Now assume that instead, $u_2\sim u_4$ touches the triangle of $u_1$. By planarity and property A, only $u_3$ can touch the $1$-path $u_1\sim u_2$. Again by planarity, only $u_4$ can touch the $1$-path $u_1\sim u_3$. 
This case is equivalent to \ncf{\DTwo}: $(u_1,u_2,u_3,u_4)$ in \ncf{\DTwo} correspond to $(u_2,u_1,u_4,u_3)$ in this case, in this order.

\item \textbf{$3$ distant problems, $p_1=1$, $p_2=1$}. We have $p_0 = 4$: the four $0$-paths can only either form a subdivision of the ``paw'' graph (a triangle with an additional edge attached to one vertex) or a subdivision of the cycle on four vertices. We can easily see that the first case is impossible: let us say the non-$0$-paths are $u_1\sim u_2, u_1\sim u_3$; there must be a path that touches two triangles of $U$, say it is $u_1\sim u_2$, that necessarily touches the triangles of $u_3$ and $u_4$. Then by planarity, the path $u_1\sim u_3$ cannot touch the triangle of $u_2$ and thus cannot be a $1$-path. Hence the $0$-paths cannot form a paw.

Now let us assume that the $0$-paths are $u_1\sim u_3, u_2\sim u_3, u_1\sim u_4, u_2\sim u_4$. Let us assume w.l.o.g. that each of $u_1, u_2$ causes a distant problem on $u_3\sim u_4$, and $u_3$ causes a distant problem on $u_1\sim u_2$.
This case can be treated as configuration \ncf{\DThree}.

\item \textbf{$4$ distant problems, $p_1=4$, $p_2=0$}. We have $p_0 = 2$, so the two $0$-paths can either be incident on one vertex or disjoint. Let us consider the first case. Assume that $u_1\sim u_4, u_2\sim u_4$ are the $0$-paths; the other four paths are $1$-paths and must each touch one triangle of $U$, so all of $u_1,u_2,u_3,u_4$ cause a distant problem. The vertex $u_3$ causes a distant problem on $u_1\sim u_2$, as it is its only potential path. Then w.l.o.g. $u_4$ causes a distant problem on $u_2\sim u_3$, and $u_1$ can only cause a distant problem on $u_3\sim u_4$. Finally, by planarity the triangle of $u_2$ cannot touch the path $u_1\sim u_3$, its only potential path left. Hence, the $0$-paths cannot be incident.

Now let us assume that the two $0$-paths are $u_1\sim u_2, u_3\sim u_4$. Assume w.l.o.g. that $u_1$ causes a distant problem on $u_2\sim u_4$. Then $u_3$ causes a distant problem on $u_1\sim u_4$ as it is its last potential path. In the same way, $u_2$ causes a distant problem on $u_1\sim u_3$ and $u_4$ on $u_2\sim u_3$. This is the configuration \ncf{\DFour}.

\item \textbf{$4$ distant problems, $p_1=2$, $p_2=1$}. Let us assume that $u_1$ and $u_2$ cause distant problems on the same $2$-path $u_3\sim u_4$. By planarity, the triangle of $u_3$ can only reach the path $u_1\sim u_2$, and so does the triangle of $u_4$. Therefore, there cannot be two distinct $1$-paths. Hence, this case is impossible.

\item \textbf{$4$ distant problems, $p_1=0$, $p_2=2$}. If we assume that the path $u_3\sim u_4$ touches the triangles of $u_1$ and $u_2$, then necessarily the path $u_1\sim u_2$ touches the triangles of $u_3$ and $u_4$. This is again configuration \ncf{\DThree}.
\end{itemize}

This concludes the proof.
\end{proof}

\ifthenelse{\equal{\isThesis}{true}}
{\section{Semi-distant configurations}}
{\subsection{Semi-distant configurations}}

We can now focus on the cases where the subdivision has up to $2$ distant problems.
%
%
%
Let us define another type of problem that we have to deal with in order to finish the proof of Lemma~\ref{lem:confred} (p.~\pageref{lem:confred}).

\begin{defn}[Close problem]
Let $G$ be a planar graph with a $4$-family $U$ and let $S$ be a $\mathcal{K}$-subdivision rooted on $U$.
A special vertex $u \in U$ causes a \emph{close problem} if it is unsettled w.r.t. $S$ and shares at least one of its remaining neighbors with at least one other unsettled special vertex.
\end{defn}

Note that by definition there are either zero or at least two special vertices causing a close problem; there cannot be a single special vertex causing a close problem on its own.
Also, note that by definition, an unsettled special vertex that does not cause a distant nor a close problem forms a {\CN} pattern that is disjoint from $S$ and that touches only {\CV} patterns, 
hence its reduction rule can be applied safely.

Let us first deal with subdivisions that have at most $2$ distant problems and no close problem, with the following \emph{semi-distant configurations} and their associated subdivision composite rules. We will then deal with subdivisions with close problems in 
\ifthenelse{\equal{\isThesis}{true}}
{Section~\ref{subs:close}.\\}
{Subsection~\ref{subs:close}.\\}

\textbf{List of the semi-distant configurations:}

\textit{The \emph{semi-distant configurations} are the configurations \ncf{\JOne}, \ncf{\JTwo}, \ncf{\JThree}, \ncf{\JFour}, \ncf{\JFive}, \ncf{\JSix} listed below.}

\textit{Each configuration describes a $4$-family $U$ and a strong $\mathcal{K}$-subdivision $S$, such that at most $2$ special vertices of $U$ cause a distant problem on $S$, and none cause close problems. For each configuration, we describe a new semi-subdivision $S'$. The routing operation is not applied to $S'$ unless stated otherwise.}

\textit{We provide for each configuration a subdivision composite rule.
We justify for each rule that the mapping is compatible w.r.t. $S'$.}

\vfill
\pagebreak

\AllSemiDistantConf

The following lemma shows that we can treat any subdivision that has at most $2$ distant problems and no close problem with one of the semi-distant configurations.

\begin{lem}[Semi-distant lemma]
Let $G$ be a planar graph with a $(C_{II})$ configuration w.r.t. a $4$-family $U$, with a strong $\mathcal{K}$-subdivision $S$ rooted on $U$. If $G$ has at most $2$ distant problems and no close problem w.r.t. $S$, then $G$ contains a configuration among $\{$\ncf{\JOne}, \ncf{\JTwo}, \ncf{\JThree}, \ncf{\JFour}, \ncf{\JFive}, \ncf{\JSix}$\}$.
\label{lem:dist2}
\end{lem}

\begin{proof}
Let us consider the case where $S$ is a $K_4$-subdivision. If it has at most one distant problem, or two distant problems on different paths of $S$, then this is configuration \ncf{\JOne}.
If $u_1,u_2\in U$ both cause distant problems on the same path $u_3\sim u_4$ of $S$, then we distinguish between $3$ cases. Let $v_1,v_2$ be the remaining neighbors of $u_1,u_2$ respectively that are on the path $u_3\sim u_4$, and let $v_1',v_2'$ be their other remaining neighbors. If $l(u_1\sim u_2) = 1$ and $u_3$ or $u_4$ has both $v_1',v_2'$ as remaining neighbors, then this is configuration \ncf{\JThree}. If $l(u_1\sim u_2) = 2$, with $w$ as the middle vertex, $w$ is adjacent to $u_3$ and $u_4$, and $v_1',v_2'$ are remaining neighbors of $u_3$ or $u_4$, then this is configuration \ncf{\JFour}. Otherwise, this is configuration \ncf{\JTwo}.

Now let us consider the case where $S$ is a $C_{4+}$-subdivision. 
By property ``$1$-linked'' and property A of $S$, the distant problems occur on parallel paths of $S$.
Thus, if there is at most one distant problem, this is configuration \ncf{\JFive}.
By Claim~\ref{clm:noCH} (p.~\pageref{clm:noCH}), two distant problems cannot be caused by $1$-linked special vertices.
Therefore, if there are two distant problems on different paths of $S$, then this is configuration \ncf{\JFive}. If there are two distant problems caused by (w.l.o.g.) $u_1,u_3$ on the same parallel $(u_2,u_4)$-path of $S$, by Claim~\ref{clm:noCH} (p.~\pageref{clm:noCH}) the remaining neighbors of $u_2,u_4$ are disjoint from $S$, and this is configuration \ncf{\JSix}.
This concludes the proof.
\end{proof}

\begin{figure}[ht]
		\tikzstyle{every node}=[thick,anchor=west,align=left]
		\begin{tikzpicture}
		  [grow via three points={one child at (0.5,-0.7) and
		  two children at (0.5,-0.7) and (0.5,-1.4)},
		  edge from parent path={(\tikzparentnode.south) |-(\tikzchildnode.west)}]
		  \node [] {$K_4$}
		  		child { node [] {$1$ distant problem: \fbox{\ncf{\JOne}}}}
		  		child { node [] {$2$ distant problems on different paths: \fbox{\ncf{\JOne}}}}
		  		child { node [] {$2$ distant problems on the same path:}
		  			child { node [] {General case: \fbox{\ncf{\JTwo}}}
		  			}
		  			child [missing] {}
		  			child { node [] {Specific case \#1\\(contact {\CU}$ + ${\CV}): \fbox{\ncf{\JThree}}}
		  			}
		  			child [missing] {}
		  			child { node [] {Specific case \#2\\(contact {\CTTNA}$ + ${\CTTNA}): \fbox{\ncf{\JFour}}}
		  			}
		  		}
		  ;

		  \node [] at (0,-6.0) {$C_{4+}$}
		  		child { node [] {$\leq 1$ distant problem: \fbox{\ncf{\JFive}}}}
				child { node [] {$2$ problems}
					child { node [] {On different paths: \fbox{\ncf{\JFive}}}}
					child { node [] {On the same path: \fbox{\ncf{\JSix}}}}
				}
		  ;
		\end{tikzpicture}
		\caption{Semi-distant lemma trees of cases}
\end{figure}
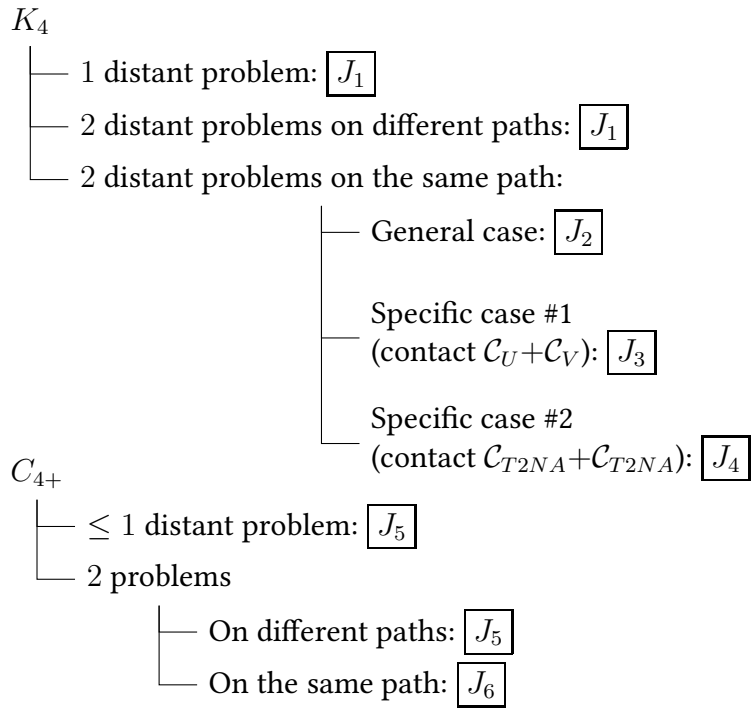

\vfill
\pagebreak

\ifthenelse{\equal{\isThesis}{true}}
{\section{Close configurations}}
{\subsection{Close configurations}}
\label{subs:close}

For simplicity, we define some macros that encapsulate several patterns and configurations from the redirection procedure.

\begin{itemize}
\item {\CDOne}: In this configuration, $u_1,u_2\in U$ are linked by a path $u_1\sim u_2$ in $S$. The vertices $u_1,u_2$ have a common remaining neighbor $v$ and have another remaining neighbor $v_1,v_2$ respectively, both adjacent to $v$. The vertices $v$ and $v_1$ are disjoint from $S$, and if $v_2$ is in $S$, it belongs to a path of $S$ incident with $u_1$ and not $u_2$.\\
First let us assume that $v_2$ is not in $S$. If the path $u_1\sim u_2$ has length $1$, then it is a {\CDA} pattern if $v_1,v_2$ are not adjacent, or a {\CDB} pattern if they are. If $u_1\sim u_2$ has length at least two, then this is forbidden by the redirection procedure, as this is a {\redirXOne} or {\redirXTwo} configuration depending on whether $v_1$ is adjacent to the neighbor $w_1$ of $u_1$ on $u_1\sim u_2$.\\
Now if there is a path $u_1\sim u'$ in $S$ that touches $v_2$, with $u'\neq u_2$, then it is a {\redirXThree} configuration, forbidden by property C.

\item {\CDTwo}: The vertices $u_1,u_2$ are linked by a path $u_1\sim u_2$ in $S$. The vertices $u_1,u_2$ have two remaining neighbors $v,v'$ in common.
No path of $S$ touches $v,v'$.\\
If $v,v'$ are not adjacent, then this is \ncf{\cfCTTNAa} or \ncf{\cfCTTNAb} depending on the parity of $v,v'$. Now assume $v,v'$ are adjacent, and let $l$ be the length of $u_1\sim u_2$. If $l=1$, then $\{u_1,u_2,v,v'\}$ form an induced $K_4$, contradicting the fact that $G$ has a $(C_{II})$ configuration by Claim~\ref{clm:inducedK4} (p.~\pageref{clm:inducedK4}). 
Then $l\geq 2$ and this is configuration {\redirXFour} from the redirection procedure, hence forbidden by property C.
%
\end{itemize}

To summarize, apart from forbidden configurations removed by the redirection procedure, a {\CDOne} macro is a {\CDA} or {\CDB} pattern.
A {\CDTwo} is a {\CTTNA} pattern and in this case $u_1,u_2$ are therefore settled if no unsettled special vertex shares remaining neighbors with them.
See Figure~\ref{fig:Macros}.

\begin{figure}[ht]
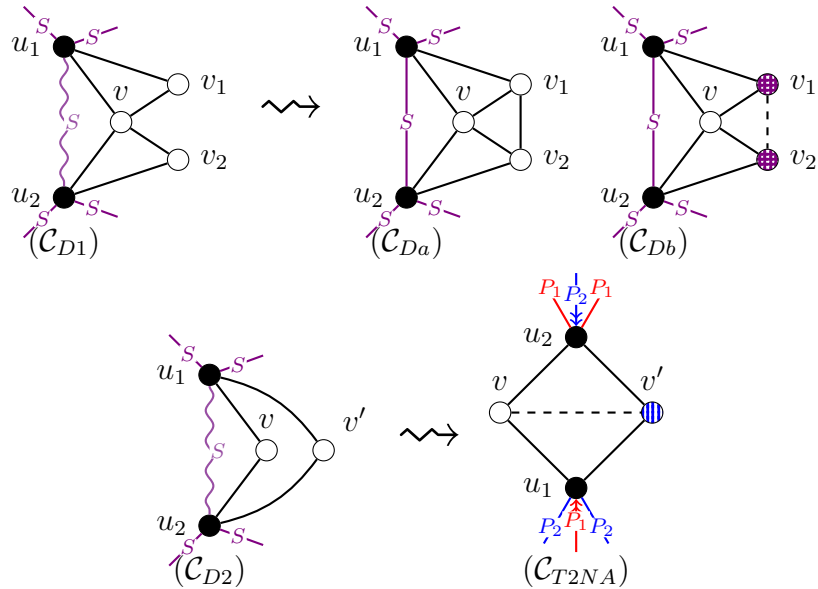


\renewcommand{\MyScale}{1}
\center{
\cfCDOnene
\tkcf{}
\flecheACentrer{1.5}
\cfCDOnenea
\tkcf{}
\cfCDOneneb
\tkcf{}

\cfCDTwowo
\tkcf{}
\flecheACentrer{1.5}
\cfCTTNA
\tkcf{}
}
\caption{Possible patterns for each macro}
\label{fig:Macros}

\end{figure}

Let us now introduce the remaining configurations, with which we treat all the cases of $\mathcal{K}$-subdivisions with close problems.

\vfill
\pagebreak

\textbf{List of the close configurations:}

\textit{The \emph{close configurations} are the configurations
\ncf{\RuOne}, \dots,
\ncf{\RuNine}
listed below.}

\textit{Each configuration describes a strong $\mathcal{K}$-subdivision $S$ rooted on a $4$-family $U = \{u_1,u_2,u_3,u_4\}$, such that at most two special vertices cause distant problems, and some special vertices cause close problems.}
\textit{We describe for each a subdivision composite rule made up of a semi-subdivision $S'$ (if not specified, $S'=S$) and a compatible mapping w.r.t. $S'$.}\\

\textbf{Remark:} When two remaining neighbors of a special vertex form a {\CN} pattern or a {\CVp} pattern, we denote it by {\CN} for simplicity. 
This does not change the case analysis.


\ \\

\AllCloseConf

Before entering the proof of the final lemma of this 
{\chapsec},
let us show a useful claim.

\begin{claim}
Let $G$ be a planar graph with a $(C_{II})$ configuration w.r.t. a $4$-family $U$, with a $\mathcal{K}$-subdivision $S$ rooted on $U$.

If $u\in U$ has no remaining neighbors in common with other special vertices, then either $u$ causes a distant problem in $S$ or $u$ is lone-settled.
\label{clm:isol}
\end{claim}

\begin{proof}
Let us assume that $u$ does not cause a distant problem in $S$. Since it does not share remaining neighbors with other special vertices, it cannot form a {\CTTNA} pattern w.r.t. $S$. Hence, since no pair of special vertices forms a {\CU} pattern by property A, if the remaining neighbors of $u$ are non-adjacent, $u$ forms a {\CV} pattern and is lone-settled.

If its remaining neighbors are adjacent, since $u$ does not share remaining neighbors and does not cause a distant problem, either none or both of its remaining neighbors belong to $S$. Then either its remaining neighbors are disjoint from $S$ and $u$ forms a {\CN} pattern, or by property B and planarity $u$ has both remaining neighbors in $S$ and forms a {\CVp} pattern. In both cases, it is lone-settled.
\end{proof}

We can now show how all the remaining cases of $\mathcal{K}$-subdivisions with close problems can be taken care of with the previous close configurations.


\begin{lem}[Close lemma]
\label{lem:close}
Let $G$ be a planar graph with a $(C_{II})$ configuration w.r.t. a $4$-family $U$, with a $\mathcal{K}$-subdivision $S$ rooted on $U$, with at most $2$ distant problems and some close problems w.r.t. $S$. Then $G$ has a configuration among 
$\{$%
\ncf{\RuOne},
\ncf{\RuTwo},
\ncf{\RuThree},
\ncf{\RuFour},
\ncf{\RuFive},
\ncf{\RuSix},
\ncf{\RuSeven},
\ncf{\RuEight},
\ncf{\RuNine},
\ncf{\JFour}%
$\}$.
\end{lem}

\begin{proof}
We start by solving \textbf{the $C_{4+}$ cases}: see Figure~\ref{C4tree} for the tree of cases.


If $S$ is a $C_{4+}^*$-subdivision, let $U = \{u_1,u_2,u_3,u_4\}$ be such that $u_1,u_2$ are $1$-linked, and $u_1,u_3$ are $2$-linked. 
We can assume that $G$ does not have a $K_4$-subdivision rooted on $U$.

If $S$ has distant problems, they are on parallel paths by property ``$1$-linked'' and property A.
By Claim~\ref{clm:noCH} (p.~\pageref{clm:noCH}), $S$ does not have two distant problems caused by $1$-linked special vertices.
If $S$ has two distant problems caused by two $0$-linked special vertices, then by property ``$0$-linked'', the other two $0$-linked special vertices do not share remaining neighbors, and are thus lone-settled by definition of close problem and Claim~\ref{clm:isol} (p.~\pageref{clm:isol}). So $S$ does not have close problems, which is a contradiction. 
%
If $S$ has two distant problem caused by $2$-linked special vertices, then by Claim~\ref{clm:noCH} (p.~\pageref{clm:noCH}), the other two special vertices have no remaining neighbor on $S$. By property ``$2$-linked'', this means that they do not share remaining neighbors, which means that they are (lone-)settled and there is no close problem, a contradiction.

So $S$ has at most one distant problem, caused by $u_3$ if so.
Let us first consider the case where remaining neighbors of $\{u_1,u_2\}$ are disjoint from the ones of $\{u_3,u_4\}$.
If $u_3$ causes a distant problem, then by Claim~\ref{clm:isol} (p.~\pageref{clm:isol}) $u_4$ is lone-settled, and thus $(u_1,u_2)$ must form a {\CDOne} or {\CDTwo} configuration to be unsettled. If $u_3$ does not cause a distant problem, at least one pair among $(u_1,u_2)$ and $(u_3,u_4)$ forms a {\CDOne} or {\CDTwo} configuration.
In all cases, this is configuration \ncf{\RuEight}.

Now, let us consider the case where $\{u_1,u_2\}$ share some remaining neighbors with $\{u_3,u_4\}$.
By property ``$0$-linked'', two $0$-linked special vertices cannot share remaining neighbors. If $u_2,u_4$ share remaining neighbors, by property ``$2$-linked'' they share exactly one and it belongs to a parallel $(u_1,u_3)$-path of $S$.
Therefore, by Claim~\ref{clm:noCH} (p.~\pageref{clm:noCH}), the remaining neighbors of $u_1,u_3$ are disjoint from $S$, thus disjoint, and this is configuration \ncf{\RuNine}.

\begin{figure}[htp]
		\tikzstyle{every node}=[thick,anchor=west,align=left]
		\begin{tikzpicture}
		  [grow via three points={one child at (0.5,-0.7) and
		  two children at (0.5,-0.7) and (0.5,-1.4)},
		  edge from parent path={(\tikzparentnode.south) |-(\tikzchildnode.west)}]
		  \node [] {$C_{4+}$}
		  		child { node [] {$2$ distant problems: impossible} 
				}
				child { node [] {$\leq 1$ distant problem}
					child { node [] {r.n. of $2$-linked disjoint: \Boxcf{\RuEight}}
					}
					child { node [] {Two $2$-linked share a r.n.: \Boxcf{\RuNine}}
					}
				}
			;
		\end{tikzpicture}
		\caption{Close lemma: tree of $C_{4+}$ cases}
		\label{C4tree}
\end{figure}
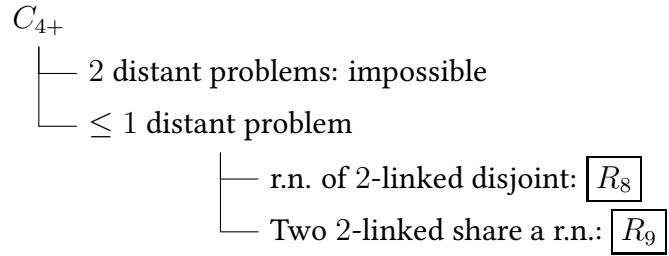

Let us now deal with \textbf{the $K_4$ cases}:
see Figure~\ref{K4tree} for the tree of cases.


We first examine the cases in which there are only two vertices involved in a close problem: we assume w.l.o.g. that these two vertices are $u_1,u_2$, they share a remaining neighbor $v$, and $u_3,u_4$ are either settled or cause distant problems. One may form a {\CN} pattern disjoint from $S$ and touching only patterns from settled vertices, but we treat it as settled. 

We first examine the case where $v\notin S$.
If $u_1,u_2$ share only one remaining neighbor, this is configuration \ncf{\RuOne}.
So now assume $u_1,u_2$ share another remaining neighbor $v'$.
If $v'\notin S$, then this is a {\CDTwo} configuration where $u_1,u_2$ form a {\CTTNA} pattern, as no other special vertex can cause a close problem by hypothesis, and by property A no path of $S$ can touch $v$ or $v'$. Thus $u_1,u_2$ are settled, a contradiction.
Otherwise, $v'\in S$ and this is configuration \ncf{\RuTwo}.

Now let us assume that $v\in S$. By property A, $v$ necessarily belongs to the path $u_3\sim u_4$. We can immediately see that if $u_1,u_2$ share another remaining neighbor $v'\notin S$, then this is a case that has already been treated, by swapping $v$ and $v'$.
Now $u_1,u_2$ may have another common remaining neighbor $v'\in S$, necessarily in the path $u_3\sim u_4$, or each of $u_1,u_2$ can have another remaining neighbor $v_1',v_2'$ respectively, both adjacent to $v$ to be unsettled. All these cases are treated as configuration \ncf{\RuThree}, except in one particular case: there is a vertex $w$ on $u_1\sim u_2$ that is a common neighbor of $u_1,u_2,u_3,u_4$, the special vertex $u_4$ is adjacent to $v_1'$, while $u_3$ is adjacent to $v_2'$, and in this case it is configuration \ncf{\JFour}.
This concludes the cases with two special vertices involved in a close problem.


Let us now examine the cases with $3$ vertices involved in a close problem, say $u_1,u_2,u_3$. We can assume w.l.o.g. that $u_1,u_3$ share a remaining neighbor $v_{13}$.
There is at most one distant problem, caused by $u_4$ if so; otherwise, $u_4$ is lone-settled by definition of distant problem.
The special vertices $u_1,u_2,u_3$ each have another remaining neighbor $v_1,v_2,v_3$ respectively, and $u_2$ has another remaining neighbor $v_2'$ (these vertices are not necessarily distinct). 

We first examine the case where $v_{13}\notin S$.

If $v_1=v_3$, we can denote this vertex by $v_{13}'$.
If $u_2$ does not have $v_{13}$ or $v_{13}'$ as a remaining neighbor, it is lone-settled by Claim~\ref{clm:isol} (p.~\pageref{clm:isol}), a contradiction with the hypothesis on $u_2$; so one of $v_{13},v_{13}'$ is also a remaining neighbor of $u_2$. By planarity, $v_{13},v_{13}'$ cannot both be remaining neighbors of $u_2$.
So let us say that $v_{13}$ is not a remaining neighbor of $u_2$, that the vertices $u_1,u_2,u_3$ have a common remaining neighbor $v_{123}$, and that $v_2 \neq v_{123}$.
Since $u_2$ is unsettled, $v_2,v_{123}$ are adjacent and this is configuration \ncf{\RuFour}.

Now let us take a look at the case where $v_1\neq v_3$. We distinguish between the case where $v_{13}$ is a shared remaining neighbor of $u_1,u_2,u_3$ or not.

\begin{itemize}
\item $v_{13}$ is also a remaining neighbor of $u_2$, and we call it $v_{123}$.
If w.l.o.g. $v_1=v_2$, then this case is equivalent to the previous one, by swapping $u_2$ and $u_3$; hence we assume that $v_1,v_2,v_3$ are pairwise distinct.
If one of $v_1,v_2,v_3$ is not adjacent to $v_{123}$, then its special vertex is lone-settled, a contradiction; so all three of $v_1,v_2,v_3$ are adjacent to $v$.
There are three {\CDOne} configurations 
on $(u_1,u_2),(u_1,u_3),(u_2,u_3)$, since by planarity and property A $v_1,v_2,v_3$ are disjoint from $S$. Each of these configurations forms a {\CDA} or {\CDB} pattern by property C, so the three paths $u_1\sim u_2, u_1\sim u_3, u_2\sim u_3$ have length $1$. Then $\{u_1,u_2,u_3\}$ is a $3$-cut in $G$ that separates two neighbors of $u_1$, a contradiction to the almost $4$-connectivity of $G$ w.r.t. the special vertices.

\item $v_{13}$ is not a remaining neighbor of $u_2$. Since by Claim~\ref{clm:isol} (p.~\pageref{clm:isol}) $u_2$ has a remaining neighbor in common with $u_1$ or $u_3$, we can assume w.l.o.g. that $v_1=v_2$ and we call this vertex $v_{12}$.

First, assume $v_2'=v_3$ and call it $v_{23}$.
Since $u_1,u_2,u_3$ are unsettled, the graph contains the edges $v_{13}v_{23}, v_{23}v_{12}, v_{12}v_{13}$.
Thus, the three {\CDOne} configurations on $(u_1,u_2)$, $(u_1,u_3)$ and $(u_2,u_3)$ are all {\CDA} patterns by property C. Thus the three paths between $u_1,u_2,u_3$ all have length $1$, which is again a contradiction to the almost $4$-connectivity of $G$ w.r.t. the special vertices.

Now assume $v_2'\neq v_3$. Due to $u_1,u_2,u_3$ being unsettled, the graph contains the edges $v_{13}v_{3}, v_{13}v_{12}, v_{12}v_2'$. By planarity, there is at most one edge among $\{v_{13}v_2', v_{12}v_3\}$, so we can assume w.l.o.g. that $v_{12}v_3$ is a non-edge and in this case we have configuration \ncf{\RuFive}.
\end{itemize}

We can now examine the case where $v_{13}\in S$.

By property A, $v_{13}$ necessarily belongs to $u_2\sim u_4$, thus $u_2$ cannot have it as a remaining neighbor by property A. The vertex $u_2$ must be involved in the close problem, so it is adjacent to $v_1$, $v_3$, or both.

First assume $u_2$ is adjacent to both $v_1,v_3$.
Since $u_1,u_2,u_3$ are unsettled, we need to have the edges $v_1v_3, v_1v_{13}, v_3v_{13}$, otherwise one of them is a settled {\CV} pattern.
However, in this case $\{v_{13},v_1,v_3,u_2,(u_1=u_3)\}$ form a $K_5$-minor, by contracting the path $u_1\sim u_3$ to a vertex, a contradiction with the planarity of $G$.

So now we assume w.l.o.g. that $u_2$ is adjacent to $v_3$ but not to $v_1$: we say that $v_3=v_2'$ and we call it $v_{23}$.
Since $u_2,u_3$ are unsettled, we need the edges $v_2v_{23}$ and $v_{13}v_{23}$.
But then this is a {\redirXThree} configuration, which contradicts property C of the subdivision.
This concludes the cases of $3$ special vertices involved in the close problem.


Let us finally examine the cases where all four special vertices are involved in one or two close problems. By definition, there are no distant problems in $S$.

For a special vertex to cause a close problem, it needs to share some of its remaining neighbors with another special vertex. Two cases may occur: either there are two independent close problems each involving two special vertices, or there is one close problem involving all four special vertices. We start with the first case: $u_1,u_2$ are involved in a close problem, as well as $u_3,u_4$, but the remaining neighbors of $\{u_1,u_2\}$ and those of $\{u_3,u_4\}$ are  disjoint. 
We examine all combinations of cases:
\begin{itemize}
\item $u_1,u_2$ share \textbf{exactly one} remaining neighbor $v_{12}$; $u_3,u_4$ share \textbf{exactly one} remaining neighbor $v_{34}$.
For all special vertices to be 
unsettled,
we assume that the other remaining neighbors $v_1,v_2$ of $u_1,u_2$
are adjacent to $v_{12}$, and the other remaining neighbors $v_3,v_4$ of $u_3,u_4$ are adjacent to $v_{34}$. If $v_{12},v_{34}\notin S$,
then none of $v_1,v_2,v_3,v_4$ can belong to $S$, otherwise by property A the graph contains a redirection configuration {\redirXThree}
, forbidden by property C. So this is configuration \ncf{\RuSix}.
By planarity, if one of $v_{12},v_{34}$ belongs to $S$, then the other one does too. If they both belong to $S$, this is configuration \ncf{\RuThree}. 

\item $u_1,u_2$ share \textbf{two} remaining neighbors $v_{12},v_{12}'$; $u_3,u_4$ share \textbf{exactly one} remaining neighbor $v_{34}$.
Since $u_3,u_4$ are unsettled, their other remaining neighbors $v_3,v_4$ are adjacent to $v_{34}$.
If none of $v_{12},v_{12}'$ belongs to $S$, then neither do $v_3,v_4,v_{34}$ by planarity. This is again configuration \ncf{\RuSix}. Now if say $v_{12}$ belongs to $S$, it belongs to $u_3\sim u_4$ by property A, and then by planarity $v_{34}$  belongs to the path $u_1\sim u_2$. Thus, by planarity, $v_{12}'$ must also belong to $S$. This is once more configuration \ncf{\RuThree}.

\item $u_1,u_2$ share \textbf{two} remaining neighbors $v_{12},v_{12}'$; $u_3,u_4$ share \textbf{two} remaining neighbors $v_{34},v_{34}'$. Using the same argument, either none of $v_{12},v_{12}',v_{34},v_{34}'$ belong to $S$, or they all do. In the former case this is again configuration \ncf{\RuSix}, in the latter this is configuration \ncf{\RuThree}.
\end{itemize}

This concludes the case with two independent close problems. Let us now assume that all four special vertices are involved in the same close problem. Let us decompose according to the case ($P_1$ or $P_2$) of the remaining neighbors of $u_1,u_2$.

\begin{itemize}
\item $u_1,u_2$ share \textbf{two} remaining neighbors $v_{12},v_{12}'$. We note that if at least one of $v_{12},v_{12}'$ belongs to $S$, then it is impossible by planarity and property A to have both $u_3,u_4$ involved. So $v_{12},v_{12}'\notin S$. Assume w.l.o.g. that $v_{12}$ is also a remaining neighbor of $u_3$, and call it $v_{123}$. If $u_3,u_4$ share a remaining neighbor $v_{34}$ (different from $v_{123}$ by planarity), then by property A if $v_{34}\in S$ it can only belong to the path $u_1\sim u_2$, but it is impossible in this case by planarity. Thus $v_{34}\notin S$, and by planarity $v_{34},v_{123}$ are not adjacent.
But then $u_3$ is a settled {\CV} pattern, a contradiction.
So $u_3,u_4$ do not share a remaining neighbor. For $u_4$ to be involved, it must have (by planarity) $v_{12}'$ as a remaining neighbor. This is configuration \ncf{\RuSeven}.

\item $u_1,u_2$ share \textbf{exactly one} remaining neighbor $v_{12}$, and each have another remaining neighbor $v_1,v_2$, both adjacent to $v_{12}$. We first assume that $v_{12}\notin S$. If $u_3$ has $v_1$ as a remaining neighbor, then by planarity $v_1,v_2,v_{12}$ must belong to the region of the graph delimited by the paths $u_1\sim u_2, u_2\sim u_3, u_1\sim u_3$, and none of $v_1,v_2,v_{12}$ can belong to $S$ (by property A and planarity). It is then impossible for $u_4$ to be involved with the close problem. So necessarily $u_3$ has $v_{12}$ as a remaining neighbor, we call it $v_{123}$. By the same argument it is impossible for $u_4$ to be involved without making $u_3$ a {\CV} pattern, thus settled, a contradiction. 
So finally assume that $v_{12}$ belongs to $S$: by property A it belongs to $u_3\sim u_4$. We can assume w.l.o.g. that $u_3$ has $v_1$ as a remaining neighbor (it cannot have $v_{12}$ by property A). By an argument used above, $u_3$ cannot share a remaining neighbor with $u_4$ without being a {\CV} pattern, so $u_4$ has $v_2$ as a remaining neighbor. This is 
again configuration \ncf{\RuThree}.
\end{itemize}

This concludes the proof.
\end{proof}

\begin{figure}[ht]
		\tikzstyle{every node}=[thick,anchor=west,align=left]
		\begin{tikzpicture}
		  [grow via three points={one child at (0.5,-0.7) and
		  two children at (0.5,-0.7) and (0.5,-1.4)},
		  edge from parent path={(\tikzparentnode.south) |-(\tikzchildnode.west)}]
		  \node [] {$K_4$}
			child { node [] {\\\\$2$ vertices involved in a close problem:\\ 
			$u_1, u_2$ share a r.n. $v$}
						child [missing] {}
				child { node [] {$v\notin S$}
					child { node [] {
						$u_1,u_2$ share only one r.n.: \Boxcf{\RuOne}}}
					child { node [] {$u_1, u_2$ share a r.n. $v'$}
								child { node [] {$v'\notin S$: impossible}
								}
								child { node [] {$v'\in S$: \Boxcf{\RuTwo}}
								}
						}
				}
				child [missing] {}
				child [missing] {}
				child [missing] {}
				child { node [] {$v\in S$}
				    child { node [] {$u_1, u_2$ share a r.n. $v'$; $v'\notin S$: see case $v\notin S$}
				    }
					child { node [] {$u_1, u_2$ share a r.n. $v'$; $v'\in S$: \Boxcf{\RuThree}}
					}
					child { node [] {$u_1$ r.n. $v_1$; $u_2$ r.n. $v_2$; $v_1'\neq v_2'$: \fbox{\Ncf{\RuThree}, \Ncf{\JFour}}}
					}
				}
			}
			child [missing] {}
			child [missing] {}
			child [missing] {}
			child [missing] {}
			child [missing] {}
			child [missing] {}
			child [missing] {}
			child [missing] {}
			child [missing] {}
			child { node [] {$3$ vertices involved in a close problem: $u_1, u_2, u_3$;\\
							 $u_1,u_3$ share a r.n. $v_{13}$}
								child { node [] {$v_{13}\notin S$}
									child { node [] {$v_1 = v_3$: \Boxcf{\RuFour}}
									}
									child { node [] {$v_1 \neq v_3$: \Boxcf{\RuFive}}
									}
								}
								child [missing] {}
								child [missing] {}
								child { node [] {$v_{13}\in S$: Impossible}}
			}
			child [missing] {}
			child [missing] {}
			child [missing] {}
			child { node [] {$4$ vertices involved\\ in close problems}
				child [missing] {}
				child { node [] {Two independent close problems\\ each involving two vertices:\\$u_1,u_2$ share one r.n. $v_{12}$;\\$u_3,u_4$ share one r.n. $v_{34}$}
				child [missing] {}
				child { node [] {$v_{12},v_{34}\notin S$: \Boxcf{\RuSix}}
						}
				child { node [] {$v_{12},v_{34}\in S$: \Boxcf{\RuThree}}
						}
				}
				child [missing] {}
				child [missing] {}
				child [missing] {}
				child { node [] {One close problem\\ involving all four vertices}
					child { node [] {$P_2$: $u_1,u_2$ share two r.n. $v_{12},v_{12}'$: \Boxcf{\RuSeven}}
					}
					child { node [] {$P_1$: $u_1,u_2$ share one r.n. $v_{12}$: \Boxcf{\RuThree}}
					}
				}
			};
		\end{tikzpicture}
		\caption{Close lemma: tree of $K_4$ cases}
		\label{K4tree}
\end{figure}
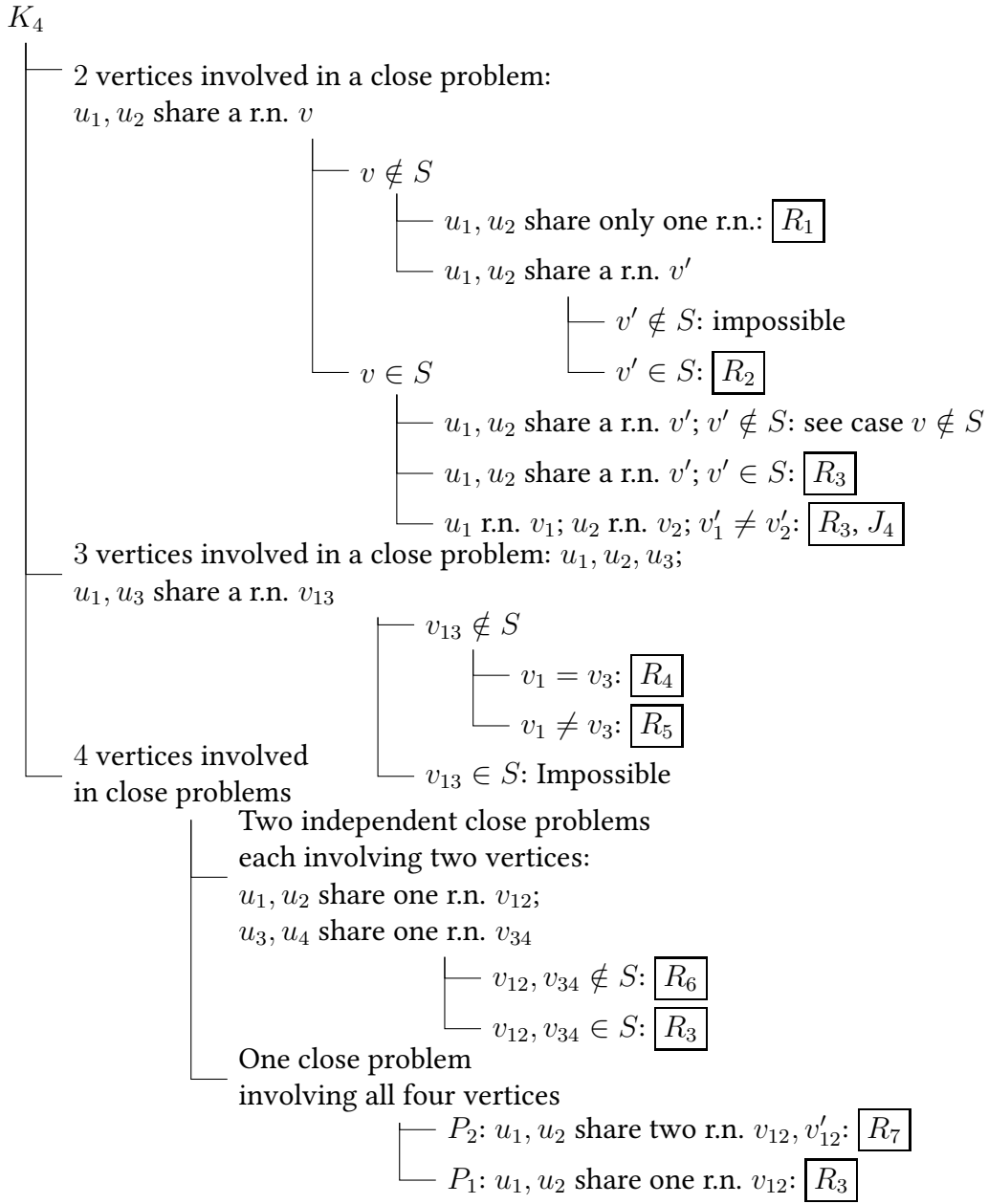
\clearpage

\pagebreak

Let us conclude this 
{\chapsec}
by proving Lemma~\ref{lem:confred} (p.~\pageref{lem:confred}) and using it to prove Lemma~\ref{lem:cii} (p.~\pageref{lem:cii}).

\begin{proof}[Proof of Lemma~\ref{lem:confred} (p.~\pageref{lem:confred})]
Let $G$ be a planar graph with a $(C_{II})$ configuration w.r.t. a $4$-family $U$, that admits a strong $\mathcal{K}$-subdivision rooted on $U$. We prove that $G$ contains a subdivision composite configuration made up of a semi-subdivision $S$ rooted on $U$ and a compatible mapping w.r.t. $S$.

By Lemma~\ref{lem:dist1} (\textit{Distant lemma}, p.~\pageref{lem:dist1}), if $G$ has at least $3$ distant problems w.r.t. its strong subdivision $S$, then it contains a configuration among $\{$\ncf{\DOne}, \ncf{\DTwo}, \ncf{\DThree}, \ncf{\DFour}$\}$.
Each of these configurations is defined 
along with a semi-subdivision and a compatible mapping.

So let us assume $G$ has at most $2$ distant problems and no close problem w.r.t. $S$. By Lemma~\ref{lem:dist2} (\textit{Semi-distant lemma}, p.~\pageref{lem:dist2}), $G$ admits a semi-distant configurations among $\{$\ncf{\JOne}, \ncf{\JTwo}, \ncf{\JThree}, \ncf{\JFour}, \ncf{\JFive}, \ncf{\JSix}$\}$, and 
we provided
a semi-subdivision and a compatible mapping for each of these configurations. If $G$ has at most $2$ distant problems and some close problems w.r.t. $S$, then by Lemma~\ref{lem:close} (\textit{Close lemma}, p.~\pageref{lem:close}), it contains a configuration among
$\{$%
\ncf{\RuOne},
\ncf{\RuTwo},
\ncf{\RuThree},
\ncf{\RuFour},
\ncf{\RuFive},
\ncf{\RuSix},
\ncf{\RuSeven},
\ncf{\RuEight},
\ncf{\RuNine},
\ncf{\JFour}%
$\}$, 
again associated with a semi-subdivision and a compatible mapping for each.
\end{proof}

\begin{proof}[Proof of Lemma~\ref{lem:cii} (p.~\pageref{lem:cii})]
Let $G$ be a minimum counterexample. By Lemma~\ref{lem:ci} (p.~\pageref{lem:ci}), it does not contain a configuration $(C_I)$.
%
%

Assume that $G$ contains a $(C_{II})$ configuration w.r.t. a $4$-family $U$.
By Claim~\ref{clm:redir3r} (p.~\pageref{clm:redir3r}), $G$ admits a strong $\mathcal{K}$-subdivision $S$ rooted on $U$.
By Lemma~\ref{lem:confred} (p.~\pageref{lem:confred}), $G$ contains a semi-subdivision $S'$ rooted on $U$ and a compatible mapping w.r.t. $S'$, which is a contradiction by Lemma~\ref{lem:safety_cii} (p.~\pageref{lem:safety_cii}).
\end{proof}

Lemma~\ref{lem:mce} (p.~\pageref{lem:mce}) ensues by combining Lemmas~\ref{lem:ci} (p.~\pageref{lem:ci}) and \ref{lem:cii} (p.~\pageref{lem:cii}).

\vfill
\pagebreak

\section{There is no minimal counterexample}\label{sec:nomce}


We complete the proof of Theorem~\ref{thm:gallai} with the following lemma:

\begin{lem}\label{lem:nomce}
Every connected planar graph on at least $3$ vertices contains a configuration $(C_I)$ or $(C_{II})$.
\end{lem}

Lemmas~\ref{lem:mce} (p.~\pageref{lem:mce}) and~\ref{lem:nomce} together guarantee that every connected planar graph other than $K_3$ and $K_5^-$ admits a good path decomposition -- it is trivial to verify for connected planar graphs on at most $2$ vertices. To prove Lemma~\ref{lem:nomce}, we first need some definitions and structural observations on planar graphs.



Given a planar graph $G$, a \emph{$1$-contraction} of $G$ is an induced subgraph $H$ of $G$ on at least $2$ vertices together with a vertex $u_1 \in V(H)$ such that all vertices have the same degree in $G$ and $H$, except possibly for $u_1$. A graph $H$ on at least $3$ vertices is a \emph{$2$-contraction} of $G$ if there exists an edge $u_1u_2 \in E(H)$ such that $H \setminus u_1u_2$ is a subgraph of $G$ and every vertex $v \not\in \{u_1,u_2\}$ satisfies $d_H(v)=d_G(v)$. Additionally, there exists a $(u_1,u_2)$-path in $G$ with all internal vertices in $V(G)\setminus V(H)$.

The \emph{damaged vertices} of a $p$-contraction ($p \in \{1,2\}$) are the vertices $\{u_1\}$, $\{u_1,u_2\}$ above. Note that any induced subgraph of $G$ (on at least $2$ vertices) can be turned into a $1$-contraction by selecting an arbitrary vertex as $u_1$. Note that a $1$-contraction with damaged vertex $u_1$ can be turned into a $2$-contraction by selecting an arbitrary neighbour of $u_1$ in $H$, unless the vertex $u_1$ has no neighbour in $H$. Note also that any $p$-contraction ($p \in \{1,2\}$) of $G$ is a minor of $G$ hence is planar.

A $2$-contraction $H'$ of $G$ is \emph{smaller} than a $2$-contraction $H$ of $G$ with damaged vertices $\{u_1,u_2\}$ if $V(H')\subsetneq V(H)$ and each of $u_1$ and $u_2$ either does not belong to $V(H')$ or is a damaged vertex of $H'$. We may simply refer to a \emph{smaller $2$-contraction than $H$} if the damaged vertices of $H$ are clear from context. A $2$-contraction $H$ of $G$ is \emph{minimal} if $G$ admits no smaller $2$-contraction.

\begin{claim}\label{cl:struct1}
Let $G$ be a connected planar graph. Any minimal $2$-contraction of $G$ either contains a non-damaged vertex of degree at most $4$ or is $3$-connected.
\end{claim}

\begin{proof}
Assume for a contradiction that $H$ with damaged vertices $u_1,u_2$ is a counterexample to the statement: $H$ is minimal, not 3-connected and every vertex in $V(H)\setminus \{u_1,u_2\}$ has degree at least $5$ in $H$ (hence in $G$). By definition of a  $2$-contraction, $V(H)\setminus \{u_1,u_2\}$ is non-empty. We note that it suffices to exhibit a $2$-contraction $H'$ of $G$ smaller than $H$.

Since $H$ is minimal, it is connected. Note that $H$ is in fact $2$-connected. Otherwise, let $x$ be a cut-vertex in $H$. There is a connected component $C$ of $H \setminus \{x\}$ containing none of $\{u_1,u_2\}$. The graph $G[C \cup \{x\}]$ is a $1$-contraction of $G$ with damaged vertex $x$. Note that $C$ contains at least $5$ vertices, as $C$ is non-empty and every vertex in $C$ has degree at least $5$ in $G$. We select an arbitrary neighbour $y$ of $x$ in $C$, and note that $G[C \cup \{x\}]$ is a $2$-contraction of $G$ with damaged vertices $x,y$ and $u_1,u_2 \not\in C$. This yields a smaller $2$-contraction than $H$, a contradiction.

Therefore, $H$ is 2-connected but not 3-connected. Let $x_1,x_2$ be a vertex cut of $H$. Since $u_1u_2 \in E(H)$, $u_1$ and $u_2$ do not belong to different connected components of $H \setminus \{x_1,x_2\}$. 
Let $C$ be a connected component $H \setminus \{x_1,x_2\}$ that contains no $u_i$. Since $H$ is $2$-connected, there is a path between $x_1$ and $x_2$ whose internal vertices belong to $V(H) \setminus C$. We obtain a $2$-contraction of $G$ with damaged vertices $\{x_1,x_2\}$ where each $u_i$ either does not belong to it or is a damaged vertex, hence a contradiction.
\end{proof}

\begin{claim}\label{cl:struct3}
Let $H$ be a minimal $2$-contraction of a planar graph $G$. Either $H$ contains a non-damaged vertex of degree at most $4$, or there are four non-damaged vertices $\{v_1,v_2,v_3,v_4\}$ of degree $5$ with respect to which $H$ is almost $4$-connected.
\end{claim}

\begin{proof}
Assume that $H$ does not contain any non-damaged vertex of degree $4$ or less, and let $u_1,u_2$ be the damaged vertices of $H$.

We first assume that $H$ is $4$-connected. Then it suffices to argue that there are four vertices of degree $5$ in $V(H)\setminus\{u_1,u_2\}$. Since $H$ is 4-connected, we have $d_H(u_1), d_H(u_2) \geq 4$. By Euler's formula, we have $\sum_{x \in V(H)}(d(x)-6)\leq -12$, hence $\sum_{x \in V(H)\setminus\{u_1,u_2\}}(d(x)-6) \leq -8$. Since every non-damaged vertex of $H$ has degree at least $5$, the conclusion follows.

Therefore, we may assume the graph $H$ is not $4$-connected. Consider an embedding of $H$ where $u_1$ and $u_2$ lie on the outer-face. Since $u_1u_2 \in E(H)$ and $H$ is planar since it is a $2$-contraction of $G$, such an embedding exists. For any vertex cut $X$ of $H$ that has size $3$, let $p$ be the number of connected components in $H \setminus X$. Consider the minor of $H$ obtained by contracting each connected component into a single vertex. Since $H$ is $3$-connected (by Claim~\ref{cl:struct1}), every resulting vertex is adjacent to all three vertices in $X$, which yields a $K_{3,p}$-minor where $p$ is the number of connected components. Since $H$ is planar, there is no $K_{3,3}$-minor, hence $p=2$.
At most one of the two connected components of $H \setminus X$ contains damaged vertices. If one of them contains damaged vertices, we let $I(X)$ be the connected component of $H \setminus X$ that does not contain damaged vertices and $E(X)$ be the connected component of $H \setminus X$ that contains a damaged vertex. If neither component contains damaged vertices, then $u_1, u_2 \in X$. Since $H$ is $3$-connected and $u_1,u_2$ belong to the outer-face, there is exactly one connected component that contains no vertex of the outer-face; We let $I(X)$ be that component. Let $E(X)$ be the other component. Observe that $E(X)$ contains at least one vertex of the outer-face.

Among all vertex cuts of $H$ that have size $3$, we select a vertex cut $\{x_1,x_2,x_3\}$ which minimizes $|I(\{x_1,x_2,x_3\})|$. We first argue that $C=I(\{x_1,x_2,x_3\})$ contains at least four vertices of degree $5$. Indeed, since $\{x_1,x_2,x_3\}$ is minimal, every $x_i$ has at least $2$ neighbors in $C$.
Hence by Euler's formula on the graph $C$, we have $\sum_{x \in C}(d(x)-6) + 3\times2 \leq -12$. Since every vertex in $C$ has degree at least $5$ in $H$, there are four non-damaged vertices $\{v_1,v_2,v_3,v_4\}$ in $C$ that have degree $5$ in $H$.

It remains to argue that $H$ is almost $4$-connected with respect to them.
We show that no $3$-cut of $H$ separates two vertices of $X\cup C$. Observe that this proves the two properties of almost $4$-connectivity.
Assume for a contradiction that there is a vertex cut $Y=\{y_1,y_2,y_3\}$ such that two vertices $z_1,z_2\in X\cup C$ are in different connected components of $H \setminus Y$. Without loss of generality, consider $z_1 \in I(Y)$ and $z_2 \in E(Y)$.

Note that $H[X\cup C]$ is connected and $z_1,z_2 \in X\cup C$. However, $z_1$ and $z_2$ are in different connected components of $H\setminus Y$. Since $z_1,z_2$ cannot be separated by a vertex cut contained in $X\cup E(X)$, at least one of $\{y_1,y_2,y_3\}$, say $y_1$, belongs to $C$.
Since $\{x_1,x_2,x_3\}$ minimizes $|I(\{x_1,x_2,x_3\})|$ over all vertex cuts of $H$ of size $3$, at least one of $\{y_1,y_2,y_3\}$, say $y_3$, belongs to $E( \{x_1,x_2,x_3\})$.

We consider two cases depending on the cardinal of $\{x_1,x_2,x_3\}\cap I(Y)$.

\begin{itemize}
    \item Assume that $I(Y)$ contains exactly one vertex in $\{x_1,x_2,x_3\}$, say $x_1$.

    If $y_2\in E(X)$, then we claim that $\{x_1,y_1\}$ is a $2$-cut that separates $z_1$ from $z_2$, a contradiction with the $3$-connectivity of $H$. To prove it, we note that in $Y \cup I(Y)$, there is a path between $z_1$ and each of $\{y_1,y_2,y_3\}$ whose internal vertices belong to $I(Y)$. However, since $\{x_1,x_2,x_3\}$ is a vertex cut of $H$ and neither $x_2$ nor $x_3$ belongs to $Y \cup I(Y)$, every path between $z_1$ and $y_2$ whose internal vertices belong to $I(Y)$ involves the vertex $x_1$. Therefore, $\{x_1,y_1\}$ separates $z_1$ and $y_2$, hence the conclusion.

	So $y_2\in X\cup C$, and we now claim that $\{x_1,y_1,y_2\}$ is a vertex cut of $H$, by the same argument: each path between $z_1$ and $y_3$ with internal vertices in $I(Y)$ involves $x_1$. Observe that it contradicts the choice of $\{x_1,x_2,x_3\}$.
    \item Assume from now on that $I(Y)$ contains both $x_2$ and $x_3$, while $x_1 \in E(Y)$.
    As above, we argue that $\{y_1,x_2,x_3\}$ is a vertex cut of $H$, which contradicts the choice of $\{x_1,x_2,x_3\}$. For completeness, we include the adapted proof. We note that in $\{x_1,x_2,x_3\} \cup I(\{x_1,x_2,x_3\})$, there is a path between $z_1$ and each of $\{x_1,x_2,x_3\}$ whose internal vertices belong to $I(\{x_1,x_2,x_3\})$. However, since $Y$ is a vertex cut of $H$ and neither $y_2$ nor $y_3$ belongs to $\{x_1,x_2,x_3\} \cup I(\{x_1,x_2,x_3\})$, every path between $z_1$ and $x_1$ whose internal vertices belong to $I(\{x_1,x_2,x_3\})$ involves the vertex $y_1$. Therefore, $\{y_1,x_2,x_3\}$ separates $z_1$ and $x_1$, hence the conclusion.
\end{itemize}%
\end{proof}

We can now use Claim~\ref{cl:struct3} to obtain Lemma~\ref{lem:nomce} (p.~\pageref{lem:nomce}).

\begin{proof}[Proof of Lemma~\ref{lem:nomce}]
Let $G$ be a non-empty connected planar graph which contains no $(C_I)$ configuration. 
Let $u_1$ be the only vertex with degree at most $4$ in $G$ if any, and an arbitrary vertex otherwise. Note that since $G$ is connected and contains at least $3$ vertices, the graph $G$ with damaged vertex $u_1$ is a $1$-contraction of $G$. Furthermore, the vertex $u_1$ has at least one neighbour $u_2$, and the graph $G$ with damaged vertices $u_1$ and $u_2$ is a $2$-contraction of $G$. Let $H$ be a minimal $2$-contraction of $G$ that is either precisely $G$ with damaged vertices $u_1,u_2$ or smaller than it.

Every non-damaged vertex in $H$ has the same degree in $H$ and in $G$, and $u_1$ is not a non-damaged vertex in $H$. Therefore, Claim~\ref{cl:struct3} applied to $H$ yields that 
$H$ is almost $4$-connected w.r.t. a $4$-family $\{v_1,v_2,v_3,v_4\}$. Hence $G$ is almost $4$-connected w.r.t. $\{v_1,v_2,v_3,v_4\}$, as desired.
\end{proof}


\printbibliography

\end{document}